\documentclass[11pt]{amsart}
\usepackage{amsmath,amssymb,amsthm,enumitem,mathrsfs,mathtools}
\usepackage[hmargin=25mm,vmargin=30mm]{geometry}
\usepackage{tikz}
\usetikzlibrary{cd}

\numberwithin{equation}{section}

\theoremstyle{plain}
\newtheorem{thm}{Theorem}[section]
\newtheorem*{thm*}{Theorem}
\newtheorem{lem}[thm]{Lemma}
\newtheorem{prop}[thm]{Proposition}

\theoremstyle{remark}
\newtheorem{rem}[thm]{Remark}

\usepackage[hidelinks]{hyperref}
\newcommand{\mat}[4]{\begin{pmatrix} #1 & #2 \\ #3 & #4 \end{pmatrix}}
\newcommand{\smat}[4]{\left(\begin{smallmatrix} #1 & #2 \\ #3 & #4 \end{smallmatrix}\right)}

\newcommand\Ad{\operatorname{Ad}}
\newcommand\AJ{\operatorname{AJ}}
\newcommand\CH{\operatorname{CH}}
\newcommand\diag{\operatorname{diag}}
\newcommand\Fil{\operatorname{Fil}}
\newcommand\Hom{\operatorname{Hom}}

\renewcommand\Im{\operatorname{Im}}
\newcommand\Ind{\operatorname{Ind}}
\newcommand\N{\operatorname{N}}
\newcommand\ord{\operatorname{ord}}
\newcommand\pr{\operatorname{pr}}
\renewcommand\Re{\operatorname{Re}}

\newcommand\Res{\operatorname{Res}}
\newcommand\supp{\operatorname{supp}}
\newcommand\tr{\operatorname{tr}}
\newcommand\vol{\operatorname{vol}}

\newcommand\dR{\mathrm{dR}}
\newcommand\fin{\mathrm{fin}}
\newcommand\GL{\mathrm{GL}}
\newcommand\GO{\mathrm{GO}}
\newcommand\GSO{\mathrm{GSO}}
\newcommand\GSp{\mathrm{GSp}}
\newcommand\GU{\mathrm{GU}}
\newcommand\HDS{\mathrm{HDS}}
\newcommand\id{\mathrm{id}}
\newcommand\M{\mathrm{M}}
\renewcommand\O{\mathrm{O}}
\newcommand\SL{\mathrm{SL}}
\newcommand\SO{\mathrm{SO}}
\newcommand\Sp{\mathrm{Sp}}
\newcommand\St{\mathrm{St}}
\newcommand\std{\mathrm{std}}
\newcommand\Sym{\mathrm{Sym}}
\newcommand\Tam{\mathrm{Tam}}

\newcommand\U{\mathrm{U}}

\newcommand\A{\mathbb{A}}
\newcommand\C{\mathbb{C}}
\newcommand\Q{\mathbb{Q}}
\newcommand\R{\mathbb{R}}
\newcommand\Z{\mathbb{Z}}

\newcommand\bG{\mathbf{G}}
\newcommand\bH{\mathbf{H}}
\newcommand\bI{\mathbf{I}}

\newcommand\bK{\mathbf{K}}

\newcommand\bP{\mathbf{P}}

\newcommand\bT{\mathbf{T}}

\newcommand\bV{\mathbf{V}}
\newcommand\bW{\mathbf{W}}

\newcommand\bff{\mathbf{f}}
\newcommand\bi{\mathbf{i}}
\newcommand\bj{\mathbf{j}}
\newcommand\bmm{\mathbf{m}}
\newcommand\bv{\mathbf{v}}

\newcommand\cE{\mathcal{E}}
\newcommand\cF{\mathcal{F}}

\newcommand\cH{\mathcal{H}}
\newcommand\cI{\mathcal{I}}

\newcommand\cK{\mathcal{K}}
\newcommand\cL{\mathcal{L}}
\newcommand\cM{\mathcal{M}}
\newcommand\cN{\mathcal{N}}
\newcommand\cO{\mathcal{O}}
\newcommand\cP{\mathcal{P}}
\newcommand\cQ{\mathcal{Q}}

\newcommand\cS{\mathcal{S}}

\newcommand\cU{\mathcal{U}}

\newcommand\cW{\mathcal{W}}

\newcommand\fH{\mathfrak{H}}

\newcommand\ff{\mathfrak{f}}
\newcommand\fg{\mathfrak{g}}
\newcommand\fk{\mathfrak{k}}
\newcommand\fl{\mathfrak{l}}
\newcommand\fo{\mathfrak{o}}
\newcommand\fp{\mathfrak{p}}
\newcommand\fq{\mathfrak{q}}
\newcommand\ft{\mathfrak{t}}
\newcommand\fu{\mathfrak{u}}

\newcommand\Qbar{\overline{\Q}}

\title{Cycles for Rankin-Selberg $L$-functions, I: automorphic periods}
\author{Atsushi Ichino}
\address{Department of Mathematics, Kyoto University, Kitashirakawa Oiwake-cho, Sakyo-ku, Kyoto 606-8502, Japan}
\email{ichino@math.kyoto-u.ac.jp}
\author{Kartik Prasanna}
\address{Department of Mathematics, University of Michigan, 2074 East Hall, 530 Church Street, Ann Arbor, MI 48109-1043, USA}
\email{kartikp@umich.edu}

\begin{document}

\begin{abstract}
In this paper, we establish an explicit formula for automorphic periods which will be used in a sequel to study special values of $p$-adic Rankin-Selberg $L$-functions. 
Our motivation is to extend the Bertolini-Darmon-Prasanna formula to the case where the archimedean local sign is opposite to that in the original setting. 
To this end, we prove a formula for the $\GU(1,1)$-period of a theta lift from $\GSO(2)$ to $\GSp_4$, which is adapted to $p$-adic interpolation. 
We also introduce a $p$-depletion Hecke operator for this theta lift, which will play a crucial role in relating the automorphic period to the image of the $p$-adic Abel-Jacobi map.
\end{abstract}

\maketitle
\tableofcontents

\section{Introduction}

\subsection{Motivation}

The purpose of this paper is to establish an explicit formula for automorphic periods which will serve as a key input in a sequel \cite{cycles2}, where we study special values of $p$-adic Rankin-Selberg $L$-functions. We begin by explaining the arithmetic motivation.

Let $f$ be a normalized cuspidal Hecke eigenform of even weight $k$.
Let $K$ be an imaginary quadratic field, and let $\chi$ be an algebraic Hecke character of $K$ of even weight $l$.
For simplicity, we assume that $f$ has trivial central character and $\chi$ is trivial on $\A_\Q^\times$.
We consider the degree $4$ Rankin-Selberg $L$-function $L(s,f,\chi)$, normalized by the functional equation
\[
 L(s,f,\chi) = \varepsilon(s,f,\chi) L(1-s,f,\chi).
\]
Then a guiding problem is to understand the arithmetic nature of $L(s,f,\chi)$ at the center of symmetry $s = \frac{1}{2}$.
The discussion naturally divides according to the global sign
\[
 \varepsilon(f,\chi) = \varepsilon(\tfrac{1}{2},f,\chi),
\]
as illustrated by the formulas of Gross-Zagier \cite{gz} and Waldspurger \cite{wald85}.

First assume that $\varepsilon(f,\chi) = +1$.
In this case, the central value $L(\frac{1}{2},f,\chi)$ should be generically nonzero and can be described in terms of automorphic periods.
Under this assumption, there exists a unique quaternion algebra $B$ over $\Q$ such that
\[
 \varepsilon_v(B) = \varepsilon_v(f, \chi)
\]
for all places $v$ of $\Q$.
Here $\varepsilon_v(B)$ is given by
\[
 \varepsilon_v(B) = 
 \begin{cases}
  +1 & \text{if $B_v$ is split;} \\
  -1 & \text{if $B_v$ is ramified}
 \end{cases}
\]
and $\varepsilon_v(f, \chi)$ denotes the local sign, so that 
\[
 \varepsilon(f, \chi) = \prod_v \varepsilon_v(f, \chi).
\]
For this $B$, $K$ embeds into $B$ and $f$ admits a Jacquet-Langlands transfer $f^B$ (viewed as an automorphic form on $B^\times(\A_\Q)$).
Then Waldspurger's formula says that 
\[
 L(\tfrac{1}{2},f,\chi) \doteq \cP(f^B,\chi)^2,
\]
where $\doteq$ denotes equality up to explicit nonzero factors and $\cP(f^B,\chi)$ is the toric period given by
\[
 \cP(f^B,\chi) = \int_{\A_\Q^\times K^\times \backslash \A_K^\times} f^B(x) \chi(x) \, dx.
\]

Next assume that $\varepsilon(f,\chi) = -1$.
In this case, the functional equation forces $L(\frac{1}{2},f,\chi) = 0$ and the Bloch-Beilinson conjecture suggests that this vanishing should be reflected in homologically trivial cycles on a variety realizing the rank $4$ motive associated to $L(s,f,\chi)$.
The geometric nature of these cycles is sensitive to the archimedean local sign
\[
 \varepsilon_\infty(f,\chi) = 
 \begin{cases}
  +1 & \text{if $k \le l$}; \\
  -1 & \text{if $k > l$}.
 \end{cases}
\]
We first recall the case $\varepsilon_\infty(f,\chi) = -1$.
In the simplest case $k=2$ and $l=0$, the Gross-Zagier formula says that 
\[
 L'(\tfrac{1}{2}, f, \chi) \doteq \langle P_{f,\chi}, P_{f,\chi} \rangle_{\mathrm{NT}}
\]
under the Heegner condition, where $P_{f,\chi}$ is the $(f,\chi)$-component of a Heegner divisor on a modular curve and $\langle \cdot, \cdot \rangle_{\mathrm{NT}}$ is the N\'eron-Tate height pairing.
Note that the Heegner condition implies that $\varepsilon_v(f,\chi) = +1$ for all finite places $v$ of $\Q$.
This formula has been extended to higher weight cases, and more generally to the setting of totally real number fields, by Zhang \cite{zhang97, zhang01} and Yuan-Zhang-Zhang \cite{yzz}.
Moreover, its $p$-adic counterparts have been developed by Perrin-Riou \cite{perrin-riou}, Nekov\'{a}\v{r} \cite{nekovar}, Bertolini-Darmon-Prasanna \cite{bdp}, and Liu-Zhang-Zhang \cite{lzz}.

By contrast, the case $\varepsilon_\infty(f,\chi) = +1$ remains largely unexplored in the context of Gross-Zagier type formulas. 
This is the case addressed in this paper and its sequel.

\subsection{The BDP setting: $\varepsilon(f,\chi) = -1$, $\varepsilon_\infty(f,\chi) = -1$}

Our starting point is the Bertolini-Darmon-Prasanna (BDP) formula \cite{bdp}, which we recall in more detail.
For simplicity, we assume that $k=2$, $l=0$, and $\varepsilon_v(f,\chi) = +1$ for all $v < \infty$.
Fix an odd rational prime $p$ which is split in $K$ and does not divide the level of $f$.
Let $\chi_{2n+2}$ vary in a $p$-adic family of algebraic Hecke characters of $K$, where $\chi_{2n+2}$ has weight $2n+2$ and specializes to $\chi$ at $n=-1$.
Let $\cL_p(f,\cdot)$ be the $p$-adic $L$-function interpolating the algebraic parts of the central values $L(\tfrac{1}{2},f,\chi_{2n+2})$ for $n \ge 0$.
Then the BDP formula says that 
\[
 \cL_p(f,\chi) \doteq (\AJ_p(\Xi_\chi)(\omega_f))^2,
\]
where
\begin{itemize}
\item $\Xi$ is a cycle on a modular curve $X$ coming from the embedding $K^\times \hookrightarrow \GL_2(\Q)$;
\item $\Xi_\chi \in \CH^1(X)_0$ is the $\chi$-component of $\Xi$;
\item $\AJ_p : \CH^1(X)_0 \rightarrow \Fil^1 H^1_{\dR}(X/\Q_p)^\vee$ is the $p$-adic Abel-Jacobi map;
\item $\omega_f \in \Fil^1 H^1_{\dR}(X/\Q_p)$ is the differential form associated to $f$.
\end{itemize}
We briefly recall the idea of the proof since it also explains the difficulty of the extension considered below.
\begin{itemize}
\item 
The first important input is Waldspurger's formula
\[
 L(\tfrac{1}{2},f,\chi_{2n+2}) \doteq \cP(\partial_\infty^n f^B, \chi_{2n+2})^2
\]
for $n \ge 0$, where $\partial_\infty$ is the Shimura-Maass differential operator, which raises the weight by $2$ and produces nearly holomorphic modular forms.
Note that $\varepsilon_\infty(f,\chi_{2n+2}) = +1$ for $n \ge 0$ and hence $B$ is \emph{indefinite}; in fact, we have $B = \M_2(\Q)$ and $f^B = f$.
\item
Another ingredient is an interpretation of the complex period $\cP$ as the $p$-adic period $\cP_p$:
\begin{equation}
\label{eqn:complexpadicequal1}
 \cP(\partial_\infty^n f, \chi_{2n+2})
 \doteq \cP_p(\partial_p^n f, \chi_{2n+2}),
\end{equation}
where $\partial_p$ is the Atkin-Serre operator on the space of $p$-adic modular forms given on $q$-expansions by 
\[
 \partial_p f = \sum_m m a_f(m) q^m
\]
for $f=\sum_m a_f(m)q^m$. The equality in \eqref{eqn:complexpadicequal1} is up to appropriate powers of certain canonical complex and $p$-adic CM periods $\Omega$ and $\Omega_p$, respectively. 

\item
The final ingredient is a description of the $p$-adic Abel-Jacobi map in terms of Coleman's theory of $p$-adic integration \cite{coleman}:
\[
 \AJ_p(\Xi_\chi)(\omega_f) \doteq 
 \lim_{n \rightarrow -1} \cP_p(\partial_p^n f, \chi_{2n+2}),
\]
where the right-hand side is understood as the $p$-adic limit as $n$ tends $p$-adically to $-1$.
Strictly speaking, this process requires the $p$-depleted modular form $f|(V U - U V)$ in place of $f$, which has the $q$-expansion
\[
 f|(V U - U V) = \sum_{p \nmid m} a_f(m) q^m
\]
and admits a Coleman primitive
\[
 \lim_{n \rightarrow -1} \partial_p^n f = \sum_{p \nmid m} m^{-1} a_f(m) q^m.
\]
\end{itemize}
Combining these, we obtain
\begin{align*}
 \cL_p(f,\chi) 
 & \doteq \lim_{n \rightarrow -1} L(\tfrac{1}{2},f,\chi_{2n+2}) \\
 & \doteq \lim_{n \rightarrow -1} \cP_\infty(\partial_\infty^n f, \chi_{2n+2})^2 \\
 & \doteq \lim_{n \rightarrow -1} \cP_p(\partial_p^n f, \chi_{2n+2})^2 \\
 & \doteq \AJ_p(\Xi_\chi)(\omega_f)^2.
\end{align*}

\subsection{Our setting: $\varepsilon(f,\chi) = -1$, $\varepsilon_\infty(f,\chi) = +1$}

We continue to assume that $\varepsilon(f,\chi) = -1$.
As explained in the previous subsection, the basic idea is to switch the archimedean local sign to make the global sign $+1$ and apply Waldspurger's formula.
In the BDP setting, where $\varepsilon_\infty(f,\chi) = -1$, one keeps $f$ fixed and varies $\chi$, whereas in our setting, where $\varepsilon_\infty(f,\chi) = +1$, we keep $\chi$ fixed and vary $f$.

We explain this for $k = l = 2$, which is the setting considered in this work.
We further assume that $f$ is ordinary at $p$.
Let $f_{2n+4}$ vary in a Hida family, where $f_{2n+4}$ has weight $2n+4$ and specializes to $f$ at $n = -1$.
Let $\cL_p(\cdot, \chi)$ be the $p$-adic $L$-function interpolating the algebraic parts of the central values $L(\tfrac{1}{2},f_{2n+4},\chi)$ for $n \ge 0$.
Our goal is to understand $\cL_p(f,\chi)$ and a first attempt is to apply Waldspurger's formula:
\[
 L(\tfrac{1}{2}, f_{2n+4}, \chi) \doteq \cP(f_{2n+4}^B,\chi)^2
\]
for $n \ge 0$.
However, since $\varepsilon_\infty(f_{2n+4},\chi) = -1$ for $n \ge 0$, $B$ is \emph{definite}.
Thus the associated Shimura variety is a finite set, leaving no room for non-trivial geometric cycles. 
This shows that the BDP argument cannot be carried over directly to our setting.

Therefore we are led to seek a setting in which the relevant motive is realized on a Shimura variety attached to a larger group.
In fact, such a setting is provided by a Shimura variety $X$ attached to $\GSp_4$ together with the theta lift $\theta(\chi)$ of $\chi$ to $\GSp_4$.
In addition, a relevant cycle is constructed from a Shimura variety $Y$ attached to $\GU(1,1)$ (for an auxiliary imaginary quadratic field $E$) and regarded as a subvariety of $X$ via the natural embedding $\GU(1,1) \hookrightarrow \GSp_4$.
Then, in this paper and its sequel, we aim to establish an extension of the BDP formula
\[
 \cL_p(f,\chi) \doteq \AJ_p(\Xi)(\omega_{\theta(\chi)} \wedge \eta_f)^2,
\]
where
\begin{itemize}
\item $\Xi \in \CH^3(X \times Y)_0$ is a cycle on $X \times Y$ coming from the diagonal embedding $Y \hookrightarrow X \times Y$;
\item $\AJ_p : \CH^3(X \times Y)_0 \rightarrow \Fil^2 H^3_{\dR}(X \times Y / \Q_p)^\vee$ is the $p$-adic Abel-Jacobi map;
\item $\omega_{\theta(\chi)} \in \Fil^2 H^2_{\dR}(X/\Q_p)$ is the differential form associated to $\theta(\chi)$;
\item $\eta_f \in H^1_{\dR}(Y/\Q_p)$ is the unit root vector associated to $f$.
\end{itemize} 
Our strategy toward this formula is as follows.
\begin{enumerate}
\item
\label{outline-1}
We begin by recalling Waldspurger's formula in a slightly more precise form:
\[
 \frac{L(\frac{1}{2}, f, \chi)}{L(1, f, \Ad)} 
 \doteq \frac{\cP(f^B, \chi)^2}{\| f^B \|^2},
\]
where $L(s,f,\Ad)$ is the degree $3$ adjoint $L$-function and $\| \cdot \|$ denotes the Petersson norm.
For $n \ge 0$, we use the seesaw diagram
\[
\begin{tikzcd}
 B^\times \arrow[r, dashed] & \GU(2) \arrow[rd,dash] & \GSp_4 \arrow[ld,dash] & \\
 K^\times \arrow[r, dashed] & \GSO(2) \arrow[u,dash] & \GU(1,1) \arrow[u,dash] & \GL_2 \times E^\times \arrow[l, dashed]
\end{tikzcd}
\]
and rewrite the toric period as a $\GU(1,1)$-period:
\[
 \cP(\theta(f_{2n+4}^\mu), \chi)
 = \cP(\partial_\infty^n \theta(\chi), f_{2n+4}^\mu).
\]
Here $\mu$ is a fixed auxiliary class group character of $E$, $f_{2n+4}^\mu = f_{2n+4} \otimes \mu$ is regarded as an automorphic form on $\GU(1,1)(\A_\Q)$, $\theta(f_{2n+4}^\mu)$ is the theta lift of $\theta(f_{2n+4}^\mu)$ to $\GU(2)$, and $\partial_\infty$ is an archimedean differential operator for $\GSp_4$ described in the next subsection.
Also, the Rallis inner product formula gives
\[
 \frac{\| \theta(f_{2n+4}^\mu) \|^2}{\| f_{2n+4}^\mu \|^2}
 \doteq L(\tfrac{1}{2}, f_{2n+4}, \mu),
\]
while the Rankin-Selberg method gives
\[
 \| f_{2n+4}^\mu \|^2 \doteq L(1, f_{2n+4}, \Ad).
\]
Combining these, we obtain
\begin{align*}
 \cP(\partial_\infty^n \theta(\chi), f_{2n+4}^\mu)^2 
 & = \cP(\theta(f_{2n+4}^\mu), \chi)^2 \\
 & \doteq \frac{L(\tfrac{1}{2}, f_{2n+4}, \chi)}{L(1, f_{2n+4}, \Ad)}
 \cdot \| \theta(f_{2n+4}^\mu) \|^2 \\
 & \doteq \frac{L(\tfrac{1}{2}, f_{2n+4}, \chi)}{L(1, f_{2n+4}, \Ad)}
 \cdot L(\tfrac{1}{2}, f_{2n+4}, \mu) \cdot \| f_{2n+4}^\mu \|^2 \\
 & \doteq L(\tfrac{1}{2}, f_{2n+4}, \chi)
 \cdot L(\tfrac{1}{2}, f_{2n+4}, \mu).
\end{align*}
\item
\label{outline-2}
As in the BDP setting, the complex period can be interpreted as a $p$-adic period:
\[
 \cP(\partial_\infty^n \theta(\chi), f_{2n+4}^\mu)
 \doteq \cP_p(\partial_p^n \theta(\chi), f_{2n+4}^\mu),
\]
where $\partial_p$ is the $p$-adic differential operator corresponding to $\partial_\infty$. 
\item
\label{outline-3}
In place of Coleman's theory used in the BDP setting, we use its extension due to Besser \cite{besser}.
Together with an additional argument, this leads to a description of the $p$-adic Abel-Jacobi map:
\[
 \AJ_p(\Xi)(\omega_{\theta(\chi)} \wedge \eta_f) 
 \doteq \lim_{n \rightarrow -1} \cP_p(\partial_p^n \theta(\chi), f_{2n+4}^\mu).
\]
Again, this process requires a $p$-depleted modular form, but a suitable theory of $p$-depletions was not available for Siegel modular forms of higher degree.
To address this difficulty, we introduce a $p$-adic Hecke operator $\tau_p$ which annihilates the cohomology class of $\omega_{\theta(\chi)}$ and has the following effect on the $q$-expansion of the particular form $\theta(\chi)$: if
\[
 \theta(\chi) = \sum_T a(T) q^T, 
\]
where $T$ runs over positive semi-definite $2 \times 2$ symmetric matrices, then $\tau_p \theta(\chi)$ has the $q$-expansion
\[
 \tau_p \theta(\chi) = \sum_{p \nmid \det(T)} a(T) q^T
\]
and admits a primitive $\lim_{n \rightarrow -1} \partial_p^n \theta(\chi)$.
\item
Finally, we deal with the auxiliary factor involving $\mu$.
Since $\mu$ has weight $0$, we have
\[
 \lim_{n \rightarrow -1} L(\tfrac{1}{2}, f_{2n+4}, \mu)
 \doteq L(\tfrac{1}{2}, f, \mu).
\]
By \cite[Theorem 1.7]{bt}, we can choose $E$ and $\mu$ so that the right-hand side is nonzero.
Hence this central value can be absorbed into the nonzero factors.
\end{enumerate}
Combining these, we obtain
\begin{align*}
 \cL_p(f,\chi) 
 & \doteq \lim_{n \rightarrow -1} L(\tfrac{1}{2},f_{2n+4},\chi) \\
 & \doteq \lim_{n \rightarrow -1} \cP(\partial_\infty^n \theta(\chi), f_{2n+4}^\mu)^2 \\
 & \doteq \lim_{n \rightarrow -1} \cP_p(\partial_p^n \theta(\chi), f_{2n+4}^\mu)^2 \\
 & \doteq \AJ_p(\Xi)(\omega_{\theta(\chi)} \wedge \eta_f)^2.
\end{align*}

The purpose of this paper is to provide the automorphic input for the argument.
More precisely, we establish \eqref{outline-1} and its analogue for the $p$-depleted modular form, and construct the $p$-adic Hecke operator $\tau_p$ in \eqref{outline-3}.
The remaining parts, namely, \eqref{outline-2} and the $p$-adic Abel-Jacobi interpretation in \eqref{outline-3}, will be treated in the sequel.

\subsection{Main results}

Now we turn to the results established in this paper and explain its organization.
We work in a minimally ramified setting: in addition to the weight condition $k=l=2$, we assume that $f$ has square-free level and $\chi$ is unramified at all finite places.
In \S\S \ref{s:modular_forms}--\ref{s:theta}, we set up the notation and recall the preliminaries.
In particular, in \S \ref{ss:diff-classical}, we introduce a differential operator $\partial_\infty$ which maps Siegel modular forms of degree $2$ and weight $\Sym^2 \otimes \det^n$ to those of weight $\Sym^2 \otimes \det^{n+1}$.

Next we give a precise definition of the theta lift $\theta(\chi)$.
This requires suitable choices of Schwartz functions in the Weil representation, which we make in \S\S \ref{s:schwartz_forms}--\ref{s:choices-schwartz}.
In particular, we construct a variant of the Kudla-Millson Schwartz forms \cite{km1} for the theta lift from $\GSO(2)$ to $\GSp_4$.
With these choices, in \S \ref{s:main_construction}, we obtain a Siegel modular form $\theta(\chi)$ of degree $2$ and weight $\Sym^2 \otimes \det$ (so that $\partial_\infty^n \theta(\chi)$ has weight $\Sym^2 \otimes \det^{n+1}$).
We also introduce a $p$-adic Hecke operator $\tau_p$ in \S \ref{ss:p-depletion} and establish the following properties of $\theta(\chi)$ (Theorems \ref{t:fourier-theta} and \ref{t:fourier-f-flat}), which provide the automorphic input for \eqref{outline-3} in the outline of the previous subsection.

\begin{thm}
\label{thm1}
If $\theta(\chi)$ has the $q$-expansion
\[
 \theta(\chi) = \sum_T a(T) q^T, 
\]
then $a(T)$ is algebraic and $p$-integral for all $T$, and $\tau_p \theta(\chi)$ has the $q$-expansion
\[
 \tau_p \theta(\chi) = \sum_{p \nmid \det(T)} a(T) q^T.
\]
\end{thm}

The proof of Theorem \ref{thm1} is given in \S \ref{s:local_fourier}.
Since the relevant Fourier coefficients are decomposable, they can be computed locally.
In particular, Proposition \ref{p:fourier-p-depletion} plays a crucial role in this computation: it ensures that $\tau_p$ has the desired effect on the Fourier coefficients.

In \S \ref{s:main_formula}, we state the following formulas for the $\GU(1,1)$-periods (Theorems \ref{t:period-L-value-1} and \ref{t:period-L-value-2}), which constitute \eqref{outline-1} in the outline of the previous subsection.

\begin{thm}
\label{thm2}
For $n \ge 0$, we have
\begin{align*}
 \cP(\partial_\infty^n \theta(\chi), f_{2n+4}^\mu)^2 
 & = L(\tfrac{1}{2}, f_{2n+4}, \chi) \cdot L(\tfrac{1}{2}, f_{2n+4}, \mu), \\
 \cP(\partial_\infty^n \tau_p \theta(\chi), f_{2n+4}^\mu)^2
 & = \cE_\infty(f_{2n+4}, \chi) \cdot \cE_p(f_{2n+4}, \chi) \cdot L(\tfrac{1}{2}, f_{2n+4}, \chi) \\
 & \times \cE_\infty(f_{2n+4}, \mu) \cdot \cE_p(f_{2n+4}, \mu) \cdot L(\tfrac{1}{2}, f_{2n+4}, \mu)
\end{align*}
up to explicit nonzero factors, 
where $\cE_\infty$ and $\cE_p$ denote the modified Euler factors at $\infty$ and $p$, respectively, as in \cite{coates89}.
\end{thm}

The rest of this paper is devoted to the proof of Theorem \ref{thm2}.
In \S \ref{s:main_seesaw}, we first reduce the computation of $\cP(\partial_\infty^n \tau_p \theta(\chi), f_{2n+4}^\mu)$ to that of $\cP(\partial_\infty^n \theta(\chi), f_{2n+4}^\mu)$, which amounts to studying the effect of $\tau_p$ on the period.
By the uniqueness of the relevant functionals, this effect can be evaluated locally and the explicit computation is carried out in \S \ref{s:local_functionals}, where the key formula is given in Lemma \ref{l:unique-hom-p-depletion}.
Then we use a seesaw identity to reduce the computation of the $\GU(1,1)$-period $\cP(\partial_\infty^n \theta(\chi), f_{2n+4}^\mu)$ to that of the toric period $\cP(\theta(f_{2n+4}^\mu), \chi)$.
Since Waldspurger's formula computes $|\cP(\theta(f_{2n+4}^\mu), \chi)|^2$ rather than $\cP(\theta(f_{2n+4}^\mu), \chi)^2$, we need to compare $\cP(\theta(f_{2n+4}^\mu), \chi)$ with its complex conjugate.
After making this comparison in \S \ref{s:complex_conjugation}, we apply Waldspurger's formula in \S \ref{s:waldspurger} to compute the toric period explicitly.
It remains to compute the Petersson norm $\| \theta(f_{2n+4}^\mu) \|^2$ of the theta lift.
We establish an explicit formula for this norm in \S\S \ref{s:rallis}--\ref{s:local_zeta} by applying the Rallis inner product formula. 
Here the main technical difficulty is the computation of the archimedean local zeta integral, which is given in Proposition \ref{p:rallis-real}.
Putting everything together, we complete the proof of Theorem \ref{thm2} in \S \ref{s:proof}.

\subsection{Acknowledgments}

A.I. was partially supported by JSPS KAKENHI Grant Number 19H01781.
K.P. was partially supported by NSF grants DMS 2001293 and DMS 2301507.

\section{Siegel modular forms}
\label{s:modular_forms}

In this section, we introduce notation for Siegel modular forms.

\subsection{Basic notation}

Let $\GSp_{2n}$ and $\Sp_{2n}$ be the symplectic similitude group and the symplectic group (regarded as group schemes over $\Z$) given by  
\[
 \GSp_{2n} = \left\{ g \in \GL_{2n} \; \middle| \; {}^t g \mat{}{I_n}{-I_n}{} g = \nu(g) \cdot \mat{}{I_n}{-I_n}{} \right\}
\]
and $\Sp_{2n} = \ker \nu$, respectively, where $\nu: \GSp_{2n} \rightarrow \mathbb{G}_m$ is the similitude character.
Put
\[
 \boldsymbol{m}(a) = \begin{pmatrix} a & \\ & {}^t a^{-1} \end{pmatrix}, \quad 
 \boldsymbol{n}(b) = \begin{pmatrix} I_n & b \\ & I_n \end{pmatrix}
\]
for $a \in \GL_n$ and $b \in \Sym_n$, where $\Sym_n = \{ b \in \M_n \mid {}^t b = b \}$.
Define a subgroup $\GSp_{2n}(\R)^+$ of $\GSp_{2n}(\R)$ by 
\[
 \GSp_{2n}(\R)^+ = \{ g \in \GSp_{2n}(\R) \mid \nu(g) > 0 \}.
\]
Let $\fH_n$ be the Siegel upper half-space given by
\[
 \fH_n = \{ z \in \Sym_n(\C) \mid \Im(z) > 0 \}.
\]
Then $\GSp_{2n}(\R)^+$ acts on $\fH_n$ transitively by
\[
 gz = (az + b) (cz + d)^{-1}
\]
for $g = \smat{a}{b}{c}{d} \in \GSp_{2n}(\R)^+$ and $z \in \fH_n$, where $a,b,c,d$ are $n \times n$ matrices.
We denote by
\[
 j : \GSp_{2n}(\R)^+ \times \fH_n \rightarrow \GL_n(\C)
\]
the factor of automorphy given by $j(g,z) = cz + d$, so that 
\[
 j(g_1 g_2, z) = j(g_1, g_2 z) j(g_2, z).
\]

Following \cite[\S 13.11]{shimura-book-arith}, we recall the definition of nearly holomorphic modular forms.
Let $(\rho,V_\rho)$ be an algebraic representation of $\GL_n(\C)$, $r$ a non-negative integer, and $\Gamma$ a congruence subgroup of $\Sp_{2n}(\Q)$.
We denote by $N_\rho^r(\Gamma)$ the space of smooth functions $f$ on $\fH_n$ valued in $V_\rho$ such that
\begin{itemize}
\item 
$f(\gamma z) = \rho(j(\gamma, z)) f(z)$ for all $\gamma \in \Gamma$ and $z \in \fH_n$;
\item
$f$ is a finite sum of the form
\[
 f(z) = \sum_k P_k(\Im(z)^{-1}) f_k(z)
\]
for $z \in \fH_n$, where $P_k$ is a polynomial of degree at most $r$ in the entries of $\Im(z)^{-1}$ and $f_k$ is a holomorphic function on $\fH_n$ valued in $V_\rho$;
\item
$f$ satisfies the cusp condition (see \cite[\S 13.12]{shimura-book-arith}), which is required only when $n=1$.
\end{itemize}
Write $M_\rho(\Gamma) = N_\rho^0(\Gamma)$, which is the space of holomorphic modular forms of weight $\rho$ and level $\Gamma$.

We also recall the Fourier expansion of modular forms.
For simplicity, we restrict ourselves to the holomorphic case.
Let $\cL$ be a lattice in $\Sym_n(\Q)$ such that $\{ \boldsymbol{n}(x) \mid x \in \cL \}$ is contained in $\Gamma$.
Let $\cL^\vee$ be the dual lattice of $\cL$ given by 
\[
 \cL^\vee = \{ T \in \Sym_n(\Q) \mid \text{$\tr(Tx) \in \Z$ for all $x \in \cL$} \}.
\]
Then every $f \in M_\rho(\Gamma)$ has a Fourier expansion of the form
\[
 f(z) = \sum_T a_f(T) e^{2 \pi \sqrt{-1} \tr(Tz)}
\]
for $z \in \fH_n$, where $T$ runs over positive semi-definite symmetric matrices in $\cL^\vee$ and $a_f(T) \in V_\rho$ is the $T$-th Fourier coefficient of $f$.

\subsection{Differential operators}
\label{ss:diff-classical}

In this subsection, we assume that $n=2$ and define differential operators which will be used later.
We use the following coordinates on $\fH_2$.
For $z \in \fH_2$, we write
\[
 z = x + \sqrt{-1} y
\]
with $x, y \in \Sym_2(\R)$ such that $y>0$.
We also write
\[
 z = \mat{z_1}{z_2}{z_2}{z_3}, \quad
 x = \mat{x_1}{x_2}{x_2}{x_3}, \quad
 y = \mat{y_1}{y_2}{y_2}{y_3} 
\]
with $z_1, z_2, z_3 \in \C$ and $x_1, x_2, x_3, y_1, y_2, y_3 \in \R$, so that
\[
 \frac{\partial}{\partial z_i} = \frac{1}{2} \left( \frac{\partial}{\partial x_i} - \sqrt{-1} \frac{\partial}{\partial y_i} \right).
\]
Put
\[
 \Delta = \det(y) = y_1 y_3 - y_2^2 \in \R_+^\times.
\]
For any non-negative integer $m$, we define differential operators $\delta_{m,1}, \delta_{m,2}, \delta_{m,3}$ on $\fH_2$ by
\begin{align*}
 \delta_{m,1} & = \frac{1}{2 \pi \sqrt{-1}} \frac{\partial}{\partial z_1} - \frac{m}{4 \pi} \frac{y_3}{\Delta}, \\
 \delta_{m,2} & = \frac{1}{2 \pi \sqrt{-1}} \frac{\partial}{\partial z_2} + \frac{m}{2 \pi} \frac{y_2}{\Delta}, \\
 \delta_{m,3} & = \frac{1}{2 \pi \sqrt{-1}} \frac{\partial}{\partial z_3} - \frac{m}{4 \pi} \frac{y_1}{\Delta}.
\end{align*}

We consider a family of irreducible representations $\{ (\rho_m, V_m) \}_{m \ge 0}$ of $\GL_2(\C)$ given by
\[
 \rho_m = \Sym^2 \otimes {\det}^m.
\]
Take a basis
\[
 \{ v_{m,i} \mid 0 \le i \le 2 \}
\]
of $V_m$ such that the associated action of $\M_2(\C)$ is given by
\begin{align*}
 d \rho_m \mat{1}{0}{0}{0} v_{m,i} & = (m+2-i) v_{m,i}, & 
 d \rho_m \mat{0}{0}{0}{1} v_{m,i} & = (m+i) v_{m,i}, \\
 d \rho_m \mat{0}{1}{0}{0} v_{m,i} & = (3-i) v_{m,i-1}, & 
 d \rho_m \mat{0}{0}{1}{0} v_{m,i} & = (i+1) v_{m,i+1},
\end{align*}
where we interpret $v_{m,i} = 0$ if $i=-1$ or $3$.
We identify $V_m$ with $\C^3$ via this basis.
Let $(\rho_m^\vee, V_m^\vee)$ be the contragredient representation of $(\rho_m,V_m)$ and take the dual basis
\[
 \{ v^*_{m,i} \mid 0 \le i \le 2 \}
\]
of $V_m^\vee$ such that $\langle v_{m,i}, v_{m,j}^* \rangle = \delta_{ij}$, where $\langle \cdot, \cdot \rangle : V_m \times V_m^\vee \rightarrow \C$ is the natural pairing.
For a smooth $V_m$-valued function $f$ on $\fH_2$, we define a smooth $V_{m+1}$-valued function $\delta_m f$ on $\fH_2$ by
\begin{align*}
 \langle \delta_m f, v_{m+1,0}^* \rangle & = 
 \langle \delta_{m,2} f, v_{m,0}^* \rangle - 2 \langle \delta_{m,1} f, v_{m,1}^* \rangle, \\
 \langle \delta_m f, v_{m+1,1}^* \rangle & = 
 \langle \delta_{m,3} f, v_{m,0}^* \rangle - \langle \delta_{m,1} f, v_{m,2}^* \rangle, \\
 \langle \delta_m f, v_{m+1,2}^* \rangle & = 
 2 \langle \delta_{m,3} f, v_{m,1}^* \rangle - \langle \delta_{m,2} f, v_{m,2}^* \rangle.
\end{align*}
Under the above identifications $V_{m+1} \simeq \C^3$ and $V_m \simeq \C^3$, we may write this as
\[
 \delta_m f = 
 \begin{pmatrix}
  \delta_{m,2} & -2 \delta_{m,1} & 0 \\
  \delta_{m,3} & 0 & - \delta_{m,1} \\
  0 & 2 \delta_{m,3} & -\delta_{m,2}
 \end{pmatrix} 
 f.
\]
This differential operator $\delta_m$ satisfies the following basic property for nearly holomorphic modular forms.

\begin{prop}
\label{p:delta_mf}
If $f \in N^r_{\rho_m}(\Gamma)$, then we have $\delta_m f \in N^{r+1}_{\rho_{m+1}}(\Gamma)$.
\end{prop}

The proof of this proposition will be given in the next section.

\section{From classical to adelic}

In this section, we relate classical modular forms to adelic automorphic forms, clarifying how the classical objects fit into the adelic setting.

\subsection{Automorphic forms}
\label{ss:autom-adelic}

Let $\A$ and $\A_f$ be the rings of adeles and finite adeles of $\Q$, respectively.
Let $\hat{\Z}$ be the profinite completion of $\Z$, so that $\hat{\Z} = \prod_p \Z_p \subset \A_f$.
Let $\psi$ be the standard additive character of $\A/\Q$, so that $\psi_\infty(x) = e^{2 \pi \sqrt{-1} x}$ for $x \in \R$.
Put $G = \GSp_{2n}(\R)$ and $G_1 = \Sp_{2n}(\R)$.
Let $K_G$ and $K_{G_1}$ be maximal compact modulo center subgroups of $G$ and $G_1$, respectively, given by
\begin{align*}
 K_G & = \left\{ z \cdot \mat{I_n}{}{}{\epsilon I_n} \; \middle| \; z \in \R^\times, \; \epsilon = \pm 1 \right\} \cdot K_{G_1}, \\
 K_{G_1} & = \left\{ \mat{a}{b}{-b}{a} \; \middle| \; a + \sqrt{-1} b \in \U(n) \right\}.
\end{align*}
Note that $K_G^0 = Z_G \cdot K_{G_1}$, where $Z_G$ is the center of $G$.
Let $\fg_0$, $\fg_{1,0}$, and $\fk_0$ be the Lie algebras of $G$, $G_1$, and $K_{G_1}$, respectively.
Then we have a Cartan decomposition $\fg_{1,0} = \fk_0 \oplus \fp_0$ with
\begin{align*}
 \fk_0 & = \left\{ \mat{A}{B}{-B}{A} \; \middle| \;
 A \in \mathrm{Alt}_n(\R), \; B \in \Sym_n(\R) \right\}, \\
 \fp_0 & = \left\{ \mat{A}{B}{B}{-A} \; \middle| \;
 A, B \in \Sym_n(\R) \right\},
\end{align*}
where $\Sym_n = \{ X \in \M_n \mid {}^t X = X \}$ and $\mathrm{Alt}_n = \{ X \in \M_n \mid {}^t X = - X \}$.
Let $\fg$, $\fg_1$, $\fk$, and $\fp$ be the complexifications of $\fg_0$, $\fg_{1,0}$, $\fk_0$, and $\fp_0$, respectively.
Let $\fp^+$ and $\fp^-$ be the subalgebras of $\fg_1$ given by 
\[
 \fp^\pm = \left\{ \mat{A}{\pm \sqrt{-1} A}{\pm \sqrt{-1} A}{-A} \; \middle| \;
 A \in \Sym_n(\C) \right\},
\]
so that $\fp = \fp^+ \oplus \fp^-$.
Note that $\fp^+$ and $\fp^-$ are commutative.

Let $(\rho,V_\rho)$ be an algebraic representation of $\GL_n(\C)$, $r$ a non-negative integer, and $\cK$ an open compact subgroup of $\GSp_{2n}(\A_f)$ such that $\nu(\cK) = \hat{\Z}^\times$.
We denote by $\cN_\rho^r(\cK)$ the space of automorphic forms $\phi$ on $\GSp_{2n}(\A)$ valued in $V_\rho$ such that 
\begin{itemize}
\item
$X \phi = 0$ for all $X \in \Sym^{r+1}(\fp^-)$;
\item 
$\phi(gzk) = \rho(a - \sqrt{-1} b)^{-1} \phi(g)$ for all $g \in \GSp_{2n}(\A)$, $z \in \R_+^\times$, and $k = \smat{a}{b}{-b}{a} \in K_{G_1}$;
\item
$\phi(g k) = \phi(g)$ for all $g \in \GSp_{2n}(\A)$ and $k \in \cK$.
\end{itemize}
Put $\Gamma = \Sp_{2n}(\Q) \cap \cK$.
Since $\GSp_{2n}(\A) = \GSp_{2n}(\Q) \GSp_{2n}(\R)^+ \cK$ by the strong approximation theorem and $\GSp_{2n}(\Q) \cap \GSp_{2n}(\R)^+ \cK \subset \Sp_{2n}(\Q)$, we have an isomorphism
\[
 \cN_\rho^r(\cK) \simeq N_\rho^r(\Gamma)
\]
sending $\phi \in \cN_\rho^r(\cK)$ to $f \in N_\rho^r(\Gamma)$ given by
\[
 f(z) = \rho(j(\nu(g)^{-1/2} \cdot g, z_0)) \phi(g)
\]
for $z \in \fH_n$, where $g \in \GSp_{2n}(\R)^+$ is any element such that $g z_0 = z$ with $z_0 = \sqrt{-1} I_n$.

Write $\cM_\rho(\cK) = \cN_\rho^0(\cK)$, so that $\cM_\rho(\cK) \simeq M_\rho(\Gamma)$.
For $f \in M_\rho(\Gamma)$, let $\phi \in \cM_\rho(\cK)$ be the corresponding adelic automorphic form.
Then for $T \in \Sym_n(\Q)$, the $T$-th Fourier coefficient $a_f(T)$ of $f$ is given by 
\begin{equation}
\label{eq:a_f(T)}
 a_f(T) = e^{-2 \pi \sqrt{-1} \tr(T z)} \rho({}^t a^{-1}) W_T(\boldsymbol{n}(b) \boldsymbol{m}(a)) 
\end{equation}
with 
\[
 W_T(g) = \int_{\Sym_n(\Q) \backslash \Sym_n(\A)} \phi(\boldsymbol{n}(x) g) \overline{\psi(\tr(Tx))} \, dx, 
\]
where $z \in \fH_n$ is arbitrary, $a \in \GL_n(\R)$ and $b \in \Sym_n(\R)$ are chosen so that $z = b + \sqrt{-1} a {}^t a$, and $dx$ is the Tamagawa measure on $\Sym_n(\A)$.

\subsection{Differential operators}
\label{ss:diff-adelic}

In this subsection, we assume that $n=2$ and relate the differential operator $\delta_m$ to the action of $\fp^+$.
Take the following basis of $\fg_{1,0}$:
\begin{align*}
 \mathtt{A}_1 & = \mathtt{E}_{11} - \mathtt{E}_{33}, &
 \mathtt{A}_2 & = \mathtt{E}_{12} - \mathtt{E}_{43}, &
 \mathtt{A}_3 & = \mathtt{E}_{21} - \mathtt{E}_{34}, &
 \mathtt{A}_4 & = \mathtt{E}_{22} - \mathtt{E}_{44}, \\
 \mathtt{B}_1 & = \mathtt{E}_{13}, &
 \mathtt{B}_2 & = \mathtt{E}_{14} + \mathtt{E}_{23}, &
 \mathtt{B}_3 & = \mathtt{E}_{24}, \\
 \mathtt{C}_1 & = \mathtt{E}_{31}, &
 \mathtt{C}_2 & = \mathtt{E}_{32} + \mathtt{E}_{41}, &
 \mathtt{C}_3 & = \mathtt{E}_{42},
\end{align*}
where $\mathtt{E}_{ij}$ denotes the $4 \times 4$ matrix with one at the $(i,j)$-th entry and zero elsewhere.
Take bases
\[
 \mathtt{H}_1, \mathtt{H}_2, \mathtt{X}, \mathtt{Y}, \quad
 \mathtt{X}_{2,0}, \mathtt{X}_{1,1}, \mathtt{X}_{0,2}, \quad
 \mathtt{X}_{-2,0}, \mathtt{X}_{-1,-1}, \mathtt{X}_{0,-2}
\]
of $\fk$, $\fp^+$, $\fp^-$, respectively, given by
\begin{align*}
 \mathtt{H}_1 & = 
 \begin{pmatrix}
  0 & 0 & - \sqrt{-1} & 0 \\
  0 & 0 & 0 & 0 \\
  \sqrt{-1} & 0 & 0 & 0 \\
  0 & 0 & 0 & 0
 \end{pmatrix}, &
 \mathtt{H}_2 & = 
 \begin{pmatrix}
  0 & 0 & 0 & 0 \\
  0 & 0 & 0 & - \sqrt{-1} \\
  0 & 0 & 0 & 0 \\
  0 & \sqrt{-1} & 0 & 0 
 \end{pmatrix}, \\
 \mathtt{X} & = \frac{1}{2} \cdot 
 \begin{pmatrix}
  0 & 1 & 0 & - \sqrt{-1} \\
  -1 & 0 & - \sqrt{-1} & 0 \\
  0 & \sqrt{-1} & 0 & 1 \\
  \sqrt{-1} & 0 & -1 & 0
 \end{pmatrix}, & 
 \mathtt{Y} & = \frac{1}{2} \cdot 
 \begin{pmatrix}
  0 & -1 & 0 & - \sqrt{-1} \\
  1 & 0 & - \sqrt{-1} & 0 \\
  0 & \sqrt{-1} & 0 & -1 \\
  \sqrt{-1} & 0 & 1 & 0
 \end{pmatrix}, \\
 \mathtt{X}_{2,0} & = \frac{1}{2} \cdot 
 \begin{pmatrix}
  1 & 0 & \sqrt{-1} & 0 \\
  0 & 0 & 0 & 0 \\
  \sqrt{-1} & 0 & -1 & 0 \\
  0 & 0 & 0 & 0 \\
 \end{pmatrix}, &
 \mathtt{X}_{-2,0} & = \frac{1}{2} \cdot 
 \begin{pmatrix}
  1 & 0 & - \sqrt{-1} & 0 \\
  0 & 0 & 0 & 0 \\
  - \sqrt{-1} & 0 & -1 & 0 \\
  0 & 0 & 0 & 0 \\
 \end{pmatrix}, \\
 \mathtt{X}_{1,1} & = \frac{1}{2} \cdot 
 \begin{pmatrix}
  0 & 1 & 0 & \sqrt{-1} \\
  1 & 0 & \sqrt{-1} & 0 \\
  0 & \sqrt{-1} & 0 & -1 \\
  \sqrt{-1} & 0 & -1 & 0 
 \end{pmatrix}, &
 \mathtt{X}_{-1,-1} & = \frac{1}{2} \cdot 
 \begin{pmatrix}
  0 & 1 & 0 & - \sqrt{-1} \\
  1 & 0 & - \sqrt{-1} & 0 \\
  0 & - \sqrt{-1} & 0 & -1 \\
  - \sqrt{-1} & 0 & -1 & 0 
 \end{pmatrix}, \\
 \mathtt{X}_{0,2} & = \frac{1}{2} \cdot 
 \begin{pmatrix}
  0 & 0 & 0 & 0 \\
  0 & 1 & 0 & \sqrt{-1} \\
  0 & 0 & 0 & 0 \\
  0 & \sqrt{-1} & 0 & -1
 \end{pmatrix}, &
 \mathtt{X}_{0,-2} & = \frac{1}{2} \cdot 
 \begin{pmatrix}
  0 & 0 & 0 & 0 \\
  0 & 1 & 0 & - \sqrt{-1} \\
  0 & 0 & 0 & 0 \\
  0 & - \sqrt{-1} & 0 & -1
 \end{pmatrix}.
\end{align*}
Note that
\begin{align*}
 [\mathtt{H}_1, \mathtt{X}_{2,0}] & = 2 \mathtt{X}_{2,0}, & 
 [\mathtt{H}_1, \mathtt{X}_{1,1}] & = \mathtt{X}_{1,1}, & 
 [\mathtt{H}_1, \mathtt{X}_{0,2}] & = 0, \\
 [\mathtt{H}_2, \mathtt{X}_{2,0}] & = 0, & 
 [\mathtt{H}_2, \mathtt{X}_{1,1}] & = \mathtt{X}_{1,1}, & 
 [\mathtt{H}_2, \mathtt{X}_{0,2}] & = 2 \mathtt{X}_{0,2}, \\ 
 [\mathtt{X}, \mathtt{X}_{2,0}] & = 0, & 
 [\mathtt{X}, \mathtt{X}_{1,1}] & = 2 \mathtt{X}_{2,0}, & 
 [\mathtt{X}, \mathtt{X}_{0,2}] & = \mathtt{X}_{1,1}, \\
 [\mathtt{Y}, \mathtt{X}_{2,0}] & = \mathtt{X}_{1,1}, &
 [\mathtt{Y}, \mathtt{X}_{1,1}] & = 2 \mathtt{X}_{0,2}, &
 [\mathtt{Y}, \mathtt{X}_{0,2}] & = 0.
\end{align*}

Fix $m \ge 0$.
Let $(\rho_m, V_m)$ be the irreducible representation of $\GL_2(\C)$ as in \S \ref{ss:diff-classical}.
Define an irreducible representation $\varrho_m$ of $K_{G_1}$ on $V_m^\vee$ by 
\[
 \varrho_m(k) = \rho_m^\vee(a - \sqrt{-1} b)
\]
for $k = \smat{a}{b}{-b}{a} \in K_{G_1}$.
Then the associated action of $\fk$ is given by the following.

\begin{lem}
\label{l:varrho_m}
We have
\begin{align*}
 d \varrho_m(\mathtt{H}_1) v_{m,i}^* & = (m+2-i) v_{m,i}^*, & 
 d \varrho_m(\mathtt{H}_2) v_{m,i}^* & = (m+i) v_{m,i}^*, \\
 d \varrho_m(\mathtt{X}) v_{m,i}^* & = i v_{m,i-1}^*, & 
 d \varrho_m(\mathtt{Y}) v_{m,i}^* & = (2-i) v_{m,i+1}^*.
\end{align*}
\end{lem}

\begin{proof}
Let $\iota: K_{G_1} \rightarrow \U(2)$ be an isomorphism given by 
\[
 \iota \mat{a}{b}{-b}{a} = a + \sqrt{-1} b.
\]
We also denote by $\iota : \fk \rightarrow \M_2(\C)$ the induced isomorphism.
Note that
\[
 \iota(\mathtt{H}_1) = \mat{1}{0}{0}{0}, \quad
 \iota(\mathtt{H}_2) = \mat{0}{0}{0}{1}, \quad
 \iota(\mathtt{X}) = \mat{0}{1}{0}{0}, \quad
 \iota(\mathtt{Y}) = \mat{0}{0}{1}{0}.
\]
Since $\varrho^\vee_m(k) = \rho_m({}^t \iota(k)^{-1})$ for $k \in K_{G_1}$, we have
\[
 d \varrho_m^\vee(X) = - d \rho_m({}^t \iota(X))
\]
for $X \in \fk$.
In particular, we have
\begin{align*}
 d \varrho_m^\vee(\mathtt{H}_1) v_{m,i} & = -(m+2-i) v_{m,i}, &
 d \varrho_m^\vee(\mathtt{H}_2) v_{m,i} & = -(m+i) v_{m,i}, \\
 d \varrho_m^\vee(\mathtt{X}) v_{m,i} & = -(i+1) v_{m,i+1}, &
 d \varrho_m^\vee(\mathtt{Y}) v_{m,i} & = -(3-i) v_{m,i-1}.
\end{align*}
This yields the assertion.
\end{proof}

Let $\phi$ be a smooth $V_m$-valued function on $G_1$ such that
\[
 \phi(gk) = \varrho_m^\vee(k)^{-1} \phi(g)
\]
for all $g \in G_1$ and $k \in K_{G_1}$.
In particular, we have
\[
 \langle XY \phi, v^* \rangle = \langle X \phi, d \varrho_m(Y) v^* \rangle
\]
for all $X \in \fg_1$, $Y \in \fk$, and $v^* \in V_m^\vee$.
Define a smooth $V_{m+1}$-valued function $\mathtt{D} \phi$ on $G_1$ by
\begin{align*}
 \langle \mathtt{D} \phi, v_{m+1,0}^* \rangle & = 
 \langle \mathtt{X}_{1,1} \phi, v_{m,0}^* \rangle 
 - 2 \langle \mathtt{X}_{2,0} \phi, v_{m,1}^* \rangle, \\
 \langle \mathtt{D} \phi, v_{m+1,1}^* \rangle & = 
 \langle \mathtt{X}_{0,2} \phi, v_{m,0}^* \rangle 
 - \langle \mathtt{X}_{2,0} \phi, v_{m,2}^* \rangle, \\
 \langle \mathtt{D} \phi, v_{m+1,2}^* \rangle & = 
 2 \langle \mathtt{X}_{0,2} \phi, v_{m,1}^* \rangle
 - \langle \mathtt{X}_{1,1} \phi, v_{m,2}^* \rangle.
\end{align*}
Under the identifications $V_{m+1} \simeq \C^3$ and $V_m \simeq \C^3$ as in \S \ref{ss:diff-classical}, we may write this as
\begin{equation}
\label{eq:D.phi}
 \mathtt{D} \phi = 
 \begin{pmatrix}
  \mathtt{X}_{1,1} & -2 \mathtt{X}_{2,0} & 0 \\
  \mathtt{X}_{0,2} & 0 &  -\mathtt{X}_{2,0} \\
  0 & 2 \mathtt{X}_{0,2} & - \mathtt{X}_{1,1}
 \end{pmatrix}
 \phi.
\end{equation}

\begin{lem}
\label{l:D.phi}
We have
\[
 \mathtt{D} \phi(gk) = \varrho_{m+1}^\vee(k)^{-1} \mathtt{D} \phi(g)
\]
for all $g \in G_1$ and $k \in K_{G_1}$.
\end{lem}

\begin{proof}
We regard $\mathtt{D}$ as the $3 \times 3$ matrix on the right-hand side of \eqref{eq:D.phi}.
Since
\begin{align*}
 \mathtt{H}_1 \mathtt{D} & = \mathtt{D} \mathtt{H}_1 + 
 \begin{pmatrix}
  \mathtt{X}_{1,1} & - 4 \mathtt{X}_{2,0} & 0 \\
  0 & 0 & - 2 \mathtt{X}_{2,0} \\
  0 & 0 & - \mathtt{X}_{1,1}  
 \end{pmatrix}, & 
 \mathtt{H}_2 \mathtt{D} & = \mathtt{D} \mathtt{H}_2 + 
 \begin{pmatrix}
  \mathtt{X}_{1,1} & 0 & 0 \\
  2 \mathtt{X}_{0,2} & 0 & 0 \\
  0 & 4 \mathtt{X}_{0,2} & - \mathtt{X}_{1,1}  
 \end{pmatrix}, \\
 \mathtt{X} \mathtt{D} & = \mathtt{D} \mathtt{X} + 
 \begin{pmatrix}
  2 \mathtt{X}_{2,0} & 0 & 0 \\
  \mathtt{X}_{1,1} & 0 & 0 \\
  0 & 2 \mathtt{X}_{1,1} & - 2 \mathtt{X}_{2,0}
 \end{pmatrix}, & 
 \mathtt{Y} \mathtt{D} & = \mathtt{D} \mathtt{Y} + 
 \begin{pmatrix}
  2 \mathtt{X}_{0,2} & - 2 \mathtt{X}_{1,1} & 0 \\
  0 & 0 & - \mathtt{X}_{1,1} \\
  0 & 0 & - 2 \mathtt{X}_{0,2}  
 \end{pmatrix},  
\end{align*}
we have
\begin{align*}
 \mathtt{H}_1 \mathtt{D} \phi & = \left( \mathtt{D}
 \begin{pmatrix}
  m+2 & 0 & 0 \\
  0 & m+1 & 0 \\
  0 & 0 & m
 \end{pmatrix} + 
 \begin{pmatrix}
  \mathtt{X}_{1,1} & - 4 \mathtt{X}_{2,0} & 0 \\
  0 & 0 & - 2 \mathtt{X}_{2,0} \\
  0 & 0 & - \mathtt{X}_{1,1}
 \end{pmatrix}
 \right) \phi \\
 & =
 \begin{pmatrix}
  m+3 & 0 & 0 \\
  0 & m+2 & 0 \\
  0 & 0 & m+1
 \end{pmatrix}
 \mathtt{D} \phi, \\
 \mathtt{H}_2 \mathtt{D} \phi & = \left( \mathtt{D}
 \begin{pmatrix}
  m & 0 & 0 \\
  0 & m+1 & 0 \\
  0 & 0 & m+2
 \end{pmatrix} + 
 \begin{pmatrix}
  \mathtt{X}_{1,1} & 0 & 0 \\
  2 \mathtt{X}_{0,2} & 0 & 0 \\
  0 & 4 \mathtt{X}_{0,2} & - \mathtt{X}_{1,1}  
 \end{pmatrix}
 \right) \phi \\
 & = 
 \begin{pmatrix}
  m+1 & 0 & 0 \\
  0 & m+2 & 0 \\
  0 & 0 & m+3
 \end{pmatrix}
 \mathtt{D} \phi, \\
 \mathtt{X} \mathtt{D} \phi & = \left( \mathtt{D}
 \begin{pmatrix}
  0 & 0 & 0 \\
  1 & 0 & 0 \\
  0 & 2 & 0
 \end{pmatrix} + 
 \begin{pmatrix}
  2 \mathtt{X}_{2,0} & 0 & 0 \\
  \mathtt{X}_{1,1} & 0 & 0 \\
  0 & 2 \mathtt{X}_{1,1} & - 2 \mathtt{X}_{2,0}
 \end{pmatrix}
 \right) \phi \\
 & = 
 \begin{pmatrix}
  0 & 0 & 0 \\
  1 & 0 & 0 \\
  0 & 2 & 0
 \end{pmatrix}
 \mathtt{D} \phi, \\
 \mathtt{Y} \mathtt{D} \phi & = \left( \mathtt{D}
 \begin{pmatrix}
  0 & 2 & 0 \\
  0 & 0 & 1 \\
  0 & 0 & 0
 \end{pmatrix} + 
 \begin{pmatrix}
  2 \mathtt{X}_{0,2} & - 2 \mathtt{X}_{1,1} & 0 \\
  0 & 0 & - \mathtt{X}_{1,1} \\
  0 & 0 & - 2 \mathtt{X}_{0,2}  
 \end{pmatrix}  
 \right) \phi \\
 & = 
 \begin{pmatrix}
  0 & 2 & 0 \\
  0 & 0 & 1 \\
  0 & 0 & 0
 \end{pmatrix}
 \mathtt{D} \phi.
\end{align*}
Hence we have
\[
 \langle X \mathtt{D} \phi, v^* \rangle =  \langle \mathtt{D} \phi, d \varrho_{m+1}(X) v^* \rangle
\]
for all $X \in \fk$ and $v^* \in V_{m+1}^\vee$.
This implies the assertion.
\end{proof}

For a smooth $V_m$-valued function $f$ on $\fH_2$, we define a smooth $V_m$-valued function $\phi$ on $G_1$ by
\[
 \phi(g) = \rho_m(j(g,z_0))^{-1} f(g z_0),
\]
where $z_0 = \sqrt{-1} I_2 \in \fH_2$.
Then we have
\[
 \phi(g k) = \varrho_m^\vee(k)^{-1} \phi(g)
\]
for all $g \in G_1$ and $k \in K_{G_1}$.
Conversely, if a smooth $V_m$-valued function $\phi$ on $G_1$ satisfies the above condition, then we may define a smooth $V_m$-valued function $f$ on $\fH_2$ by
\[
 f(z) = \rho_m(j(g,z_0)) \phi(g),
\]
where $g \in G_1$ is any element such that $g z_0 = z$.

\begin{prop}
\label{p:delta-f}
If $f$ corresponds to $\phi$, then $\delta_m f$ corresponds to $-\frac{1}{4 \pi} \mathtt{D} \phi$.
\end{prop}

To prove this proposition, we need some preparation.
Define a smooth $V_m$-valued function $\varphi$ on $G_1$ by
\[
 \varphi(g) = f(g z_0).
\]
For $z \in \fH_2$, put
\[
 g_z =
 \boldsymbol{n} \mat{x_1}{x_2}{x_2}{x_3}
 \boldsymbol{m} \mat{1}{b}{}{1}
 \boldsymbol{m} \mat{a_1}{}{}{a_2}
 \in G_1,
\]
where $a_1, a_2 \in \R_+^\times$ and $b \in \R$ are the unique elements such that
\[
 y_1 = a_1^2 + a_2^2 b^2, \quad
 y_2 = a_2^2 b, \quad
 y_3 = a_2^2.
\]
Note that $g_z z_0 = z$.

\begin{lem}
\label{l:modform-varphi}
We have
\begin{align*}
 \mathtt{A}_1 \varphi(g_z) & =
 2 \Delta y_3^{-1} \frac{\partial f}{\partial y_1}(z), \\
 \mathtt{A}_2 \varphi(g_z) = \mathtt{A}_3 \varphi(g_z) & =
 \sqrt{\Delta} \left( 2 y_2 y_3^{-1} \frac{\partial f}{\partial y_1} + \frac{\partial f}{\partial y_2} \right)(z), \\
 \mathtt{A}_4 \varphi(g_z) & =
 2 \left( y_2^2 y_3^{-1} \frac{\partial f}{\partial y_1}
 + y_2 \frac{\partial f}{\partial y_2}
 + y_3 \frac{\partial f}{\partial y_3} \right)(z), \\
 \mathtt{B}_1 \varphi(g_z) = \mathtt{C}_1 \varphi(g_z) & =
 \Delta y_3^{-1} \frac{\partial f}{\partial x_1}(z), \\
 \mathtt{B}_2 \varphi(g_z) = \mathtt{C}_2 \varphi(g_z) & =
 \sqrt{\Delta} \left( 2 y_2 y_3^{-1} \frac{\partial f}{\partial x_1} + \frac{\partial f}{\partial x_2} \right)(z), \\
 \mathtt{B}_3 \varphi(g_z) = \mathtt{C}_3 \varphi(g_z) & =
 \left( y_2^2 y_3^{-1} \frac{\partial f}{\partial x_1}
 + y_2 \frac{\partial f}{\partial x_2}
 + y_3 \frac{\partial f}{\partial x_3} \right)(z).
\end{align*}
\end{lem}

\begin{proof}
Since $X \varphi = 0$ for $X \in \fk_0$, we have
\[
 (\mathtt{A}_2 - \mathtt{A}_3) \varphi = 
 (\mathtt{B}_1 - \mathtt{C}_1) \varphi = 
 (\mathtt{B}_2 - \mathtt{C}_2) \varphi = 
 (\mathtt{B}_3 - \mathtt{C}_3) \varphi = 0.
\]
From this and a direct computation, we can deduce the assertion.
\end{proof}

Put
\[
 j_z = j(g_z, z_0) = \mat{1}{}{-b}{1} \mat{a_1^{-1}}{}{}{a_2^{-1}}.
\]
We also denote by $j : \fg_{1,0} \times \fH_2 \rightarrow \M_2(\C)$ the map induced by the factor of automorphy.

\begin{lem}
\label{l:modform-j}
We have
\begin{align*}
 j(\Ad (g_z) \mathtt{A}_1, z) & = \Ad(j_z) \mat{-1}{0}{0}{0}, &
 j(\Ad (g_z) \mathtt{A}_2, z) & = \Ad(j_z) \mat{0}{0}{-1}{0}, \\
 j(\Ad (g_z) \mathtt{A}_3, z) & = \Ad(j_z) \mat{0}{-1}{0}{0}, &
 j(\Ad (g_z) \mathtt{A}_4, z) & = \Ad(j_z) \mat{0}{0}{0}{-1}, \\
 j(\Ad (g_z) \mathtt{B}_1, z) & = 0, &
 j(\Ad (g_z) \mathtt{B}_2, z) & = 0, \\
 j(\Ad (g_z) \mathtt{B}_3, z) & = 0, \\
 j(\Ad (g_z) \mathtt{C}_1, z) & = \Ad(j_z) \mat{\sqrt{-1}}{0}{0}{0}, &
 j(\Ad (g_z) \mathtt{C}_2, z) & = \Ad(j_z) \mat{0}{\sqrt{-1}}{\sqrt{-1}}{0}, \\
 j(\Ad (g_z) \mathtt{C}_3, z) & = \Ad(j_z) \mat{0}{0}{0}{\sqrt{-1}}.
\end{align*}
\end{lem}

\begin{proof}
The lemma follows from a direct computation.
\end{proof}

\begin{lem}
\label{l:modform-X-phi}
We have
\begin{align*}
 \langle \mathtt{X}_{2,0} \phi(g_z), v_{m,i}^* \rangle & = 
 \Delta y_3^{-1} \left\langle \left( \sqrt{-1} \frac{\partial f}{\partial x_1} + \frac{\partial f}{\partial y_1} \right)(z), \rho_m^\vee(j_z) v_{m,i}^* \right\rangle \\
 & + (m+2-i) \langle f(z), \rho_m^\vee(j_z) v_{m,i}^* \rangle, \\
 \langle \mathtt{X}_{1,1} \phi(g_z), v_{m,i}^* \rangle & =
 2 \sqrt{\Delta} y_2 y_3^{-1} \left\langle \left( \sqrt{-1} \frac{\partial f}{\partial x_1} + \frac{\partial f}{\partial y_1} \right)(z), \rho_m^\vee(j_z) v_{m,i}^* \right\rangle \\
 & + \sqrt{\Delta} \left\langle \left( \sqrt{-1} \frac{\partial f}{\partial x_2} + \frac{\partial f}{\partial y_2} \right)(z), \rho_m^\vee(j_z) v_{m,i}^* \right\rangle \\
 & + i \langle f(z), \rho_m^\vee(j_z) v_{m,i-1}^* \rangle \\
 & - (i-2) \langle f(z), \rho_m^\vee(j_z) v_{m,i+1}^* \rangle, \\
 \langle \mathtt{X}_{0,2} \phi(g_z), v_{m,i}^* \rangle & = 
 y_2^2 y_3^{-1} \left\langle \left( \sqrt{-1} \frac{\partial f}{\partial x_1} + \frac{\partial f}{\partial y_1} \right)(z), \rho_m^\vee(j_z) v_{m,i}^* \right\rangle \\
 & + y_2 \left\langle \left( \sqrt{-1} \frac{\partial f}{\partial x_2} + \frac{\partial f}{\partial y_2} \right)(z), \rho_m^\vee(j_z) v_{m,i}^* \right\rangle \\
 & + y_3 \left\langle \left( \sqrt{-1} \frac{\partial f}{\partial x_3} + \frac{\partial f}{\partial y_3} \right)(z), \rho_m^\vee(j_z) v_{m,i}^* \right\rangle \\
 & + (m+i) \langle f(z), \rho_m^\vee(j_z) v_{m,i}^* \rangle.
\end{align*}
\end{lem}

\begin{proof}
We have
\begin{align*}
 \langle \phi(g \exp(tX)), v^* \rangle & = 
 \langle \varphi(g \exp(tX)), \rho_m^\vee(j(g \exp(tX), z_0)) v^* \rangle \\
 & = \langle \varphi(g \exp(tX)), \rho_m^\vee(j(g \exp(tX) g^{-1}, g z_0)) \rho_m^\vee(j(g, z_0)) v^* \rangle,
\end{align*}
so that 
\[
 \langle X \phi(g_z), v^* \rangle
 = \langle X \varphi(g_z), \rho_m^\vee(j_z) v^* \rangle
 + \langle f(z), d \rho_m^\vee(j(\Ad(g_z) X, z)) \rho_m^\vee(j_z) v^* \rangle
\]
for $X \in \fg_{1,0}$, $z \in \fH_2$, and $v^* \in V_m^\vee$.
Hence, noting that
\begin{align*}
 \mathtt{X}_{2,0} & = \frac{1}{2} (\mathtt{A}_1 + \sqrt{-1} \mathtt{B}_1 + \sqrt{-1} \mathtt{C}_1), \\ 
 \mathtt{X}_{1,1} & = \frac{1}{2} (\mathtt{A}_2 + \mathtt{A}_3 + \sqrt{-1} \mathtt{B}_2 + \sqrt{-1} \mathtt{C}_2), \\
 \mathtt{X}_{0,2} & = \frac{1}{2} (\mathtt{A}_4 + \sqrt{-1} \mathtt{B}_3 + \sqrt{-1} \mathtt{C}_3),
\end{align*}
we can deduce the assertion from Lemmas \ref{l:modform-varphi} and \ref{l:modform-j}.
\end{proof}

Now we prove Proposition \ref{p:delta-f}.
Put
\begin{align*}
 h_1 & = 2 \sqrt{-1} \Delta y_3^{-1} \frac{\partial f}{\partial z_1}, \\
 h_2 & = 2 \sqrt{-1} \sqrt{\Delta} \left( 2 y_2 y_3^{-1} \frac{\partial f}{\partial z_1} + \frac{\partial f}{\partial z_2} \right), \\
 h_3 & = 2 \sqrt{-1} \left( y_2^2 y_3^{-1} \frac{\partial f}{\partial z_1} + y_2 \frac{\partial f}{\partial z_2} + y_3 \frac{\partial f}{\partial z_3} \right).
\end{align*}
Then by Lemma \ref{l:modform-X-phi}, we have
\begin{align*}
 \langle \mathtt{D} \phi(g_z), v_{m+1,0}^* \rangle & = 
 \langle h_2(z), \rho_m^\vee(j_z) v_{m,0}^* \rangle
 - 2 \langle h_1(z), \rho_m^\vee(j_z) v_{m,1}^* \rangle \\
 & - 2m \langle f(z), \rho_m^\vee(j_z) v_{m,1}^* \rangle, \\
 \langle \mathtt{D} \phi(g_z), v_{m+1,1}^* \rangle & =
 \langle h_3(z), \rho_m^\vee(j_z) v_{m,0}^* \rangle
 - \langle h_1(z), \rho_m^\vee(j_z) v_{m,2}^* \rangle \\
 & + m \langle f(z), \rho_m^\vee(j_z) v_{m,0}^* \rangle
 - m \langle f(z), \rho_m^\vee(j_z) v_{m,2}^* \rangle, \\
 \langle \mathtt{D} \phi(g_z), v_{m+1,2}^* \rangle & = 
 2 \langle h_3(z), \rho_m^\vee(j_z) v_{m,1}^* \rangle
 - \langle h_2(z), \rho_m^\vee(j_z) v_{m,2}^* \rangle \\
 & + 2m \langle f(z), \rho_m^\vee(j_z) v_{m,1}^* \rangle.
\end{align*}
We may write this as
\[
 \mathtt{D} \phi(g_z) = (A_1 P_m h_1 + A_2 P_m h_2 + A_3 P_m h_3 + m B P_m f)(z)
\]
with
\[
 A_1 = 
 \begin{pmatrix}
  0 & -2 & 0 \\
  0 & 0 & -1 \\
  0 & 0 & 0
 \end{pmatrix}, \quad
 A_2 = 
 \begin{pmatrix}
  1 & 0 & 0 \\
  0 & 0 & 0 \\
  0 & 0 & -1
 \end{pmatrix}, \quad
 A_3 = 
 \begin{pmatrix}
  0 & 0 & 0 \\
  1 & 0 & 0 \\
  0 & 2 & 0
 \end{pmatrix}, \quad
 B = A_1 + A_3, 
\]
and
\[
 P_m = \rho_m(j_z)^{-1} = a_1^m a_2^m 
 \begin{pmatrix}
  a_1^2 & & \\
  & a_1 a_2 & \\
  & & a_2^2
 \end{pmatrix}
 \begin{pmatrix}
  1 & & \\
  b & 1 & \\
  b^2 & 2b & 1
 \end{pmatrix}.
\]
Since
\begin{align*}
 P_{m+1}^{-1} A_1 P_m & = a_1^{-2}
 \begin{pmatrix}
  -2b & -2 & 0 \\
  b^2 & 0 & -1 \\
  0 & 2b^2 & 2b
 \end{pmatrix}
 = \Delta^{-1} y_3
 \begin{pmatrix}
  -2 y_2 y_3^{-1} & -2 & 0 \\
  y_2^2 y_3^{-2} & 0 & -1 \\
  0 & 2 y_2^2 y_3^{-2} & 2 y_2 y_3^{-1}
 \end{pmatrix}, \\
 P_{m+1}^{-1} A_2 P_m & = a_1^{-1} a_2^{-1}
 \begin{pmatrix}
  1 & 0 & 0 \\
  -b & 0 & 0 \\
  0 & -2b & -1
 \end{pmatrix}
 = \sqrt{\Delta}^{-1}
 \begin{pmatrix}
  1 & 0 & 0 \\
  - y_2 y_3^{-1} & 0 & 0 \\
  0 & -2 y_2 y_3^{-1} & -1
 \end{pmatrix}, \\
 P_{m+1}^{-1} A_3 P_m & = a_2^{-2}
 \begin{pmatrix}
  0 & 0 & 0 \\
  1 & 0 & 0 \\
  0 & 2 & 0
 \end{pmatrix}
 = y_3^{-1}
 \begin{pmatrix}
  0 & 0 & 0 \\
  1 & 0 & 0 \\
  0 & 2 & 0
 \end{pmatrix},
\end{align*}
we have
\begin{align*}
 P_{m+1}^{-1} (A_1 P_m h_1 + A_2 P_m h_2 + A_3 P_m h_3)
 & = 2 \sqrt{-1} \left( A_1 \frac{\partial f}{\partial z_1} + A_2 \frac{\partial f}{\partial z_2} + A_3 \frac{\partial f}{\partial z_3} \right), \\
 P_{m+1}^{-1} B P_m & = \Delta^{-1} (y_3 A_1 - 2 y_2 A_2 + y_1 A_3).
\end{align*}
Hence we have
\[
 \rho_{m+1}(j_z) \mathtt{D} \phi(g_z) 
 = - 4 \pi (A_1 \delta_{m,1} f + A_2 \delta_{m,2} f + A_3 \delta_{m,3} f)(z).
\]
This completes the proof of Proposition \ref{p:delta-f}.

Finally, we prove Proposition \ref{p:delta_mf}.
The proof consists of the following two parts.
\begin{itemize}
\item 
Let $f$ be a smooth $V_m$-valued function on $\fH_2$ such that
\[
 f(\gamma z) = \rho_m(j(\gamma,z)) f(z)
\]
for all $\gamma \in \Gamma$ and $z \in \fH_2$.
Then the corresponding function $\phi$ on $G_1$ satisfies
\[
 \phi(\gamma g) = \phi(g)
\]
for all $\gamma \in \Gamma$ and $g \in G_1$.
Hence, by definition, we have
\[
 \texttt{D} \phi(\gamma g) = \texttt{D} \phi(g)
\]
for all $\gamma \in \Gamma$ and $g \in G_1$.
From this and Proposition \ref{p:delta-f}, we deduce that 
\[
 \delta_m f(\gamma z) = \rho_{m+1}(j(\gamma,z)) \delta_m f(z)
\]
for all $\gamma \in \Gamma$ and $z \in \fH_2$.
\item
For $z \in \fH_2$, we write 
\[
 \Im(z)^{-1} = \mat{Y_1}{Y_2}{Y_2}{Y_3},
\]
so that 
\[
 Y_1 = \frac{y_3}{\Delta}, \quad
 Y_2 = - \frac{y_2}{\Delta}, \quad
 Y_3 = \frac{y_1}{\Delta}.
\]
Then we have
\begin{align*}
 \frac{\partial Y_1}{\partial z_1} & = - \frac{Y_1^2}{2 \sqrt{-1}}, &
 \frac{\partial Y_2}{\partial z_1} & = - \frac{Y_1 Y_2}{2 \sqrt{-1}}, &
 \frac{\partial Y_3}{\partial z_1} & = - \frac{Y_2^2}{2 \sqrt{-1}}, \\
 \frac{\partial Y_1}{\partial z_2} & = - \frac{2 Y_1 Y_2}{2 \sqrt{-1}}, &
 \frac{\partial Y_2}{\partial z_2} & = - \frac{Y_1 Y_3 + Y_2^2}{2 \sqrt{-1}}, &
 \frac{\partial Y_3}{\partial z_2} & = - \frac{2 Y_2 Y_3}{2 \sqrt{-1}}, \\
 \frac{\partial Y_1}{\partial z_3} & = - \frac{Y_2^2}{2 \sqrt{-1}}, &
 \frac{\partial Y_2}{\partial z_3} & = - \frac{Y_2 Y_3}{2 \sqrt{-1}}, &
 \frac{\partial Y_3}{\partial z_3} & = - \frac{Y_3^2}{2 \sqrt{-1}}.
\end{align*}
Hence, if $f$ is a finite sum of the form
\[
 f = \sum_k P_k f_k,
\]
where $P_k$ is a polynomial of degree at most $r$ in $Y_1, Y_2, Y_3$ and $f_k$ is a holomorphic function on $\fH_2$, then $\delta_{m,i} f$ is again a finite sum of this form, where $P_k$ is of degree at most $r+1$.
\end{itemize}
This completes the proof of Proposition \ref{p:delta_mf}.

\subsection{Periods}
\label{ss:periods-adelic}

Let $E$ be an imaginary quadratic field and put $E_\infty = E \otimes_\Q \R$.
Let $\A_E$ and $\A_{E,f}$ be the rings of adeles and finite adeles of $E$, respectively.
Put
\[
 \bG = (\GL_2 \times \Res_{E/\Q}(\mathbb{G}_m))/\mathbb{G}_m,
\]
where $\mathbb{G}_m$ embeds into $\GL_2 \times \Res_{E/\Q}(\mathbb{G}_m)$ diagonally.
For $(g,z) \in \GL_2 \times \Res_{E/\Q}(\mathbb{G}_m)$, we denote by $[g,z]$ its image in $\bG$.
Choose a trace zero element $\bi \in E^\times$ and put $u = \bi^2 \in \Q^\times$ (so that $u<0$).
Then we define an embedding $\iota : \bG \hookrightarrow \GSp_4$ by
\[
 \iota \left( \left[ \mat{a}{b}{c}{d}, x + y \bi \right] \right) = 
 \begin{pmatrix}
  a & 0 & b & 0 \\
  0 & a & 0 & -bu \\
  c & 0 & d & 0 \\
  0 & -\frac{c}{u} & 0 & d
 \end{pmatrix}
 \cdot 
 \frac{1}{x^2 - y^2 u}
 \begin{pmatrix}
  x & -y & 0 & 0 \\
  -yu & x & 0 & 0 \\
  0 & 0 & x & yu \\
  0 & 0 & y & x
 \end{pmatrix}
\]
(see also \S \ref{ss:weil-seesaw} below).
Put
\[
 \cP^{\bG}(\phi, \boldsymbol{\phi}) = \int_{Z(\A) \bG(\Q) \backslash \bG(\A)} \phi(\iota(g)) \boldsymbol{\phi}(g) \, dg
\]
(whenever the integral makes sense) for automorphic forms $\phi$ and $\boldsymbol{\phi}$ on $\GSp_4(\A)$ and $\bG(\A)$, respectively, where $Z = \{ [1,z] \mid z \in \mathbb{G}_m \}$ and $dg$ is the Tamagawa measure on $Z(\A) \backslash \bG(\A)$.

Now we describe the period explicitly when the automorphic forms correspond to classical modular forms.
Let $\cK$ and $\cK'$ be open compact subgroups of $\GSp_4(\A_f)$ and $\GL_2(\A_f)$, respectively, such that $\nu(\cK) = \nu(\cK') = \hat{\Z}^\times$.
Let $\phi \in \cN^r_{\rho_m}(\cK)$ and $\phi' \in \cN^{r'}_{2m+2}(\cK')$, where the subscript $2m+2$ corresponds to the $1$-dimensional representation of $\GL_1(\C)$ given by $z \mapsto z^{2m+2}$.
We assume that
\begin{itemize}
\item $\phi(zg) = \phi(g)$ for all $z \in \A^\times$ and $g \in \GSp_4(\A)$;
\item $\phi'(zg) = \phi'(g)$ for all $z \in \A^\times$ and $g \in \GL_2(\A)$.
\end{itemize}
Let $f \in N_{\rho_m}^r(\Gamma)$ and $f' \in N^{r'}_{2m+2}(\Gamma')$ be the modular forms corresponding to $\phi$ and $\phi'$, respectively, where $\Gamma = \Sp_4(\Q) \cap \cK$ and $\Gamma' = \SL_2(\Q) \cap \cK'$.
Let $\mu$ be a character of $\A_E^\times/E^\times$ such that
\begin{itemize}
\item $\mu|_{\A^\times} = 1$;
\item $\mu_\infty = 1$.
\end{itemize}
Let $\cU$ be an open compact subgroup of $\A_{E,f}^\times$ such that $\mu|_{\cU} = 1$ and fix a set of representatives $\{ a_1, \dots, a_M \}$ for $E^\times \backslash \A_{E,f}^\times / \cU$.
For $1 \le i \le M$, we define an automorphic form $\phi_i \in \cN^r_{\rho_m}(\cK_i)$ by $\phi_i(g) = \phi(g g_i)$, where $g_i = \iota([1,a_i]) \in \GSp_4(\A_f)$ and $\cK_i = g_i \cK g_i^{-1}$.
Let $f_i \in N_{\rho_m}^r(\Gamma_i)$ be the modular form corresponding to $\phi_i$, where $\Gamma_i = \Sp_4(\Q) \cap \cK_i$.
We consider the period $\cP^{\bG}(\ell \circ \tilde{\phi}, \boldsymbol{\phi})$, where
\begin{itemize}
\item $\ell: V_m \rightarrow \C$ is a linear form given by 
\[
 \ell(v) = \langle v, v_{m,0}^* + v_{m,2}^* \rangle;
\]
\item $\tilde{\phi}$ is an automorphic form on $\GSp_4(\A)$ valued in $V_m$ given by
\[
 \tilde{\phi}(g) = \phi(g g_0) 
\]
with 
\[
 g_0 = \boldsymbol{m} \mat{1}{0}{0}{|u|^{1/2}} \in \Sp_4(\R);
\]
\item $\boldsymbol{\phi}$ is an automorphic form on $\bG(\A)$ given by
\[
 \boldsymbol{\phi}([g,z]) = \mu(z) \overline{\phi'(g)}.
\]
\end{itemize}

\begin{prop}
\label{p:periods-classical}
Assume that $\cK'$ contains $\{ z I_2 \mid z \in \hat{\Z}^\times \}$ and
\[
 \phi(g \iota([k,z])) = \phi(g)
\]
for all $g \in \GSp_4(\A)$, $k \in \cK'$, and $z \in \cU$.
Assume further that the integral $\cP^{\bG}(\ell \circ \tilde{\phi}, \boldsymbol{\phi})$ is absolutely convergent.
Then we have
\begin{align*}
 \cP^{\bG}(\ell \circ \tilde{\phi}, \boldsymbol{\phi})
 & = 4 M^{-1} \vol(\Gamma' \backslash \fH_1)^{-1} |u|^{m/2} \\
 & \times \sum_{i=1}^M \mu(a_i) \int_{\Gamma' \backslash \fH_1} \left\langle f_i \mat{z}{0}{0}{-uz}, v_{m,0}^* - u v_{m,2}^* \right\rangle \overline{f'(z)} \Im(z)^{2m+2} \, d \nu(z),
\end{align*}
where $d \nu(z) = y^{-2} \, dx \, dy$ with the Lebesgue measures $dx, dy$ for $z = x + \sqrt{-1} y$.
\end{prop}

\begin{proof}
We have
\begin{align*}
 (\ell \circ \tilde{\phi})(gk)
 & = \langle \phi(g k g_0), v^*_{m,0} + v^*_{m,2} \rangle \\
 & = \langle \varrho_m^\vee(g_0^{-1} k g_0)^{-1} \phi(g g_0), v^*_{m,0} + v^*_{m,2} \rangle \\
 & = \langle \tilde{\phi}(g), \varrho_m(g_0^{-1} k g_0)(v^*_{m,0} + v^*_{m,2}) \rangle \\
\end{align*}
for $g \in \GSp_4(\A)$ and $k \in g_0 K_{G_1} g_0^{-1}$.
Since
\[
 g_0^{-1} \iota \left( \left[ \mat{a}{b}{c}{d}, x + \sqrt{-1} y \right] \right) g_0 =
 \begin{pmatrix}
  a & 0 & b & 0 \\
  0 & a & 0 & b \\
  c & 0 & d & 0 \\
  0 & c & 0 & d
 \end{pmatrix}
 \cdot \frac{1}{x^2 + y^2}
 \begin{pmatrix}
  x & -y & 0 & 0 \\
  y & x & 0 & 0 \\
  0 & 0 & x & -y \\
  0 & 0 & y & x
 \end{pmatrix}
\]
under the identification $E_\infty \simeq \C$ such that $\bi = \sqrt{-1} |u|^{1/2}$, and
\begin{align*}
 d \varrho_m(\mathtt{H}_1 + \mathtt{H}_2)(v^*_{m,0} + v^*_{m,2})
 & = (2m+2)(v^*_{m,0} + v^*_{m,2}), \\
 d \varrho_m(\mathtt{X}-\mathtt{Y})(v^*_{m,0} + v^*_{m,2})
 & = 0 
\end{align*}
by Lemma \ref{l:varrho_m}, we have
\[
 (\ell \circ \tilde{\phi})(g \iota([k,z])) = e^{\sqrt{-1} (2m+2) \theta} (\ell \circ \tilde{\phi})(g)
\]
for $g \in \GSp_4(\A)$, $k = \smat{\cos \theta}{\sin \theta}{-\sin \theta}{\cos \theta} \in \SO(2)$, and $z \in E_\infty^\times$.
Also, we have
\[
\boldsymbol{\phi}(g [k,z]) = e^{- \sqrt{-1} (2m+2) \theta} \boldsymbol{\phi}(g)
\]
for $g \in \bG(\A)$, $k = \smat{\cos \theta}{\sin \theta}{-\sin \theta}{\cos \theta} \in \SO(2)$, and $z \in E_\infty^\times$.
Hence, noting that 
\[
 \vol(\A^\times \GL_2(\Q) \backslash \GL_2(\A)) = \vol(\A^\times E^\times \backslash \A_E^\times) = 2, 
\]
we have
\begin{align*}
 \cP^{\bG}(\ell \circ \tilde{\phi}, \boldsymbol{\phi}) 
 & = \int_{\A^\times E^\times \backslash \A_E^\times} \int_{\A^\times \GL_2(\Q) \backslash \GL_2(\A)} (\ell \circ \tilde{\phi})(\iota([g,z])) \boldsymbol{\phi}([g,z]) \, dg \, dz \\
 & = 2 M^{-1} \sum_{i=1}^M \int_{\A^\times \GL_2(\Q) \backslash \GL_2(\A)} (\ell \circ \tilde{\phi})(\iota([g,a_i])) \mu(a_i) \overline{\phi'(g)} \, dg \\
 & = 2 M^{-1} \sum_{i=1}^M \mu(a_i) \int_{\A^\times \GL_2(\Q) \backslash \GL_2(\A)} (\ell \circ \tilde{\phi}_i)(\iota([g,1])) \overline{\phi'(g)} \, dg \\
 & = 4 M^{-1} \vol(\Gamma' \backslash \fH_1)^{-1} \sum_{i=1}^M \mu(a_i) \int_{\Gamma' \backslash \fH_1} (\ell \circ \tilde{\phi}_i)(\iota([g_z,1])) \overline{\phi'(g_z)} \, d \nu(z),
\end{align*}
where $\tilde{\phi}_i(g) = \tilde{\phi}(g g_i) = \phi_i(g g_0)$ and
\[
 g_z = \mat{1}{x}{0}{1} \mat{y^{1/2}}{0}{0}{y^{-1/2}}
\]
for $z = x + \sqrt{-1} y$.
Since
\begin{align*}
 \tilde{\phi}_i(\iota([g_z,1]))
 & = \rho_m(j(\iota([g_z,1]) g_0, z_0))^{-1} f_i(\iota([g_z,1]) g_0 z_0) \\
 & = \rho_m \mat{y^{1/2}}{0}{0}{|u|^{1/2} y^{1/2}} f_i \mat{z}{0}{0}{|u|z},
\end{align*}
we have
\begin{align*}
 (\ell \circ \tilde{\phi}_i)(\iota([g_z,1]))
 & = \left\langle \rho_m \mat{y^{1/2}}{0}{0}{|u|^{1/2} y^{1/2}} f_i \mat{z}{0}{0}{|u|z}, v^*_{m,0} + v^*_{m,2} \right\rangle \\
 & = \left\langle f_i \mat{z}{0}{0}{|u|z}, \rho_m^\vee \mat{y^{1/2}}{0}{0}{|u|^{1/2} y^{1/2}}^{-1} (v^*_{m,0} + v^*_{m,2}) \right\rangle \\
 & = \left\langle f_i \mat{z}{0}{0}{|u|z}, |u|^{m/2} y^{m+1} v^*_{m,0} + |u|^{m/2+1} y^{m+1} v^*_{m,2} \right\rangle.
\end{align*}
Similarly, we have
\[
 \phi'(g_z) = y^{m+1} f'(z).
\]
This implies the assertion.
\end{proof}

\section{Cohomological representations}
\label{s:cohrep}

In this section, we describe some unitary representations which contribute to the cohomology.
Put $G = \GSp_4(\R)$ and $G_1 = \Sp_4(\R)$.
Define a Cartan involution $\theta$ of $G$ by $\theta(g) = {}^t g^{-1}$.
Let $K_{G_1}$ be a maximal compact subgroup of $G_1$ given by
\[
 K_{G_1} = \{ g \in G_1 \mid \theta(g) = g \} \simeq \U(2)
\]
and put $K_G = Z_G \cdot K_{G_1}$, where $Z_G$ is the center of $G$.
Let $\fg_0$ and $\fg_{1,0}$ be the Lie algebras of $G$ and $G_1$, respectively.
We have a Cartan decomposition
\[
 \fg_{1,0} = \fk_0 \oplus \fp_0,
\]
where $\fk_0$ and $\fp_0$ are the $+1$ and $-1$ eigenspaces of $\theta$, respectively.
Note that $\fk_0$ is the Lie algebra of $K_{G_1}$.
Let $\ft_0$ be a Cartan subalgebra of $\fk_0$ given by
\[
 \ft_0 = \left\{
 \begin{pmatrix}
  & & \vartheta_1 & \\
  & & & \vartheta_2 \\
  -\vartheta_1 & & & \\
  & -\vartheta_2 & & 
 \end{pmatrix}
 \; \middle| \; \vartheta_1, \vartheta_2 \in \R \right\}.
\]
Let $\fg$, $\fg_1$, $\fk$, $\fp$, $\ft$ be the complexifications of $\fg_0$, $\fg_{1,0}$, $\fk_0$, $\fp_0$, $\ft_0$, respectively.
Let $e_1,e_2$ be a basis of $\ft^*$ given by 
\[
 e_i
 \begin{pmatrix}
  & & \vartheta_1 & \\
  & & & \vartheta_2 \\
  -\vartheta_1 & & & \\
  & -\vartheta_2 & & 
 \end{pmatrix} = \sqrt{-1} \vartheta_i, 
\]
via which we identify $\ft^*$ with $\C^2$.
For any subspace $\ff$ of $\fg_1$ which is stable under the adjoint action of $\ft$, we write $\Delta(\ff)$ for the set of roots of $\ft$ in $\ff$ and put
\[
 \rho(\ff) = \frac{1}{2} \sum_{\alpha \in \Delta(\ff)} \alpha.
\]
Let $\fp^+$ and $\fp^-$ be the $+\sqrt{-1}$ and $-\sqrt{-1}$ eigenspaces of the linear map $J:\fp \rightarrow \fp$ induced by the complex structure on $G_1/K_{G_1}$, respectively, so that
\[
 \Delta(\fp^+) = \{ e_1 + e_2, 2 e_1, 2 e_2 \}, \quad
 \Delta(\fp^-) = \{ - e_1 - e_2, -2 e_1, - 2 e_2 \}.
\]
Let $\fq = \fl \oplus \fu$ and $\fq' = \fl' \oplus \fu'$ be $\theta$-stable parabolic subalgebras of $\fg_1$ given by 
\begin{align*}
 \fl & = \left\{
 \begin{pmatrix}
  & & \vartheta & \\
  & a & & b \\
  -\vartheta & & & \\
  & c & & -a
 \end{pmatrix}
 \; \middle| \; \vartheta, a, b, c \in \C \right\} \simeq \mathfrak{u}(1) \oplus \mathfrak{sl}_2, & 
 \Delta(\fu) & = \{ e_1-e_2, e_1+e_2, 2e_1 \}, \\
 \fl' & = \left\{ 
 \begin{pmatrix}
  a & & b & \\
  & & & \vartheta \\
  c & & -a & \\
  & -\vartheta & &
 \end{pmatrix}
 \; \middle| \; \vartheta, a, b, c \in \C \right\} \simeq \mathfrak{sl}_2 \oplus \mathfrak{u}(1), & 
 \Delta(\fu') & = \{ e_1-e_2, -e_1-e_2, -2e_2 \}.
\end{align*}
We denote by $A_\fq$ and $A_{\fq'}$ the cohomological representations as defined in \cite{vz}.
These are irreducible unitary $(\fg_1,K_{G_1})$-modules but we may extend them to $(\fg,K_G)$-modules by letting $Z_G \simeq \R^\times$ act trivially.
For any $(\fg, K_G)$-module $\pi$, we denote by $H^*(\fg, K_G; \pi)$ the relative Lie algebra cohomology of $\pi$.

\begin{lem}
We have
\begin{align*}
 \dim H^{p,q}(\fg, K_G; A_{\fq}) & =
 \begin{cases}
  1 & \text{if $(p,q) = (2,0), (3,1)$;} \\
  0 & \text{otherwise,}
 \end{cases} \\
 \dim H^{p,q}(\fg, K_G; A_{\fq'}) & = 
 \begin{cases}
  1 & \text{if $(p,q) = (0,2), (1,3)$;} \\
  0 & \text{otherwise.}
 \end{cases}
\end{align*}
\end{lem}

\begin{proof}
By definition, we have
\[
 H^{p,q}(\fg, K_G; A_{\fq}) = H^{p,q}(\fg_1, K_{G_1}; A_{\fq}), \quad
 H^{p,q}(\fg, K_G; A_{\fq'}) = H^{p,q}(\fg_1, K_{G_1}; A_{\fq'}).
\]
Hence, noting that 
\[
 \Delta(\fu \cap \fp) = \{ e_1 + e_2, 2 e_1 \}, \quad 
 \Delta(\fu' \cap \fp) = \{ - e_1 - e_2, - 2 e_2 \},
\]
we can deduce the assertion from \cite[Proposition 6.19]{vz}.
\end{proof}

\begin{lem}
\label{l:Aq-K-types}
We have
\[
 A_{\fq}|_{K_{G_1}} \simeq \bigoplus_{a, b \ge 0} \Sym^{2a+2} \otimes {\det}^{b+1}, \quad
 A_{\fq'}|_{K_{G_1}} \simeq \bigoplus_{a, b \ge 0} (\Sym^{2a+2} \otimes {\det}^{b+1})^\vee,
\]
where $\Sym^a \otimes {\det}^b$ denotes the irreducible representation of $\U(2)$ with highest weight $(a+b,b)$.
\end{lem}

\begin{proof}
We only compute $A_{\fq}|_{K_{G_1}}$.
Noting that $\fl \cap \fk = \ft$, we take $\Delta^+(\fk) = \Delta(\fu \cap \fk) = \{ e_1 - e_2 \}$ as a positive system of $\Delta(\fk)$.
Put
\[
 \mu(\fq) = 2 \rho(\fu \cap \fp) = (3,1).
\]
Then by \cite[Theorem 2.5]{vz}, $A_{\fq}$ has a minimal $K_{G_1}$-type with highest weight $\mu(\fq)$ and the other $K_{G_1}$-types have highest weights of the form 
\[
 \mu(\fq) + \sum_{\beta \in \Delta(\fu \cap \fp)} n_\beta \beta
\]
for some non-negative integer $n_\beta$.
For non-negative integers $a,b$, put
\[
 \mu_{a,b} = (2a+b,b).
\]
It remains to compute the multiplicity
\[
 m_{a,b} = \dim \Hom_{K_{G_1}}(\tau_{a,b}, A_{\fq}),
\]
where $\tau_{a,b}$ is the irreducible representation of $K_{G_1}$ with highest weight $\mu(\fq) + \mu_{a,b}$.

Let $W = \langle s \rangle \simeq \Z/2\Z$ be the Weyl group of $\Delta(\fk)$, where $s$ is the simple reflection with respect to $e_1-e_2$.
Define a subset $W^1$ of $W$ by 
\[
 W^1 = \{ w \in W \mid \Delta^+(w) \subset \Delta(\fu \cap \fk) \}, \quad
 \Delta^+(w) = \{ \alpha \in \Delta^+(\fk) \mid w^{-1} \alpha < 0 \}.
\]
In fact, we have $W^1 = W$.
For $\mu \in \ft^*$, let $\cP(\mu)$ be the multiplicity of $\mu$ in the symmetric algebra of $\fu \cap \fp$.
Since $\Delta(\fu \cap \fp) = \{ e_1 + e_2, 2 e_1 \}$, we have
\[
 \cP(\mu) =
 \begin{cases}
  1 & \text{if $\mu = \mu_{a,b}$ for some $a,b \ge 0$;} \\
  0 & \text{otherwise.}
 \end{cases}
\]
Put
\[
 \rho_{K_{G_1}} = \frac{1}{2} \sum_{\alpha \in \Delta^+(\fk)} \alpha = \left( \frac{1}{2}, -\frac{1}{2} \right).
\]
Then by the Blattner multiplicity formula \cite[Theorem 8.29]{kv}, we have
\begin{align*}
 m_{a,b} & = \sum_{w \in W^1} \det(w) \cdot \cP(w(\mu(\fq) + \mu_{a,b} + \rho_{K_{G_1}}) - (\mu(\fq) + \rho_{K_{G_1}})) \\
 & = \cP(2a+b,b) - \cP(b-3,2a+b+3) \\
 & = 1
\end{align*}
for $a, b \ge 0$.
This completes the proof.
\end{proof}

\section{Weil representations}
\label{s:weil}

In this section, we introduce Weil representations for symplectic-orthogonal and unitary dual pairs used in this paper. 
Let $F$ be a local field of characteristic zero and fix a nontrivial additive character $\psi$ of $F$.
For $t \in F^\times$, we define another nontrivial additive character $t \psi$ by $(t \psi)(x) = \psi(tx)$.

\subsection{The symplectic-orthogonal case}
\label{ss:weil-sp-o}

Let $K$ be an \'etale quadratic algebra over $F$.
We denote by $x \mapsto \bar{x}$ the nontrivial automorphism of $K$ over $F$.
Let $\tr_{K/F}$ and $\N_{K/F}$ be the trace and norm maps from $K$ to $F$, respectively.
Let $\xi_{K/F}$ be the (possibly trivial) quadratic character of $F^{\times}$ associated to $K/F$ by class field theory.

Let $V = K$ be the $2$-dimensional (right) $F$-vector space equipped with the quadratic form $\langle x, y \rangle = \tr_{K/F}(\bar{x} y)$.
Let $\GL(V)$ act on $V$ on the left.
Then we have $\GO(V) = K^{\times} \rtimes \langle h_0 \rangle$, where $h_0 \in \O(V) \smallsetminus \SO(V)$ is given by $h_0(x) = \bar{x}$.
Let $\nu : \GO(V) \rightarrow F^\times$ be the similitude character, so that $\nu(h) = \N_{K/F}(h)$ for $h \in K^\times$.
Let $W = F^4$ be the $4$-dimensional (left) $F$-vector space equipped with the symplectic form
\[
 \langle (x_1, \dots, x_4), (y_1, \dots, y_4) \rangle = x_1 y_3 + x_2 y_4 - x_3 y_1 - x_4 y_2.
\]
Let $\GL(W)$ act on $W$ on the right.
Then we have
\[
 \GSp(W) = \left\{ g \in \GL_4(F) \; \middle| \; g  \mat{}{I_2}{-I_2}{} {}^t g = \nu(g) \cdot \mat{}{I_2}{-I_2}{}, \; \nu(g) \in F^\times \right\}.
\]

Let $\cS(K^2)$ be the space of Schwartz functions on $K^2$ and put
\[
 q(x) = \mat{\N_{K/F}(x_1)}{\frac{1}{2} \tr_{K/F}(\bar{x}_1 x_2)}{\frac{1}{2} \tr_{K/F}(\bar{x}_1 x_2)}{\N_{K/F}(x_2)}
\]
for $x = (x_1, x_2) \in K^2$.
As in \cite[\S 5]{kudla-splitting}, we define a Weil representation $\omega_\psi$ of $\Sp(W) \times \O(V)$ on $\cS(K^2)$ relative to $\psi$, so that
\begin{align*}
 \omega_\psi(\boldsymbol{m}(a)) \varphi(x) & = \xi_{K/F}(\det(a)) |\det(a)| \varphi(x a), \\ 
 \omega_\psi(\boldsymbol{n}(b)) \varphi(x) & = \psi(\tr(q(x) b)) \varphi(x), \\
 \omega_\psi \mat{}{-I_2}{I_2}{} \varphi(x) & = \xi_{K/F}(-1) \int_{K^2} \varphi(y) \psi(- \tr_{K/F}(\bar{x} {}^t y)) \, dy, \\
 \omega_\psi(h) \varphi(x) & = \varphi(h^{-1} x)
\end{align*}
for $\varphi \in \cS(K^2)$, $x \in K^2$, $a \in \GL_2(F)$, $b \in \Sym_2(F)$, and $h \in \O(V)$, where $dy$ is the product of the self-dual Haar measures on $K$ with respect to $\psi \circ \tr_{K/F}$.
Note that 
\[
 \omega_{t \psi}(g) = \omega_\psi \left( \mat{I_2}{}{}{t^{-1} I_2} g \mat{I_2}{}{}{t I_2} \right)
\]
for $t \in F^\times$ and $g \in \Sp(W)$.

Let $\cS(K^2 \times F^\times)$ be the space of Schwartz functions on $K^2 \times F^\times$.
As in \cite[\S I.3]{wald85}, we define a representation $\Omega_\psi$ of $\GSp(W) \times \GO(V)$ on $\cS(K^2 \times F^\times)$ by
\begin{align*}
 \Omega_\psi(g) \varphi (x,t) & = (\omega_{t \psi}(g) \varphi_t)(x), \\ 
 \Omega_\psi \mat{I_2}{}{}{\nu I_2} \varphi(x,t) & = |\nu|^{-1} \varphi(x,\nu^{-1} t), \\
 \Omega_\psi(h) \varphi(x,t) & = \varphi(h^{-1} x, \nu(h) t)
\end{align*}
for $\varphi \in \cS(K^2 \times F^\times)$, $x \in K^2$, $t \in F^\times$, $g \in \Sp(W)$, $\nu \in F^\times$, and $h \in \GO(V)$, where $\varphi_t \in \cS(K^2)$ is given by $\varphi_t(x) = \varphi(x,t)$.
When $F = \R$, $K = \C$, and $\psi(x) = e^{2 \pi \sqrt{-1} x}$, we denote by $S(K^2 \times F^\times)$ the subspace of Schwartz functions $\varphi$ on $K^2 \times F^\times = \C^2 \times \R^\times$ of the form
\[
 \varphi(x,t) = P(x_1, \bar{x}_1, x_2, \bar{x}_2) \phi(t) e^{- 2 \pi |t| (x_1 \bar{x}_1 + x_2 \bar{x}_2)}
\]
for $x = (x_1, x_2) \in \C^2$ and $t \in \R^\times$, where $P$ is a polynomial and $\phi$ is a smooth compactly supported function on $\R^\times$.
Then $S(K^2 \times F^\times)$ is invariant under the action of $(\mathfrak{gsp}(W), K_{\GSp(W)}) \times \GO(V)$, where $\mathfrak{gsp}(W)$ is the complexified Lie algebra of $\GSp(W)$ and $K_{\GSp(W)}$ is the standard maximal compact modulo center subgroup of $\GSp(W)$.
When $F$ is non-archimedean, we also write $S(K^2 \times F^\times) = \cS(K^2 \times F^\times)$ for uniformity of notation.

\subsection{The unitary case}
\label{ss:weil-u}

Let $B$ be a quaternion algebra over $F$ such that $K$ embeds into $B$.
We denote by $*$ the main involution on $B$.
Let $\tr_{B/F}$ and $\nu$ be the reduced trace and norm maps from $B$ to $F$, respectively.
Write $B = K + K \bi$ and $K = F + F \bj$ for some trace zero elements $\bi \in B^\times$ and $\bj \in K^\times$.
Then $E = F + F \bi$ is an \'etale quadratic algebra over $F$ such that $B = E + E \bj$.
Let $\pr : B \rightarrow E$ be the associated projection.
We denote by $x \mapsto \bar{x}$ the nontrivial automorphism of $E$ over $F$.

Let $\bV = B$ be the $2$-dimensional (right) $E$-vector space equipped with the hermitian form $(x, y) = 2 \pr(x^* y)$.
Let $\GL(\bV)$ act on $\bV$ on the left.
Then we have $\GU(\bV) = (B^\times \times E^\times)/F^\times$, where $F^\times$ embeds into $B^\times \times E^\times$ diagonally.
For $(h',z) \in B^\times \times E^\times$, we denote by $[h',z]$ its image in $\GU(\bV)$.
Let $\nu : \GU(\bV) \rightarrow F^\times$ be the similitude character, so that $\nu([h',z]) = \nu(h') \N_{E/F}(z)^{-1}$.
Let $\bW = E^2$ be the $2$-dimensional (left) $E$-vector space equipped with the skew-hermitian form
\[
 ((x_1, x_2), (y_1, y_2)) = x_1 \bar{y}_2 - x_2 \bar{y}_1.
\]
Let $\GL(\bW)$ act on $\bW$ on the right.
Then we have
\[
 \GU(\bW) = \left\{ g \in \GL_2(E) \; \middle| \; g  \mat{}{1}{-1}{} {}^t \bar{g} = \nu(g) \cdot \mat{}{1}{-1}{}, \; \nu(g) \in F^\times \right\}.
\]
We may also identify $\GU(\bW)$ with $(\GL_2(F) \times E^\times)/F^\times$, where $F^\times$ embeds into $\GL_2(F) \times E^\times$ diagonally.
For $(g',z) \in \GL_2(F) \times E^\times$, we denote by $[g',z]$ its image in $\GU(\bW)$ under this identification.

Let $\cS(B)$ be the space of Schwartz functions on $B$.
As in \cite[\S 5]{kudla-splitting}, we define a Weil representation $\omega_\psi$ of $\U(\bW) \times \U(\bV)$ on $\cS(B)$ relative to $\psi$ and the trivial character of $E^\times$, so that 
\begin{align*}
 \omega_\psi \mat{a}{}{}{\bar{a}^{-1}} \varphi(x) & = |a| \varphi(x a), \\ 
 \omega_\psi \mat{1}{b}{}{1} \varphi(x) & = \psi(\nu(x) b) \varphi(x), \\
 \omega_\psi \mat{}{-1}{1}{} \varphi(x) & = \gamma_B \int_B \varphi(y) \psi(-\tr_{B/F}(x^* y)) \, dy, \\
 \omega_\psi(h) \varphi(x) & = \varphi(h^{-1} x)
\end{align*}
for $\varphi \in \cS(B)$, $x \in B$, $a \in E^\times$, $b \in F$, and $h \in \U(\bV)$, where
\[
 \gamma_B = 
 \begin{cases}
  +1 & \text{if $B$ is split;} \\
  -1 & \text{if $B$ is ramified}
 \end{cases}
\]
and $dy$ is the self-dual Haar measure on $B$ with respect to $\psi \circ \tr_{B/F}$.
Note that 
\[
 \omega_{t \psi}(g) = \omega_\psi \left( \mat{1}{}{}{t^{-1}} g \mat{1}{}{}{t} \right)
\]
for $t \in F^\times$ and $g \in \U(\bW)$.

Let $\cS(B \times F^\times)$ be the space of Schwartz functions on $B \times F^\times$.
As in \cite[\S I.3]{wald85}, we define a representation $\Omega_\psi$ of $\GU(\bW) \times \GU(\bV)$ on $\cS(B \times F^\times)$ by
\begin{align*}
 \Omega_\psi(g) \varphi (x,t) & = (\omega_{t \psi}(g) \varphi_t)(x), \\  
 \Omega_\psi \mat{1}{}{}{\nu} \varphi(x,t) & = |\nu|^{-1} \varphi(x,\nu^{-1} t), \\
 \Omega_\psi(h) \varphi(x,t) & = \varphi(h^{-1} x, \nu(h) t)
\end{align*}
for $\varphi \in \cS(B \times F^\times)$, $x \in B$, $t \in F^\times$, $g \in \U(\bW)$, $\nu \in F^\times$, and $h \in \GU(\bV)$, where $\varphi_t \in \cS(B)$ is given by $\varphi_t(x) = \varphi(x,t)$.
When $F = \R$, $K = \C$, $B = \mathbb{H}$ (so that $E = \C$), and $\psi(x) = e^{2 \pi \sqrt{-1} x}$, we denote by $S(B \times F^\times)$ the subspace of Schwartz functions $\varphi$ on $B \times F^\times = \mathbb{H} \times \R^\times$ of the form
\[
 \varphi(x,t) = P(x_1, x_2, x_3, x_4) \phi(t) e^{- 2 \pi |t| \nu(x)}
\]
for $x = x_1 + x_2 \bi + x_3 \bj + x_4 \bi \bj$ with $x_1, x_2, x_3, x_4 \in \R$ and $t \in \R^\times$, where $P$ is a polynomial and $\phi$ is a smooth compactly supported function on $\R^\times$.
Then $S(B \times F^\times)$ is invariant under the action of $(\mathfrak{gu}(\bW), K_{\GU(\bW)}) \times \GU(\bV)$, where $\mathfrak{gu}(\bW)$ is the complexified Lie algebra of $\GU(\bW)$ and $K_{\GU(\bW)}$ is the standard maximal compact modulo center subgroup of $\GU(\bW)$.
When $F$ is non-archimedean, we also write $S(B \times F^\times) = \cS(B \times F^\times)$ for uniformity of notation.

In the rest of this subsection, we denote by $c$ the nontrivial automorphism of $E$ over $F$ to distinguish it from the complex conjugation.
Then $c$ induces $E$-anti-linear automorphisms of $\bW$ and $\bV$, which in turn induce automorphisms of $\GU(\bW)$ and $\GU(\bV)$.
Explicitly, we have
\[
 g^c = [g',z^c], \quad
 h^c = [\bj h' \bj^{-1}, z^c]
\]
for $g = [g',z] \in \GU(\bW)$, $h = [h', z] \in \GU(\bV)$ with $g' \in \GL_2(F)$, $h' \in B^\times$, and $z \in E^\times$.
For $\varphi \in \cS(B \times F^\times)$, we define Schwartz functions $\varphi^c, \bar{\varphi} \in \cS(B \times F^\times)$ by
\[
 \varphi^c(x,t) = \varphi(x^c,t), \quad 
 \bar{\varphi}(x,t) = \overline{\varphi(x,t)},
\]
where the bar denotes the complex conjugation.
Then a direct computation shows the following.

\begin{lem}
\label{l:weil-conjugation}
We have
\[
 [\Omega_\psi(g,h)\varphi]^c = \Omega_\psi(g^c,h^c) \varphi^c, \quad
 \overline{\Omega_\psi(g,h)\varphi} = \Omega_\psi(\delta g \delta, h) \bar{\varphi}
\]
for $\varphi \in \cS(B \times F^\times)$, $g \in \GU(\bW)$, and $h \in \GU(\bV)$, where 
\[
 \delta = \left[ \mat{-1}{}{}{1}, 1 \right] \in \GU(\bW).
\]
\end{lem}

\subsection{Change of polarizations}
\label{ss:weil-pft}

When $K$ is split, we also need to introduce a different model of the Weil representation for the unitary dual pair.
We identify $B$ with $\M_2(F)$ via the map
\[
 a + b \bi + c \bj + d \bi \bj \mapsto \mat{a + cJ'}{b - dJ'}{(b + dJ')u}{a - cJ'},
\]
where we write $\bj^2 = (J')^2$ for some $J' \in F^\times$ and put $u = \bi^2 \in F^\times$.
Under this identification, we have
\[
 \mat{x_1}{x_2}{x_3}{x_4}^* = \mat{x_4}{-x_2}{-x_3}{x_1}.
\]
Take a basis $\bv_1, \bv_2$ of $\bV$ given by
\[
 \bv_1 = \mat{1}{0}{0}{0}, \quad
 \bv_2 = \mat{0}{0}{1}{0}.
\]
Note that
\[
 (\bv_1, \bv_1) = (\bv_2, \bv_2) = 0, \quad
 (\bv_1, \bv_2) = \frac{1}{u} \bi.
\]
We define a partial Fourier transform
\[
 \cF_\psi : \cS(B) \rightarrow \cS(E^2)
\]
by
\[
 \cF_\psi(\varphi)(x_1, x_2) = \int_E \varphi(\bv_1 y + \bv_2 x_1) \psi_E(x_2 \bar{y}) \, dy,
\]
where $\psi_E$ is a nontrivial additive character of $E$ given by 
\[
 \psi_E(y) = \psi \left(\frac{1}{2u} \tr_{E/F}(\bi y) \right)
\]
and $dy$ is the self-dual Haar measure on $E$ with respect to $\psi_E$.
This induces the Weil representation $\hat{\omega}_\psi$ of $\U(\bW)$ on $\cS(E^2)$, i.e.
\[
 \hat{\omega}_\psi(g) = \cF_\psi \circ \omega_\psi(g) \circ \cF_\psi^{-1}.
\]

\begin{lem}
\label{l:weil-pft}
We have
\[
 \hat{\omega}_\psi(g) \hat{\varphi}(x) = \hat{\varphi}(xg)
\]
for $\hat{\varphi} \in \cS(E^2)$, $x \in E^2$, and $g \in \U(\bW)$.
\end{lem}

\begin{proof}
Write $\hat{\varphi} = \cF_\psi(\varphi)$ with $\varphi \in \cS(B)$.
For $a \in E^\times$, we have
\begin{align*}
 \hat{\omega}_\psi \mat{a}{}{}{\bar{a}^{-1}} \hat{\varphi}(x_1, x_2)
 & = |a| \int_E \varphi((\bv_1 y + \bv_2 x_1) a) \psi_E(x_2 \bar{y}) \, dy \\
 & = \int_E \varphi(\bv_1 y + \bv_2 x_1 a) \psi_E(x_2 \bar{y} \bar{a}^{-1}) \, dy \\
 & = \hat{\varphi}(x_1 a, x_2 \bar{a}^{-1}).
\end{align*}
For $b \in F$, we have 
\begin{align*}
 \hat{\omega}_\psi\mat{1}{b}{}{1} \hat{\varphi}(x_1, x_2) 
 & = \int_E \varphi(\bv_1 y + \bv_2 x_1) \psi \left( \frac{1}{2} (\bv_1 y + \bv_2 x_1, \bv_1 y + \bv_2 x_1) b \right) \psi_E(x_2 \bar{y}) \, dy \\
 & = \int_E \varphi(\bv_1 y + \bv_2 x_1) \psi_E(\bar{y} x_1 b) \psi_E(x_2 \bar{y}) \, dy \\
 & = \hat{\varphi}(x_1, x_1 b + x_2).
\end{align*}
Finally, we have
\begin{align*}
 \hat{\omega}_\psi\mat{}{-1}{1}{} \hat{\varphi}(x_1, x_2)
 & = \int_E \int_{E^2} \varphi(\bv_1 y_1 + \bv_2 y_2) \psi \left( - \frac{1}{2} \tr_{E/F} (\bv_1 y + \bv_2 x_1, \bv_1 y_1 + \bv_2 y_2) \right) dy_1 \, dy_2 \\
 & \qquad \times \psi_E(x_2 \bar{y}) \, dy \\
 & = \int_E \int_{E^2} \varphi(\bv_1 y_1 + \bv_2 y_2) \psi_E(- \bar{y} y_2 - x_1 \bar{y}_1) \, dy_1 \, dy_2 \, \psi_E(x_2 \bar{y}) \, dy \\
 & = \int_E \varphi(\bv_1 y_1 + \bv_2 x_2) \psi_E(- x_1 \bar{y}_1) \, dy_1 \\
 & = \hat{\varphi}(x_2, -x_1).
\end{align*}
This completes the proof.
\end{proof}

We extend $\cF_\psi$ to a partial Fourier transform
\[
 \cF_\psi: \cS(B \times F^\times) \rightarrow \cS(E^2 \times F^\times)
\]
by 
\[
 \cF_\psi(\varphi)(x,t) = \cF_{t \psi}(\varphi_t)(x).
\]
This induces the Weil representation $\hat{\Omega}_\psi$ of $\GU(\bW) \times \GU(\bV)$ on $\cS(E^2 \times F^\times)$, i.e.
\[
 \hat{\Omega}_\psi(g,h) = \cF_\psi \circ \Omega_\psi(g,h) \circ \cF_\psi^{-1}.
\]

\begin{lem}
\label{l:weil-pft2}
We have
\begin{align*}
 \hat{\Omega}_\psi(g) \hat{\varphi}(x,t) & = \hat{\varphi}(xg, \nu(g)^{-1} t), \\
 \hat{\Omega}_\psi(h_1) \hat{\varphi}(x,t) & = |a| \hat{\varphi}(x, a t), \\
 \hat{\Omega}_\psi(h_2) \hat{\varphi}(x,t) & = \psi_E(t \bar{x}_1 x_2 b) \hat{\varphi}(x, t) 
\end{align*}
for $\hat{\varphi} \in \cS(E^2 \times F^\times)$, $x = (x_1, x_2) \in E^2$, $t \in F^\times$, $g \in \GU(\bW)$, and
\[
 h_1 = \left[ \mat{a}{}{}{1}, 1 \right],
 h_2 = \left[ \mat{1}{b}{}{1}, 1 \right] \in \GU(\bV) 
\]
with $a \in F^\times$, $b \in F$.
\end{lem}

\begin{proof}
Write $\hat{\varphi} = \cF_\psi(\varphi)$ with $\varphi \in \cS(B \times F^\times)$.
For $g \in \U(\bW)$, we have
\begin{align*}
 \hat{\Omega}_\psi(g) \hat{\varphi}(x, t) & = \cF_{t \psi}(\omega_{t \psi}(g) \varphi_t)(x) \\
 & = \hat{\omega}_{t \psi}(g) \cF_{t \psi} (\varphi_t) (x) \\
 & = \cF_{t \psi} (\varphi_t) (xg) \\
 & = \hat{\varphi}(xg, t)
\end{align*}
by Lemma \ref{l:weil-pft}.
For $\nu \in F^\times$, we have
\begin{align*}
 \hat{\Omega}_\psi \mat{1}{}{}{\nu} \hat{\varphi}(x_1, x_2, t)
 & = |\nu|^{-1} \cF_{t\psi}(\varphi_{\nu^{-1} t})(x_1, x_2) \\
 & = |\nu|^{-1} \int_E \varphi_{\nu^{-1} t}(\bv_1 y + \bv_2 x_1) \psi_E(t x_2 \bar{y}) |t| \, dy \\
 & = \cF_{\nu^{-1} t \psi}(\varphi_{\nu^{-1} t})(x_1, \nu x_2) \\
 & = \hat{\varphi}(x_1,\nu x_2, \nu^{-1} t).
\end{align*}
For $h \in \GU(\bV)$, we have
\[
 \hat{\Omega}_\psi(h) \hat{\varphi}(x_1, x_2, t) = \int_E \varphi_{\nu(h) t}(h^{-1}(\bv_1 y + \bv_2 x_1)) \psi_E(t x_2 \bar{y}) |t| \, dy.
\]
If $h = \left[ \smat{a}{}{}{1}, 1 \right]$ with $a \in F^\times$, then we have
\begin{align*}
 \hat{\Omega}_\psi(h) \hat{\varphi}(x_1, x_2, t) 
 & = \int_E \varphi_{a t}(\bv_1 a^{-1} y + \bv_2 x_1) \psi_E(t x_2 \bar{y}) |t| \, dy \\
 & = |a|^2 \int_E \varphi_{a t}(\bv_1 y + \bv_2 x_1) \psi_E(t x_2 a \bar{y}) |t| \, dy \\
 & = |a| \cF_{at \psi}(\varphi_{at})(x_1, x_2) \\
 & = |a| \hat{\varphi}(x_1, x_2, at).
\end{align*}
If $h = \left[ \smat{1}{b}{}{1}, 1 \right]$ with $b \in F$, then we have
\begin{align*}
 \hat{\Omega}_\psi(h) \hat{\varphi}(x_1, x_2, t)
 & = \int_E \varphi_t(\bv_1 y + (-\bv_1 b + \bv_2) x_1) \psi_E(t x_2 \bar{y}) |t| \, dy \\
 & = \int_E \varphi_t(\bv_1 y + \bv_2 x_1) \psi_E(t x_2 (\bar{y} + b \bar{x}_1)) |t| \, dy \\
 & = \psi_E(t \bar{x}_1 x_2 b) \cF_{t \psi}(\varphi_t)(x_1, x_2) \\
 & = \psi_E(t \bar{x}_1 x_2 b) \hat{\varphi}(x_1, x_2, t).
\end{align*}
This completes the proof.
\end{proof}

\subsection{Comparison}
\label{ss:weil-seesaw}

Now we compare the two Weil representations as in \S \ref{ss:weil-sp-o} and \S \ref{ss:weil-u}.
We consider the $2$-dimensional (right) $E$-vector space $V \otimes_F E$ equipped with the hermitian form induced by $\langle \cdot, \cdot \rangle$.
Since this space is isometric to $\bV$, we have a natural embedding $\iota : \GO(V) \hookrightarrow \GU(\bV)$.
Explicitly, the associated map $\iota : K^{\times} \rtimes \langle h_0 \rangle \hookrightarrow (B^\times \times E^\times)/F^\times$ is given by
\[
 \iota(h) = [h,1], \quad
 \iota(h_0) = [\bi,\bi]
\]
for $h \in K^\times$.
We also regard $\bW$ as the $4$-dimensional (left) $F$-vector space equipped with the symplectic form $\frac{1}{2} {\tr_{E/F}} \circ (\cdot, \cdot)$.
Since this space is isometric to $W$, we have a natural embedding $\iota : \GU(\bW) \hookrightarrow \GSp(W)$.
Explicitly, if we identify $\bW$ with $W$ (as $F$-vector spaces) via the map
\[
 (a_1 + b_1 \bi, a_2 + b_2 \bi) \mapsto (a_1, b_1, a_2, - b_2 u), 
\]
then we have
\[
 \iota \mat{a_1 + b_1 \bi}{a_2 + b_2 \bi}{a_3 + b_3 \bi}{a_4 + b_4 \bi} = 
 \begin{pmatrix}
  a_1 & b_1 & a_2 & - b_2 u \\
  b_1 u & a_1 & b_2 u & - a_2 u \\
  a_3 & b_3 & a_4 & - b_4 u \\
  -b_3 & - \frac{a_3}{u} & -b_4 & a_4
 \end{pmatrix}, 
\]
where $a_i, b_j \in F$ and $u = \bi^2 \in F^\times$.
Then we have an identification $V \otimes_F W = \bV \otimes_E \bW$ as symplectic $F$-spaces and the following seesaw diagram:
\[
\begin{tikzcd}
 \GSp(W) \arrow[rd,dash] & \GU(\bV) \arrow[ld,dash] \\
 \GU(\bW) \arrow[u,dash] & \GO(V) \arrow[u,dash]
\end{tikzcd}
\]
We write $\Omega_\psi^K$ and $\Omega_\psi^B$ for the Weil representations of $\GSp(W) \times \GO(V)$ and $\GU(\bW) \times \GU(\bV)$, respectively, realized on the same space $\cS(K^2 \times F^\times) = \cS(B \times F^\times)$, where we identify $K^2$ with $B$ (as $F$-vector spaces) via the map
\[
 (x_1, x_2) \mapsto x_1 + x_2 \bi
\]
for $x_1, x_2 \in K$.

\begin{lem}
\label{l:weil-seesaw}
We have
\[
 \Omega_\psi^B(g, \iota(h)) = \xi_{K/F}(\N_{E/F}(z)) \cdot \Omega_\psi^K(\iota(g), h)
\]
for $g = [g',z] \in \GU(\bW)$ with $(g',z) \in \GL_2(F) \times E^\times$ and $h \in \GO(V)$.
\end{lem}

\begin{proof}
The assertion follows from a direct computation and the fact that $\xi_{K/F}(u) = \gamma_B$.
\end{proof}

\section{Theta lifting}
\label{s:theta}

In this section, we describe the global theta lifting for the dual pairs introduced in \S \ref{s:weil}.
Let $F$ be a totally real number field and put $F_\infty = F \otimes_\Q \R$.
Let $\A$ and $\A_f$ be the rings of adeles and finite adeles of $F$, respectively. 
Fix a nontrivial additive character $\psi$ of $\A/F$.

\subsection{From $K^\times$ to $\GSp_4$}
\label{ss:theta-Sp4}

Let $K$ be a totally imaginary quadratic extension of $F$ and $\A_K$ the ring of adeles of $K$.
Let $\xi_{K/F}$ be the quadratic character of $\A^\times/F^\times$ associated to $K/F$ by class field theory.
Let $V = K$ and $W = F^4$ be the quadratic and symplectic $F$-spaces, respectively, as in \S \ref{ss:weil-sp-o}.
Put
\[
 G = \GSp(W), \quad
 \tilde{H} = \GO(V), \quad
 H = \GO(V)^0.
\]
Let $Z_G \simeq F^\times$ and $Z_{\tilde{H}} \simeq F^\times$ be the centers of $G$ and $\tilde{H}$, respectively. 
Put $G_\infty = G(F_\infty)$.
Let $\fg_\infty$ be the complexified Lie algebra of $G_\infty$ and $K_{G_\infty}$ the standard maximal compact modulo center subgroup of $G_\infty$.
As in \S \ref{ss:weil-sp-o}, we define a Weil representation $\Omega_\psi$ of $G(\A) \times \tilde{H}(\A)$ on the Schwartz space $\cS(\A_K^2 \times \A^\times)$ and a $(\fg_\infty, K_{G_\infty}) \times G(\A_f) \times \tilde{H}(\A)$-invariant subspace $S(\A_K^2 \times \A^\times)$ of $\cS(\A_K^2 \times \A^\times)$.
Note that $\Omega_\psi$ has trivial central character on $\{ (z \cdot \id_W, z \cdot \id_V) \mid z \in \A^\times \}$.
For $\varphi \in S(\A_K^2 \times \A^\times)$, we may form a theta function on $G(\A) \times \tilde{H}(\A)$:
\[
 \Theta_\varphi(g, h) = \sum_{x \in K^2} \sum_{t \in F^\times} \Omega_\psi(g,h) \varphi(x,t).
\]

Let $\chi$ be a character of $\A_K^\times/K^\times$, which is regarded as an automorphic character of $H(\A) = \A_K^\times$.
For $\varphi \in S(\A_K^2 \times \A^\times)$, we may define an automorphic form $\theta_\varphi(\chi)$ on $G(\A)$ by 
\[
 \theta_\varphi(\chi)(g) = \int_{H(F) \backslash H(\A)} \Theta_\varphi(g, h) \chi(h) \, dh,
\]
where $dh$ is the Tamagawa measure on $H(\A)$.
Let $\Theta(\chi)$ be the automorphic representation of $G(\A)$ spanned by $\theta_\varphi(\chi)$ for all $\varphi \in S(\A_K^2 \times \A^\times)$.
(Strictly speaking, $\Theta(\chi)$ is only a $(\fg_\infty, K_{G_\infty}) \times G(\A_f)$-module.)
Then $\Theta(\chi)$ has central character $\chi|_{\A^\times}$.

\begin{prop}
\begin{enumerate}
\item
The global theta lift $\Theta(\chi)$ is nonzero and semisimple.
\item
Let $\pi$ be an irreducible component of $\Theta(\chi)$.
Then we have
\[
 \pi \simeq \otimes_v \theta(\sigma_v),
\]
where $\sigma_v$ is some irreducible component of $\Ind^{\tilde{H}_v}_{H_v}(\chi_v)$ and $\theta(\sigma_v)$ is the local theta lift of $\sigma_v$ (see Proposition \ref{p:howe-duality-Sp4} below).
\end{enumerate}
\end{prop}

\begin{proof}
It follows from the argument in the proof of \cite[Proposition 10.1]{yamana} that $\Theta(\chi)$ is nonzero and contained in the space of automorphic forms on $G(\A)$ which are square-integrable modulo the center.
Hence the proposition follows from the Howe duality (see Proposition \ref{p:howe-duality-Sp4} below), noting that for any irreducible component $\pi \simeq \otimes_v \pi_v$ of $\Theta(\chi)$, we have
\[
 \Hom_{G_v \times \tilde{H}_v}(\Omega_{\psi,v} \otimes \Ind^{\tilde{H}_v}_{H_v}(\chi_v), \pi_v) = \Hom_{G_v \times H_v}(\Omega_{\psi,v} \otimes \chi_v, \pi_v) \ne 0.
\]
\end{proof}

We also need to consider the local theta lifting.
For the moment, we fix a place $v$ of $F$ and suppress the subscript $v$ from the notation.
Thus $F$ is non-archimedean, or $F = \R$ and $K = \C$.
When $F = \R$, we abuse terminology and mean a $(\fg,K_G)$-module by a representation of $G$, where $\fg$ is the complexified Lie algebra of $G$ and $K_G$ is the standard maximal compact modulo center subgroup of $G$.
For any irreducible representation $\sigma$ of $\tilde{H}$, we write
\[
 \Theta(\sigma) = (\Omega_\psi \otimes \sigma)_{\tilde{H}}
\]
for the big theta lift of $\sigma$, where the subscript $\tilde{H}$ indicates that we take $\tilde{H}$-coinvariants.

\begin{prop}
\label{p:howe-duality-Sp4}
The big theta lift $\Theta(\sigma)$ is a nonzero representation of $G$ of finite length and has a unique irreducible quotient $\theta(\sigma)$.
In particular, $\theta(\sigma)$ is the unique irreducible representation of $G$ such that
\[
 \Hom_{G \times \tilde{H}}(\Omega_\psi \otimes \sigma, \theta(\sigma)) \ne 0.
\]
\end{prop}

\begin{proof}
Put $G_1 = \Sp(W)$ and $\tilde{H}_1 = \O(V)$.
By the result of Roberts \cite{roberts}, the Howe duality \cite{howe-jams,waldspurger-howe-duality,gt16} for $(G_1, \tilde{H}_1)$ and the conservation relation \cite{sun-zhu} imply the Howe duality for $(G, \tilde{H})$.
Moreover, since $(G_1, \tilde{H}_1)$ is in the stable range \cite{li-invent}, $\Theta(\sigma)$ is always nonzero.
This yields the assertion.
\end{proof}

\begin{rem}
\label{r:big-theta-Sp4}
If $\sigma$ is unitary, then it follows from \cite[Theorem A]{loke-ma}, \cite[Corollary 7.3]{chen-zou}, combined with the argument in the proof of \cite[Lemma 2.2]{gt11}, that $\Theta(\sigma)$ is irreducible.
\end{rem}

We describe $\theta(\sigma)$ explicitly in the unramified case and the real case.
First suppose that $F$ is non-archimedean and non-dyadic, $K$ is split or inert, and $\psi$ is of order zero.
Let $\fo_F$ and $\fo_K$ be the rings of integers of $F$ and $K$, respectively.
Let $K_G = \GSp_4(\fo_F)$ and $K_{\tilde{H}} = \fo_K^\times \rtimes \mu_2$ be the standard maximal compact subgroups of $G = \GSp_4(F)$ and $\tilde{H} = K^\times \rtimes \mu_2$, respectively.
Let $T$ be the maximal torus of $G$ consisting of diagonal matrices and $B_G$ the standard Borel subgroup of $G$ containing $T$.

\begin{lem}
\label{l:theta-unram}
Let $\chi$ be a unitary unramified character of $K^\times$.
Let $\chi_1$ and $\chi_2$ be unitary unramified characters of $F^\times$ such that
\begin{itemize}
\item $\chi = \chi_1 \otimes \chi_2$ if $K$ is split;
\item $\chi = \chi_1 \circ \N_{K/F}$ and $\chi_2 = \chi_1 \xi_{K/F}$ if $K$ is inert.
\end{itemize}
Let $\sigma$ be an irreducible component of $\Ind^{\tilde{H}}_H(\chi)$ and assume that $\theta(\sigma)$ has a nonzero $K_G$-fixed vector, which is the case if and only if $\sigma$ has a nonzero $K_{\tilde{H}}$-fixed vector.
Then $\theta(\sigma)$ is the irreducible unramified component of $\Ind^G_{B_G}(\chi_T)$, where $\chi_T$ is the pullback of the character 
\[
 \xi_{K/F} |\cdot|^{1/2} \otimes \chi_1 \otimes |\cdot|^{-1/2} \otimes \chi_2
\]
via the natural embedding $T \hookrightarrow (F^\times)^4$.
In particular, $\theta(\sigma)$ is of type
\[
\begin{cases}
 \text{IIIb} & \text{if $K$ is split and $\chi_1 \ne \chi_2$;} \\
 \text{VId} & \text{if $K$ is split and $\chi_1 = \chi_2$;} \\
 \text{Vd} & \text{if $K$ is inert}
\end{cases} 
\]
in the list of Roberts-Schmidt \cite{roberts-schmidt}.
\end{lem}

\begin{proof}
By \cite[Proposition A.8]{gi11}, \cite[Proposition 6.5]{gurevich-szpruch}, $\theta(\sigma)$ is the irreducible unramified component of 
\[
 \Ind^G_{B_G}(\xi_{K/F}|\cdot| \otimes \chi_1 \chi_2^{-1} \otimes \chi_2 |\cdot|^{-1/2}),
\]
where we identify $T$ with $(F^\times)^3$ via the map
\[
 (t_1, t_2, t_0) \mapsto \diag(t_1, t_2, t_0 t_1^{-1}, t_0 t_2^{-1}).
\]
This implies the assertion.
\end{proof}

Next suppose that $F = \R$, $K = \C$, and $\psi(x) = e^{2 \pi \sqrt{-1} x}$.

\begin{lem}
\label{l:theta-real}
Let $\chi$ be a character of $\C^\times$ given by $\chi(z) = \bar{z}/z$ and put $\sigma = \Ind^{\tilde{H}}_H(\chi)$, so that $\sigma$ is irreducible.
Then we have
\[
 \theta(\sigma) = \Ind^{(\fg, K_G)}_{(\fg, K_G^0)}(A_\fq), 
\]
where $A_\fq$ is the cohomological representation as in \S \ref{s:cohrep}.
\end{lem}

\begin{proof}
Put $\sigma_1 = \sigma|_{\tilde{H}_1}$.
Note that $\sigma_1$ is irreducible.
Let $\theta(\sigma_1)$ be the theta lift of $\sigma_1$ to $G_1$.
Then by \cite[Theorem 6.2]{li-duke}, we have
\[
 \theta(\sigma_1) = A_\fq
\]
as representations of $G_1$.
From this, we can deduce the assertion.
\end{proof}

Now we consider the global theta lifting again.
For our purposes, it is more natural to define it with respect to the standard measure on $H(\A) = \A_K^\times$ given as follows.
For each place $v$ of $F$, we define the standard measure $dh_v$ on $K_v^\times$ as follows:
\begin{itemize} 
\item if $v$ is real, then $dh_v = 2 \, dh_v^{\mathrm{Leb}} / |h_v|$, where $dh_v^{\mathrm{Leb}}$ is the Lebesgue measure on $K_v = \C$;
\item if $v$ is finite, then $dh_v$ is the Haar measure on $K_v^\times$ such that $\vol(\fo_{K_v}^\times) = 1$.
\end{itemize}
We define the standard measure $dh^\std$ on $\A_K^\times$ as the product measure $\prod_v dh_v$.
Then the Tamagawa measure on $\A_K^\times$ is given by $C_K^{-1} \cdot dh^\std$ with 
\[
 C_K = |D_K|^{1/2} \cdot \Res_{s=1} \zeta_K(s), 
\]
where $D_K$ is the discriminant of $K$ and $\zeta_K(s)$ is the Dedekind zeta function of $K$.

For $\varphi \in S(\A_K^2 \times \A^\times)$, we write $\theta_\varphi(\chi)^\std$ for the theta lift defined with respect to $dh^\std$, so that $\theta_\varphi(\chi)^\std = C_K \cdot \theta_\varphi(\chi)$.
Then it has a Fourier expansion of the form
\[
 \theta_\varphi(\chi)^\std(g) = \sum_{T \in \Sym_2(F)} W_T(g),
\]
where
\[
 W_T(g) = \int_{\Sym_2(F) \backslash \Sym_2(\A)} \theta_\varphi(\chi)^\std (\boldsymbol{n}(b) g) \overline{\psi(\tr(Tb))} \, db
\]
with the Tamagawa measure $db$ on $\Sym_2(\A)$.
For $x \in K^2 \smallsetminus \{ 0 \}$ and $t \in F^\times$, we may define a function $\cW_{x,t}$ on $G(\A)$ by 
\[
 \cW_{x,t}(g) = \int_{\A_K^\times} \Omega_\psi(g) \varphi(h^{-1} x, \nu(h) t) \chi(h) \, dh^\std.
\]

\begin{lem}
\label{l:W_T}
Let $T \in \Sym_2(F)$.
\begin{enumerate}
\item If $T \ne 0$, then we have $W_T = 0$ unless $T = t_0 q(x_0)$ for some $x_0 \in K^2 \smallsetminus \{ 0 \}$ and $t_0 \in F^\times$, in which case
\[
 W_T =
 \begin{cases}
  \cW_{x_0,t_0} + \cW_{h_0 x_0,t_0} & \text{if $\det(T) \ne 0$;} \\
  \cW_{x_0,t_0} & \text{if $\det(T) = 0$.}
 \end{cases}
\]
\item 
If $T = 0$, then we have $W_T = 0$ unless $\chi|_{\A_K^1} = 1$.
\end{enumerate}
\end{lem}

\begin{proof}
For brevity, we write $dh = dh^\std$ in this proof.
By definition, we have
\begin{align*}
 W_T(1) & = \int_{\Sym_2(F) \backslash \Sym_2(\A)} \int_{K^\times \backslash \A_K^\times} \sum_{x \in K^2} \sum_{t \in F^\times} \Omega_\psi(\boldsymbol{n}(b), h) \varphi(x,t) \chi(h) \overline{\psi(\tr(Tb))} \, dh \, db \\
 & = \int_{\Sym_2(F) \backslash \Sym_2(\A)} \int_{K^\times \backslash \A_K^\times} \sum_{x \in K^2} \sum_{t \in F^\times} \psi(t \tr(q(x) b)) \varphi(h^{-1} x, \nu(h) t) \chi(h) \overline{\psi(\tr(Tb))} \, dh \, db \\
 & = \int_{K^\times \backslash \A_K^\times} \sum_{(x,t) \in \Xi_T} \varphi(h^{-1} x, \nu(h) t) \chi(h) \, dh,
\end{align*}
where $\Xi_T = \{ (x,t) \in K^2 \times F^\times \mid t q(x) = T \}$.
In particular, we have $W_T(1) = 0$ unless $\Xi_T \ne \varnothing$.
Suppose that $T = t_0 q(x_0)$ for some $x_0 \in K^2$ and $t_0 \in F^\times$.
Define an action of $K^\times$ on $K^2 \times F^\times$ by 
\[
 h \cdot (x,t) = (h x, \nu(h)^{-1} t).
\]
Note that $T \ne 0$ if and only if $x_0 \ne 0$, in which case the stabilizer of $(x_0, t_0)$ in $K^\times$ is trivial.
If $\det(T) \ne 0$, then we have $\Xi_T = K^\times (x_0, t_0) \sqcup K^\times (h_0 x_0, t_0)$, so that 
\begin{align*}
 W_T(1) & = \int_{K^\times \backslash \A_K^\times} \sum_{\gamma \in K^\times} (\varphi(h^{-1} \gamma^{-1} x_0, \nu(\gamma h) t_0) + \varphi(h^{-1} \gamma^{-1} h_0 x_0, \nu(\gamma h) t_0)) \chi(h) \, dh \\
 & = \int_{\A_K^\times} (\varphi(h^{-1} x_0, \nu(h) t_0) + \varphi(h^{-1} h_0 x_0, \nu(h) t_0)) \chi(h) \, dh.
\end{align*}
If $T \ne 0$ and $\det(T) = 0$, then we have $\Xi_T = K^\times (x_0, t_0)$, so that 
\begin{align*}
 W_T(1) & = \int_{K^\times \backslash \A_K^\times} \sum_{\gamma \in K^\times} \varphi(h^{-1} \gamma^{-1} x_0, \nu(\gamma h) t_0) \chi(h) \, dh \\
 & = \int_{\A_K^\times} \varphi(h^{-1} x_0, \nu(h) t_0) \chi(h) \, dh.
\end{align*}
If $T = 0$, then we have $\Xi_T = \{ (0,t) \mid t \in F^\times \}$, so that 
\begin{align*}
 W_T(1) & = \int_{K^\times \backslash \A_K^\times} \sum_{t \in F^\times} \varphi(0, \nu(h) t) \chi(h) \, dh \\ 
 & = \int_{K^\times \A_K^1 \backslash \A_K^\times} \sum_{t \in F^\times} \varphi(0, \nu(h) t) \int_{K^1 \backslash \A_K^1} \chi(h_1 h) \, dh_1 \, dh.
\end{align*}
This completes the proof.
\end{proof}

\subsection{From $\GU(1,1)$ to $\GU(B)$}
\label{ss:theta-U11}

Let $B$ be a totally definite quaternion algebra over $F$ such that $K$ embeds into $B$.
Write $B = K + K \bi$ and $K = F + F \bj$ for some trace zero elements $\bi \in B^\times$ and $\bj \in K^\times$.
Then $E = F + F \bi$ is a totally imaginary quadratic extension of $F$ such that $B = E + E \bj$.
Let $\bV = B$ and $\bW = E^2$ be the hermitian and split skew-hermitian $E$-spaces, respectively, as in \S \ref{ss:weil-u}.
Put
\[
 \bG = \GU(\bW), \quad \bH = \GU(\bV).
\]
Put $\bG_\infty = \bG(F_\infty)$.
Let $\boldsymbol{\fg}_\infty$ be the complexified Lie algebra of $\bG_\infty$ and $K_{\bG_\infty}$ the standard maximal compact modulo center subgroup of $\bG_\infty$.
As in \S \ref{ss:weil-u}, we define a Weil representation $\Omega_\psi$ of $\bG(\A) \times \bH(\A)$ on the Schwartz space $\cS(B(\A) \times \A^\times)$ and a $(\boldsymbol{\fg}_\infty, K_{\bG_\infty}) \times \bG(\A_f) \times \bH(\A)$-invariant subspace $S(B(\A) \times \A^\times)$ of $\cS(B(\A) \times \A^\times)$.
Note that $\Omega_\psi$ has trivial central character on $\{ (z \cdot \id_\bW, z \cdot \id_\bV) \mid z \in \A^\times_E \}$.
For $\varphi \in S(B(\A) \times \A^\times)$, we may form a theta function on $\bG(\A) \times \bH(\A)$:
\[
 \Theta_\varphi(g, h) = \sum_{x \in B} \sum_{t \in F^\times} \Omega_\psi(g,h) \varphi(x,t).
\]

Let $\pi$ be an irreducible cuspidal automorphic representation of $\GL_2(\A)$ with central character $\xi_\pi$ and $\mu$ a character of $\A_E^\times/E^\times$.
Assume that $\xi_\pi \cdot \mu|_{\A^\times} = 1$, so that $\pi \boxtimes \mu$ can be regarded as an irreducible cuspidal automorphic representation of $\bG(\A) = (\GL_2(\A) \times \A_E^\times)/\A^\times$ with central character $\mu^{-1}$.
(Strictly speaking, $\pi \boxtimes \mu$ is only a $(\boldsymbol{\fg}_\infty, K_{\bG_\infty}) \times \bG(\A_f)$-module.)
For $f \in \pi$, we denote by $f^\mu$ an automorphic form on $\bG(\A)$ given by
\[
 f^\mu(g) = \mu(z) f(g')
\]
for $g = [g',z] \in \bG(\A)$ with $(g',z) \in \GL_2(\A) \times \A_E^\times$.
For $\varphi \in S(B(\A) \times \A^\times)$ and $f \in \pi$, we may define an automorphic form $\theta_\varphi(f^\mu)$ on $\bH(\A)$ by 
\[
 \theta_\varphi(f^\mu)(h) = \int_{\bG(F) \backslash \bG(\A)} \Theta_\varphi(g, h) f^\mu(g) \, dg,
\]
where $dg$ is the Tamagawa measure on $\bG(\A)$.
Let $\Theta(\pi \boxtimes \mu)$ be the automorphic representation of $\bH(\A)$ spanned by $\theta_\varphi(f^\mu)$ for all $\varphi \in S(B(\A) \times \A^\times)$ and $f \in \pi$.
Then $\Theta(\pi \boxtimes \mu)$ has central character $\mu^{-1}$.

We also need to consider the local theta lifting. 
For the moment, we fix a place $v$ of $F$ and suppress the subscript $v$ from the notation.
Thus $F$ is non-archimedean, or $F = \R$ and $E = \C$.
Let $\pi$ be an irreducible representation of $\GL_2(F)$ with central character $\xi_\pi$ and $\mu$ a character of $E^\times$.
Assume that $\xi_\pi \cdot \mu|_{F^\times} = 1$, so that $\pi \boxtimes \mu$ can be regarded as an irreducible representation of $\bG$.
We write
\[
 \Theta(\pi \boxtimes \mu) = (\Omega_\psi \otimes (\pi \boxtimes \mu))_{\bG}
\]
for the big theta lift of $\pi \boxtimes \mu$, where the subscript $\bG$ indicates that we take $\bG$-coinvariants.
Let $\pi_E$ be the base change of $\pi$ to $\GL_2(E)$.
We denote by 
\[
 \varepsilon(\pi_E \times \mu)
\]
the value of the local Rankin-Selberg $\varepsilon$-factor $\varepsilon(s, \pi_E \times \mu, \psi \circ \tr_{E/F})$ at $s=\frac{1}{2}$ (which does not depend on $\psi$).
Put 
\[
 \varepsilon(B) = 
 \begin{cases}
  +1 & \text{if $B$ is split;} \\
  -1 & \text{if $B$ is ramified.}
 \end{cases} 
\]

\begin{prop}
\label{p:howe-duality-U11}
The big theta lift $\Theta(\pi \boxtimes \mu)$ is nonzero if and only if
\[
 \varepsilon(\pi_E \times \mu) \cdot \xi_\pi(-1) = \varepsilon(B),
\]
in which case it is a representation of $\bG$ of finite length and has a unique irreducible quotient $\theta(\pi \boxtimes \mu)$.
In particular, $\theta(\pi \boxtimes \mu)$ is the unique irreducible representation of $\bG$ such that 
\[
 \Hom_{\bG \times \bH}(\Omega_\psi \otimes (\pi \boxtimes \mu), \theta(\pi \boxtimes \mu)) \ne 0.
\]
Moreover, we have
\[
 \theta(\pi \boxtimes \mu) = \pi^B \boxtimes \mu, 
\]
where $\pi^B$ is the Jacquet-Langlands transfer of $\pi$ to $B^\times$.
\end{prop}

\begin{proof}
Put $\bG_1 = \U(\bW)$ and $\bH_1 = \U(\bV)$.
By the result of Zhang \cite{zhang_chong}, the Howe duality \cite{howe-jams,waldspurger-howe-duality,minguez,gt16} for $(\bG_1, \bH_1)$ and the conservation relation \cite{sun-zhu} imply the Howe duality for $(\bG, \bH)$.
(Note that $E$ is assumed to be a quadratic field extension of $F$ in \cite{zhang_chong}, but the same argument works when $E = F \oplus F$.)
Moreover, by \cite[Theorem 4.1]{harris-jams}, $\Theta(\pi \boxtimes \mu)$ is nonzero if and only if $\varepsilon(\pi_E \times \mu) \cdot \xi_\pi(-1) = \varepsilon(B)$, in which case $\pi^B$ exists and $\theta(\pi \boxtimes \mu) = \pi^B \boxtimes \mu$.
(Note that the convention in \cite{harris-jams} differs from ours.)
This yields the assertion.
\end{proof}

\begin{rem}
\label{r:big-theta-U11}
If $F$ is non-archimedean and $\pi$ is generic, then it follows from \cite[Theorem 6.1]{chen-zou}, \cite[Theorem 1.7]{fsx}, combined with the argument in the proof of \cite[Lemma 2.2]{gt11}, that $\Theta(\pi \boxtimes \mu)$ is zero or irreducible.
\end{rem}

Now we consider the global theta lifting again.
We assume that 
\[
 \varepsilon_v(\pi_E \times \mu) \cdot \xi_{\pi,v}(-1) = \varepsilon_v(B)
\]
for all places $v$ of $F$, where $\pi_E$ is the base change of $\pi$ to $\GL_2(\A_E)$, and 
\[
 \varepsilon_v(\pi_E \times \mu) = \varepsilon(\pi_{E,v} \times \mu_v), \quad
 \varepsilon_v(B) = \varepsilon(B_v).
\]
Then the Jacquet-Langlands transfer $\pi^B$ of $\pi$ to $B^\times(\A)$ exists.
Moreover, by \cite[Theorem 4.5]{harris-jams} (see also \cite[Theorem 10.1 and Lemma 10.2]{yamana}), we have:

\begin{prop}
\label{p:theta-u11-u2}
The global theta lift $\Theta(\pi \boxtimes \mu)$ is nonzero if and only if $L(\frac{1}{2}, \pi_E \times \mu) \ne 0$, where $L(s, \pi_E \times \mu)$ is the global Rankin-Selberg $L$-function.
If this is the case, then we have
\[
 \Theta(\pi \boxtimes \mu) = \pi^B \boxtimes \mu
\]
as spaces of automorphic forms on $\bH(\A) = (B^\times(\A) \times \A_E^\times)/\A^\times$.
\end{prop}

\subsection{A seesaw identity}
\label{ss:seesaw}

As in \S \ref{ss:weil-seesaw}, we consider the seesaw diagram
\[
\begin{tikzcd}
 G \arrow[rd,dash] & \bH \arrow[ld,dash] \\
 \bG \arrow[u,dash] & \tilde{H} \arrow[u,dash]
\end{tikzcd}
\]
and identify $K^2$ with $B$ (as $F$-vector spaces) via the map
\[
 (x_1, x_2) \mapsto x_1 + x_2 \bi
\]
for $x_1, x_2 \in K$.
Put
\begin{align*}
 \cP^{\bG}(\phi, \boldsymbol{\phi}) & = \int_{Z_G(\A) \bG(F) \backslash \bG(\A)} \phi(g) \boldsymbol{\phi}(g) \, dg, \\
 \cP^{H}(\boldsymbol{\phi}', \chi) & = \int_{Z_{\tilde{H}}(\A) H(F) \backslash H(\A)} \boldsymbol{\phi}'(h) \chi(h) \, dh
\end{align*}
(whenever the integrals make sense) for automorphic forms $\phi$, $\boldsymbol{\phi}$, $\boldsymbol{\phi}'$ on $G(\A)$, $\bG(\A)$, $\bH(\A)$, respectively, where $dg$ and $dh$ are the Tamagawa measures on $Z_G(\A) \backslash \bG(\A)$ and $Z_{\tilde{H}}(\A) \backslash H(\A)$, respectively.

\begin{lem}
\label{l:seesaw}
Assume that $\xi_\pi \cdot \chi|_{\A^\times} = \xi_\pi \cdot \mu|_{\A^\times} = 1$.
Then we have
\[
 \cP^{\bG}(\theta_\varphi(\chi), f^{\mu'}) = \cP^H(\theta_\varphi(f^\mu), \chi)
\]
for $\varphi \in S(\A_K^2 \times \A^\times) = S(B(\A) \times \A^\times)$ and $f \in \pi$, where $\mu'$ is a character of $\A_E^\times/E^\times$ given by
\[
 \mu' = \mu \cdot (\xi_{K/F} \circ \N_{E/F}).
\]
\end{lem}

\begin{proof}
We write $\Theta^K_\varphi$ and $\Theta^B_\varphi$ for the theta functions on $G(\A) \times \tilde{H}(\A)$ and $\bG(\A) \times \bH(\A)$, respectively.
Put 
\[
 I(\varphi,f) = \int_{Z(\A) (\bG \times H)(F) \backslash (\bG \times H)(\A)} \Theta^B_\varphi(g,h) f^\mu(g) \chi(h) \, dg \, dh, 
\]
where 
\[
 Z = \{ (z \cdot \id_W, z \cdot \id_V) \mid z \in F^\times \} \subset Z_G \times Z_{\tilde{H}}.
\]
Then we have
\begin{align*}
 I(\varphi,f) & = \int_{Z_{\tilde{H}}(\A) H(F) \backslash H(\A)} \int_{\bG(F) \backslash \bG(\A)} \Theta^B_\varphi(g,h) f^\mu(g) \chi(h) \, dg \,dh \\
 & = \int_{Z_{\tilde{H}}(\A) H(F) \backslash H(\A)} \theta_\varphi(f^\mu)(h) \chi(h) \, dh.
\end{align*}
On the other hand, by Lemma \ref{l:weil-seesaw}, we have
\[
 \Theta^B_\varphi(g,h) = \xi_{K/F}(\N_{E/F}(z)) \cdot \Theta^K_\varphi(g,h)
\]
for $g = [g',z] \in \bG(\A)$ with $(g',z) \in \GL_2(\A) \times \A_E^\times$ and $h \in \tilde{H}(\A)$.
Hence we have
\begin{align*}
 I(\varphi,f) & = \int_{Z_G(\A) \bG(F) \backslash \bG(\A)} \int_{H(F) \backslash H(\A)} \Theta^K_\varphi(g,h) f^{\mu'}(g) \chi(h) \, dh \, dg \\
 & = \int_{Z_G(\A) \bG(F) \backslash \bG(\A)} \theta_\varphi(\chi)(g) f^{\mu'}(g) \, dg.
\end{align*}
This completes the proof.
\end{proof}

Finally, we record an application of the local version of the seesaw identity, which will be used in a sequel to this paper.
We fix a finite place $v$ of $F$ and suppress the subscript $v$ from the notation.
Thus $F$ is a non-archimedean local field of characteristic zero, $K$ and $E$ are \'etale quadratic algebras over $F$, and $B$ is the quaternion algebra over $F$ associated to $K$ and $E$ as above.
Let $\chi$ be a unitary character of $H = K^\times$.
Define a character $\chi'$ of $H$ by $\chi'(x) = \chi(\bar{x})$.
Then $\Ind^{\tilde{H}}_H(\chi)$ is described as follows:
\begin{itemize}
\item if $\chi' \ne \chi$, then $\Ind^{\tilde{H}}_H(\chi)$ is irreducible;
\item if $\chi' = \chi$, then $\Ind^{\tilde{H}}_H(\chi) = \sigma \oplus \sigma'$, where $\sigma$ and $\sigma'$ are the two distinct extensions of $\chi$ to $\tilde{H}$.
\end{itemize}
Let $\pi$ be an irreducible unitary generic representation of $\GL_2(F)$ with central character $\xi_\pi$ and $\mu$ a unitary character of $E^\times$.
Define a character $\mu'$ of $E^\times$ by $\mu' = \mu \cdot (\xi_{K/F} \circ \N_{E/F})$.

\begin{lem}
\label{l:seesaw-local}
Assume that $\xi_\pi \cdot \chi|_{F^\times} = \xi_\pi \cdot \mu|_{F^\times} = 1$.
\begin{enumerate}
\item 
If the condition 
\[
 \varepsilon(\pi_K \times \chi) \cdot \xi_\pi(-1) = \varepsilon(\pi_E \times \mu) \cdot \xi_\pi(-1) = \varepsilon(B) 
\]
holds, then there exists a unique irreducible component $\sigma$ of $\Ind^{\tilde{H}}_H(\chi)$ such that 
\[
 \Hom_{\bG}(\theta(\sigma) \otimes (\pi \boxtimes \mu'), \C) \ne 0.
\]
Moreover, this Hom space is $1$-dimensional.
\item 
If the above condition does not hold, then we have
\[
 \Hom_{\bG}(\theta(\sigma) \otimes (\pi \boxtimes \mu'), \C) = 0
\]
for all irreducible components $\sigma$ of $\Ind^{\tilde{H}}_H(\chi)$.
\end{enumerate}
\end{lem}

\begin{proof}
Recall that $\Theta(\sigma) = (\Omega^K_\psi \otimes \sigma)_{\tilde{H}}$ and $\Theta(\pi \boxtimes \mu) = (\Omega^B_\psi \otimes (\pi \boxtimes \mu))_{\bG}$.
By Lemma \ref{l:weil-seesaw}, we have
\begin{align*}
 \Hom_{\bG}(\Theta(\sigma) \otimes (\pi \boxtimes \mu'), \C) 
 & \simeq \Hom_{\bG \times \tilde{H}}(\Omega^K_\psi \otimes (\pi \boxtimes \mu') \otimes \sigma, \C) \\
 & \simeq \Hom_{\bG \times \tilde{H}}(\Omega^B_\psi \otimes (\pi \boxtimes \mu) \otimes \sigma, \C) \\
 & \simeq \Hom_{\tilde{H}}(\Theta(\pi \boxtimes \mu) \otimes \sigma, \C).
\end{align*}
Also, by Propositions \ref{p:howe-duality-Sp4}, \ref{p:howe-duality-U11} and Remarks \ref{r:big-theta-Sp4}, \ref{r:big-theta-U11}, the big theta lifts are given as follows:
\begin{itemize}
\item $\Theta(\sigma) \ne 0$ and $\Theta(\sigma) = \theta(\sigma)$;
\item $\Theta(\pi \boxtimes \mu) \ne 0$ if and only if $\varepsilon(\pi_E \times \mu) \cdot \xi_\pi(-1) = \varepsilon(B)$, in which case $\Theta(\pi \boxtimes \mu) = \pi^B \boxtimes \mu$.
\end{itemize}
From this, we deduce that
\[
 \Hom_{\bG}(\theta(\sigma) \otimes (\pi \boxtimes \mu'), \C) = 0
\]
unless $\varepsilon(\pi_E \times \mu) \cdot \xi_\pi(-1) = \varepsilon(B)$, in which case 
\[
 \Hom_{\bG}(\theta(\sigma) \otimes (\pi \boxtimes \mu'), \C) \simeq
 \Hom_{\tilde{H}}((\pi^B \boxtimes \mu) \otimes \sigma, \C).
\]
Hence we may assume that $\varepsilon(\pi_E \times \mu) \cdot \xi_\pi(-1) = \varepsilon(B)$, so that 
\[
 \bigoplus_\sigma  \Hom_{\bG}(\theta(\sigma) \otimes (\pi \boxtimes \mu'), \C) \simeq 
 \Hom_{\tilde{H}}((\pi^B \boxtimes \mu) \otimes \Ind^{\tilde{H}}_H(\chi), \C),
\]
where $\sigma$ runs over irreducible components of $\Ind^{\tilde{H}}_H(\chi)$.
Then we have
\begin{align*}
 \Hom_{\tilde{H}}((\pi^B \boxtimes \mu) \otimes \Ind^{\tilde{H}}_H(\chi), \C)
 & \simeq \Hom_{\tilde{H}}(\pi^B \boxtimes \mu, \Ind^{\tilde{H}}_H(\chi^{-1})) \\
 & \simeq \Hom_H(\pi^B \boxtimes \mu, \chi^{-1}) \\
 & \simeq \Hom_H((\pi^B \boxtimes \mu) \otimes \chi, \C) \\
 & \simeq \Hom_{K^\times}(\pi^B \otimes \chi, \C).
\end{align*}
By the result of Tunnell-Saito \cite{tunnell, saito}, the right-hand side is nonzero if and only if $\varepsilon(\pi_K \times \chi) \cdot \xi_\pi(-1) = \varepsilon(B)$, in which case this Hom space is $1$-dimensional.
This yields the assertion.
\end{proof}

\section{Schwartz forms}
\label{s:schwartz_forms}

In this section, we construct a Schwartz form \`a la Kudla-Millson \cite{km1} for the dual pair introduced in \S \ref{ss:weil-sp-o}.

\subsection{Setup}
\label{ss:schwartz-setup}

Put $G_1 = \Sp_4(\R)$.
Let $K_{G_1} \simeq \U(2)$ be the standard maximal compact subgroup of $G_1$.
Let $\fg_1$ and $\fk$ be the complexified Lie algebras of $G_1$ and $K_{G_1}$, respectively.
Recall the decomposition $\fg_1 = \fk \oplus \fp^+ \oplus \fp^-$ as in \S \ref{ss:autom-adelic} and the bases
\[
 \mathtt{A}_1, \mathtt{A}_2, \mathtt{A}_3, \mathtt{A}_4, \mathtt{B}_1, \mathtt{B}_2, \mathtt{B}_3, \mathtt{C}_1, \mathtt{C}_2, \mathtt{C}_3, \quad
 \mathtt{H}_1, \mathtt{H}_2, \mathtt{X}, \mathtt{Y}, \quad
 \mathtt{X}_{2,0}, \mathtt{X}_{1,1}, \mathtt{X}_{0,2}, \quad
 \mathtt{X}_{-2,0}, \mathtt{X}_{-1,-1}, \mathtt{X}_{0,-2}
\]
of $\fg_1$, $\fk$, $\fp^+$, $\fp^-$, respectively, as in \S \ref{ss:diff-adelic}.

Fix $t \in \R^\times$ and take an additive character $\psi$ of $\R$ given by $\psi(x) = e^{2 \pi \sqrt{-1} x}$.
Let $\omega$ be the Weil representation of $G_1 \times \C^1$ on $\cS(\C^2)$ relative to $t \psi$ as in \S \ref{ss:weil-sp-o}.
More explicitly, we have
\begin{align*}
 \omega(\boldsymbol{m}(a)) \varphi(x) & = \det(a) \cdot \varphi(xa), \\ 
 \omega(\boldsymbol{n}(b)) \varphi(x) & = e^{2 \pi t \sqrt{-1} \tr(q(x) b)} \cdot \varphi(x), \\
 \omega \mat{}{-I_2}{I_2}{} \varphi(x) & = - t^2 \cdot \int_{\C^2} \varphi(y) e^{- 2 \pi t \sqrt{-1} (\bar{x} {}^t y + x {}^t \bar{y})} \, dy, \\
 \omega(h) \varphi(x) & = \varphi(h^{-1} x)
\end{align*}
for $\varphi \in \cS(\C^2)$, $x \in \C^2$, $a \in \GL_2(\R)$, $b \in \Sym_2(\R)$, and $h \in \C^1$, where $dy$ is four times the Lebesgue measure on $\C^2$.
This induces a representation $\omega$ of $\fg_1$ given by 
\begin{align*}
 \omega(\mathtt{A}_1) & = 1 + x_1 \frac{\partial}{\partial x_1} + \bar{x}_1 \frac{\partial}{\partial \bar{x}_1}, &
 \omega(\mathtt{A}_2) & = x_1 \frac{\partial}{\partial x_2} + \bar{x}_1 \frac{\partial}{\partial \bar{x}_2}, \\
 \omega(\mathtt{A}_3) & = x_2 \frac{\partial}{\partial x_1} + \bar{x}_2 \frac{\partial}{\partial \bar{x}_1}, &
 \omega(\mathtt{A}_4) & = 1 + x_2 \frac{\partial}{\partial x_2} + \bar{x}_2 \frac{\partial}{\partial \bar{x}_2}, \\
 \omega(\mathtt{B}_1) & = 2 \pi t \sqrt{-1} \cdot x_1 \bar{x}_1, &
 \omega(\mathtt{B}_2) & = 2 \pi t \sqrt{-1} \cdot (x_1 \bar{x}_2 + \bar{x}_1 x_2), \\
 \omega(\mathtt{B}_3) & = 2 \pi t \sqrt{-1} \cdot x_2 \bar{x}_2, \\
 \omega(\mathtt{C}_1) & = \frac{\sqrt{-1}}{2 \pi t} \cdot \frac{\partial^2}{\partial x_1 \partial \bar{x}_1}, &
 \omega(\mathtt{C}_2) & = \frac{\sqrt{-1}}{2 \pi t} \cdot \left( \frac{\partial^2}{\partial x_1 \partial \bar{x}_2} + \frac{\partial^2}{\partial \bar{x}_1 \partial x_2} \right), \\
 \omega(\mathtt{C}_3) & = \frac{\sqrt{-1}}{2 \pi t} \cdot \frac{\partial^2}{\partial x_2 \partial \bar{x}_2},
\end{align*}
where we write $x = (x_1, x_2) \in \C^2$.

\subsection{Derivatives of Schwartz functions}
\label{ss:schwartz-derivatives}

From now on, we assume for simplicity that $t > 0$.
Let $\varphi^0 \in \cS(\C^2)$ be the Gaussian given by 
\[
 \varphi^0(x) = e^{- 2 \pi t (x_1 \bar{x}_1 + x_2 \bar{x}_2)}.
\]
Define Schwartz functions $\varphi_0, \varphi'_0, \varphi''_0 \in \cS(\C^2)$ by
\[
 \varphi_0(x) = \bar{x}_1^2 \cdot \varphi^0(x), \quad
 \varphi'_0(x) = \bar{x}_1 \bar{x}_2 \cdot \varphi^0(x), \quad
 \varphi''_0(x) = \bar{x}_2^2 \cdot \varphi^0(x).
\]
We can easily compute the action of $\fg_1$ on these functions. 

\begin{lem}
\label{l:schwartz-table}
Put $\kappa = - 2 \pi t$.
Then we have 
\[
 \omega(*) \varphi_0^\bullet = P_*^\bullet \cdot \varphi^0,
\]
where $P_*^\bullet$ is a polynomial in $x_1, \bar{x}_1, x_2, \bar{x}_2$ given by the following table.
\[
{\renewcommand{\arraystretch}{1.2}
\begin{array}{c|l|l|l}
 * \ \diagdown \ \varphi_0^\bullet & \varphi_0 & \varphi_0' & \varphi_0'' \\ \hline
 \mathtt{A}_1 & 2 \kappa x_1 \bar{x}_1^3 + 3 \bar{x}_1^2 & 2 \kappa x_1 \bar{x}_1^2 \bar{x}_2 + 2 \bar{x}_1 \bar{x}_2 & 2 \kappa x_1 \bar{x}_1 \bar{x}_2^2 + \bar{x}_2^2 \\
 \mathtt{A}_2 & \kappa (x_1 \bar{x}_1^2 \bar{x}_2 + \bar{x}_1^3 x_2) & \kappa (x_1 \bar{x}_1 \bar{x}_2^2 + \bar{x}_1^2 x_2 \bar{x}_2) + \bar{x}_1^2 & \kappa (x_1 \bar{x}_2^3 + \bar{x}_1 x_2 \bar{x}_2^2) + 2 \bar{x}_1 \bar{x}_2 \\
 \mathtt{A}_3 & \kappa (x_1 \bar{x}_1^2 \bar{x}_2 + \bar{x}_1^3 x_2) + 2 \bar{x}_1 \bar{x}_2 & \kappa (x_1 \bar{x}_1 \bar{x}_2^2 + \bar{x}_1^2 x_2 \bar{x}_2) + \bar{x}_2^2 & \kappa (x_1 \bar{x}_2^3 + \bar{x}_1 x_2 \bar{x}_2^2) \\
 \mathtt{A}_4 & 2 \kappa \bar{x}_1^2 x_2 \bar{x}_2 + \bar{x}_1^2 & 2 \kappa \bar{x}_1 x_2 \bar{x}_2^2 + 2 \bar{x}_1 \bar{x}_2 & 2 \kappa x_2 \bar{x}_2^3 + 3 \bar{x}_2^2 \\
 \sqrt{-1} \mathtt{B}_1 & \kappa x_1 \bar{x}_1^3 & \kappa x_1 \bar{x}_1^2 \bar{x}_2 & \kappa x_1 \bar{x}_1 \bar{x}_2^2 \\
 \sqrt{-1} \mathtt{B}_2 & \kappa (x_1 \bar{x}_1^2 \bar{x}_2 + \bar{x}_1^3 x_2) & \kappa (x_1 \bar{x}_1 \bar{x}_2^2 + \bar{x}_1^2 x_2 \bar{x}_2) & \kappa (x_1 \bar{x}_2^3 + \bar{x}_1 x_2 \bar{x}_2^2) \\
 \sqrt{-1} \mathtt{B}_3 & \kappa \bar{x}_1^2 x_2 \bar{x}_2 & \kappa \bar{x}_1 x_2 \bar{x}_2^2 & \kappa x_2 \bar{x}_2^3 \\
 \sqrt{-1} \mathtt{C}_1 & \kappa x_1 \bar{x}_1^3 + 3 \bar{x}_1^2 & \kappa x_1 \bar{x}_1^2 \bar{x}_2 + 2 \bar{x}_1 \bar{x}_2 & \kappa x_1 \bar{x}_1 \bar{x}_2^2 + \bar{x}_2^2 \\
 \sqrt{-1} \mathtt{C}_2 & \kappa (x_1 \bar{x}_1^2 \bar{x}_2 + \bar{x}_1^3 x_2) + 2 \bar{x}_1 \bar{x}_2 & \kappa (x_1 \bar{x}_1 \bar{x}_2^2 + \bar{x}_1^2 x_2 \bar{x}_2) + \bar{x}_1^2 + \bar{x}_2^2 & \kappa (x_1 \bar{x}_2^3 + \bar{x}_1 x_2 \bar{x}_2^2) + 2 \bar{x}_1 \bar{x}_2 \\
 \sqrt{-1} \mathtt{C}_3 & \kappa \bar{x}_1^2 x_2 \bar{x}_2 + \bar{x}_1^2 & \kappa \bar{x}_1 x_2 \bar{x}_2^2 + 2 \bar{x}_1 \bar{x}_2 & \kappa x_2 \bar{x}_2^3 + 3 \bar{x}_2^2
\end{array}
}
\]
\end{lem}

\begin{lem}
\label{l:schwartz-3dim-0}
We have
\begin{align*}
 \omega(\mathtt{H}_1) \varphi_0 & = 3 \varphi_0, &
 \omega(\mathtt{H}_1) \varphi'_0 & = 2 \varphi'_0, &
 \omega(\mathtt{H}_1) \varphi''_0 & = \varphi''_0, \\
 \omega(\mathtt{H}_2) \varphi_0 & = \varphi_0, &
 \omega(\mathtt{H}_2) \varphi'_0 & = 2 \varphi'_0, &
 \omega(\mathtt{H}_2) \varphi''_0 & = 3 \varphi''_0, \\
 \omega(\mathtt{X}) \varphi_0 & = 0, &
 \omega(\mathtt{X}) \varphi'_0 & = \varphi_0, &
 \omega(\mathtt{X}) \varphi''_0 & = 2 \varphi'_0, \\
 \omega(\mathtt{Y}) \varphi_0 & = 2 \varphi'_0, &
 \omega(\mathtt{Y}) \varphi'_0 & = \varphi''_0, &
 \omega(\mathtt{Y}) \varphi''_0 & = 0.
\end{align*}
\end{lem}

\begin{proof}
This is an immediate consequence of Lemma \ref{l:schwartz-table}, noting that
\begin{align*}
 \mathtt{H}_1 & = - \sqrt{-1} \mathtt{B}_1 + \sqrt{-1} \mathtt{C}_1, &
 \mathtt{H}_2 & = - \sqrt{-1} \mathtt{B}_3 + \sqrt{-1} \mathtt{C}_3, \\
 \mathtt{X} & = \frac{1}{2} (\mathtt{A}_2 - \mathtt{A}_3 - \sqrt{-1} \mathtt{B}_2 + \sqrt{-1} \mathtt{C}_2), &
 \mathtt{Y} & = \frac{1}{2} (- \mathtt{A}_2 +\mathtt{A}_3 - \sqrt{-1} \mathtt{B}_2 + \sqrt{-1} \mathtt{C}_2).
\end{align*}
\end{proof} 

For any non-negative integer $n$, we define Schwartz functions $\varphi_n, \varphi'_n, \varphi''_n \in \cS(\C^2)$ by
\[
 \begin{pmatrix}
  \varphi_n \\
  \varphi'_n \\
  \varphi''_n 
 \end{pmatrix} 
 = \omega(\mathtt{D})^n
 \begin{pmatrix}
  \varphi_0 \\
  \varphi'_0 \\
  \varphi''_0 
 \end{pmatrix},
\]
where 
\[
 \mathtt{D} = 
 \begin{pmatrix}
  \mathtt{X}_{1,1} & -2 \mathtt{X}_{2,0} & 0 \\
  \mathtt{X}_{0,2} & 0 &  -\mathtt{X}_{2,0} \\
  0 & 2 \mathtt{X}_{0,2} & - \mathtt{X}_{1,1}
 \end{pmatrix}.
\]

\begin{lem}
\label{l:schwartz-3dim-n}
We have
\begin{align*}
 \omega(\mathtt{H}_1) \varphi_n & = (n+3) \varphi_n, &
 \omega(\mathtt{H}_1) \varphi_n' & = (n+2) \varphi_n', &
 \omega(\mathtt{H}_1) \varphi_n'' & = (n+1) \varphi_n'', \\
 \omega(\mathtt{H}_2) \varphi_n & = (n+1) \varphi_n, &
 \omega(\mathtt{H}_2) \varphi_n' & = (n+2) \varphi_n', &
 \omega(\mathtt{H}_2) \varphi_n'' & = (n+3) \varphi_n'', \\
 \omega(\mathtt{X}) \varphi_n & = 0, &
 \omega(\mathtt{X}) \varphi_n' & = \varphi_n, &
 \omega(\mathtt{X}) \varphi_n'' & = 2 \varphi_n', \\
 \omega(\mathtt{Y}) \varphi_n & = 2 \varphi_n', &
 \omega(\mathtt{Y}) \varphi_n' & = \varphi_n'', &
 \omega(\mathtt{Y}) \varphi_n'' & = 0.
\end{align*}
\end{lem}

\begin{proof}
We proceed by induction on $n$.
For $n=0$, the assertion follows from Lemma \ref{l:schwartz-3dim-0}.
For general $n$, put
\[
 \vec{\varphi}_n = 
 \begin{pmatrix}
  \varphi_n \\
  \varphi'_n \\
  \varphi''_n 
 \end{pmatrix}.
\]
Then it follows from the induction hypothesis and the computation in the proof of Lemma \ref{l:D.phi} that
\begin{align*}
 \omega(\mathtt{H}_1) \vec{\varphi}_{n+1} & = \omega \left( \mathtt{D}
 \begin{pmatrix}
  n+3 & 0 & 0 \\
  0 & n+2 & 0 \\
  0 & 0 & n+1
 \end{pmatrix} + 
 \begin{pmatrix}
  \mathtt{X}_{1,1} & - 4 \mathtt{X}_{2,0} & 0 \\
  0 & 0 & - 2 \mathtt{X}_{2,0} \\
  0 & 0 & - \mathtt{X}_{1,1}
 \end{pmatrix}
 \right) \vec{\varphi}_n \\
 & = \omega \left(
 \begin{pmatrix}
  n+4 & 0 & 0 \\
  0 & n+3 & 0 \\
  0 & 0 & n+2
 \end{pmatrix}
 \mathtt{D} \right) \vec{\varphi}_n, \\
 \omega(\mathtt{H}_2) \vec{\varphi}_{n+1} & = \omega \left( \mathtt{D}
 \begin{pmatrix}
  n+1 & 0 & 0 \\
  0 & n+2 & 0 \\
  0 & 0 & n+3
 \end{pmatrix} + 
 \begin{pmatrix}
  \mathtt{X}_{1,1} & 0 & 0 \\
  2 \mathtt{X}_{0,2} & 0 & 0 \\
  0 & 4 \mathtt{X}_{0,2} & - \mathtt{X}_{1,1}  
 \end{pmatrix}
 \right) \vec{\varphi}_n \\
 & = \omega \left(
 \begin{pmatrix}
  n+2 & 0 & 0 \\
  0 & n+3 & 0 \\
  0 & 0 & n+4
 \end{pmatrix}
 \mathtt{D} \right) \vec{\varphi}_n, \\
 \omega(\mathtt{X}) \vec{\varphi}_{n+1} & = \omega \left( \mathtt{D}
 \begin{pmatrix}
  0 & 0 & 0 \\
  1 & 0 & 0 \\
  0 & 2 & 0
 \end{pmatrix} + 
 \begin{pmatrix}
  2 \mathtt{X}_{2,0} & 0 & 0 \\
  \mathtt{X}_{1,1} & 0 & 0 \\
  0 & 2 \mathtt{X}_{1,1} & - 2 \mathtt{X}_{2,0}
 \end{pmatrix}
 \right) \vec{\varphi}_n \\
 & = \omega \left(
 \begin{pmatrix}
  0 & 0 & 0 \\
  1 & 0 & 0 \\
  0 & 2 & 0
 \end{pmatrix}
 \mathtt{D} \right) \vec{\varphi}_n, \\
 \omega(\mathtt{Y}) \vec{\varphi}_{n+1} & = \omega \left( \mathtt{D}
 \begin{pmatrix}
  0 & 2 & 0 \\
  0 & 0 & 1 \\
  0 & 0 & 0
 \end{pmatrix} + 
 \begin{pmatrix}
  2 \mathtt{X}_{0,2} & - 2 \mathtt{X}_{1,1} & 0 \\
  0 & 0 & - \mathtt{X}_{1,1} \\
  0 & 0 & - 2 \mathtt{X}_{0,2}  
 \end{pmatrix}  
 \right) \vec{\varphi}_n \\
 & = \omega \left(
 \begin{pmatrix}
  0 & 2 & 0 \\
  0 & 0 & 1 \\
  0 & 0 & 0
 \end{pmatrix}
 \mathtt{D} \right) \vec{\varphi}_n.
\end{align*}
This completes the proof.
\end{proof}

Let $\cS(\C^2)_n$ be the space spanned by $\varphi_n, \varphi_n', \varphi_n''$, which is nonzero by Proposition \ref{p:phi-sharp} below.
Hence, by Lemma \ref{l:schwartz-3dim-n}, $\cS(\C^2)_n$ is stable under the action of $K_{G_1} \simeq \U(2)$ and is isomorphic to the $3$-dimensional representation 
\[
 \Sym^2 \otimes {\det}^{n+1}.
\]
Moreover, the spaces $\cS(\C^2)_n$ for all $n \ge 0$ form a ladder in the $K_{G_1}$-types of the cohomological representation $A_\fq$ described in Lemma \ref{l:Aq-K-types}.
Indeed, by Lemma \ref{l:theta-real}, we may realize $A_\fq$ on the space
\[
 \{ \varphi \in \cS(\C^2) \mid \text{$\omega(h) \varphi = h^2 \varphi$ for all $h \in \C^1$} \},
\]
which contains $\cS(\C^2)_n$ for all $n \ge 0$.
Put 
\[
 k(z) =
 \begin{pmatrix}
  \cos \theta & & \sin \theta & \\
  & \cos \theta & & \sin \theta \\
  - \sin \theta & & \cos \theta & \\
  & - \sin \theta & & \cos \theta 
 \end{pmatrix}, \quad
 k'(z) =
 \begin{pmatrix}
  \cos \theta & \sin \theta & & \\
  - \sin \theta & \cos \theta & & \\
  & & \cos \theta & \sin \theta \\
  & & - \sin \theta & \cos \theta 
 \end{pmatrix}
\]
for $z = e^{\sqrt{-1} \theta}$ with $\theta \in \R$.

\begin{lem}
\label{l:S(C^2)_n}
\begin{enumerate}
\item 
\label{S(C^2)_n-1}
We have
\[
 \omega(k(z)) \varphi = z^{2n+4} \varphi
\]
for all $z \in \C^1$ and $\varphi \in \cS(\C^2)_n$.
\item 
\label{S(C^2)_n-2}
The space
\[
 \{ \varphi \in \cS(\C^2)_n \mid \text{$\omega(k'(z)) \varphi = \varphi$ for all $z \in \C^1$} \}
\]
is spanned by $\varphi_n + \varphi_n''$.
\end{enumerate}
\end{lem}

\begin{proof}
We have 
\[
 k(z) = \exp(\sqrt{-1} \theta (\mathtt{H}_1 + \mathtt{H}_2)), \quad
 k'(z) = \exp(\theta (\mathtt{X} - \mathtt{Y}))
\]
for $z = e^{\sqrt{-1} \theta}$ with $\theta \in \R$.
From this and Lemma \ref{l:schwartz-3dim-n}, we can deduce the assertion.
\end{proof}

\subsection{Construction}
\label{ss:schwartz-def}

Recall the basis $\{ \mathtt{X}_\alpha \mid \alpha \in \Delta \}$ of $\fp$, where
\[
 \Delta = \{ \pm (2,0), \pm (1,1), \pm (0,2) \}.
\]
Let $\{ \omega_\alpha \mid \alpha \in \Delta \}$ be its dual basis of $\fp^*$.
As representations of $K_{G_1} \simeq \U(2)$, we have $\fp^+ \simeq \Sym^2$ and hence 
\[
 \wedge^2 (\fp^+)^* \simeq \Sym^2 \otimes {\det}^{-3}.
\]
Indeed, we can easily compute the action of $\fk$ on the following basis of $\wedge^2 (\fp^+)^*$:
\[
 \eta = \omega_{2,0} \wedge \omega_{1,1}, \quad
 \eta' = \omega_{2,0} \wedge \omega_{0,2}, \quad
 \eta'' = \omega_{1,1} \wedge \omega_{0,2}.
\]

\begin{lem}
\label{l:diff-form-K-type}
We have
\begin{align*}
 \mathtt{H}_1 \cdot \eta & = - 3 \eta, &
 \mathtt{H}_1 \cdot \eta' & = - 2 \eta', &
 \mathtt{H}_1 \cdot \eta'' & = - \eta'', \\
 \mathtt{H}_2 \cdot \eta & = - \eta, &
 \mathtt{H}_2 \cdot \eta' & = - 2 \eta', &
 \mathtt{H}_2 \cdot \eta'' & = - 3 \eta'', \\
 \mathtt{X} \cdot \eta & = - \eta', &
 \mathtt{X} \cdot \eta' & = - 2 \eta'', &
 \mathtt{X} \cdot \eta'' & = 0, \\
 \mathtt{Y} \cdot \eta & = 0, &
 \mathtt{Y} \cdot \eta' & = - 2 \eta, &
 \mathtt{Y} \cdot \eta'' & = - \eta'.
\end{align*}
\end{lem}

Recall that the relative Lie algebra cohomology
\[
 H^\bullet(\fg_1,K_{G_1};\cS(\C^2))
\]
is defined as the cohomology of the complex $(\cS(\C^2) \otimes \wedge^\bullet \fp^*)^{K_{G_1}}$ with differential
\[
 d = \sum_{\alpha \in \Delta} \omega(\mathtt{X}_\alpha) \otimes \varepsilon(\omega_\alpha),
\]
where $\varepsilon(\omega_\alpha)$ is the left exterior product by $\omega_\alpha$.
We regard
\[
 \vec{\varphi}_0 = 
 \begin{pmatrix}
  \varphi_0 \\
  \varphi'_0 \\
  \varphi''_0 
 \end{pmatrix}
\]
as an element in $\cS(\C^2) \otimes \wedge^2 \fp^*$ given by
\[
 \vec{\varphi}_0 = \varphi_0 \otimes \eta + \varphi_0' \otimes \eta' + \varphi_0'' \otimes \eta''.
\]
By Lemmas \ref{l:schwartz-3dim-0} and \ref{l:diff-form-K-type}, we have
\[
 \vec{\varphi}_0 \in (\cS(\C^2) \otimes \wedge^2 \fp^*)^{K_{G_1}}.
\]

\begin{lem}
\label{l:phi-closed}
We have
\[
 d \vec{\varphi}_0 = 0.
\]
\end{lem}

\begin{proof}
Put $\kappa = - 2 \pi t$.
Since 
\begin{align*}
 \mathtt{X}_{2,0} & = \frac{1}{2} (\mathtt{A}_1 + \sqrt{-1} \mathtt{B}_1 + \sqrt{-1} \mathtt{C}_1), & 
 \mathtt{X}_{-2,0} & = \frac{1}{2} (\mathtt{A}_1 - \sqrt{-1} \mathtt{B}_1 - \sqrt{-1} \mathtt{C}_1), \\ 
 \mathtt{X}_{1,1} & = \frac{1}{2} (\mathtt{A}_2 + \mathtt{A}_3 + \sqrt{-1} \mathtt{B}_2 + \sqrt{-1} \mathtt{C}_2), &
 \mathtt{X}_{-1,-1} & = \frac{1}{2} (\mathtt{A}_2 + \mathtt{A}_3 - \sqrt{-1} \mathtt{B}_2 - \sqrt{-1} \mathtt{C}_2), \\
 \mathtt{X}_{0,2} & = \frac{1}{2} (\mathtt{A}_4 + \sqrt{-1} \mathtt{B}_3 + \sqrt{-1} \mathtt{C}_3), & 
 \mathtt{X}_{0,-2} & = \frac{1}{2} (\mathtt{A}_4 - \sqrt{-1} \mathtt{B}_3 - \sqrt{-1} \mathtt{C}_3),
\end{align*}
it follows from Lemma \ref{l:schwartz-table} that 
\[
 \omega(*) \varphi_0^\bullet = P_*^\bullet \cdot \varphi^0,
\]
where $P_*^\bullet$ is a polynomial in $x_1, \bar{x}_1, x_2, \bar{x}_2$ given by the following table.
\[
{\renewcommand{\arraystretch}{1.2}
\begin{array}{c|l|l|l}
 * \ \diagdown \ \varphi_0^\bullet & \varphi_0 & \varphi_0' & \varphi_0'' \\ \hline
 \mathtt{X}_{2,0} & 2 \kappa x_1 \bar{x}_1^3 + 3 \bar{x}_1^2 & 2 \kappa x_1 \bar{x}_1^2 \bar{x}_2 + 2 \bar{x}_1 \bar{x}_2 & 2 \kappa x_1 \bar{x}_1 \bar{x}_2^2 + \bar{x}_2^2 \\
 \mathtt{X}_{1,1} & 2 \kappa (x_1 \bar{x}_1^2 \bar{x}_2 + \bar{x}_1^3 x_2) + 2 \bar{x}_1 \bar{x}_2 & 2 \kappa (x_1 \bar{x}_1 \bar{x}_2^2 + \bar{x}_1^2 x_2 \bar{x}_2) + \bar{x}_1^2 + \bar{x}_2^2 & 2 \kappa (x_1 \bar{x}_2^3 + \bar{x}_1 x_2 \bar{x}_2^2) + 2 \bar{x}_1 \bar{x}_2 \\
 \mathtt{X}_{0,2} & 2 \kappa \bar{x}_1^2 x_2 \bar{x}_2 + \bar{x}_1^2 & 2 \kappa \bar{x}_1 x_2 \bar{x}_2^2 + 2 \bar{x}_1 \bar{x}_2 & 2 \kappa x_2 \bar{x}_2^3 + 3 \bar{x}_2^2 \\
 \mathtt{X}_{-2,0} & 0 & 0 & 0 \\
 \mathtt{X}_{-1,-1} & 0 & 0 & 0 \\
 \mathtt{X}_{0,-2} & 0 & 0 & 0
\end{array}
}
\]
Hence we have 
\begin{align*}
 d(\varphi_0 \otimes \eta) & = \omega(\mathtt{X}_{0,2}) \varphi_0 \otimes (\omega_{0,2} \wedge \omega_{2,0} \wedge \omega_{1,1}), \\
 d(\varphi_0' \otimes \eta') & = \omega(\mathtt{X}_{1,1}) \varphi_0' \otimes (\omega_{1,1} \wedge \omega_{2,0} \wedge \omega_{0,2}), \\
 d(\varphi_0'' \otimes \eta'') & = \omega(\mathtt{X}_{2,0}) \varphi_0'' \otimes (\omega_{2,0} \wedge \omega_{1,1} \wedge \omega_{0,2}),
\end{align*}
and 
\[
 \omega(\mathtt{X}_{0,2}) \varphi_0 - \omega(\mathtt{X}_{1,1}) \varphi_0' + \omega(\mathtt{X}_{2,0}) \varphi_0'' = 0.
\]
This proves the lemma.
\end{proof}

Finally, we define a Schwartz form $\vec{\varphi} \in \cS(\C^2 \times \R^\times) \otimes \wedge^2 \fp^*$ by
\[
 \vec{\varphi}(x,t) = t \phi(t) \cdot \vec{\varphi}_{0,t}(x) = t \phi(t) \cdot (\varphi_{0,t}(x) \eta + \varphi_{0,t}'(x) \eta' + \varphi_{0,t}''(x) \eta''), 
\]
where we fix a smooth compactly supported real-valued function $\phi$ on $\R^\times$ such that $\supp(\phi) \subset \R^\times_+$ and
\[
 \int_0^\infty \phi(t) \, \frac{dt}{t} = \frac{1}{2 \pi}
\]
with the Lebesgue measure $dt$, and write $\vec{\varphi}_0 = \vec{\varphi}_{0,t}$, etc., to indicate the dependence on $t$.

\section{Schwartz functions}
\label{s:choices-schwartz}

In this section, we choose Schwartz functions explicitly which will be used later in the construction of modular forms.

\subsection{Setup}

Let $F$ be a non-archimedean local field of characteristic zero and fix a nontrivial additive character $\psi$ of $F$ of order zero.
Let $\fo_F$ be the ring of integers of $F$, $\fp_F$ the maximal ideal of $\fo_F$, and $\varpi_F$ a uniformizer of $\fo_F$.
For $i \in \Z$, put $\fp_F^i = \varpi_F^i \fo_F$.
Let $K$ and $E$ be \'etale quadratic algebras over $F$.
Write $K = F + F \bj$ and $E = F + F \bi$ for some trace zero elements $\bj \in K^\times$ and $\bi \in E^\times$.
Put $J = \bj^2 \in F^\times$ and $u = \bi^2 \in F^\times$.
Let $B$ be the quaternion algebra over $F$ of the form $B = K + K \bi = E + E \bj$.
We identify $B$ with $K^2$ (as $F$-vector spaces) via the map $(x_1,x_2) \mapsto x_1 + x_2 \bi$.
Recall the Weil representations $\Omega_\psi^B$ and $\Omega_\psi^K$ of $\GU(\bW) \times \GU(\bV)$ and $\GSp(W) \times \GO(V)$, respectively, on $\cS(B \times F^\times) = \cS(K^2 \times F^\times)$, where 
\begin{align*}
 \GU(\bW) & = (\GL_2(F) \times E^\times)/F^\times, & 
 \GU(\bV) & = (B^\times \times E^\times)/F^\times, \\
 \GSp(W) & = \GSp_4(F), &
 \GO(V) & = K^\times \rtimes \mu_2.
\end{align*}
We consider the six cases in the following table.
\begin{center}
{\renewcommand{\arraystretch}{1.3}
\begin{tabular}{cccc}
 $F$ & $K$ & $E$ & $B$ \\ \hline
 non-dyadic & split & any & split \\
 non-dyadic & split & split & split \\
 non-dyadic & inert & split or inert & split \\
 non-dyadic & inert & ramified & ramified \\
 non-dyadic & ramified & split & split \\
 dyadic & split & split & split \\
\end{tabular}
}
\end{center}

\subsection{The split case I}
\label{ss:choices-split-1}

Suppose that
\begin{itemize}
 \item $2 \in \fo_F^\times$;
 \item $J \in (\fo_F^\times)^2$;
 \item $u \in \fo_F^\times \cup \varpi_F \fo_F^\times$.
\end{itemize}
Fix $J' \in \fo_F^\times$ such that $(J')^2 = J$.
We identify $K$ and $B$ with $F \oplus F$ and $\M_2(F)$, respectively, via the maps
\[
 a + c \bj \mapsto (a+cJ', a-cJ'), \quad  
 a + b \bi + c \bj + d \bi \bj \mapsto \mat{a + cJ'}{b - dJ'}{(b + dJ')u}{a - cJ'}.
\]
Under this identification, we define a maximal order $\fo_B$ in $B$ by $\fo_B = \M_2(\fo_F)$.
Let $\varphi \in \cS(B \times F^\times)$ be the characteristic function of $\fo_B \times \fo_F^\times$.
Then we have
\[
 \Omega_\psi^B([k,z], [k',z']) \varphi = \varphi
\]
for $k \in \GL_2(\fo_F)$, $k' \in \fo_B^\times$, and $z, z' \in \fo_E^\times$.
If we regard $\varphi$ as an element in $\cS(K^2 \times F^\times)$, then $\varphi$ is the characteristic function of
\[
 (\fo_K \oplus (1, u)^{-1} \fo_K) \times \fo_F^\times.
\]
Hence, putting
\[
 \cK = 
 \begin{cases}
 \GSp_4(\fo_F) & \text{if $u \in \fo_F^\times$;} \\
 \left\{ g \in \GSp_4(\fo_F) \; \middle| \; g \equiv 
 \begin{pmatrix}
  * & * & * & 0 \\
  0 & * & 0 & 0 \\
  * & * & * & 0 \\
  * & * & * & * 
 \end{pmatrix}
 \bmod \fp_F \right\} & \text{if $u \in \varpi_F \fo_F^\times$,}
 \end{cases}
\]
we have
\[
 \Omega_\psi^K(g, h) \varphi = \varphi
\]
for $g \in \cK$ and
\[
 h \in 
 \begin{cases}
  \fo_K^\times \rtimes \mu_2 & \text{if $u \in \fo_F^\times$;} \\
  \fo_K^\times & \text{if $u \in \varpi_F \fo_F^\times$.}
 \end{cases}
\]

\subsection{The split case II}
\label{ss:choices-split-2}

Suppose that 
\begin{itemize}
 \item $2 \in \fo_F^\times$;
 \item $J \in (\fo_F^\times)^2$;
 \item $u \in (\fo_F^\times)^2$.
\end{itemize}
Fix $J', u' \in \fo_F^\times$ such that $(J')^2 = J$, $(u')^2 = u$.
We identify $K$ and $B$ with $F \oplus F$ and $\M_2(F)$, respectively, as in \S \ref{ss:choices-split-1}.
Under this identification, we define an order $\mathfrak{I}_B$ in $B$ as an Iwahori subalgebra $\mathfrak{I}$ of $\M_2(F)$ given by
\[
 \mathfrak{I} = \left\{ x \in \M_2(\fo_F) \; \middle| \; x \equiv \mat{*}{*}{0}{*} \bmod \fp_F \right\}.
\]
Put
\[
 k_0 = \mat{1}{\frac{1}{u'}}{1}{-\frac{1}{u'}}, 
\]
so that
\[
 k_0 (a + b\bi) k_0^{-1} = \mat{a + bu'}{0}{0}{a - bu'}.
\]
Let $\varphi \in \cS(B \times F^\times)$ be the characteristic function of $\mathfrak{I}_B k_0 \times \fo_F^\times$.
Then we have
\[
 \Omega_\psi^B([k,z], [k',z']) \varphi = \varphi
\]
for $k \in \mathfrak{I}^\times$, $k' \in \mathfrak{I}_B^\times$, and $z, z' \in \fo_E^\times$.
If we regard $\varphi$ as an element in $\cS(K^2 \times F^\times)$, then $\varphi$ is the characteristic function of
\[
 \{ (x_1, x_2) \in \fo_K \oplus \fo_K \mid x_1 + x_2 u' \in \fo_F \oplus \fp_F \} \times \fo_F^\times.
\]
Hence, putting
\[
 \cK = \left\{ g \in \GSp_4(\fo_F) \; \middle| \; g \equiv 
 \begin{pmatrix}
  a & 0 & * & * \\
  0 & a & * & * \\
  0 & 0 & * & 0 \\
  0 & 0 & 0 & * 
 \end{pmatrix}
 \bmod \fp_F, \; a \in \fo_F^\times \right\},
\]
we have
\[
 \Omega_\psi^K(g, h) \varphi = \varphi
\]
for $g \in \cK$ and $h \in \fo_K^\times$.

\subsection{The inert case I}
\label{ss:choices-inert-1}

Suppose that
\begin{itemize}
 \item $2 \in \fo_F^\times$;
 \item $J \in \fo_F^\times \smallsetminus (\fo_F^\times)^2$;
 \item $u \in \fo_F^\times$.
\end{itemize}
We define a maximal order $\fo_B$ in $B$ by
\[
 \fo_B = \fo_F + \fo_F \bi + \fo_F \bj + \fo_F \bi \bj.
\]
Let $\varphi \in \cS(B \times F^\times)$ be the characteristic function of $\fo_B \times \fo_F^\times$.
Then we have
\[
 \Omega_\psi^B([k,z], [k',z']) \varphi = \varphi
\]
for $k \in \GL_2(\fo_F)$, $k' \in \fo_B^\times$, and $z, z' \in \fo_E^\times$.
If we regard $\varphi$ as an element in $\cS(K^2 \times F^\times)$, then $\varphi$ is the characteristic function of
\[
 (\fo_K \oplus \fo_K) \times \fo_F^\times.
\]
Hence, putting $\cK = \GSp_4(\fo_F)$, we have
\[
 \Omega_\psi^K(g, h) \varphi = \varphi
\]
for $g \in \cK$ and $h \in \fo_K^\times \rtimes \mu_2$.

\subsection{The inert case II}
\label{ss:choices-inert-2}

Suppose that
\begin{itemize}
 \item $2 \in \fo_F^\times$;
 \item $J \in \fo_F^\times \smallsetminus (\fo_F^\times)^2$;
 \item $u \in \varpi_F \fo_F^\times$.
\end{itemize}
Let $\fo_B$ be the unique maximal order in $B$, so that 
\[
 \fo_B = \fo_F + \fo_F \bi + \fo_F \bj + \fo_F \bi \bj.
\]
Let $\varphi \in \cS(B \times F^\times)$ be the characteristic function of $\fo_B \times \fo_F^\times$.
Then we have
\[
 \Omega_\psi^B([k,z], [k',z']) \varphi = \varphi
\]
for $k \in \mathfrak{I}^\times$, $k' \in \fo_B^\times$, and $z, z' \in \fo_E^\times$, where $\mathfrak{I}$ is the Iwahori subalgebra of $\M_2(F)$ as in \S \ref{ss:choices-split-2}.
If we regard $\varphi$ as an element in $\cS(K^2 \times F^\times)$, then $\varphi$ is the characteristic function of
\[
 (\fo_K \oplus \fo_K) \times \fo_F^\times.
\]
Hence, putting $\cK = \GSp_4(\fo_F)$, we have
\[
 \Omega_\psi^K(g, h) \varphi = \varphi
\]
for $g \in \cK$ and $h \in \fo_K^\times \rtimes \mu_2$.

\subsection{The ramified case}
\label{ss:choices-ramified}

Suppose that 
\begin{itemize}
 \item $2 \in \fo_F^\times$;
 \item $J \in \varpi_F \fo_F^\times$;
 \item $u \in (\fo_F^\times)^2$.
\end{itemize}
Fix $u' \in \fo_F^\times$ such that $(u')^2 = u$.
We identify $E$ and $B$ with $F \oplus F$ and $\M_2(F)$, respectively, via the maps
\[
 a + b \bi \mapsto (a+bu', a-bu'), \quad  
 a + b \bi + c \bj + d \bi \bj \mapsto \mat{a + bu'}{c + du'}{(c - du')J}{a - bu'}. 
\]
Under this identification, we define a maximal order $\fo_B$ in $B$ by $\fo_B = \M_2(\fo_F)$.
Let $\varphi \in \cS(B \times F^\times)$ be the characteristic function of $\fo_B \times \fo_F^\times$.
Then we have
\[
 \Omega_\psi^B([k,z], [k',z']) \varphi = \varphi
\]
for $k \in \GL_2(\fo_F)$, $k' \in \fo_B^\times$, and $z, z' \in \fo_E^\times$.
If we regard $\varphi$ as an element in $\cS(K^2 \times F^\times)$, then $\varphi$ is the characteristic function of
\[
 \{ (x_1, x_2) \in \fp_K^{-1} \oplus \fp_K^{-1} \mid x_1 - x_2 u' \in \fo_K \} \times \fo_F^\times.
\]
Hence, putting 
\[
 \cK = \left\{ g \in \GSp_4(\fo_F) \; \middle| \; g \equiv 
 \begin{pmatrix}
  a & 0 & 0 & 0 \\
  0 & a & 0 & 0 \\
  0 & 0 & * & 0 \\
  0 & 0 & 0 & * 
 \end{pmatrix}
 \bmod \fp_F, \; a \in \fo_F^\times \right\},
\]
we have
\[
 \Omega_\psi^K(g, h) \varphi = \varphi
\]
for $g \in \cK$ and $h \in \fo_K^\times \rtimes \mu_2$.

\subsection{The dyadic case}
\label{ss:choices-dyadic}

Suppose that
\begin{itemize}
 \item $2 \in \fp_F$;
 \item $J \in (\fo_F^\times)^2$;
 \item $u \in (\fo_F^\times)^2$.
\end{itemize}
Fix $J', u' \in \fo_F^\times$ such that $(J')^2 = J$, $(u')^2 = u$.
We identify $K$ and $B$ with $F \oplus F$ and $\M_2(F)$, respectively, as in \S \ref{ss:choices-split-1}.
Under this identification, we define a maximal order $\fo_B$ in $B$ by $\fo_B = \M_2(\fo_F)$.
Let $\varphi \in \cS(B \times F^\times)$ be the characteristic function of $\fo_B k_0 \times 2^{-1} \fo_F^\times$, where $k_0$ is the element in $B^\times$ as in \S \ref{ss:choices-split-2}.
Then we have
\[
 \Omega_\psi^B([k,z], [k',z']) \varphi = \varphi
\]
for $k \in \GL_2(\fo_F)$, $k' \in \fo_B^\times$, and $z, z' \in \fo_E^\times$.
If we regard $\varphi$ as an element in $\cS(K^2 \times F^\times)$, then $\varphi$ is the characteristic function of 
\[
 \{ (x_1, x_2) \in \fo_K \oplus \fo_K \mid x_1 + x_2 u' \in 2 \fo_K \} \times 2^{-1} \fo_F^\times.
\]
Hence, putting 
\[
 \cK = \left\{ g \in \GSp_4(\fo_F) \; \middle| \; g \equiv 
 \begin{pmatrix}
  a & 0 & 0 & 0 \\
  0 & a & 0 & 0 \\
  0 & 0 & * & 0 \\
  0 & 0 & 0 & * 
 \end{pmatrix}
 \bmod 2 \fo_F, \; a \in \fo_F^\times \right\},
\]
we have
\[
 \Omega_\psi^K(g, h) \varphi = \varphi
\]
for $g \in \cK$ and $h \in \fo_K^\times \rtimes \mu_2$.

\section{The main construction of modular forms}
\label{s:main_construction}

In this section, we construct an algebraic and $p$-integral Siegel modular form of degree $2$ contributing to $H^{2,0}$.
We also introduce a $p$-depletion $f^\flat$ of $f$ in the following sense: if $a_f(T)$ and $a_{f^\flat}(T)$ are the $T$-th Fourier coefficients of $f$ and $f^\flat$ for $T \in \Sym_2(\Q)$, respectively, then we have $a_{f^\flat}(T) = 0$ unless $T \in \Sym_2(\Z_p)$ and $\det(T) \in \Z_p^\times$, in which case
\[
 a_{f^\flat}(T) = a_f(T).
\]

\subsection{Setup}
\label{ss:choices}

Fix an odd rational prime $p$.
Let $\Sigma$ be a finite set of odd rational primes such that $p \notin \Sigma$.
Let $\Qbar$ be the algebraic closure of $\Q$ in $\C$.
Fix an algebraic closure $\Qbar_p$ of $\Q_p$ and an embedding $\Qbar \hookrightarrow \Qbar_p$.

Let $K$ be an imaginary quadratic field such that
\begin{itemize}
\item $2$ and $p$ are split in $K$;
\item $v$ is split or inert in $K$ for all $v \in \Sigma$.
\end{itemize}
Fix an embedding $K \hookrightarrow \Qbar$, which determines embeddings $i_\infty : K \hookrightarrow \C$ and $i_p : K \hookrightarrow \Q_p$.
For each place $v$ of $\Q$, put $K_v = K \otimes_\Q \Q_v$.
We use the identification $K_\infty \simeq \C$ such that its restriction to $K$ agrees with $i_\infty$.
We also use the identification $K_p \simeq \Q_p \oplus \Q_p$ such that the restriction to $K$ of its first projection agrees with $i_p$.
Choose a trace zero element $\bj \in K^\times$ such that $K = \Q + \Q \bj$.
We may assume that $J = \bj^2$ is a negative square-free integer (so that $J$ is the discriminant of $K$).
Let $\Sigma_s^K$ (resp.~$\Sigma_i^K$, resp.~$\Sigma_r^K$) be the set of odd rational primes which are split (resp.~inert, resp.~ramified) in $K$.
Put $\Sigma^+ = \Sigma \cap \Sigma_s^K$ and $\Sigma^- = \Sigma \cap \Sigma_i^K$.
We assume that $\# \Sigma^-$ is odd.
For each $v \in \Sigma_s^K \cup \{ 2 \}$, we fix $J'_v \in \Z_v^\times$ such that $(J'_v)^2 = J$ and identify $K_v$ with $\Q_v \oplus \Q_v$ via the map $a + b \bj \mapsto (a + b J'_v, a - b J'_v)$.
(When $v=p$, we choose $J_v'$ such that it determines the above identification induced by $i_p$.)

Let $B$ be the quaternion algebra over $\Q$ which is ramified precisely at $\Sigma^- \cup \{ \infty \}$, so that $K$ embeds into $B$.
Choose a trace zero element $\bi \in B^\times$ such that $B = K + K \bi$.
Then $E = \Q + \Q \bi$ is an imaginary quadratic field such that
\begin{itemize}
\item $v$ is split or inert in $E$ for all $v \in \Sigma_i^K \smallsetminus \Sigma^-$;
\item $v$ is ramified in $E$ for all $v \in \Sigma^-$. 
\end{itemize}
In particular, we have $E \ne K$.
By replacing $\bi$ by $\alpha \bi$ with some $\alpha \in K^\times$ if necessary, we may assume that
\begin{itemize}
\item $p$ is split or inert in $E$;
\item $v$ is split in $E$ for all $v \in \Sigma^+ \cup \Sigma_r^K \cup \{ 2 \}$.
\end{itemize}
We may further assume that $u = \bi^2$ is a negative square-free integer (so that $u$ is the discriminant of $E$).
Let $\Sigma_s^E$ (resp.~$\Sigma_i^E$, resp.~$\Sigma_r^E$) be the set of odd rational primes which are split (resp.~inert, resp.~ramified) in $E$.
For each $v \in \Sigma_s^E \cup \{ 2 \}$, we fix $u'_v \in \Z_v^\times$ such that $(u'_v)^2 = u$ and identify $E_v$ with $\Q_v \oplus \Q_v$ via the map $a + b \bi \mapsto (a + b u'_v, a - b u'_v)$.

\subsection{Construction}
\label{ss:theta_construction}

Now we apply the theta lifting described in \S \ref{ss:theta-Sp4} (with $F = \Q$) to construct a Siegel modular form of degree $2$.
Let $\chi$ be a character of $\A_K^\times/K^\times$ such that
\begin{itemize}
\item $\chi|_{\A^\times} = 1$;
\item $\chi_\infty(z) = \bar{z}/z$;
\item $\chi_v$ is unramified for all $v \ne \infty$.
\end{itemize}
Let $\rho$ be the $3$-dimensional representation of $\GL_2(\C)$ on $V_\rho = \wedge^2 (\fp^+)^*$ such that $\rho(a - \sqrt{-1} b)$ agrees with the natural action of $\smat{a}{b}{-b}{a}$ for all $a + \sqrt{-1} b \in \U(2)$.
We identify $V_\rho$ with $\C^3$ via the basis $\eta, \eta', \eta''$ as in \S \ref{ss:schwartz-def}.
Let $\vec{\varphi} = \vec{\varphi}_\infty \otimes (\otimes_{v \ne \infty} \varphi_v) \in S(\A_K^2 \times \A^\times) \otimes V_\rho$ be the Schwartz form such that 
\begin{itemize}
\item $\vec{\varphi}_\infty \in S(\C^2 \times \R^\times) \otimes V_\rho$ is the Schwartz form as in \S \ref{ss:schwartz-def};
\item $\varphi_v \in S(K_v^2 \times \Q_v^\times)$ is the Schwartz function as in \S \ref{ss:choices-split-1} (resp.~\S \ref{ss:choices-split-2}, resp.~\S \ref{ss:choices-inert-1}, resp.~\S \ref{ss:choices-inert-2}, resp.~\S \ref{ss:choices-ramified}, resp.~\S \ref{ss:choices-dyadic}) if $v \in \Sigma_s^K \smallsetminus \Sigma^+$ (resp.~$v \in \Sigma^+$, resp.~$v \in \Sigma_i^K \smallsetminus \Sigma^-$, resp.~$v \in \Sigma^-$, resp.~$v \in \Sigma_r^K$, resp.~$v=2$).
\end{itemize}
Let $\phi = \theta_{\vec{\varphi}}(\chi)^\std$ be the theta lift (with respect to the standard additive character of $\A/\Q$ and the standard measure on $\A_K^\times$).
Then we have $\phi \in \cM_\rho(\cK)$, where $\cK = \prod_{v \ne \infty} \cK_v$ is the open compact subgroup of $\GSp_4(\A_f)$ such that $\cK_v$ is as in \S \ref{ss:choices-split-1} (resp.~\S \ref{ss:choices-split-2}, resp.~\S \ref{ss:choices-inert-1}, resp.~\S \ref{ss:choices-inert-2}, resp.~\S \ref{ss:choices-ramified}, resp.~\S \ref{ss:choices-dyadic}) if $v \in \Sigma_s^K \smallsetminus \Sigma^+$ (resp.~$v \in \Sigma^+$, resp.~$v \in \Sigma_i^K \smallsetminus \Sigma^-$, resp.~$v \in \Sigma^-$, resp.~$v \in \Sigma_r^K$, resp.~$v=2$).
Let $f \in M_\rho(\Gamma)$ be the classical modular form corresponding to $\phi$, where $\Gamma = \Sp_4(\Q) \cap \cK$.

To describe properties of the Fourier coefficients of $f$, we need to introduce more notation.
Let $\Q(\chi)$ be the field generated by $\chi(a)$ for all $a \in \A_{K,f}^\times$, where $\A_{K,f}$ is the ring of finite adeles of $K$.

\begin{lem}
We have $\Q(\chi) \supset K$ and $[\Q(\chi):\Q] < \infty$.
\end{lem}

\begin{proof}
Let $h = h_K$ be the class number of $K$ and fix a set of representatives $\{ a_1, \dots, a_h \}$ for $K^\times \backslash \A_{K,f}^\times / \hat{\fo}_K^\times$, where $\hat{\fo}_K = \prod_{v \ne \infty} \fo_{K_v}$.
Then $\Q(\chi)$ is generated by $\chi(a_1), \dots, \chi(a_h)$ and $\chi_\infty(b) = \bar{b}/b$ for all $b \in K^\times$.
In particular, we have $\Q(\chi) \supset K$, noting that $[K:\Q] = 2$.
On the other hand, we have $a_i^h \in b_i \hat{\fo}_K^\times$ for some $b_i \in K^\times$, so that $\chi(a_i)^h = \chi_\infty(b_i)^{-1} \in K$.
Hence $\Q(\chi) = K(\chi(a_1), \dots, \chi(a_h))$ is a finite extension of $K$.
This completes the proof.
\end{proof}

By this lemma, we may regard $\Q(\chi)$ as a subfield of $\Qbar_p$ via the fixed embedding $\Qbar \hookrightarrow \Qbar_p$.
Let $\cO$ be the completion of the ring of integers of $\Q(\chi)$ in $\Qbar_p$.

\begin{lem}
\label{l:chi-values-integral}
Let $v \ne \infty$ and $a \in K_v^\times$.
\begin{enumerate}
\item 
If $v \ne p$, then we have $\chi_v(a) \in \cO^\times$.
\item
If $v = p$ and $a = (p,1)$, then we have $\chi_v(a) \in p \cO^\times$.
\end{enumerate}
\end{lem}

\begin{proof}
We have $a^h \in b \hat{\fo}_K^\times$ for some $b \in K^\times$, so that $\chi_v(a)^h = \chi_\infty(b)^{-1} = b/\bar{b}$.
We regard $b,\bar{b}$ as elements in $\Q_p$.
Then we have
\[
 \begin{cases}
  b, \bar{b} \in \Z_p^\times & \text{if $v \ne p$;} \\
  b \in p^h \Z_p^\times, \; \bar{b} \in \Z_p^\times & \text{if $v = p$ and $a = (p,1)$.}
 \end{cases}
\]
This yields the assertion.
\end{proof}

Recall that
\[
 q(x) = \mat{x_1 \bar{x}_1}{\frac{1}{2} (x_1 \bar{x}_2 + \bar{x}_1 x_2)}{\frac{1}{2} (x_1 \bar{x}_2 + \bar{x}_1 x_2)}{x_2 \bar{x}_2}
\]
and $h_0 x = (\bar{x}_1, \bar{x}_2)$ for $x = (x_1, x_2) \in K^2$.

\begin{thm}
\label{t:fourier-theta}
Let $T \in \Sym_2(\Q)$.
\begin{enumerate}
\item 
\label{fourier-theta-1}
We have $a_f(T) = 0$ unless $T = t q(x)$ for some $x = (x_1,x_2) \in K^2 \smallsetminus \{ 0 \}$ and $t \in \Q^\times$ (so that $T \ne 0$), in which case
\[
 a_f(T) = \prod_{v \ne \infty} \cW_{x,t,v}(1) \cdot
 \begin{pmatrix}
  t \bar{x}_1^2 \\
  t \bar{x}_1 \bar{x}_2 \\
  t \bar{x}_2^2
 \end{pmatrix}
 + \prod_{v \ne \infty} \cW_{h_0 x,t,v}(1) \cdot
 \begin{pmatrix}
  t x_1^2 \\
  t x_1 x_2 \\
  t x_2^2
 \end{pmatrix}
\]
if $\det(T) \ne 0$ and
\[
 a_f(T) = \prod_{v \ne \infty} \cW_{x,t,v}(1) \cdot
 \begin{pmatrix}
  t \bar{x}_1^2 \\
  t \bar{x}_1 \bar{x}_2 \\
  t \bar{x}_2^2
 \end{pmatrix}
\]
if $\det(T) = 0$, where $\cW_{x,t,v}$ is a function on $\GSp_4(\Q_v)$ given by
\[
 \cW_{x,t,v}(g) = \int_{K_v^\times} \Omega_{\psi,v}(g) \varphi_v(h_v^{-1} x, \nu(h_v) t) \chi_v(h_v) \, dh_v
\]
with the standard measure $dh_v$ on $K_v^\times$.
\item
\label{fourier-theta-2}
We have $a_f(T) \in \Q(\chi)^3$.
\item 
\label{fourier-theta-3}
We have $a_f(T) \in \cO^3$.
\end{enumerate}
\end{thm}

\begin{proof}
Since $\chi_\infty(z) = \bar{z}/z$, we have $\chi|_{\A_K^1} \ne 1$.
Hence \eqref{fourier-theta-1} follows from \eqref{eq:a_f(T)}, Lemma \ref{l:W_T}, and Lemma \ref{l:fourier-real} below.
By Lemmas \ref{l:fourier-split-1}, \ref{l:fourier-split-2}, \ref{l:fourier-inert}, \ref{l:fourier-ramified}, \ref{l:fourier-dyadic} below, we have $\cW_{x,t,v}(1) \in \Q(\chi)$ for all $v \ne \infty$.
This proves \eqref{fourier-theta-2}.
Moreover, by combining these lemmas with Lemma \ref{l:chi-values-integral}, we have $\cW_{x,t,v}(1) \in \cO$ for all $v \ne \infty, p$ and
\[
 \cW_{x,t,p}(1) = \sum_{i=-m_2-m_0}^{m_1} \alpha^{2i+m_0} 
\]
with $\alpha = \chi_p(p,1) \in p \cO^\times$, where we regard $x_1, x_2, t$ as elements in $\Q_p$ and put
\[
 m_1 = \min \{ \ord_{\Q_p}(x_1), \ord_{\Q_p}(x_2) \}, \quad
 m_2 = \min \{ \ord_{\Q_p}(\bar{x}_1), \ord_{\Q_p}(\bar{x}_2) \}, \quad
 m_0 = \ord_{\Q_p}(t).
\]
In particular, we have
\[
 \cW_{x,t,p}(1) \cdot
 \begin{pmatrix}
  t \bar{x}_1^2 \\
  t \bar{x}_1 \bar{x}_2 \\
  t \bar{x}_2^2
 \end{pmatrix}
 \in \cO^3.
\]
This proves \eqref{fourier-theta-3}.
\end{proof}

\subsection{$p$-depletion}
\label{ss:p-depletion}

Define an element $\tau_p$ in the group ring $\Q[\Sp_4(\Q_p)]$ by 
\[
 \tau_p = I_4 - p^{-1} \sum_{a \in \GL_2(\Z_p) \backslash X_p} \boldsymbol{m}(a^{-1}) + p^{-1} \cdot \boldsymbol{m}(p^{-1} I_2)
\]
and put 
\[
 \phi^\flat = \tau_p \phi,
\]
where $X_p = \{ a \in \M_2(\Z_p) \mid \ord_{\Q_p}(\det(a)) = 1 \}$ and $\Q[\Sp_4(\Q_p)]$ acts on $\phi$ by right translation.
(Strictly speaking, we choose a set of representatives for $\GL_2(\Z_p) \backslash X_p$, but $\phi^\flat$ does not depend on this choice.)
Then we have $\phi^\flat \in \cM_\rho(\cK^\flat)$, where $\cK^\flat = \cK_p^\flat \prod_{v \ne \infty, p} \cK_v$ is the open compact subgroup of $\GSp_4(\A_f)$ with 
\[
 \cK_p^\flat = \left\{ g \in \GSp_4(\Z_p) \; \middle| \; g \equiv 
 \begin{pmatrix}
  * & * & * & * \\
  * & * & * & * \\
  0 & 0 & * & * \\
  0 & 0 & * & * 
 \end{pmatrix}
 \bmod p^2 \Z_p \right\}.
\]
Let $f^\flat \in M_\rho(\Gamma^\flat)$ be the classical modular form corresponding to $\phi^\flat$, where $\Gamma^\flat = \Sp_4(\Q) \cap \cK^\flat$.
We call $f^\flat$ the $p$-depletion of $f$, which is justified by the following.

\begin{thm}
\label{t:fourier-f-flat}
Let $T \in \Sym_2(\Q)$.
Then we have $a_{f^\flat}(T) = 0$ unless $T \in \Sym_2(\Z_p)$ and $\det(T) \in \Z_p^\times$, in which case 
\[
 a_{f^\flat}(T) = a_f(T).
\]
\end{thm}

\begin{proof}
The assertion follows from \eqref{eq:a_f(T)}, Lemma \ref{l:W_T}, and Proposition \ref{p:fourier-p-depletion} below.
\end{proof}

\begin{rem}
By Theorem \ref{t:fourier-f-flat}, the Fourier coefficients of $f^\flat$ at the cusp $\sqrt{-1} \infty$ are supported on non-singular symmetric matrices.
On the other hand, $\phi^\flat$ belongs to the automorphic representation generated by $\phi$, which is not cuspidal.
Hence $f^\flat$ has a nonzero singular Fourier coefficient at some other cusp.
\end{rem}

\section{Local Fourier coefficients}
\label{s:local_fourier}

To finish the proof of Theorems \ref{t:fourier-theta} and \ref{t:fourier-f-flat}, we need to compute the local integrals appearing in the Fourier coefficients of the theta lift explicitly.
We fix a place $v$ of $F$ and omit the subscript $v$ from the notation.
Let $\chi$ be a character of $K^\times$ such that $\chi|_{F^\times} = 1$ and $\varphi$ a Schwartz function on $K^2 \times F^\times$.
Put $T = t q(x)$ for $x \in K^2$ and $t \in F^\times$, i.e.~$T = \smat{T_1}{T_2/2}{T_2/2}{T_3}$ with 
\[
 T_1 = t \N_{K/F}(x_1), \quad
 T_2 = t \tr_{K/F}(\bar{x}_1 x_2), \quad
 T_3 = t \N_{K/F}(x_2).
\]
Assume that $T \ne 0$.
Then we may define a function $\cW_{x,t}$ on $\GSp_4(F)$ by
\[
 \cW_{x,t}(g) = \int_{K^\times} \Omega_\psi(g) \varphi(h^{-1} x, \nu(h) t) \chi(h) \, dh,
\]
where $dh$ is the standard measure on $K^\times$.

\subsection{The split case I}

We use the setting of \S \ref{ss:choices-split-1}.
In particular, $\varphi$ is the characteristic function of
\[
 (\fo_K \oplus (1, u)^{-1} \fo_K) \times \fo_F^\times
\]
with $u \in \fo_F^\times \cup \varpi_F \fo_F^\times$.
Put
\[
 k_{ij} = \ord_F(x_{ij}), \quad
 m_1 = \min \{ k_{11}, k_{21} \}, \quad
 m_2 = \min \{ k_{12}, k_{22} + \ord_F(u) \}
\]
for $x = (x_1, x_2)$ with $x_i = (x_{i1}, x_{i2})$ and $m_0 = \ord_F(t)$.

\begin{lem}
\label{l:fourier-split-1}
Assume that $\chi$ is unramified and put $\alpha = \chi(\varpi_F,1)$.
Then we have $\cW_{x,t}(1) = 0$ unless $T_1, T_2, T_3 \in u^{-1} \fo_F$, in which case
\[
 \cW_{x,t}(1) = \sum_{i=-m_2-m_0}^{m_1} \alpha^{2i+m_0}.
\]
\end{lem}

\begin{proof}
We have
\[
 \cW_{x,t}(1) = \sum_{n_1, n_2} \chi(\varpi_F^{n_1}, \varpi_F^{n_2}),
\]
where $n_1, n_2$ run over integers such that
\begin{align*}
 k_{11} - n_1 & \ge 0, & 
 k_{12} - n_2 & \ge 0, \\
 k_{21} - n_1 & \ge 0, & 
 k_{22} - n_2 & \ge - \ord_F(u),
\end{align*}
and $m_0 + n_1 + n_2 = 0$.
Hence, noting that $\chi(\varpi_F,\varpi_F) = 1$, we have
\[
 \cW_{x,t}(1) = \sum_{i=-m_2-m_0}^{m_1} \alpha^{2i+m_0}.
\]
On the other hand, we have
\begin{align*}
 \ord_F(T_1) & = k_{11} + k_{12} + m_0, \\
 \ord_F(T_2) & \ge \min \{ k_{11} + k_{22}, k_{12} + k_{21} \} + m_0, \\
 \ord_F(T_3) & = k_{21} + k_{22} + m_0, 
\end{align*}
so that
\begin{align*}
 \min\{ \ord_F(T_1), \ord_F(T_2), \ord_F(T_3) \} 
 & \ge \min \{ k_{11}, k_{21} \} + \min \{ k_{12}, k_{22} \} + m_0 \\
 & \ge m_1 + m_2 + m_0 - \ord_F(u).
\end{align*}
This completes the proof. 
\end{proof}

Suppose that $F = \Q_p$ with $p$ an odd prime and $u \in \Z_p^\times$.
Then $\cW_{x,t}$ is right $\GSp_4(\Z_p)$-invariant.
Define a function $\cW_{x,t}^\flat$ on $\GSp_4(\Q_p)$ by
\[
 \cW_{x,t}^\flat(g) = \cW_{x,t}(g) - p^{-1} \sum_{a \in \GL_2(\Z_p) \backslash X_p} \cW_{x,t}(g \boldsymbol{m}(a^{-1})) + p^{-1} \cW_{x,t}(g \boldsymbol{m}(p^{-1} I_2)),
\]
where $X_p = \{ a \in \M_2(\Z_p) \mid \ord_{\Q_p}(\det(a)) = 1 \}$.

\begin{prop}
\label{p:fourier-p-depletion}
We have $\cW_{x,t}^\flat(1) = 0$ unless $T_1, T_2, T_3 \in \Z_p$ and $\det(T) \in \Z_p^\times$, in which case
\[
 \cW_{x,t}^\flat(1) = \cW_{x,t}(1).
\]
\end{prop}

The rest of this subsection is devoted to the proof of Proposition \ref{p:fourier-p-depletion}.
We identify $x$ with the matrix
\[
 \mat{x_{11}}{x_{21}}{x_{12}}{x_{22}},
\]
on which $\GL_2(\Q_p)$ acts by right multiplication.
Under this identification, we have
\[
 \det(T) = - \frac{t^2 (x_1 \bar{x}_2 - \bar{x}_1 x_2)^2}{4}
 = -\frac{t^2 (x_{11} x_{22} - x_{12} x_{21})^2}{4}
 = -\frac{(t \det(x))^2}{4}.
\]
Put $w(x,t) = \cW_{x,t}(1)$, so that
\[
 w(x,t) = \sum_{i=-m_2-m_0}^{m_1} \alpha^{2i+m_0}
\]
by Lemma \ref{l:fourier-split-1}, where
\[
 m_1 = \min \{ k_{11}, k_{21} \}, \quad
 m_2 = \min \{ k_{12}, k_{22} \}, \quad 
 m_0 = \ord_{\Q_p}(t)
\]
with $k_{ij} = \ord_{\Q_p}(x_{ij})$.
Note that
\[
 w(xa, \nu t) = w(x,t) 
\]
for $a \in \GL_2(\Z_p)$ and $\nu \in \Z_p^\times$.
Put
\[
 \tilde{w}(x,t) = \sum_{a \in \GL_2(\Z_p) \backslash X_p} w(xa^{-1}, t).
\]

\begin{lem}
\label{l:tilde_w}
Assume that $\det(x) \ne 0$ and $x_{12} = 0$.
\begin{enumerate}
\item
If $k_{11} > k_{21}$, then we have $\tilde{w}(x,t) = 0$ unless $m_1 + m_2 + m_0 \ge 0$, in which case
\[
 \tilde{w}(x,t) = 
 \sum_{i=-m_2-m_0}^{m_1} \alpha^{2i+m_0}
 + p \sum_{i=-m_2-m_0+1}^{m_1-1} \alpha^{2i+m_0}.
\]
\item
If $x_{21} = 0$, then we have $\tilde{w}(x,t) = 0$ unless $m_1 + m_2 + m_0 > 0$, in which case
\[
 \tilde{w}(x,t) = 
 \sum_{i=-m_2-m_0}^{m_1} \alpha^{2i+m_0}
 + p \sum_{i=-m_2-m_0+1}^{m_1-1} \alpha^{2i+m_0}.
\]
\end{enumerate}
\end{lem}

\begin{proof}
Since $\{ a^{-1} \mid a \in X_p \} = p^{-1} X_p$ and
\[
 X_p = \mat{1}{0}{0}{p} \GL_2(\Z_p) \sqcup \bigsqcup_{y=0}^{p-1} \mat{p}{y}{0}{1} \GL_2(\Z_p), 
\]
we have
\begin{align*}
 \tilde{w}(x,t)
 & = w \left( x \mat{p^{-1}}{0}{0}{1}, t \right)
 + \sum_{y=0}^{p-1} w \left( x \mat{1}{p^{-1} y}{0}{p^{-1}}, t \right) \\
 & = w \left( \mat{p^{-1} x_{11}}{x_{21}}{0}{x_{22}}, t \right)
 + \sum_{y=0}^{p-1} w \left( \mat{x_{11}}{p^{-1} x_{11} y + p^{-1} x_{21}}{0}{p^{-1} x_{22}}, t \right).
\end{align*}
Hence we have
\[
 \tilde{w}(x,t) = 
 \sum_{i=-m_2-m_0}^{m_1} \alpha^{2i+m_0}
 + p \sum_{i=-m_2-m_0+1}^{m_1-1} \alpha^{2i+m_0}
\]
if $k_{11} > k_{21}$ and 
\[
 \tilde{w}(x,t) = 
 \sum_{i=-m_2-m_0}^{m_1-1} \alpha^{2i+m_0}
 + (p-1) \sum_{i=-m_2-m_0+1}^{m_1-1} \alpha^{2i+m_0}
 + \sum_{i=-m_2-m_0+1}^{m_1} \alpha^{2i+m_0}
\]
if $x_{21} = 0$.
This implies the assertion.
\end{proof}

Put $w^\flat(x,t) = \cW_{x,t}^\flat(1)$.
Since $\cW_{x,t}(\boldsymbol{m}(a)) = |\det(a)| \cdot \cW_{xa,t}(1)$ for $a \in \GL_2(\Q_p)$, we have
\[
 w^\flat(x,t) = w(x,t) - \tilde{w}(x,t) + p w(p^{-1} x,t).
\]

\begin{lem}
\label{l:flat_w}
We have $w^\flat(x,t) = 0$ unless $t \det(x) \in \Z_p^\times$, in which case
\[
 w^\flat(x,t) = w(x,t).
\]
\end{lem}

\begin{proof}
First assume that $\det(x) \ne 0$.
By replacing $x$ by an element in its $\GL_2(\Z_p)$-orbit, we may assume that $x_{12} = 0$ and either $k_{11} > k_{21}$ or $x_{21} = 0$.
Then, noting that
\[
 w(p^{-1} x,t) = \sum_{i=-m_2-m_0+1}^{m_1-1} \alpha^{2i+m_0}, 
\]
we can deduce from Lemma \ref{l:tilde_w} the following.
\begin{itemize}
\item 
If $m_1 + m_2 + m_0 < 0$, then we have
\[
 w^\flat(x,t) = w(x,t) = 0.
\]
\item 
If $m_1 + m_2 + m_0 \ge 0$ and $k_{11} > k_{21}$ (so that $\ord_{\Q_p}(t \det(x)) > m_1 + m_2 + m_0$), then we have
\[
 w^\flat(x,t) = 0.
\]
\item 
If $m_1 + m_2 + m_0 \ge 0$ and $x_{21} = 0$ (so that $\ord_{\Q_p}(t \det(x)) = m_1 + m_2 + m_0$), then we have
\[
 w^\flat(x,t) =
 \begin{cases}
  w(x,t) & \text{if $m_1 + m_2 + m_0 = 0$;} \\
  0 & \text{if $m_1 + m_2 + m_0 > 0$.}
 \end{cases}
\]
\end{itemize}
Next assume that $\det(x) = 0$.
By replacing $x$ by an element in its $\GL_2(\Z_p)$-orbit, we may assume that $x_{11} = x_{12} = 0$.
Then as in the proof of Lemma \ref{l:tilde_w}, we have
\begin{align*}
 \tilde{w}(x,t)
 & = w \left( \mat{0}{x_{21}}{0}{x_{22}}, t \right)
 + \sum_{y=0}^{p-1} w \left( \mat{0}{p^{-1} x_{21}}{0}{p^{-1} x_{22}}, t \right) \\
 & = w(x,t) + p w(p^{-1} x, t).
\end{align*}
This completes the proof. 
\end{proof}

Now Proposition \ref{p:fourier-p-depletion} follows from Lemmas \ref{l:fourier-split-1} and \ref{l:flat_w}.

\subsection{The split case II}

We use the setting of \S \ref{ss:choices-split-2}.
In particular, $\varphi$ is the characteristic function of
\[
 \{ (x_1, x_2) \in \fo_K \oplus \fo_K \mid x_1 + x_2 u' \in \fo_F \oplus \fp_F \} \times \fo_F^\times
\]
with $u' \in \fo_F^\times$.
Put
\[
 k_{ij} = \ord_F(x_{ij}), \quad
 m_1 = \min \{ k_{11}, k_{21} \}, \quad
 m_2 = \min \{ k_{12}, k_{22} \}
\]
for $x = (x_1, x_2)$ with $x_i = (x_{i1}, x_{i2})$ and $m_0 = \ord_F(t)$.
Put
\[
 m_2' = 
 \begin{cases}
  m_2 - 1 & \text{if $\ord_F(x_{12} + x_{22} u') = m_2$;} \\
  m_2 & \text{if $\ord_F(x_{12} + x_{22} u') > m_2$.}
 \end{cases}
\]

\begin{lem}
\label{l:fourier-split-2}
Assume that $\chi$ is unramified and put $\alpha = \chi(\varpi_F,1)$.
Then we have $\cW_{x,t}(1) = 0$ unless $T_1, T_2, T_3 \in \fo_F$, in which case
\[
 \cW_{x,t}(1) = \sum_{i=-m'_2-m_0}^{m_1} \alpha^{2i + m_0}.
\]
\end{lem}

\begin{proof}
We have
\[
 \cW_{x,t}(1) = \sum_{n_1, n_2} \mathbb{I}_{\fp_F}(\varpi_F^{-n_2}(x_{12} + x_{22} u')) \chi(\varpi_F^{n_1}, \varpi_F^{n_2}),
\]
where $n_1, n_2$ run over integers such that
\begin{align*}
 k_{11} - n_1 & \ge 0, & 
 k_{12} - n_2 & \ge 0, \\
 k_{21} - n_1 & \ge 0, & 
 k_{22} - n_2 & \ge 0,
\end{align*}
and $m_0 + n_1 + n_2 = 0$.
Hence, noting that $\chi(\varpi_F,\varpi_F) = 1$, we have
\begin{align*}
 \cW_{x,t}(1) & = \sum_{i=-m_2-m_0}^{m_1} \mathbb{I}_{\fp_F^{-i-m_0+1}}(x_{12} + x_{22} u') \alpha^{2i + m_0} \\
 & = \sum_{i=-m'_2-m_0}^{m_1} \alpha^{2i + m_0}.
\end{align*}
On the other hand, we have
\begin{align*}
 \ord_F(T_1) & = k_{11} + k_{12} + m_0, \\
 \ord_F(T_2) & \ge \min \{ k_{11} + k_{22}, k_{12} + k_{21} \} + m_0, \\
 \ord_F(T_3) & = k_{21} + k_{22} + m_0, 
\end{align*}
so that
\begin{align*}
 \min\{ \ord_F(T_1), \ord_F(T_2), \ord_F(T_3) \}  
 & \ge \min \{ k_{11}, k_{21} \} + \min \{ k_{12}, k_{22} \} + m_0 \\
 & \ge m_1 + m'_2 + m_0.
\end{align*}
This completes the proof.
\end{proof}

\subsection{The inert case}

We use the setting of \S \ref{ss:choices-inert-1} or \S \ref{ss:choices-inert-2}.
In particular, $\varphi$ is the characteristic function of
\[
 (\fo_K \oplus \fo_K) \times \fo_F^\times.
\]
Put $k_i = \ord_K(x_i)$ for $x = (x_1,x_2)$ and $m_0 = \ord_F(t)$.

\begin{lem}
\label{l:fourier-inert}
Assume that $\chi$ is unramified (and hence trivial).
Then we have $\cW_{x,t}(1) = 0$ unless $T_1, T_2, T_3 \in \fo_F$, in which case 
\[
 \cW_{x,t}(1) = 
 \begin{cases}
  1 & \text{if $m_0$ is even;} \\
  0 & \text{otherwise.}
 \end{cases}
\]
\end{lem}

\begin{proof}
We have
\[
 \cW_{x,t}(1) = \sum_n \chi(\varpi_F^n),
\]
where $n$ runs over integers such that
\[
 k_1 - n \ge 0, \quad
 k_2 - n \ge 0, \quad
 m_0 + 2n = 0.
\]
Hence we have
\[
 \cW_{x,t}(1) = 
 \begin{cases}
  1 & \text{if $m_0$ is even and $2 \min \{ k_1, k_2 \} + m_0 \ge 0$;} \\
  0 & \text{otherwise.}
 \end{cases}
\]
On the other hand, we have
\begin{align*}
 \ord_F(T_1) & = 2 k_1 + m_0, \\
 \ord_F(T_2) & \ge k_1 + k_2 + m_0, \\
 \ord_F(T_3) & = 2 k_2 + m_0, 
\end{align*}
so that 
\[
 \min\{ \ord_F(T_1), \ord_F(T_2), \ord_F(T_3) \} = 2 \min\{ k_1, k_2 \} + m_0.
\]
This completes the proof. 
\end{proof}

\subsection{The ramified case}

We use the setting of \S \ref{ss:choices-ramified}.
In particular, $\varphi$ is the characteristic function of
\[
 \{ (x_1, x_2) \in \fp_K^{-1} \oplus \fp_K^{-1} \mid x_1 - x_2 u' \in \fo_K \} \times \fo_F^\times
\]
with $u' \in \fo_F^\times$.
Put $k_i = \ord_K(x_i)$ for $x = (x_1,x_2)$ and $m_0 = \ord_F(t)$.

\begin{lem}
\label{l:fourier-ramified}
Assume that $\chi$ is unramified and put $\alpha = \chi(\varpi_K)$ (so that $\alpha = \pm 1$).
Then we have $\cW_{x,t}(1) = 0$ unless $T_1, T_2, T_3 \in \fp_F^{-1}$, in which case
\[
 \cW_{x,t}(1) =  
 \begin{cases}
  \alpha^{m_0} & \text{if $\ord_K(x_1 - x_2 u') + m_0 \ge 0$;} \\
  0 & \text{otherwise.}
 \end{cases}
\]
\end{lem}

\begin{proof}
We have
\[
 \cW_{x,t}(1) = \sum_n \mathbb{I}_{\fo_K}(\varpi_K^{-n}(x_1 - x_2 u')) \chi(\varpi_K^n),
\]
where $n$ runs over integers such that
\[
 k_1 - n \ge -1, \quad
 k_2 - n \ge -1, \quad
 m_0 + n = 0.
\]
Hence we have
\[
 \cW_{x,t}(1) =  
 \begin{cases}
  \alpha^{-m_0} & \text{if $\min \{ k_1, k_2 \} + m_0 + 1 \ge 0$ and $\ord_K(x_1 - x_2 u') + m_0 \ge 0$;} \\
  0 & \text{otherwise.}
 \end{cases}
\]
On the other hand, we have
\begin{align*}
 \ord_F(T_1) & = k_1 + m_0, \\
 \ord_F(T_2) & \ge \left[ \frac{k_1 + k_2 + 1}{2} \right] + m_0, \\
 \ord_F(T_3) & = k_2 + m_0,
\end{align*}
so that
\[
 \min\{ \ord_F(T_1), \ord_F(T_2), \ord_F(T_3) \} = \min\{ k_1, k_2 \} + m_0.
\]
This completes the proof. 
\end{proof}

\subsection{The dyadic case}

We use the setting of \S \ref{ss:choices-dyadic}.
In particular, $\varphi$ is the characteristic function of
\[
 \{ (x_1, x_2) \in \fo_K \oplus \fo_K \mid x_1 + x_2 u' \in 2 \fo_K \} \times 2^{-1} \fo_F^\times
\]
with $u' \in \fo_F^\times$.
Put
\[
 k_{ij} = \ord_F(x_{ij}), \quad
 m_1 = \min \{ k_{11}, k_{21} \}, \quad
 m_2 = \min \{ k_{12}, k_{22} \}
\]
for $x = (x_1, x_2)$ with $x_i = (x_{i1}, x_{i2})$ and $m_0 = \ord_F(2t)$.
Put
\[
 m_i' = \min \{ m_i, \ord_F(2^{-1} (x_{1i} + x_{2i} u')) \}.
\]

\begin{lem}
\label{l:fourier-dyadic}
Assume that $\chi$ is unramified and put $\alpha = \chi(\varpi_F,1)$.
Then we have $\cW_{x,t}(1) = 0$ unless $T_1, T_2, T_3 \in 2^{-1} \fo_F$, in which case
\[
 \cW_{x,t}(1) = \sum_{i=-m'_2-m_0}^{m'_1} \alpha^{2i + m_0}.
\]
\end{lem}

\begin{proof}
We have
\[
 \cW_{x,t}(1) = \sum_{n_1, n_2} \mathbb{I}_{2 \fo_F}(\varpi_F^{-n_1} (x_{11} + x_{21} u')) \mathbb{I}_{2 \fo_F}(\varpi_F^{-n_2} (x_{12} + x_{22} u')) \chi(\varpi_F^{n_1}, \varpi_F^{n_2}),
\]
where $n_1, n_2$ run over integers such that
\begin{align*}
 k_{11} - n_1 & \ge 0, & 
 k_{12} - n_2 & \ge 0, \\
 k_{21} - n_1 & \ge 0, & 
 k_{22} - n_2 & \ge 0,
\end{align*}
and $m_0 + n_1 + n_2 = 0$.
Hence, noting that $\chi(\varpi_F,\varpi_F) = 1$, we have
\begin{align*}
 \cW_{x,t}(1) & = \sum_{i=-m_2-m_0}^{m_1} \mathbb{I}_{2 \fp_F^i}(x_{11} + x_{21} u') \mathbb{I}_{2 \fp_F^{-i - m_0}}(x_{12} + x_{22} u') \alpha^{2i + m_0} \\
 & = \sum_{i=-m'_2-m_0}^{m_1'} \alpha^{2i + m_0}.
\end{align*}
On the other hand, we have
\begin{align*}
 \ord_F(2 T_1) & = k_{11} + k_{12} + m_0, \\
 \ord_F(2 T_2) & \ge \min \{ k_{11} + k_{22}, k_{12} + k_{21} \} + m_0, \\
 \ord_F(2 T_3) & = k_{21} + k_{22} + m_0, 
\end{align*}
so that
\begin{align*}
 \min\{ \ord_F(2 T_1), \ord_F(2 T_2), \ord_F(2 T_3) \} 
 & \ge \min \{ k_{11}, k_{21} \} + \min \{ k_{12}, k_{22} \} + m_0 \\
 & \ge m'_1 + m'_2 + m_0.
\end{align*}
This completes the proof. 
\end{proof}

\subsection{The real case}

We use the setting of \S \ref{ss:schwartz-def}.
In particular, we have
\[
 \vec{\varphi}(x,t) = t \phi(t) \cdot e^{- 2 \pi t (x_1 \bar{x}_1 + x_2 \bar{x}_2)} \cdot \vec{x}
\]
with
\[
 \vec{x} = 
 \begin{pmatrix}
  \bar{x}_1^2 \\
  \bar{x}_1 \bar{x}_2 \\
  \bar{x}_2^2
 \end{pmatrix}.
\]
This Schwartz form gives rise to a $\C^3$-valued function $\cW_{x,t}$ on $\GSp_4(\R)$ as before.

\begin{lem}
\label{l:fourier-real}
Assume that $\chi(z) = \bar{z}/z$.
Let $a \in \GL_2(\R)$ and $b \in \Sym_2(\R)$.
Then we have
\[
 \cW_{x,t}(\boldsymbol{n}(b) \boldsymbol{m}(a)) = 
 \begin{cases}
  t \cdot e^{2 \pi \sqrt{-1} \tr(T z)} \cdot \rho({}^t a) \vec{x} & \text{if $t>0$;} \\
  0 & \text{if $t<0$,}
 \end{cases}
\]
where $T = t q(x)$ and $z = b + \sqrt{-1} a {}^t a$.
\end{lem}

\begin{proof}
We have
\begin{align*}
 & \Omega_\psi(\boldsymbol{n}(b) \boldsymbol{m}(a)) \vec{\varphi}(h^{-1}x, \nu(h)t) \\
 & = \psi(t \tr(q(x) b)) \cdot \det(a) \cdot \vec{\varphi}(h^{-1}xa, \nu(h)t) \\
 & = \nu(h) t \cdot \phi(\nu(h) t) \cdot e^{2 \pi t \sqrt{-1} \tr(q(x) b)} \cdot e^{-2 \pi t \tr(q(x) a {}^t a)} \cdot \bar{h}^{-2} \cdot \rho({}^t a) \vec{x} \\
 & = \chi(h)^{-1} t \cdot \phi(\nu(h) t) \cdot e^{2 \pi \sqrt{-1} \tr(T z)} \cdot \rho({}^t a) \vec{x}, 
\end{align*}
so that
\[
 \cW_{x,t}(\boldsymbol{n}(b) \boldsymbol{m}(a))
 = t \cdot \int_{\C^\times} \phi(\nu(h) t) \, dh \cdot e^{2 \pi \sqrt{-1} \tr(T z)} \cdot \rho({}^t a) \vec{x}.
\]
Since 
\[
 \int_{\C^\times} \phi(\nu(h) t) \, dh = 2 \cdot 2 \pi \cdot \int_0^\infty \phi(r^2 t) \, \frac{dr}{r} = 
 \begin{cases}
  1 & \text{if $t>0$;} \\
  0 & \text{if $t<0$,}
 \end{cases}
\]
the desired identity follows.
\end{proof}

\section{The main formula for periods}
\label{s:main_formula}

In this section, we give a formula for the square of the period in terms of central values of Rankin-Selberg $L$-functions.

\subsection{Setup}
\label{ss:choices-2}

We use the setting of \S \ref{ss:choices}.
Let $n$ be a non-negative integer.
Let $\pi$ be an irreducible cuspidal automorphic representation of $\GL_2(\A_\Q)$ such that
\begin{itemize}
\item the central character of $\pi$ is trivial;
\item $\pi_\infty$ is the discrete series representation of weight $2n+4$;
\item $\pi_v = \St \otimes \xi_v$ for all $v \in \Sigma$, where $\St$ is the Steinberg representation of $\GL_2(\Q_v)$ and $\xi_v$ is an unramified (possibly trivial) quadratic character of $\Q_v^\times$;
\item $\pi_v$ is an unramified principal series representation for all $v \notin \Sigma \cup \{ \infty \}$.
\end{itemize}
Let $\chi$ be the character of $\A_K^\times/K^\times$ as in \S \ref{ss:theta_construction} such that 
\begin{itemize}
\item $\chi|_{\A_\Q^\times} = 1$;
\item $\chi_\infty(z) = \bar{z}/z$;
\item $\chi_v$ is unramified for all $v \ne \infty$.
\end{itemize}
Let $\mu$ be an auxiliary character of $\A_E^\times/E^\times$ such that
\begin{itemize}
\item $\mu|_{\A_\Q^\times} = 1$;
\item $\mu_\infty = 1$;
\item $\mu_v = \xi_v \circ \N_{E_v/\Q_v}$ for all $v \in \Sigma^-$;
\item $\mu_v$ is unramified for all $v \ne \infty$.
\end{itemize}
Define a character $\mu'$ of $\A_E^\times/E^\times$ by $\mu' = \mu \cdot (\xi_{K/\Q} \circ \N_{E/\Q})$, so that
\begin{itemize}
\item $\mu'|_{\A_\Q^\times} = 1$;
\item $\mu'_\infty = 1$.
\end{itemize}

\subsection{Periods}
\label{ss:main_periods}

Let $(\rho,V_\rho)$ be the $3$-dimensional representation of $\GL_2(\C)$ as in \S \ref{ss:theta_construction}.
By Lemmas \ref{l:varrho_m} and \ref{l:diff-form-K-type}, we may identify $(\rho \otimes \det^n, V_\rho \otimes V_{\det}^{\otimes n})$ with the representation $(\rho_{n+1}, V_{n+1})$ as in \S \ref{ss:diff-classical} so that
\[
 \eta \otimes \lambda^{\otimes n} = v_{n+1,0}, \quad
 \eta' \otimes \lambda^{\otimes n} = v_{n+1,1}, \quad
 \eta'' \otimes \lambda^{\otimes n} = v_{n+1,2},
\]
where $\eta, \eta', \eta''$ are the basis of $V_\rho$ as in \S \ref{ss:schwartz-def} and $\lambda$ is a fixed basis of $V_{\det}$.
Hence, by Proposition \ref{p:delta_mf}, the differential operator $\delta^n = \delta_n \circ \cdots \circ \delta_1$ induces a map
\[
 \delta^n : M_{\rho}(\Gamma) \rightarrow N_{\rho \otimes \det^n}^n(\Gamma)
\]
for any congruence subgroup $\Gamma$ of $\Sp_4(\Q)$.
Let $f \in M_\rho(\Gamma)$ and $f^\flat \in M_\rho(\Gamma^\flat)$ be the modular forms as in \S \ref{ss:theta_construction} and \S \ref{ss:p-depletion}, respectively, so that $\delta^n f \in N_{\rho \otimes \det^n}^n(\Gamma)$ and $\delta^n f^\flat \in N_{\rho \otimes \det^n}^n(\Gamma^\flat)$.
Let $\cU$ be an open compact subgroup of $\A_{E,f}^\times$ such that $\mu'|_{\cU} = 1$ and fix a set of representatives $\{ a_1, \dots, a_M \}$ for $E^\times \backslash \A_{E,f}^\times / \cU$.
As in \S \ref{ss:periods-adelic}, we define modular forms $(\delta^n f)_1, \dots, (\delta^n f)_M$ and $(\delta^n f^\flat)_1, \dots, (\delta^n f^\flat)_M$.

Let $\bff \in S_k(\Gamma_0(N))$ be the normalized Hecke eigenform corresponding to $\pi$, where $k=2n+4$ and $N = \prod_{q \in \Sigma} q$.
Put
\begin{align*}
 & P(\delta^n f, \bff, \mu') = M^{-1} \vol(\Gamma_0(N) \backslash \fH_1)^{-1} \\
 & \times \sum_{i=1}^M \mu'(a_i) \int_{\Gamma_0(N) \backslash \fH_1} \left\langle (\delta^n f)_i \mat{z}{0}{0}{-uz}, v_{n+1,0}^* - u v_{n+1,2}^* \right\rangle \overline{\bff(z)} \Im(z)^{2n+4} \, d \nu(z).
\end{align*}
Recall the (completed) Rankin-Selberg $L$-functions
\begin{align*}
 L(s, \pi_K \times \chi) & = L_\infty(s, \pi_K \times \chi) \cdot L_{\fin}(s, \pi_K \times \chi), \\
 L(s, \pi_E \times \mu) & = L_\infty(s, \pi_E \times \mu) \cdot L_{\fin}(s, \pi_E \times \mu),
\end{align*}
where $\pi_K$ and $\pi_E$ are the base changes of $\pi$ to $\GL_2(\A_K)$ and $\GL_2(\A_E)$, respectively.
Note that
\begin{align*}
 L_\infty(s, \pi_K \times \chi) & = \Gamma_\C \left( s+\frac{k+1}{2} \right) \Gamma_\C \left( s+\frac{k-3}{2} \right), \\
 L_\infty(s, \pi_E \times \mu) & = \Gamma_\C \left( s+\frac{k-1}{2} \right)^2,
\end{align*}
where $\Gamma_\C(s) = 2 (2\pi)^{-s} \Gamma(s)$.

\begin{thm}
\label{t:period-L-value-1}
We have
\[
 P(\delta^n f, \bff, \mu')^2 = C \cdot L(\tfrac{1}{2}, \pi_K \times \chi) \cdot L(\tfrac{1}{2}, \pi_E \times \mu),
\]
where
\begin{align*}
 C & = (-1)^{n+1} 2^{-2n-11} \pi^{4} \cdot |D_K|^{1/2} \cdot |D_E|^{-n-1/2} \rho_E^{-2} \cdot \zeta(2)^{-2} \\
 & \times \prod_{q \in \Sigma_s^K \cap \Sigma_r^E} \chi_q(1,q)
 \cdot \prod_{q \in \Sigma^+} \frac{\chi_q(q,1) \mu_q(1,q)}{(q+1)^2}
 \cdot \prod_{q \in \Sigma^-} \frac{2 \xi_q(q)}{(q+1)^2}
 \cdot \prod_{q \in \Sigma_r^K} \mu_q(q,1).
\end{align*}
Here $D_K$ and $D_E$ are the discriminants of $K$ and $E$, respectively, $\rho_E$ is the residue of the Dedekind zeta function $\zeta_E(s)$ of $E$ at $s=1$, and $\zeta(s)$ is the Riemann zeta function.
\end{thm}

For the $p$-adic interpolation, we need to consider the period of the $p$-depleted modular form against the $p$-stabilized modular form.
We adopt the following convention for the slash operator:
\[
 (\bff|_k g)(z) = \det(g)^{k/2} (cz+d)^{-k} \bff(gz)
\]
for $g = \smat{a}{b}{c}{d} \in \GL_2(\R)^+$ and $z \in \fH_1$.
We write the Fourier expansion of $\bff$ as
\[
 \bff(z) = \sum_{m=1}^\infty a_{\bff}(m) e^{2 \pi \sqrt{-1} mz}.
\]
Let $\alpha, \beta$ be the roots of the Hecke polynomial
\[
 X^2 - a_{\bff}(p) X + p^{k-1}.
\]
Let $\bff_\alpha \in S_k(\Gamma_0(Np))$ be the $p$-stabilization of $\bff$ with respect to $\alpha$:
\[
 \bff_\alpha(z) = \bff(z) - \beta \bff(pz).
\]
Let $\bff_\alpha^\rho \in S_k(\Gamma_0(Np))$ be the complex conjugate of $\bff_\alpha$:
\[
 \bff_\alpha^\rho(z) = \sum_{m=1}^\infty \overline{a_{\bff_\alpha}(m)} e^{2 \pi \sqrt{-1} mz}.
\]
Define a cusp form $\tilde{\bff}_\alpha \in S_k(\Gamma_0(Np^2))$ by
\[
 \tilde{\bff}_\alpha(z) = \left( \bff_\alpha^\rho|_k \mat{0}{-1}{Np}{0} \right)(pz).
\]
Put
\begin{align*}
 & P(\delta^n f^\flat, \tilde{\bff}_\alpha, \mu') = M^{-1} \vol(\Gamma_0(Np^2) \backslash \fH_1)^{-1} \\
 & \times \sum_{i=1}^M \mu'(a_i) \int_{\Gamma_0(Np^2) \backslash \fH_1} \left\langle (\delta^n f^\flat)_i \mat{z}{0}{0}{-uz}, v_{n+1,0}^* - u v_{n+1,2}^* \right\rangle \overline{\tilde{\bff}_\alpha(z)} \Im(z)^{2n+4} \, d \nu(z).
\end{align*}
We also need to modify the Rankin-Selberg $L$-functions by including the factors at $\infty$ and $p$ as in \cite[\S 1 and \S 2]{coates89}:
\begin{align*}
 \cE_\infty(\pi_K \times \chi) & = (-1)^n, \\
 \cE_\infty(\pi_E \times \mu) & = (-1)^n, \\
 \cE_p(\pi_K \times \chi) & = (1 - \xi_p(p) \chi_p(p,1) p^{-1/2})^2 (1 - \xi_p(p) \chi_p(1,p) p^{-1/2})^2, \\
 \cE_p(\pi_E \times \mu) & = 
 \begin{cases}
  (1 - \xi_p(p) \mu_p(p,1) p^{-1/2})^2 (1 - \xi_p(p) \mu_p(1,p) p^{-1/2})^2 & \text{if $p$ is split in $E$;} \\
  (1 - \xi_p(p)^2 p^{-1})^2 & \text{if $p$ is inert in $E$,} 
 \end{cases}
\end{align*}
where $\xi_p$ is the unitary unramified character of $\Q_p^\times$ such that $\xi_p(p) = \beta p^{-(k-1)/2}$ (so that $\pi_p = \Ind(\xi_p \otimes \xi_p^{-1})$).

\begin{thm}
\label{t:period-L-value-2}
We have
\begin{align*}
 P(\delta^n f^\flat, \tilde{\bff}_\alpha, \mu')^2 & = C \cdot \alpha^4 p^{-6n-12} (1+p^{-1})^{-2} \\
 & \times \cE_\infty(\pi_K \times \chi) \cdot \cE_p(\pi_K \times \chi) \cdot L(\tfrac{1}{2}, \pi_K \times \chi) \\
 & \times \cE_\infty(\pi_E \times \mu) \cdot \cE_p(\pi_E \times \mu) \cdot L(\tfrac{1}{2}, \pi_E \times \mu),
\end{align*}
where $C$ is the constant in Theorem \ref{t:period-L-value-1}.
\end{thm}

The rest of this paper is devoted to the proof of Theorems \ref{t:period-L-value-1} and \ref{t:period-L-value-2}.

\section{Expression as toric periods}
\label{s:main_seesaw}

In this section, we express the periods in Theorems \ref{t:period-L-value-1} and \ref{t:period-L-value-2} in terms of toric periods.

\subsection{From classical to adelic}
\label{ss:adelic_period}

Let $\vec{\varphi} = \vec{\varphi}_\infty \otimes (\otimes_{v \ne \infty} \varphi_v) \in S(\A_K^2 \times \A^\times) \otimes V_\rho \otimes V_{\det}^{\otimes n}$ be the Schwartz form as in \S \ref{ss:theta_construction} but with the following modification: $\vec{\varphi}_\infty$ is given by
\[
 \vec{\varphi}_\infty(x,t) = t \phi(t) \cdot (\varphi_{n,t}(x) \eta + \varphi_{n,t}'(x) \eta' + \varphi_{n,t}''(x) \eta'') \otimes \lambda^{\otimes n},
\]
where $\varphi_{n,t}, \varphi'_{n,t}, \varphi_{n,t}'' \in \cS(\C^2)$ are as in \S \ref{ss:schwartz-derivatives} (with a subscript indicating the dependence on $t$) and $\phi \in C_c^\infty(\R^\times)$ is as in \S \ref{ss:schwartz-def}.
This induces a Schwartz function $\varphi = \varphi_\infty \otimes (\otimes_{v \ne \infty} \varphi_v) \in S(\A_K^2 \times \A^\times)$ given by
\[
 \varphi_\infty = \ell \circ \Omega_{\psi,\infty}(g_0) \vec{\varphi}_\infty,
\]
where $\ell : V_\rho \otimes V_{\det}^{\otimes n} \rightarrow \C$ and $g_0 \in \Sp_4(\R)$ are as in \S \ref{ss:periods-adelic}.
Note that $\varphi_\infty = \Omega_{\psi,\infty}(g_0) (\ell \circ \vec{\varphi}_\infty)$, where
\[
 (\ell \circ \vec{\varphi}_\infty)(x,t) = t \phi(t) \cdot (\varphi_{n,t}(x) + \varphi''_{n,t}(x)).
\]
We also modify $\varphi$ at $p$ and define a Schwartz function $\varphi^\flat = \varphi_p^\flat \otimes (\otimes_{v \ne p} \varphi_v) \in S(\A_K^2 \times \A^\times)$ by 
\[
 \varphi_p^\flat = \tau_p \varphi_p,
\]
where $\tau_p$ is the Hecke operator as in \S \ref{ss:p-depletion}.

We write $f_0 \in \pi$ for the automorphic form on $\GL_2(\A)$ corresponding to $\bar{\bff}$:
\[
 f_0(\gamma g h) = \overline{(\bff|_k g)(\sqrt{-1})}
\]
for $\gamma \in \GL_2(\Q)$, $g \in \GL_2(\R)^+$, and $h \in \cK_0(N)$, where
\[
 \cK_0(N) = \left\{ h \in \GL_2(\hat{\Z}) \; \middle| \; h \equiv \mat{*}{*}{0}{*} \bmod N \hat{\Z} \right\}.
 \]
Note that $f_0$ is decomposable.
We also write $f_\alpha$ for the automorphic form on $\GL_2(\A)$ corresponding to $\overline{\tilde{\bff}_\alpha}$.
Put $\bG = (\GL_2(\Q) \times E^\times)/\Q^\times$.
We extend $f_0$ and $f_\alpha$ to automorphic forms $f_0^{\mu'}$ and $f_\alpha^{\mu'}$ on $\bG(\A)$, respectively, as in \S \ref{ss:theta-U11}.

\begin{lem}
\label{l:adelic_period}
We have
\begin{align*}
 P(\delta^n f, \bff, \mu') & = C' \cdot \cP^{\bG}(\theta_{\varphi}(\chi)^\std, f_0^{\mu'}), \\
 P(\delta^n f^\flat, \tilde{\bff}_\alpha, \mu') & = C' \cdot \cP^{\bG}(\theta_{\varphi^\flat}(\chi)^\std, f^{\mu'}_\alpha), 
\end{align*}
where
\[
 C' = (-1)^n 2^{-2n-2} \pi^{-n} |u|^{-(n+1)/2}.
\]
\end{lem}

\begin{proof}
Put $\phi_0 = \theta_{\vec{\varphi}}(\chi)^\std$.
By Proposition \ref{p:delta-f}, $\delta^n f$ corresponds to the adelic automorphic form $(-4\pi)^{-n} \cdot \phi_0$.
Hence $\delta^n f^\flat$ corresponds to the adelic automorphic form $(-4\pi)^{-n} \cdot \tau_p \phi_0$, noting that the differential operator commutes with the Hecke operator.
By Lemma \ref{l:weil-seesaw} and the choice of $\vec{\varphi}$, we have
\begin{align*}
 \phi_0(g \iota([h,z])) & = \phi_0(g), \\
 \tau_p \phi_0(g \iota([h',z])) & = \tau_p \phi_0(g)
\end{align*}
for all $g \in \GSp_4(\A)$, $h \in \cK_0(N)$, $h' \in \cK_0(Np^2)$, and $z \in \cU$.
Hence, by Proposition \ref{p:periods-classical}, we have
\begin{align*}
 (-4\pi)^{-n} \cdot \cP^{\bG}(\ell \circ \tilde{\phi}_0, f_0^{\mu'}) & = 4 |u|^{(n+1)/2} \cdot P(\delta^n f, \bff, \mu'), \\
 (-4\pi)^{-n} \cdot \cP^{\bG}(\ell \circ \tau_p \tilde{\phi}_0, f_\alpha^{\mu'}) & = 4 |u|^{(n+1)/2} \cdot P(\delta^n f^\flat, \tilde{\bff}_\alpha, \mu'),
\end{align*}
where $\tilde{\phi}_0(g) = \phi_0(g g_0)$.
 Since $\ell \circ \tilde{\phi}_0 = \theta_{\varphi}(\chi)^\std$ and $\ell \circ \tau_p \tilde{\phi}_0 = \theta_{\varphi^\flat}(\chi)^\std$, the assertion follows.
\end{proof}

\subsection{A seesaw identity}
\label{ss:toric_period}

Put $H = K^\times$ and $\bH = (B^\times \times E^\times)/\Q^\times$.
As in \S \ref{ss:weil-seesaw}, we regard $H$ as a subgroup of $\bH$ via the map $h \mapsto [h,1]$.
We also regard $\varphi$ as an element in $S(B(\A) \times \A^\times)$ and consider the theta lift $\theta_\varphi(f_0^\mu)$ as in \S \ref{ss:theta-U11}.
Then Lemma \ref{l:seesaw} gives the following.

\begin{lem}
\label{l:seesaw-explicit}
We have
\[
 \cP^{\bG}(\theta_{\varphi}(\chi)^\std, f^{\mu'}_0)
 = C_K \cdot \cP^H(\theta_\varphi(f_0^\mu), \chi), 
\]
where
\[
 C_K = |D_K|^{1/2} \cdot \Res_{s=1} \zeta_K(s).
\]
\end{lem}

\subsection{Behavior under $p$-depletion and $p$-stabilization}

To compute $\cP^{\bG}(\theta_{\varphi^\flat}(\chi)^\std, f^{\mu'}_\alpha)$, we relate it to $\cP^{\bG}(\theta_{\varphi}(\chi)^\std, f^{\mu'}_0)$.
We first express $f_\alpha$ in terms of $f_0$.

\begin{lem}
\label{l:p-stabilization}
If we write $f_0 = \otimes_v f_v \in \pi$, then we have
\[
 f_\alpha = \varepsilon_{\fin}(\pi) p^{-k/2} \cdot \tilde{f}_p \otimes (\otimes_{v \ne p} f_v) \in \pi, 
\]
where $\varepsilon_{\fin}(\pi)$ is the value of the standard $\varepsilon$-function $\varepsilon_{\fin}(s, \pi, \psi)$ at $s = \tfrac{1}{2}$ (which does not depend on $\psi$) and
\[
 \tilde{f}_p = \pi_p(t_p^{-2}) f_p - \xi_p(p) p^{-1/2} \pi_p(t_p^{-1}) f_p \in \pi_p
\]
with 
\[
 t_p = \mat{p}{0}{0}{1} \in \GL_2(\Q_p).
\]
\end{lem}

\begin{proof}
Since $a_{\bff}(m) \in \R$ for all $m \ge 1$, we have
\[
 \bff_\alpha^\rho = \bff - \alpha p^{-k/2} \bff|_k \mat{p}{0}{0}{1}.
\]
Hence we have
\begin{align*}
 \tilde{\bff}_\alpha
 & = p^{-k/2} \bff_\alpha^\rho|_k \mat{0}{-1}{Np}{0} \mat{p}{0}{0}{1} \\ 
 & = p^{-k/2} \bff|_k \mat{0}{-1}{Np}{0} \mat{p}{0}{0}{1}
 - \alpha p^{-k} \bff|_k \mat{p}{0}{0}{1} \mat{0}{-1}{Np}{0} \mat{p}{0}{0}{1} \\
 & = p^{-k/2} \bff|_k \mat{0}{-1}{N}{0} \mat{p^2}{0}{0}{1}
 - \alpha p^{-k} \bff|_k \mat{0}{-1}{N}{0} \mat{p^2}{0}{0}{p}.
\end{align*}
Moreover, since
\[
 \bff|_k \mat{0}{-1}{N}{0} = \varepsilon_{\fin}(\pi) \bff,
\]
we have
\[
 \tilde{\bff}_\alpha = \varepsilon_{\fin}(\pi) p^{-k/2}
 \left( \bff|_k \mat{p^2}{0}{0}{1} - \alpha p^{-k/2} \bff|_k \mat{p}{0}{0}{1} \right), 
\]
so that 
\[
 \overline{\tilde{\bff}_\alpha} = \varepsilon_{\fin}(\pi) p^{-k/2}
 \left( \overline{\bff|_k \mat{p^2}{0}{0}{1}} - \beta p^{-k/2} \overline{\bff|_k \mat{p}{0}{0}{1}} \right).
\]
This yields the assertion.
\end{proof}

\begin{lem}
\label{l:kill-p-depletion}
We have
\[
 \cP^{\bG}(\theta_{\varphi^\flat}(\chi)^\std, f^{\mu'}_\alpha)
 = \varepsilon_{\fin}(\pi) \cdot C_p \cdot \cP^{\bG}(\theta_{\varphi}(\chi)^\std, f^{\mu'}_0), 
\]
where 
\begin{align*}
 C_p & = \alpha^2 p^{-3k/2} (1+p^{-1})^{-1} \\
 & \times (1 - \xi_p(p) \chi_p(p,1) p^{-1/2}) (1 - \xi_p(p) \chi_p(1,p) p^{-1/2}) \\
 & \times 
 \begin{cases}
  (1 - \xi_p(p) \mu_p(p,1) p^{-1/2}) (1 - \xi_p(p) \mu_p(1,p) p^{-1/2}) & \text{if $p$ is split in $E$;} \\
  (1 - \xi_p(p)^2 p^{-1}) & \text{if $p$ is inert in $E$.} 
 \end{cases}
\end{align*}
\end{lem}

\begin{proof}
By abuse of notation, we write $f_0^{\mu'} = \otimes_v f_v$, where $f_v$ is regarded as an element in $\pi_v \boxtimes \mu'_v$.
For $\varphi'_p \in S(K_p^2 \times \Q_p^\times)$ and $f_p' \in \pi_p \boxtimes \mu'_p$, put
\[
 \cQ^{\mathrm{aut}}(\varphi'_p,f'_p) = \cP^{\bG}(\theta_{\varphi'}(\chi)^\std, f'),
\]
where 
\[
 \varphi' = \varphi_p' \otimes (\otimes_{v \ne p} \varphi_v), \quad
 f' = f'_p \otimes (\otimes_{v \ne p} f_v).
\]
Then we have
\[
 \cQ^{\mathrm{aut}}(\varphi_p, f_p) = \cP^{\bG}(\theta_{\varphi}(\chi)^\std, f_0^{\mu'}).
\]
Also, by Lemma \ref{l:p-stabilization}, we have
\[
 \cQ^{\mathrm{aut}}(\varphi_p^\flat, \tilde{f}_p) = \frac{\cP^{\bG}(\theta_{\varphi^\flat}(\chi)^\std, f_\alpha^{\mu'})}{\varepsilon_{\fin}(\pi) p^{-k/2}}.
\]
On the other hand, by Lemma \ref{l:weil-seesaw}, $\cQ^{\mathrm{aut}}$ defines an element in
\[
 \Hom_{\bG_p \times H_p}(\Omega_{\psi,p} \otimes (\pi_p \boxtimes \mu_p) \otimes \chi_p, \C) \simeq
 \Hom_{H_p}(\Theta(\pi_p \boxtimes \mu_p) \otimes \chi_p, \C), 
\]
where $\Theta(\pi_p \boxtimes \mu_p)$ is the big theta lift of $\pi_p \boxtimes \mu_p$.
As shown in the proof of Lemma \ref{l:seesaw-local}, the above Hom space is $1$-dimensional.
Pick any nonzero element $\cQ$ in this space.
By Lemmas \ref{l:unique-hom-unram-inert} and \ref{l:unique-hom-unram-split} below, we have $\cQ(\varphi_p, f_p) \ne 0$ and hence 
\[
 \cQ^{\mathrm{aut}}(\varphi_p^\flat, \tilde{f}_p)
 = \cQ^{\mathrm{aut}}(\varphi_p, f_p) \cdot 
 \frac{\cQ(\varphi_p^\flat, \tilde{f}_p)}{\cQ(\varphi_p, f_p)}.
\]
From this and Lemma \ref{l:unique-hom-p-depletion} below, we can deduce the assertion.
\end{proof}

\section{Local functionals}
\label{s:local_functionals}

In this section, we finish the proof of Lemma \ref{l:kill-p-depletion}.
We write $F = \Q_p$ with $p$ an odd prime and omit the subscript $p$ from the notation.

\subsection{An explicit formula}

We use the setting of \S \ref{ss:choices-split-1}.
In particular, the natural embedding $K \hookrightarrow B$ is given by
\[
 (x_1, x_2) \mapsto \mat{x_1}{}{}{x_2}
\]
under the identifications $K = F \oplus F$ and $B = \M_2(F)$ as in \S \ref{ss:choices-split-1}.
Suppose further that $u \in \fo_F^\times$.
When $E$ is split (so that $u \in (\fo_F^\times)^2$), we fix $u' \in \fo_F^\times$ such that $(u')^2 = u$ and identify $E$ with $F \oplus F$ via the map $a+b\bi \mapsto (a+bu',a-bu')$.

As in \S \ref{ss:weil-seesaw}, we realize the Weil representation $\Omega_\psi$ of $\bG \times \bH$ on $\cS(K^2 \times F^\times)$.
Let $\pi = \Ind(\xi \otimes \xi^{-1})$ be a principal series representation of $\GL_2(F)$, where $\xi$ is a unitary unramified characters of $F^\times$, and $\mu$ a unitary unramified character of $E^\times$ such that $\mu|_{F^\times} = 1$.
Note that $\mu$ is trivial if $E$ is inert.
We regard $\pi \boxtimes \mu$ as an irreducible representation of $\bG$.
Let $\chi$ be a unitary unramified character of $K^\times$ such that $\chi|_{F^\times} = 1$.
Then as explained in the proof of Lemma \ref{l:kill-p-depletion}, the space
\begin{equation}
\label{eq:hom-space-p-depletion}
 \Hom_{\bG \times H}(\Omega_\psi \otimes (\pi \boxtimes \mu) \otimes \chi, \C) 
\end{equation}
is $1$-dimensional.
Pick any nonzero element $\cQ$ in this space.
Let $\varphi_0 \in \cS(K^2 \times F^\times)$ be the characteristic function of $(\fo_K \oplus \fo_K) \times \fo_F^\times$ and put $\varphi^\flat = \tau \varphi_0$, where $\tau \in \Q[\Sp_4(F)]$ is the element as in \S \ref{ss:p-depletion}.
Let $w_0 \in \pi \boxtimes \mu$ be the unique (up to scalars) nonzero $\bK$-fixed vector and put
\[
 \tilde{w}_0 = (\pi \boxtimes \mu)([t_p^{-2}, 1]) w_0 - \xi(p) p^{-1/2} \cdot (\pi \boxtimes \mu)([t_p^{-1}, 1]) w_0, 
\]
where $\bK$ is a maximal compact subgroup of $\bG$ given by
\[
 \bK = \{ [k,z] \mid k \in \GL_2(\fo_F), \; z \in \fo_E^\times \}
\]
and 
\[
 t_p = \mat{p}{0}{0}{1} \in \GL_2(F).
\]
Note that $\tilde{w}_0$ is $\bK_0$-fixed, where
\[
 \bK_0 = \{ [k,z] \mid k \in K_0, \; z \in \fo_E^\times \}
\]
with 
\[
 K_0 = \left\{ k \in \GL_2(\fo_F) \; \middle| \; k \equiv \mat{*}{*}{0}{*} \bmod p^2 \fo_F \right\}.
\]
By Lemmas \ref{l:unique-hom-unram-inert} and \ref{l:unique-hom-unram-split} below, we have $\cQ(\varphi_0,w_0) \ne 0$.
Moreover, we have:

\begin{lem}
\label{l:unique-hom-p-depletion}
\begin{enumerate}
\item 
\label{unique-hom-p-depletion-inert}
If $E$ is inert, then we have
\begin{align*}
 \frac{\cQ(\varphi^\flat,\tilde{w}_0)}{\cQ(\varphi_0,w_0)} 
 = \xi(p)^{-2} p^{-1} (1+p^{-1})^{-1}
 & \times (1 - \xi(p) \chi(p,1) p^{-1/2}) (1 - \xi(p) \chi(1,p) p^{-1/2}) \\
 & \times (1 - \xi(p)^2 p^{-1}).
\end{align*}
\item
\label{unique-hom-p-depletion-split}
If $E$ is split, then we have
\begin{align*}
 \frac{\cQ(\varphi^\flat,\tilde{w}_0)}{\cQ(\varphi_0,w_0)}
 = \xi(p)^{-2} p^{-1} (1+p^{-1})^{-1}
 & \times (1 - \xi(p) \chi(p,1) p^{-1/2}) (1 - \xi(p) \chi(1,p) p^{-1/2}) \\
 & \times (1 - \xi(p) \mu(p,1) p^{-1/2})(1 - \xi(p) \mu(1,p) p^{-1/2}).
\end{align*}
\end{enumerate}
\end{lem} 

The rest of this section is devoted to the proof of Lemma \ref{l:unique-hom-p-depletion}.
For this, we will construct $\cQ$ explicitly, and compute $\cQ(\varphi_0,w_0)$ and $\cQ(\varphi^\flat,\tilde{w}_0)$ directly.
We distinguish two cases according to whether $E$ is inert or split, except that we compute $\varphi^\flat$ uniformly as follows.

As in \S \ref{ss:weil-pft}, we realize the Weil representation $\hat{\Omega}_\psi$ of $\bG \times \bH$ on $\cS(E^2 \times F^\times)$ via the partial Fourier transform
\[
 \cF_\psi : \cS(K^2 \times F^\times) = \cS(B \times F^\times) \rightarrow \cS(E^2 \times F^\times), 
\]
where we identify $K^2$ with $B$ via the map
\[
 ((x_{11},x_{12}), (x_{21},x_{22})) \mapsto \mat{x_{11}}{x_{21}}{x_{22}u}{x_{12}}.
\]
Explicitly, we have
\[
 \cF_\psi(\varphi)(r_1 + s_1 \bi, r_2 + s_2 \bi, t) = \int_{F^2} \varphi((r, s_1), (s, r_1 u^{-1}), t) \psi(t(s_2 r -r_2 s)) |t| \, dr \, ds,
\]
where $dr, ds$ are the self-dual Haar measures on $F$ with respect to $\psi$.
Put $\hat{\varphi}_0 = \cF_\psi(\varphi_0)$ and $\hat{\varphi}^\flat = \cF_\psi(\varphi^\flat)$.
Then $\hat{\varphi}_0$ is the characteristic function of $(\fo_E \oplus \fo_E) \times \fo_F^\times$ and $\hat{\varphi}^\flat$ is given as follows.

\begin{lem}
\label{l:varphi-hat-flat}
We have
\[
 \hat{\varphi}^\flat(x_1, x_2, t) =
 \begin{cases}
  1 - p^{-1} & \text{if $x_1 \in \fo_E \smallsetminus p \fo_E$, $x_2 \in \fo_E$, $t \in \fo_F^\times$;} \\
  - p^{-1} & \text{if $x_1 \in \fo_E \smallsetminus p \fo_E$, $x_2 \in p^{-1} \fo_E \smallsetminus \fo_E$, $r_1 r_2 - s_1 s_2 u \in \fo_F$, $t \in \fo_F^\times$;} \\
  0 & \text{otherwise}
 \end{cases}
\]
for $x_1, x_2 \in E$ with $x_i = r_i + s_i \bi$ and $t \in F^\times$.
(Note that when $E$ is split and $x_i = (x_{i1}, x_{i2})$, we have $r_1 r_2 - s_1 s_2 u \in \fo_F$ if and only if $x_{11} x_{22} + x_{12} x_{21} \in \fo_F$.)
In particular, $\hat{\varphi}^\flat$ is $\bK_0$-fixed.
\end{lem}

\begin{proof}
Recall that
\[
 \varphi^\flat = \varphi_0 - p^{-1} \Omega_\psi \left(\boldsymbol{m} \mat{p^{-1}}{0}{0}{1} \right) \varphi_0 - p^{-1} \sum_{k=0}^{p-1} \Omega_\psi \left(\boldsymbol{m} \mat{1}{p^{-1} k}{0}{p^{-1}} \right) \varphi_0 + p^{-1} \Omega_\psi(\boldsymbol{m}(p^{-1} I_2)) \varphi_0, 
\]
so that 
\[
 \varphi^\flat(x_1,x_2,t) = \varphi_0(x_1,x_2,t)
 - \varphi_0(p^{-1} x_1,x_2,t)
 - \sum_{k=0}^{p-1} \varphi_0(x_1,p^{-1}(kx_1 + x_2),t)
 + p \varphi_0(p^{-1}x_1,p^{-1}x_2,t)
\]
for $x_1, x_2 \in K$ and $t \in F^\times$.
From this, we deduce that
\[
 \hat{\varphi}^\flat = \mathbb{I}_{\fo_E \oplus \fo_E} \otimes \mathbb{I}_{\fo_F^\times}
 - \sum_{k=0}^p \cF_\psi(\varphi'_k) \otimes \mathbb{I}_{\fo_F^\times} + 
 p^{-1} \mathbb{I}_{p \fo_E \oplus p^{-1} \fo_E} \otimes \mathbb{I}_{\fo_F^\times},
\]
where $\varphi'_k \in \cS(K^2)$ is a Schwartz function given by 
\[
 \varphi'_k(x_1, x_2) = \mathbb{I}_{\fo_K}(x_1) \mathbb{I}_{p \fo_K}(k x_1 + x_2)
\]
for $0 \le k \le p-1$ and 
\[
 \varphi'_p(x_1, x_2) = \mathbb{I}_{p \fo_K}(x_1) \mathbb{I}_{\fo_K}(x_2), 
\]
and $\cF_\psi(\varphi_k') \in \cS(E^2)$ is its partial Fourier transform, so that
\[
 \cF_\psi(\varphi_k')(r_1 + s_1 \bi, r_2 + s_2 \bi) = \int_{F^2} \varphi_k'((r, s_1), (s, r_1 u^{-1})) \psi(s_2 r -r_2 s) \, dr \, ds.
\]
By a direct computation, we have
\[
 \cF_\psi(\varphi_k')(r_1 + s_1 \bi, r_2 + s_2 \bi) = p^{-1} \mathbb{I}_{\fo_F}(s_1) \mathbb{I}_{p \fo_F}(k s_1 + r_1 u^{-1}) \mathbb{I}_{\fo_F}(s_2 + k r_2) \mathbb{I}_{p^{-1} \fo_F}(r_2)
\]
for $0 \le k \le p-1$ and 
\[
 \cF_\psi(\varphi_p')(r_1 + s_1 \bi, r_2 + s_2 \bi) = p^{-1} \mathbb{I}_{p \fo_F}(s_1) \mathbb{I}_{\fo_F}(r_1) \mathbb{I}_{p^{-1} \fo_F}(s_2) \mathbb{I}_{\fo_F}(r_2).
\]
Hence we have
\[
 \hat{\varphi}^\flat(x_1, x_2, t) = p^{-1} I(x_1, x_2) \mathbb{I}_{\fo_F^\times}(t)
\]
for $x_1, x_2 \in E$ and $t \in F^\times$, where 
\[
 I(x_1, x_2) = p \mathbb{I}_{\fo_E}(x_1) \mathbb{I}_{\fo_E}(x_2) + \mathbb{I}_{p \fo_E}(x_1) \mathbb{I}_{p^{-1} \fo_E}(x_2) - \sum_{k=0}^p \mathbb{I}_{U_k}(x_1) \mathbb{I}_{U'_k}(x_2)
\]
with 
\begin{align*}
 U_k & = \{ r + s \bi \mid r, s \in \fo_F, \; r + ksu \in p \fo_F \}, \\
 U'_k & = \{ r + s \bi \mid r, s \in p^{-1} \fo_F, \; kr + s \in \fo_F \}
\end{align*}
for $0 \le k \le p-1$ and 
\begin{align*}
 U_p & = \{ r + s \bi \mid r \in \fo_F, \; s \in p \fo_F \}, \\
 U'_p & = \{ r + s \bi \mid r \in \fo_F, \; s \in p^{-1} \fo_F \}. 
\end{align*}
Since
\[
 \fo_E = p \fo_E \sqcup \bigsqcup_{k=0}^p (U_k \smallsetminus p \fo_E), \quad
 p^{-1} \fo_E = \fo_E \sqcup \bigsqcup_{k=0}^p (U'_k \smallsetminus \fo_E),
\]
we have
\[
 I(x_1,x_2) = 
 \begin{cases}
  p + 1 - (p+1) = 0 & \text{if $x_1 \in p \fo_E$ and $x_2 \in \fo_E$;} \\  
  0 + 1 - 1 = 0 & \text{if $x_1 \in p \fo_E$ and $x_2 \in p^{-1} \fo_E \smallsetminus \fo_E$;} \\
  p + 0 - 1 = p-1 & \text{if $x_1 \in \fo_E \smallsetminus p \fo_E$ and $x_2 \in \fo_E$;} \\
  0 + 0 - 1 = -1 & \text{if $x_1 \in U_k \smallsetminus p \fo_E$ and $x_2 \in U'_k \smallsetminus \fo_E$;} \\
  0 + 0 - 0 = 0 & \text{if $x_1 \in U_k \smallsetminus p \fo_E$ and $x_2 \in U'_{k'} \smallsetminus \fo_E$ with $k \ne k'$;} \\
  0 & \text{otherwise.}
 \end{cases}
\]
Moreover, writing $x_i = r_i + s_i \bi$ for $x_1 \in \fo_E$ and $x_2 \in p^{-1} \fo_E$, we have $x_1 \in U_k$ and $x_2 \in U'_k$ for some $0 \le k \le p$ if and only if $r_1 r_2 - s_1 s_2 u \in \fo_F$.
This yields the desired identity.
Moreover, by combining this with Lemma \ref{l:weil-pft2}, we can deduce that $\hat{\varphi}^\flat$ is $\bK_0$-fixed.
This completes the proof.
\end{proof}

We also introduce some notation which will be used in both cases.
Put 
\[
 k_{(1,m)} = \mat{1}{0}{m}{1}, \quad
 k_{(pm,1)} = \mat{pm}{-1}{1}{0}.
\]
Then we can take
\[
 \{ k_{(1,m)} \mid 0 \le m \le p^2-1 \} \cup
 \{ k_{(pm,1)} \mid 0 \le m \le p-1 \}
\]
as a set of representatives for $\GL_2(\fo_F)/K_0$, noting that the map $\smat{a}{b}{c}{d} \mapsto (a,c) \bmod p^2 \fo_F$ induces a bijection $\GL_2(\fo_F)/K_0 \simeq \mathbb{P}^1(\Z/p^2 \Z)$.

\subsection{The inert case}

Suppose that $E$ is inert (so that $u \in \fo_F^\times \smallsetminus (\fo_F^\times)^2$).
Let $\bT$ be a maximal torus of $\bG$ given by
\[
 \bT = \{ [\iota(z), z'] \mid z, z' \in E^\times \},
\]
where
\[
 \iota(a + b \bi) = \mat{a}{bu}{b}{a}.
\]
Put $x_0 = (1, \bi) \in E^2$ and let $\bT_0$ be the stabilizer of $x_0$ in $\bG$, so that 
\[
 \bT_0 = \{ [\iota(z), z] \mid z \in E^\times \}.
\]
Following \cite[Lemme 8]{wald85}, we take the model $\mathfrak{W}(\pi \boxtimes \mu)$ defined as the image of a nonzero element in the $1$-dimensional space $\Hom_{\bG}(\pi \boxtimes \mu, \Ind^{\bG}_{\bT}(\mu_{\bT}))$, where $\mu_\bT$ is the trivial character of $\bT$.
Since $\bT \subset Z_{\mathbf{G}} \cdot \mathbf{K}$, any $\mathbf{K}$-fixed vector in $(\pi \boxtimes \mu)^\vee$ belongs to $\Hom_{\bT}(\pi \boxtimes \mu, \mu_{\bT})$, so that $\mathfrak{W}(\pi \boxtimes \mu)$ agrees with the space of left $\mathbf{K}$-invariant matrix coefficients of $\pi \boxtimes \mu$.
We normalize the $\bK$-fixed vector $w_0 \in \mathfrak{W}(\pi \boxtimes \mu)$ (i.e.~the $\mathbf{K}$-bi-invariant matrix coefficient of $\pi \boxtimes \mu$) so that $w_0(1) = 1$.
Put
\[
 \cQ(\varphi,w)
 = \int_{F^\times \backslash K^\times} \int_{\bT_0 \backslash \bG} 
 \hat{\Omega}_\psi(g,h) \cF_\psi(\varphi)(x_0,1) w(g) \chi(h) \, dg \, dh
\]
for $\varphi \in \cS(K^2 \times F^\times)$ and $w \in \mathfrak{W}(\pi \boxtimes \mu)$, where we normalize the measures so that $\vol(\bT_0 \backslash \bT_0 \bK) = \vol(F^\times \backslash F^\times \fo_K^\times) = 1$.

\begin{lem}
\label{l:abs-conv-inert}
The integral $\cQ(\varphi,w)$ is absolutely convergent.
\end{lem} 

\begin{proof}
Since $\cQ(\varphi, (\pi \boxtimes \mu)(g) w) = \cQ(\Omega_\psi(g^{-1}) \varphi, w)$ for all $g \in \bG$, we may assume that $w$ is $\bK$-fixed.
Since
\[
 \cQ(\varphi,w) = \int_{F^\times} \int_{\bT_0 \backslash \bG} |a| \cF_\psi(\varphi)(x_0 g, \nu(g)^{-1} a) w(g) \chi(a,1) \, dg \, da 
\]
by Lemma \ref{l:weil-pft2}, we may further assume that $\cF_\psi(\varphi)$ is the characteristic function of $p^i (\fo_E \oplus \fo_E) \times p^l \fo_F^\times$ for some $i,l \in \Z$.
Then, noting that $\hat{\varphi}_0$ is the characteristic function of $(\fo_E \oplus \fo_E) \times \fo_F^\times$, we have
\begin{align*}
 \cQ(\varphi,w) & = \int_{F^\times} \int_{\bT_0 \backslash \bG} |a| \hat{\varphi}_0(p^{-i} x_0 g, p^{-l} \nu(g)^{-1} a) w(g) \chi(a,1) \, dg \, da \\
 & = \int_{F^\times} \int_{\bT_0 \backslash \bG} |p^{2i+l} a| \hat{\varphi}_0(x_0 g, \nu(g)^{-1} a) w(g \cdot [1,p^{-i}]) \chi(p^{2i+l} a,1) \, dg \, da \\
 & = |p^{2i+l}| \chi(p^{2i+l},1) \cQ(\varphi_0,w).
\end{align*}
Hence the absolute convergence of $\cQ(\varphi,w)$ follows from the computation in the proof of Lemma \ref{l:unique-hom-unram-inert} below.
\end{proof}

By this lemma, the above integral defines an element $\cQ$ in the Hom space \eqref{eq:hom-space-p-depletion}.

\begin{lem}
\label{l:unique-hom-unram-inert}
We have
\[
 \cQ(\varphi_0,w_0) = \frac{1}{(1 - \xi(p) \chi(p,1) p^{-1/2}) (1 - \xi(p)^{-1} \chi(p,1) p^{-1/2})}.
\]
In particular, $\cQ$ is nonzero.
\end{lem} 

\begin{proof}
Put $\alpha = \xi(p)$ and $\beta = \chi(p,1)$.
Put $g_{i,k} = [g_k, p^{-i}] \in \bG$ with
\[
 g_k = \mat{p^k}{0}{0}{1}
\]
and $h_l = (p^l, 1) \in K^\times$.
Then we can take $\{ g_{i,k} \mid i,k \in \Z, \; k \ge 0 \}$ and $\{ h_l \mid l \in \Z \}$ as sets of representatives for $\bT_0 \backslash \bG / \bK$ and $F^\times \backslash K^\times / \fo_K^\times$, respectively (see \cite[Lemma 3.4.1]{furusawa93}), so that
\[
 \cQ(\varphi_0,w_0) = \sum_{k=0}^\infty \sum_{i,l \in \Z} \hat{\Omega}_\psi(g_{i,k},h_l) \hat{\varphi}_0(x_0,1) w_0(g_{i,k}) \chi(h_l) \vol(\bT_0 \backslash \bT_0 g_{i,k} \bK) \vol(F^\times \backslash F^\times h_l \fo_K^\times).
\]
By Lemma \ref{l:weil-pft2}, we have
\begin{align*}
 \hat{\Omega}_\psi(g_{i,k},h_l) \hat{\varphi}_0(x_0,1)
 & = p^{-l} \hat{\varphi}_0(p^{i+k},p^i\bi, p^{-2i-k+l}) \\
 & = 
 \begin{cases}
  p^{-2i-k} & \text{if $i \ge 0$ and $l=2i+k$;} \\
  0 & \text{otherwise.} 
 \end{cases}
\end{align*}
Also, by Macdonald's formula, we have
\[
 w_0(g_{i,k}) = \frac{p^{-k/2}}{1+p^{-1}} A_k
\]
for $k \ge 0$, where
\[
 A_k = \alpha^k \frac{1 - \alpha^{-2} p^{-1}}{1-\alpha^{-2}}
 + \alpha^{-k} \frac{1 - \alpha^2 p^{-1}}{1-\alpha^2}.
\]
Moreover, it is easy to see that
\begin{align*}
 \vol(\bT_0 \backslash \bT_0 g_{i,k} \bK)
 & = \vol(\bT_0 \cap g_{i,k} \bK g_{i,k}^{-1})^{-1} \\
 & = 
 \begin{cases}
  1 & \text{if $k=0$;} \\
  p^k (1+p^{-1}) & \text{if $k \ge 1$}
 \end{cases}
\end{align*}
(see \cite[Lemma 3.5.3]{furusawa93}) and
\begin{equation}
\label{eq:v0-hl-inert}
 \chi(h_l) = \beta^l, \quad
 \vol(F^\times \backslash F^\times h_l \fo_K^\times) = 1.
\end{equation}
Hence we have
\begin{align*}
 \cQ(\varphi_0,w_0) 
 & = \sum_{i=0}^\infty p^{-2i} \beta^{2i}
 + \sum_{k=1}^\infty \sum_{i=0}^\infty p^{-2i-k/2} A_k \beta^{2i+k} \\
 & = \frac{1}{1 - \beta^2 p^{-2}} 
 \left( 1 + \sum_{k=1}^\infty p^{-k/2} A_k \beta^k \right).
\end{align*}
Finally, we have
\begin{align*}
 1 + \sum_{k=1}^\infty p^{-k/2} A_k \beta^k 
 & = 1 + \frac{1 - \alpha^{-2} p^{-1}}{1-\alpha^{-2}} \frac{\alpha \beta p^{-1/2}}{1 - \alpha \beta p^{-1/2}}
 + \frac{1 - \alpha^2 p^{-1}}{1-\alpha^2} \frac{\alpha^{-1} \beta p^{-1/2}}{1 - \alpha^{-1} \beta p^{-1/2}} \\
 & = \frac{1-\beta^2 p^{-2}}{(1 - \alpha \beta p^{-1/2})(1 - \alpha^{-1} \beta p^{-1/2})},
\end{align*}
so that 
\[
 \cQ(\varphi_0,w_0) = \frac{1}{(1 - \alpha \beta p^{-1/2})(1 - \alpha^{-1} \beta p^{-1/2})}.
\]
This completes the proof.
\end{proof}

\begin{lem}
\label{l:unique-hom-p-depletion-inert}
We have
\[
 \cQ(\varphi^\flat,\tilde{w}_0)
 = \xi(p)^{-2} p^{-1} \cdot \frac{(1 - \xi(p)^2 p^{-1}) (1 - \xi(p) \chi(1,p) p^{-1/2})}{(1+p^{-1})(1 - \xi(p)^{-1} \chi(p,1) p^{-1/2})}.
\]
\end{lem}

\begin{proof}
We use the notation of the proof of Lemma \ref{l:unique-hom-unram-inert}.
By a direct computation, we have
\[
 \bT_0 g_{i,k} \bK = \bigsqcup_{\bmm \in \mathfrak{M}_k} \bT_0 g_{i,k,\bmm} \bK_0,
\]
where
\[
 \mathfrak{M}_k = 
 \begin{cases}
  \{ (1,0) \} & \text{if $k=0$;} \\
  \{ (1, pm) \mid 0 \le m \le p-1 \} \cup \{ (0,1) \} & \text{if $k \ge 1$}
 \end{cases}
\]
and $g_{i,k,\bmm} = [g_k k_\bmm, p^{-i}]$.
This allows us to write
\[
 \cQ(\varphi^\flat, \tilde{w}_0) = \sum_{k=0}^\infty \sum_{\bmm \in \mathfrak{M}_k} \cQ_{k,\bmm}
\]
with 
\[
 \cQ_{k,\bmm} = \sum_{i,l \in \Z} \hat{\Omega}_\psi(g_{i,k,\bmm},h_l) \hat{\varphi}^\flat(x_0,1) \tilde{w}_0(g_{i,k,\bmm}) \chi(h_l) \vol(\bT_0 \backslash \bT_0 g_{i,k,\bmm} \bK_0) \vol(F^\times \backslash F^\times h_l \fo_K^\times).
\]
By Lemma \ref{l:weil-pft2} and \eqref{eq:v0-hl-inert}, we have
\[
 \cQ_{k,\bmm} = \sum_{i,l \in \Z} p^{-l} \hat{\varphi}^\flat(x_0 g_{i,k,\bmm}, p^{-2i-k+l}) \tilde{w}_0(g_{i,k,\bmm}) \beta^l \vol(\bT_0 \backslash \bT_0 g_{i,k,\bmm} \bK_0).
\]
By Lemma \ref{l:varphi-hat-flat}, we have $\hat{\varphi}^\flat(x_0 g_{i,k,\bmm}, p^{-2i-k+l}) = 0$ unless $l=2i+k$, so that
\[
 \cQ_{k,\bmm} = \sum_{i \in \Z} p^{-2i-k} \hat{\varphi}^\flat(x_0 g_{i,k,\bmm}, 1) \tilde{w}_0([g_k k_\bmm, 1]) \beta^{2i+k} \vol(\bT_0 \backslash \bT_0 g_{i,k,\bmm} \bK_0).
\]
Moreover, we have
\[
 \hat{\varphi}^\flat(x_0 g_{i,k,\bmm}, 1) =
 \begin{cases}
  1-p^{-1} & \text{if $i \in \mathfrak{I}_{k,\bmm}$;} \\
  -p^{-1} & \text{if $i \in \mathfrak{I}'_{k,\bmm}$;} \\
  0 & \text{otherwise,}
 \end{cases}
\]
where $\mathfrak{I}_{k,\bmm}, \mathfrak{I}'_{k,\bmm}$ are the subsets of $\Z$ given as follows.
\begin{itemize}
\item
If $k=0$ and $\bmm=(1,0)$, then we have
\[
 \hat{\varphi}^\flat(x_0 g_{i,k,\bmm}, 1) = \hat{\varphi}^\flat(p^i, p^i \bi, 1), 
\]
so that 
\begin{align*}
 \mathfrak{I}_{k,\bmm} & = \{ 0 \}, \\
 \mathfrak{I}'_{k,\bmm} & = \varnothing.
\end{align*}
\item
If $k \ge 1$ and $\bmm=(1,0)$, then we have
\[
 \hat{\varphi}^\flat(x_0 g_{i,k,\bmm}, 1) = \hat{\varphi}^\flat(p^{i+k}, p^i \bi, 1), 
\]
so that
\begin{align*}
 \mathfrak{I}_{k,\bmm} & = \varnothing, \\
 \mathfrak{I}'_{k,\bmm} & =
 \begin{cases}
  \{ -1 \} & \text{if $k=1$;} \\
  \varnothing & \text{if $k \ge 2$.}
 \end{cases}
\end{align*}
\item 
If $k \ge 1$ and $\bmm=(1,pm)$ with $1 \le m \le p-1$, then we have
\[
 \hat{\varphi}^\flat(x_0 g_{i,k,\bmm}, 1) = \hat{\varphi}^\flat(p^{i+k} + p^{i+1} m \bi, p^i \bi, 1), 
\]
so that
\begin{align*}
 \mathfrak{I}_{k,\bmm} & = \varnothing, \\
 \mathfrak{I}'_{k,\bmm} & = \varnothing.
\end{align*}
\item 
If $k \ge 1$ and $\bmm=(0,1)$, then we have
\[
 \hat{\varphi}^\flat(x_0 g_{i,k,\bmm}, 1) = \hat{\varphi}^\flat(p^i \bi, -p^{i+k},  1),
\]
so that
\begin{align*}
 \mathfrak{I}_{k,\bmm} & = \{ 0 \}, \\
 \mathfrak{I}'_{k,\bmm} & = \varnothing.
\end{align*}
\end{itemize}
In particular, we have
\[
 \cQ_{k,\bmm} = 0
\]
unless
\begin{itemize}
\item $k=0$ and $\bmm = (1,0)$; or
\item $k=1$ and $\bmm = (1,0)$; or
\item $k \ge 1$ and $\bmm = (0,1)$.
\end{itemize}
As in the proof of Lemma \ref{l:unique-hom-unram-inert}, $\tilde{w}_0([g_k k_\bmm, 1])$ is given as follows.
\begin{itemize}
\item 
If $k=0$ and $\bmm=(1,0)$, then we have
\begin{align*}
 g_k k_\bmm t_p^{-2} & = \mat{p^{-2}}{0}{0}{1}, \\
 g_k k_\bmm t_p^{-1} & = \mat{p^{-1}}{0}{0}{1},
\end{align*}
so that 
\[
 \tilde{w}_0([g_k k_\bmm, 1]) = p^{-1} (1+p^{-1})^{-1} (A_2 - \alpha A_1).
\]
\item
If $k=1$ and $\bmm=(1,0)$, then we have
\begin{align*}
 g_k k_\bmm t_p^{-2} & = \mat{p^{-1}}{0}{0}{1}, \\
 g_k k_\bmm t_p^{-1} & = \mat{1}{0}{0}{1},
\end{align*}
so that 
\[
 \tilde{w}_0([g_k k_\bmm, 1]) = p^{-1/2} (1+p^{-1})^{-1} (A_1 - \alpha A_0).
\]
\item
If $k \ge 1$ and $\bmm=(0,1)$, then we have
\begin{align*}
 g_k k_\bmm t_p^{-2} & = \mat{p^k}{0}{0}{p^{-2}} \mat{0}{-1}{1}{0}, \\
 g_k k_\bmm t_p^{-1} & = \mat{p^k}{0}{0}{p^{-1}} \mat{0}{-1}{1}{0},
\end{align*}
so that
\[
 \tilde{w}_0([g_k k_\bmm, 1]) = p^{-k/2-1} (1+p^{-1})^{-1} (A_{k+2} - \alpha A_{k+1}).
\]
\end{itemize}
Moreover, it is easy to see that
\begin{align*}
 \vol(\bT_0 \backslash \bT_0 g_{i,k,\bmm} \bK_0) 
 & = [\bK:\bK_0]^{-1} \cdot \vol(\bT_0 \cap g_{i,k,\bmm} \bK_0 g_{i,k,\bmm}^{-1})^{-1} \\
 & =
 \begin{cases}
  1 & \text{if $k=0$ and $\bmm=(1,0)$;} \\ 
  p^{k-2} & \text{if $k \ge 1$ and $\bmm=(1,pm)$;} \\ 
  p^k & \text{if $k \ge 1$ and $\bmm=(0,1)$.}
 \end{cases}
\end{align*}
Hence, noting that
\[
 A_k - \alpha A_{k-1} = \alpha^{-k} (1 - \alpha^2 p^{-1})
\]
for $k \ge 1$, we can compute $\cQ_{k,\bmm}$ as follows.
\begin{itemize}
\item 
If $k=0$ and $\bmm=(1,0)$, then we have
\[
 \cQ_{k,\bmm} = p^{-1} (1-p^{-1}) (1+p^{-1})^{-1} (1 - \alpha^2 p^{-1}) \alpha^{-2}.
\]
\item 
If $k=1$ and $\bmm=(1,0)$, then we have
\[
 \cQ_{k,\bmm} = - p^{-3/2} (1+p^{-1})^{-1} (1 - \alpha^2 p^{-1}) \alpha^{-1} \beta^{-1}.
\]
\item
If $k \ge 1$ and $\bmm=(0,1)$, then we have
\[
 \cQ_{k,\bmm} = p^{-k/2-1} (1-p^{-1}) (1+p^{-1})^{-1} (1 - \alpha^2 p^{-1}) \alpha^{-k-2} \beta^k.
\]
\end{itemize}
Thus we obtain
\begin{align*}
 \cQ(\varphi^\flat, \tilde{w}_0)
 & = \cQ_{0,(1,0)} + \cQ_{1,(1,0)} + \sum_{k=1}^\infty \cQ_{k,(0,1)} \\
 & = \frac{1 - \alpha^2 p^{-1}}{1+p^{-1}}
 \left( (1-p^{-1}) p^{-1} \alpha^{-2} - p^{-3/2} \alpha^{-1} \beta^{-1} + \frac{(1-p^{-1}) p^{-3/2} \alpha^{-3} \beta}{1 - \alpha^{-1} \beta p^{-1/2}} \right) \\
 & = \frac{\alpha^{-2} p^{-1} (1 - \alpha^2 p^{-1}) (1 - \alpha \beta^{-1} p^{-1/2})}{(1+p^{-1}) (1 - \alpha^{-1} \beta p^{-1/2})}.
\end{align*}
This completes the proof.
\end{proof}

Now Lemma \ref{l:unique-hom-p-depletion}\eqref{unique-hom-p-depletion-inert} follows from Lemmas \ref{l:unique-hom-unram-inert} and \ref{l:unique-hom-p-depletion-inert}.

\subsection{The split case}

Suppose that $E$ is split (so that $u \in (\fo_F^\times)^2$).
Let $\bT$ be a maximal torus of $\bG$ given by
\[
 \bT = \left\{ \left[ \mat{a_1}{}{}{a_2}, z \right] \; \middle| \; a_1, a_2 \in F^\times, \; z \in E^\times \right\}.
\]
Put $x_0 = ((1,0), (0,1)) \in E^2$ and let $\bT_0$ be the stabilizer of $x_0$ in $\bG$, so that 
\[
 \bT_0 = \left\{ \left[ \mat{a_1}{}{}{a_2}, (a_1, a_2) \right] \; \middle| \; a_1, a_2 \in F^\times \right\}.
\]
Following \cite[Lemme 8]{wald85}, we take the model $\mathfrak{W}(\pi \boxtimes \mu)$ defined as the image of a nonzero element in the $1$-dimensional space $\Hom_{\bG}(\pi \boxtimes \mu, \Ind^{\bG}_{\bT}(\mu_{\bT}))$, where $\mu_\bT$ is a character of $\bT$ given by
\[
 \mu_\bT\left( \left[ \mat{a_1}{}{}{a_2}, z \right] \right) = \mu((a_1, a_2)^{-1} z).
\]
We normalize the $\bK$-fixed vector $w_0 \in \mathfrak{W}(\pi \boxtimes \mu)$ so that $w_0(1) = 1$ (see also the proof of Lemma \ref{l:unique-hom-unram-split} below).
Put
\[
 \cQ(\varphi,w)
 = \int_{F^\times \backslash K^\times} \int_{\bT_0 \backslash \bG} 
 \hat{\Omega}_\psi(g,h) \cF_\psi(\varphi)(x_0,1) w(g) \chi(h) \, dg \, dh
\]
for $\varphi \in \cS(K^2 \times F^\times)$ and $w \in \mathfrak{W}(\pi \boxtimes \mu)$, where we normalize the measures so that $\vol(\bT_0 \backslash \bT_0 \bK) = \vol(F^\times \backslash F^\times \fo_K^\times) = 1$.

\begin{lem}
The integral $\cQ(\varphi,w)$ is absolutely convergent.
\end{lem} 

\begin{proof}
The lemma follows from the argument in the proof of Lemma \ref{l:abs-conv-inert} and the computation in the proof of Lemma \ref{l:unique-hom-unram-split} below.
\end{proof}

By this lemma, the above integral defines an element $\cQ$ in the Hom space \eqref{eq:hom-space-p-depletion}.

\begin{lem}
\label{l:unique-hom-unram-split}
We have 
\[
 \cQ(\varphi_0,w_0) = \frac{1}{(1 - \xi(p) \chi(p,1) p^{-1/2}) (1 - \xi(p)^{-1} \chi(p,1) p^{-1/2})}.
\]
In particular, $\cQ$ is nonzero.
\end{lem} 

\begin{proof}
Put $\alpha = \xi(p)$, $\beta = \chi(p,1)$, and $\gamma = \mu(p,1)$.
Put $g_{i,j,k} = [g_k, (p^{-i}, p^{-j})] \in \bG$ with
\[
 g_k =
 \begin{cases}
  1 & \text{if $k=0$;} \\
  {\renewcommand{\arraystretch}{1} \mat{p^k}{1}{0}{1}} & \text{if $k \ge 1$}
 \end{cases}
\]
and $h_l = (p^l, 1) \in K^\times$.
Then we can take $\{ g_{i,j,k} \mid i,j,k \in \Z, \; k \ge 0 \}$ and $\{ h_l \mid l \in \Z \}$ as sets of representatives for $\bT_0 \backslash \bG / \bK$ and $F^\times \backslash K^\times / \fo_K^\times$, respectively, so that
\[
 \cQ(\varphi_0,w_0) = \sum_{k=0}^\infty \sum_{i,j,l \in \Z} \hat{\Omega}_\psi(g_{i,j,k},h_l) \hat{\varphi}_0(x_0,1) w_0(g_{i,j,k}) \chi(h_l) \vol(\bT_0 \backslash \bT_0 g_{i,j,k} \bK) \vol(F^\times \backslash F^\times h_l \fo_K^\times).
\]
By Lemma \ref{l:weil-pft2}, we have
\begin{align*}
 \hat{\Omega}_\psi(g_{i,j,k},h_l) \hat{\varphi}_0(x_0,1)
 & =
 \begin{cases}
  p^{-l} \hat{\varphi}_0((p^i, 0), (0, p^j), p^{-i-j+l}) & \text{if $k=0$;} \\
  p^{-l} \hat{\varphi}_0((p^{i+k}, 0), (p^i, p^j), p^{-i-j-k+l}) & \text{if $k \ge 1$}
 \end{cases} \\
 & = 
 \begin{cases}
  p^{-i-j-k} & \text{if $i,j \ge 0$ and $l=i+j+k$;} \\
  0 & \text{otherwise.} 
 \end{cases}
\end{align*}
Also, by \cite[Theorem 1.1]{bff}, we have
\[
 w_0(g_{i,j,k}) = \frac{p^{-k/2}}{1-p^{-1}} A_k \gamma^{-i+j}
\]
for $k \ge 0$, where
\[
 A_k = \alpha^k \frac{(1 - \alpha^{-1} \gamma p^{-1/2})(1 - \alpha^{-1} \gamma^{-1} p^{-1/2})}{1-\alpha^{-2}} + \alpha^{-k} \frac{(1 - \alpha \gamma p^{-1/2})(1 - \alpha \gamma^{-1} p^{-1/2})}{1-\alpha^2}.
\]
(This also implies that $w(1) \ne 0$ for any nonzero $\bK$-fixed vector $w \in \mathfrak{W}(\pi \boxtimes \mu)$.)
Moreover, it is easy to see that
\begin{align*}
 \vol(\bT_0 \backslash \bT_0 g_{i,j,k} \bK)
 & = \vol(\bT_0 \cap g_{i,j,k} \bK g_{i,j,k}^{-1})^{-1} \\
 & = 
 \begin{cases}
  1 & \text{if $k=0$;} \\
  p^k (1-p^{-1}) & \text{if $k \ge 1$} 
 \end{cases}
\end{align*}
and 
\begin{equation}
\label{eq:v0-hl-split}
 \chi(h_l) = \beta^l, \quad
 \vol(F^\times \backslash F^\times h_l \fo_K^\times) = 1.
\end{equation}
Hence we have
\begin{align*}
 \cQ(\varphi_0,w_0)
 & = \sum_{i=0}^\infty \sum_{j=0}^\infty p^{-i-j} \gamma^{-i+j} \beta^{i+j}
 + \sum_{k=1}^\infty \sum_{i=0}^\infty \sum_{j=0}^\infty p^{-i-j-k/2} A_k \gamma^{-i+j} \beta^{i+j+k} \\
 & = \frac{1}{(1 - \beta \gamma p^{-1}) (1 - \beta \gamma^{-1} p^{-1})}
 \left( 1 + \sum_{k=1}^\infty p^{-k/2} A_k \beta^k \right).
\end{align*}
Finally, we have
\begin{align*}
 1 + \sum_{k=1}^\infty p^{-k/2} A_k \beta^k
 & = 1 + \frac{(1 - \alpha^{-1} \gamma p^{-1/2})(1 - \alpha^{-1} \gamma^{-1} p^{-1/2}) }{1-\alpha^{-2}} \frac{\alpha \beta p^{-1/2}}{1 - \alpha \beta p^{-1/2}} \\
 & \phantom{{}=1} + \frac{(1 - \alpha \gamma p^{-1/2})(1 - \alpha \gamma^{-1} p^{-1/2}) }{1-\alpha^2} \frac{\alpha^{-1} \beta p^{-1/2}}{1 - \alpha^{-1} \beta p^{-1/2}} \\
 & = \frac{(1 - \beta \gamma p^{-1}) (1 - \beta \gamma^{-1} p^{-1})}{(1 - \alpha \beta p^{-1/2}) (1 - \alpha^{-1} \beta p^{-1/2})},
\end{align*}
so that 
\[
 \cQ(\varphi_0,w_0) = \frac{1}{(1 - \alpha \beta p^{-1/2}) (1 - \alpha^{-1} \beta p^{-1/2})}.
\]
This completes the proof.
\end{proof}

\begin{lem}
\label{l:unique-hom-p-depletion-split} 
We have
\[
 \cQ(\varphi^\flat,\tilde{w}_0) = \xi(p)^{-2} p^{-1} \cdot
 \frac{(1 - \xi(p) \chi(1,p) p^{-1/2})(1 - \xi(p) \mu(p,1) p^{-1/2})(1 - \xi(p) \mu(1,p) p^{-1/2})}{(1+p^{-1})(1 - \xi(p)^{-1} \chi(p,1) p^{-1/2})}.
\]
\end{lem}

\begin{proof}
We use the notation of the proof of Lemma \ref{l:unique-hom-unram-split}.
By a direct computation, we have
\[
 \bT_0 g_{i,j,k} \bK = \bigsqcup_{\bmm \in \mathfrak{M}_k} \bT_0 g_{i,j,k,\bmm} \bK_0,
\]
where 
\[
 \mathfrak{M}_k = 
 \begin{cases}
  \{ (1,0), (1,1), (1,p), (0,1), (p,1) \} & \text{if $k=0$;} \\
  \{ (1, pm) \mid 0 \le m \le p-1 \} \cup \{ (0,1) \} & \text{if $k \ge 1$}
 \end{cases}
\]
and $g_{i,j,k,\bmm} = [g_k k_\bmm, (p^{-i}, p^{-j})]$.
This allows us to write
\[
 \cQ(\varphi^\flat, \tilde{w}_0) = \sum_{k=0}^\infty \sum_{\bmm \in \mathfrak{M}_k} \cQ_{k,\bmm}
\]
with 
\[
 \cQ_{k,\bmm} = \sum_{i,j,l \in \Z} \hat{\Omega}_\psi(g_{i,j,k,\bmm},h_l) \hat{\varphi}^\flat(x_0,1) \tilde{w}_0(g_{i,j,k,\bmm}) \chi(h_l) \vol(\bT_0 \backslash \bT_0 g_{i,j,k,\bmm} \bK_0) \vol(F^\times \backslash F^\times h_l \fo_K^\times).
\]
By Lemma \ref{l:weil-pft2} and \eqref{eq:v0-hl-split}, we have
\[
 \cQ_{k,\bmm} = \sum_{i,j,l \in \Z} p^{-l} \hat{\varphi}^\flat(x_0 g_{i,j,k,\bmm}, p^{-i-j-k+l}) \tilde{w}_0(g_{i,j,k,\bmm}) \beta^l \vol(\bT_0 \backslash \bT_0 g_{i,j,k,\bmm} \bK_0).
\]
By Lemma \ref{l:varphi-hat-flat}, we have $\hat{\varphi}^\flat(x_0 g_{i,j,k,\bmm}, p^{-i-j-k+l}) = 0$ unless $l=i+j+k$, so that
\[
 \cQ_{k,\bmm} = \sum_{i,j \in \Z} p^{-i-j-k} \hat{\varphi}^\flat(x_0 g_{i,j,k,\bmm}, 1) \tilde{w}_0([g_k k_\bmm, 1]) \gamma^{-i+j} \beta^{i+j+k} \vol(\bT_0 \backslash \bT_0 g_{i,j,k,\bmm} \bK_0).
\]
Moreover, we have
\[
 \hat{\varphi}^\flat(x_0 g_{i,j,k,\bmm}, 1) =
 \begin{cases}
  1-p^{-1} & \text{if $(i,j) \in \mathfrak{I}_{k,\bmm}$;} \\
  -p^{-1} & \text{if $(i,j) \in \mathfrak{I}'_{k,\bmm}$;} \\
  0 & \text{otherwise,}
 \end{cases}
\]
where $\mathfrak{I}_{k,\bmm}, \mathfrak{I}'_{k,\bmm}$ are the subsets of $\Z^2$ given as follows.
\begin{itemize}
\item
If $k=0$ and $\bmm=(1,0)$, then we have
\[
 \hat{\varphi}^\flat(x_0 g_{i,j,k,\bmm}, 1) = \hat{\varphi}^\flat((p^i,0), (0,p^j), 1), 
\]
so that 
\begin{align*}
 \mathfrak{I}_{k,\bmm} & = \{ (0,j) \mid j \ge 0 \}, \\
 \mathfrak{I}'_{k,\bmm} & = \varnothing.
\end{align*}
\item 
If $k=0$ and $\bmm=(1,1)$, then we have
\[
 \hat{\varphi}^\flat(x_0 g_{i,j,k,\bmm}, 1) = \hat{\varphi}^\flat((p^i,p^j), (0,p^j), 1), 
\]
so that 
\begin{align*}
 \mathfrak{I}_{k,\bmm} & = \{ (0,j) \mid j \ge 0 \} \cup \{ (i,0) \mid i \ge 1 \}, \\
 \mathfrak{I}'_{k,\bmm} & = \varnothing.
\end{align*}
\item
If $k=0$ and $\bmm=(1,p)$, then we have
\[
 \hat{\varphi}^\flat(x_0 g_{i,j,k,\bmm}, 1) = \hat{\varphi}^\flat((p^i,p^{j+1}), (0,p^j), 1),
\]
so that
\begin{align*}
 \mathfrak{I}_{k,\bmm} & = \{ (0,j) \mid j \ge 0 \}, \\
 \mathfrak{I}'_{k,\bmm} & = \{ (i,-1) \mid i \ge 1 \}.
\end{align*}
\item 
If $k=0$ and $\bmm=(0,1)$, then we have
\[
 \hat{\varphi}^\flat(x_0 g_{i,j,k,\bmm}, 1) = \hat{\varphi}^\flat((0,p^j), (-p^i,0), 1),
\]
so that
\begin{align*}
 \mathfrak{I}_{k,\bmm} & = \{ (i,0) \mid i \ge 0 \}, \\
 \mathfrak{I}'_{k,\bmm} & = \varnothing.
\end{align*}
\item 
If $k=0$ and $\bmm=(p,1)$, then we have
\[
 \hat{\varphi}^\flat(x_0 g_{i,j,k,\bmm}, 1) = \hat{\varphi}^\flat((p^{i+1},p^j), (-p^i,0), 1),
\]
so that
\begin{align*}
 \mathfrak{I}_{k,\bmm} & = \{ (i,0) \mid i \ge 0 \}, \\
 \mathfrak{I}'_{k,\bmm} & = \{ (-1,j) \mid j \ge 1 \}.
\end{align*}
\item 
If $k \ge 1$ and $\bmm=(1,0)$, then we have
\[
 \hat{\varphi}^\flat(x_0 g_{i,j,k,\bmm}, 1) = \hat{\varphi}^\flat((p^{i+k}, 0), (p^i,p^j), 1), 
\]
so that
\begin{align*}
 \mathfrak{I}_{k,\bmm} & = \varnothing, \\
 \mathfrak{I}'_{k,\bmm} & =
 \begin{cases}
  \{ (-1,j) \mid j \ge 0 \} & \text{if $k=1$;} \\
  \varnothing & \text{if $k \ge 2$.}
 \end{cases}
\end{align*}
\item 
If $k \ge 1$ and $\bmm=(1,pm)$ with $1 \le m \le p-1$, then we have
\[
 \hat{\varphi}^\flat(x_0 g_{i,j,k,\bmm}, 1) = \hat{\varphi}^\flat((p^{i+k}+p^{i+1}m, p^{j+1}m), (p^i,p^j), 1),
\]
so that
\begin{align*}
 \mathfrak{I}_{k,\bmm} & = \varnothing, \\
 \mathfrak{I}'_{k,\bmm} & =
 \begin{cases}
  \{ (-1,j) \mid j \ge -1 \} \cup \{ (i,-1) \mid i \ge 0 \} & \text{if $k=1$ and $m=\frac{p-1}{2}$;} \\
  \{ (i,-1) \mid i \ge 0 \} & \text{if $k=1$ and $m=p-1$;} \\
  \{ (-1,j) \mid j \ge 0 \} \cup \{ (i,-1) \mid i \ge 0 \} & \text{if $k=1$ and $m \ne \frac{p-1}{2},p-1$;} \\
  \{ (-1,j) \mid j \ge 0 \} \cup \{ (i,-1) \mid i \ge 0 \} & \text{if $k \ge 2$.}
 \end{cases}
\end{align*}
\item 
If $k \ge 1$ and $\bmm=(0,1)$, then we have
\[
 \hat{\varphi}^\flat(x_0 g_{i,j,k,\bmm}, 1) = \hat{\varphi}^\flat((p^i,p^j), (-p^{i+k},0), 1),
\]
so that
\begin{align*}
 \mathfrak{I}_{k,\bmm} & = \{ (0,j) \mid j \ge 0 \} \cup \{ (i,0) \mid i \ge 1 \}, \\
 \mathfrak{I}'_{k,\bmm} & = \varnothing.
\end{align*}
\end{itemize}
As in the proof of Lemma \ref{l:unique-hom-unram-split}, $\tilde{w}_0([g_k k_\bmm, 1])$ is given as follows.
\begin{itemize}
\item
If $k=0$ and $\bmm=(1,0)$, then we have
\begin{align*}
 g_k k_\bmm t_p^{-2} & = \mat{p^{-2}}{0}{0}{1}, \\
 g_k k_\bmm t_p^{-1} & = \mat{p^{-1}}{0}{0}{1},
\end{align*}
so that 
\[
 \tilde{w}_0([g_k k_\bmm, 1]) = \gamma^2 - \alpha \gamma p^{-1/2}.
\]
\item
If $k=0$ and $\bmm=(1,1)$, then we have
\begin{align*}
 g_k k_\bmm t_p^{-2} & = \mat{1}{p^{-2}}{0}{p^{-2}} \mat{0}{-1}{1}{p^2}, \\
 g_k k_\bmm t_p^{-1} & = \mat{1}{p^{-1}}{0}{p^{-1}} \mat{0}{-1}{1}{p},
\end{align*}
so that
\[
 \tilde{w}_0([g_k k_\bmm, 1]) = p^{-1} (1-p^{-1})^{-1} A_2
 - \alpha p^{-1} (1-p^{-1})^{-1} A_1.
\]
\item
If $k=0$ and $\bmm=(1,p)$, then we have
\begin{align*}
 g_k k_\bmm t_p^{-2} & = \mat{p^{-1}}{p^{-2}}{0}{p^{-1}} \mat{0}{-1}{1}{p}, \\
 g_k k_\bmm t_p^{-1} & = \mat{p^{-1}}{0}{0}{1} \mat{1}{0}{1}{1},
\end{align*}
so that
\[
 \tilde{w}_0([g_k k_\bmm, 1]) = \gamma p^{-1/2} (1-p^{-1})^{-1} A_1 - \alpha \gamma p^{-1/2}.
\]
\item
If $k=0$ and $\bmm=(0,1)$, then we have
\begin{align*}
 g_k k_\bmm t_p^{-2} & = \mat{1}{0}{0}{p^{-2}} \mat{0}{-1}{1}{0}, \\
 g_k k_\bmm t_p^{-1} & = \mat{1}{0}{0}{p^{-1}} \mat{0}{-1}{1}{0},
\end{align*}
so that
\[
 \tilde{w}_0([g_k k_\bmm, 1]) = \gamma^{-2} - \alpha \gamma^{-1} p^{-1/2}.
\]
\item
If $k=0$ and $\bmm=(p,1)$, then we have
\begin{align*}
 g_k k_\bmm t_p^{-2} & = \mat{1}{p^{-1}}{0}{p^{-2}} \mat{0}{-1}{1}{0}, \\
 g_k k_\bmm t_p^{-1} & = \mat{1}{0}{0}{p^{-1}} \mat{1}{-1}{1}{0},
\end{align*}
so that
\[
 \tilde{w}_0([g_k k_\bmm, 1]) = \gamma^{-1} p^{-1/2} (1-p^{-1})^{-1} A_1 - \alpha \gamma^{-1} p^{-1/2}.
\]
\item
If $k \ge 1$ and $\bmm=(1,0)$, then we have
\begin{align*}
 g_k k_\bmm t_p^{-2} & = \mat{p^{k-2}}{1}{0}{1}, \\
 g_k k_\bmm t_p^{-1} & = \mat{p^{k-1}}{1}{0}{1},
\end{align*}
so that
\[
 \tilde{w}_0([g_k k_\bmm, 1]) = 
 \begin{cases}
  \gamma - \alpha p^{-1/2} & \text{if $k=1$;} \\
  p^{-k/2+1} (1-p^{-1})^{-1} A_{k-2} - \alpha p^{-k/2} (1-p^{-1})^{-1} A_{k-1} & \text{if $k \ge 2$.}
 \end{cases}
\]
\item
If $k \ge 1$ and $\bmm=(1,pm)$ with $1 \le m \le p-1$, then we have
\begin{align*}
 g_k k_\bmm t_p^{-2} & = \mat{p^{k-1}}{p^{k-2}m^{-1} + p^{-1}}{0}{p^{-1}} \mat{0}{-m^{-1}}{m}{p}, \\
 g_k k_\bmm t_p^{-1} & = \mat{p^{k-1}}{1}{0}{1} \mat{1}{0}{m}{1},
\end{align*}
so that
\[
 \tilde{w}_0([g_k k_\bmm, 1]) = 
 \begin{cases}
  \gamma^{-1} - \alpha p^{-1/2} & \text{if $k = 1$ and $m = p-1$;} \\
  p^{-1/2} (1-p^{-1})^{-1} A_1 - \alpha p^{-1/2} & \text{if $k = 1$ and $m \ne p-1$;} \\
  p^{-k/2} (1-p^{-1})^{-1} A_k - \alpha p^{-k/2} (1-p^{-1})^{-1} A_{k-1} & \text{if $k \ge 2$.}
 \end{cases}
\]
\item
If $k \ge 1$ and $\bmm=(0,1)$, then we have
\begin{align*}
 g_k k_\bmm t_p^{-2} & = \mat{p^k}{p^{-2}}{0}{p^{-2}} \mat{0}{-1}{1}{0}, \\
 g_k k_\bmm t_p^{-1} & = \mat{p^k}{p^{-1}}{0}{p^{-1}} \mat{0}{-1}{1}{0},
\end{align*}
so that
\[
 \tilde{w}_0([g_k k_\bmm, 1]) = p^{-k/2-1} (1-p^{-1})^{-1} A_{k+2} 
 - \alpha p^{-k/2-1} (1-p^{-1})^{-1} A_{k+1}.
\]
\end{itemize}
Moreover, it is easy to see that
\begin{align*}
 \vol(\bT_0 \backslash \bT_0 g_{i,j,k,\bmm} \bK_0)
 & = [\bK:\bK_0]^{-1} \cdot \vol(\bT_0 \cap g_{i,j,k,\bmm} \bK_0 g_{i,j,k,\bmm}^{-1})^{-1} \\
 & = p^{-2} (1+p^{-1})^{-1} \times 
 \begin{cases}
  1 & \text{if $k=0$ and $\bmm=(1,0)$;} \\
  p^2 (1-p^{-1}) & \text{if $k=0$ and $\bmm=(1,1)$;} \\
  p (1-p^{-1}) & \text{if $k=0$ and $\bmm=(1,p)$;} \\
  1 & \text{if $k=0$ and $\bmm=(0,1)$;} \\
  p (1-p^{-1}) & \text{if $k=0$ and $\bmm=(p,1)$;} \\
  p^k (1-p^{-1}) & \text{if $k \ge 1$ and $\bmm=(1,pm)$;} \\
  p^{k+2} (1-p^{-1}) & \text{if $k \ge 1$ and $\bmm=(0,1)$.}
 \end{cases}
\end{align*}
Hence, noting that
\[
 A_k - \alpha A_{k-1} = \alpha^{-k} (1 - \alpha \gamma p^{-1/2})(1 - \alpha \gamma^{-1} p^{-1/2})
\]
for $k \ge 1$, we can compute $\cQ_{k,\bmm}$ as follows.
\begin{itemize}
\item
If $k=0$ and $\bmm=(1,0)$, then we have
\begin{align*}
 \cQ_{k,\bmm} & = \sum_{j=0}^\infty p^{-j-2} (1-p^{-1}) (1+p^{-1})^{-1}
 (\gamma^2 - \alpha \gamma p^{-1/2}) \gamma^j \beta^j \\
 & = \frac{\gamma^2 p^{-2} (1-p^{-1}) (1 - \alpha \gamma^{-1} p^{-1/2})}{(1+p^{-1})(1 - \beta \gamma p^{-1})}.
\end{align*}
\item
If $k=0$ and $\bmm=(1,1)$, then we have
\begin{align*}
 \cQ_{k,\bmm} & = \sum_{j=0}^\infty p^{-j-1} (1-p^{-1}) (1+p^{-1})^{-1} (A_2 - \alpha A_1) \gamma^j \beta^j \\
 & + \sum_{i=1}^\infty p^{-i-1} (1-p^{-1}) (1+p^{-1})^{-1} (A_2 - \alpha A_1) \gamma^{-i} \beta^i \\
 & = \frac{\alpha^{-2} p^{-1} (1-p^{-1}) (1 - \alpha \gamma p^{-1/2})(1 - \alpha \gamma^{-1} p^{-1/2})}{1+p^{-1}}
 \left( \frac{1}{1 - \beta \gamma p^{-1}} + \frac{\beta \gamma^{-1} p^{-1}}{1 - \beta \gamma^{-1} p^{-1}} \right) \\
 & = \frac{\alpha^{-2} p^{-1} (1-p^{-1}) (1 - \alpha \gamma p^{-1/2})(1 - \alpha \gamma^{-1} p^{-1/2}) (1 - \beta^2 p^{-2})}{(1+p^{-1}) (1 - \beta \gamma p^{-1}) (1 - \beta \gamma^{-1} p^{-1})}.
\end{align*}
\item
If $k=0$ and $\bmm=(1,p)$, then we have
\begin{align*}
 \cQ_{k,\bmm} & = \sum_{j=0}^\infty p^{-j-3/2} (1-p^{-1}) (1+p^{-1})^{-1} (\gamma A_1 - \alpha \gamma (1-p^{-1})) \gamma^j \beta^j \\
 & - \sum_{i=1}^\infty p^{-i-3/2} (1+p^{-1})^{-1} (\gamma A_1 - \alpha \gamma (1-p^{-1})) \gamma^{-i-1}\beta^{i-1} \\
 & = \frac{\alpha^{-1} \gamma p^{-3/2} (1 - \alpha \gamma p^{-1/2})(1 - \alpha \gamma^{-1} p^{-1/2})}{1 + p^{-1}} \left( \frac{1-p^{-1}}{1-\beta \gamma p^{-1}} - \frac{\gamma^{-2} p^{-1}}{1 - \beta \gamma^{-1} p^{-1}} \right) \\
 & = \frac{\alpha^{-1} p^{-3/2} (1 - \alpha \gamma p^{-1/2})(1 - \alpha \gamma^{-1} p^{-1/2}) (\gamma - (\gamma + \gamma^{-1} + \beta) p^{-1} + 2 \beta p^{-2})}{(1 + p^{-1}) (1-\beta \gamma p^{-1}) (1 - \beta \gamma^{-1} p^{-1})}.
\end{align*}
\item
If $k=0$ and $\bmm=(0,1)$, then we have
\begin{align*}
 \cQ_{k,\bmm} & = \sum_{i=0}^\infty p^{-i-2} (1-p^{-1}) (1+p^{-1})^{-1} (\gamma^{-2} - \alpha \gamma^{-1} p^{-1/2}) \gamma^{-i} \beta^i \\
 & = \frac{\gamma^{-2} p^{-2} (1-p^{-1}) (1 - \alpha \gamma p^{-1/2})}{(1+p^{-1}) (1 - \beta \gamma^{-1} p^{-1})}.
\end{align*}
\item
If $k=0$ and $\bmm=(p,1)$, then we have
\begin{align*}
 \cQ_{k,\bmm} & = \sum_{i=0}^\infty p^{-i-3/2} (1-p^{-1}) (1+p^{-1})^{-1} (\gamma^{-1} A_1 - \alpha \gamma^{-1} (1-p^{-1})) \gamma^{-i} \beta^i \\
 & - \sum_{j=1}^\infty p^{-j-3/2} (1+p^{-1})^{-1} (\gamma^{-1} A_1 - \alpha \gamma^{-1} (1-p^{-1})) \gamma^{j+1} \beta^{j-1} \\
 & = \frac{\alpha^{-1} \gamma^{-1} p^{-3/2} (1 - \alpha \gamma p^{-1/2})(1 - \alpha \gamma^{-1} p^{-1/2})}{1+p^{-1}} \left( \frac{1 - p^{-1}}{1 - \beta \gamma^{-1} p^{-1}} - \frac{\gamma^2 p^{-1}}{1 - \beta \gamma p^{-1}} \right) \\
 & = \frac{\alpha^{-1} p^{-3/2} (1 - \alpha \gamma p^{-1/2})(1 - \alpha \gamma^{-1} p^{-1/2}) (\gamma^{-1} - (\gamma + \gamma^{-1} + \beta)p^{-1} + 2 \beta p^{-2})}{(1+p^{-1}) (1 - \beta \gamma p^{-1}) (1 - \beta \gamma^{-1} p^{-1})}.
\end{align*}
\item
If $k=1$ and $\bmm=(1,0)$, then we have
\begin{align*}
 \cQ_{k,\bmm} & = - \sum_{j=0}^\infty p^{-j-2} (1-p^{-1}) (1+p^{-1})^{-1} (\gamma - \alpha p^{-1/2}) \gamma^{j+1} \beta^j \\
 & = - \frac{\gamma^2 p^{-2} (1-p^{-1}) (1 - \alpha \gamma^{-1} p^{-1/2})}{(1+p^{-1}) (1 - \beta \gamma p^{-1})}.
\end{align*}
\item
If $k \ge 2$ and $\bmm=(1,0)$, then we have
\[
 \cQ_{k,\bmm} = 0.
\]
\item
If $k=1$ and $\bmm=(1, \frac{p(p-1)}{2})$, then we have
\begin{align*}
 \cQ_{k,\bmm} & = - \sum_{j=-1}^\infty p^{-j-5/2} (1+p^{-1})^{-1} (A_1 - \alpha(1-p^{-1})) \gamma^{j+1} \beta^j \\
 & \phantom{={}} - \sum_{i=0}^\infty p^{-i-5/2} (1+p^{-1})^{-1} (A_1 - \alpha(1-p^{-1})) \gamma^{-i-1} \beta^i \\
 & = - \frac{\alpha^{-1} p^{-5/2} (1 - \alpha \gamma p^{-1/2})(1 - \alpha \gamma^{-1} p^{-1/2})}{1+p^{-1}} \left( \frac{\beta^{-1} p}{1 - \beta \gamma p^{-1}} + \frac{\gamma^{-1}}{1 - \beta \gamma^{-1} p^{-1}} \right) \\
 & = - \frac{\alpha^{-1} \beta^{-1} p^{-3/2} (1 - \alpha \gamma p^{-1/2})(1 - \alpha \gamma^{-1} p^{-1/2}) (1 - \beta^2 p^{-2})}{(1+p^{-1}) (1 - \beta \gamma p^{-1}) (1 - \beta \gamma^{-1} p^{-1})}.
\end{align*}
\item
If $k=1$ and $\bmm=(1,p(p-1))$, then we have
\begin{align*}
 \cQ_{k,\bmm} & = - \sum_{i=0}^\infty p^{-i-2} (1-p^{-1}) (1+p^{-1})^{-1} (\gamma^{-1} - \alpha p^{-1/2}) \gamma^{-i-1} \beta^i \\
 & = - \frac{\gamma^{-2} p^{-2} (1-p^{-1}) (1 - \alpha \gamma p^{-1/2})}{(1+p^{-1}) (1 - \beta \gamma^{-1} p^{-1})}.
\end{align*}
\item
If $k=1$ and $\bmm=(1,pm)$ with $1 \le m \le p-2$, $m \ne \frac{p-1}{2}$, then we have
\begin{align*}
 \cQ_{k,\bmm} & = - \sum_{j=0}^\infty p^{-j-5/2} (1+p^{-1})^{-1} (A_1 - \alpha(1-p^{-1})) \gamma^{j+1} \beta^j \\
 & \phantom{={}} - \sum_{i=0}^\infty p^{-i-5/2} (1+p^{-1})^{-1} (A_1 - \alpha(1-p^{-1})) \gamma^{-i-1} \beta^i \\
 & = - \frac{\alpha^{-1} p^{-5/2} (1 - \alpha \gamma p^{-1/2})(1 - \alpha \gamma^{-1} p^{-1/2})}{1+p^{-1}} \left( \frac{\gamma}{1 - \beta \gamma p^{-1}} + \frac{\gamma^{-1}}{1 - \beta \gamma^{-1} p^{-1}} \right) \\
 & = - \frac{\alpha^{-1} p^{-5/2} (1 - \alpha \gamma p^{-1/2})(1 - \alpha \gamma^{-1} p^{-1/2}) (\gamma + \gamma^{-1} - 2 \beta p^{-1})}{(1+p^{-1}) (1 - \beta \gamma p^{-1}) (1 - \beta \gamma^{-1} p^{-1})}.
\end{align*}
\item
If $k \ge 2$ and $\bmm=(1,pm)$ with $1 \le m \le p-1$, then we have
\begin{align*}
 \cQ_{k,\bmm} 
 & = - \sum_{j=0}^\infty p^{-j-k/2-2} (1+p^{-1})^{-1} (A_k - \alpha A_{k-1}) \gamma^{j+1} \beta^{j+k-1} \\
 & \phantom{={}} - \sum_{i=0}^\infty p^{-i-k/2-2} (1+p^{-1})^{-1} (A_k - \alpha A_{k-1}) \gamma^{-i-1} \beta^{i+k-1} \\
 & = - \frac{\alpha^{-k} p^{-k/2-2} (1 - \alpha \gamma p^{-1/2})(1 - \alpha \gamma^{-1} p^{-1/2})}{1+p^{-1}} \left( \frac{\beta^{k-1} \gamma}{1-\beta \gamma p^{-1}} + \frac{\beta^{k-1} \gamma^{-1}}{1 - \beta \gamma^{-1} p^{-1}}\right) \\
 & = - \frac{\alpha^{-k} \beta^{k-1} p^{-k/2-2} (1 - \alpha \gamma p^{-1/2})(1 - \alpha \gamma^{-1} p^{-1/2}) (\gamma + \gamma^{-1} - 2 \beta p^{-1})}{(1+p^{-1}) (1-\beta \gamma p^{-1}) (1 - \beta \gamma^{-1} p^{-1})}, 
\end{align*}
so that 
\[
 \sum_{k=2}^\infty \cQ_{k,\bmm} = 
 - \frac{\alpha^{-2} \beta p^{-3} (1 - \alpha \gamma p^{-1/2})(1 - \alpha \gamma^{-1} p^{-1/2}) (\gamma + \gamma^{-1} - 2 \beta p^{-1})}{(1+p^{-1}) (1 - \alpha^{-1} \beta p^{-1/2}) (1-\beta \gamma p^{-1}) (1 - \beta \gamma^{-1} p^{-1})}.
\]
\item
If $k \ge 1$ and $\bmm=(0,1)$, then we have
\begin{align*}
 \cQ_{k,\bmm} & = 
 \sum_{j=0}^\infty p^{-j-k/2-1} (1-p^{-1}) (1+p^{-1})^{-1} (A_{k+2} - \alpha A_{k+1}) \gamma^j \beta^{j+k} \\
 & + \sum_{i=1}^\infty p^{-i-k/2-1} (1-p^{-1}) (1+p^{-1})^{-1} (A_{k+2} - \alpha A_{k+1}) \gamma^{-i} \beta^{i+k} \\
 & = \frac{\alpha^{-k-2} p^{-k/2-1} (1-p^{-1}) (1 - \alpha \gamma p^{-1/2})(1 - \alpha \gamma^{-1} p^{-1/2})}{1+p^{-1}} \left( \frac{\beta^k}{1 - \beta \gamma p^{-1}} + \frac{\beta^{k+1} \gamma^{-1} p^{-1} }{1 - \beta \gamma^{-1} p^{-1}} \right) \\
 & = \frac{\alpha^{-k-2} \beta^k p^{-k/2-1} (1-p^{-1}) (1 - \alpha \gamma p^{-1/2})(1 - \alpha \gamma^{-1} p^{-1/2}) (1 - \beta^2 p^{-2})}{(1+p^{-1}) (1 - \beta \gamma p^{-1}) (1 - \beta \gamma^{-1} p^{-1})}, 
\end{align*}
so that 
\[
 \sum_{k=1}^\infty \cQ_{k,\bmm} = 
 \frac{\alpha^{-3} \beta p^{-3/2} (1-p^{-1}) (1 - \alpha \gamma p^{-1/2})(1 - \alpha \gamma^{-1} p^{-1/2}) (1 - \beta^2 p^{-2})}{(1+p^{-1}) (1 - \alpha^{-1} \beta p^{-1/2})(1 - \beta \gamma p^{-1}) (1 - \beta \gamma^{-1} p^{-1})}.
\]
\end{itemize}
Thus we obtain
\[
 \cQ_{0,(1,0)} + \cQ_{1,(1,0)} = \cQ_{0,(0,1)} + \cQ_{1,(1,p(p-1))} = 0
\]
and
\[
 \cQ_{0,(1,1)} + \cQ_{0,(1,p)} + \cQ_{0,(p,1)} + \sum_{m=1}^{p-2} \cQ_{1,(p,pm)} = 
 \frac{\alpha^{-2} p^{-1} (1 - \alpha \gamma p^{-1/2})(1 - \alpha \gamma^{-1} p^{-1/2})}{(1+p^{-1}) (1 - \beta \gamma p^{-1}) (1 - \beta \gamma^{-1} p^{-1})}
 \cdot \cQ'
\]
with
\begin{align*}
 \cQ' & = (1-p^{-1}) (1-\beta^2 p^{-2}) \\
 & + \alpha p^{-1/2} (\gamma - (\gamma + \gamma^{-1} + \beta) p^{-1} + 2 \beta p^{-2}) \\
 & + \alpha p^{-1/2} (\gamma^{-1} - (\gamma + \gamma^{-1} + \beta) p^{-1} + 2 \beta p^{-2}) \\
 & - \alpha \beta^{-1} p^{-1/2} (1-\beta^2 p^{-2}) \\
 & - (p-3) \cdot \alpha p^{-3/2} (\gamma + \gamma^{-1} - 2 \beta p^{-1}) \\
 & = 1 - \alpha \beta^{-1} p^{-1/2} - p^{-1} + (\alpha \gamma + \alpha \gamma^{-1}) p^{-3/2} - \beta^2 p^{-2} - \alpha \beta p^{-5/2} + \beta^2 p^{-3},
\end{align*}
so that
\[
 \cQ(\varphi^\flat, \tilde{w}_0) = 
 \frac{\alpha^{-2} p^{-1} (1 - \alpha \gamma p^{-1/2})(1 - \alpha \gamma^{-1} p^{-1/2})}{(1+p^{-1}) (1 - \alpha^{-1} \beta p^{-1/2})(1 - \beta \gamma p^{-1}) (1 - \beta \gamma^{-1} p^{-1})} \cdot \cQ''
\]
with 
\begin{align*}
 \cQ'' & = (1 - \alpha^{-1} \beta p^{-1/2}) \cdot \cQ' \\
 & - (p-1) \cdot \beta p^{-2} (\gamma + \gamma^{-1} - 2 \beta p^{-1}) \\
 & + \alpha^{-1} \beta p^{-1/2} (1-p^{-1}) (1-\beta^2 p^{-2}) \\
 & = (1 - \alpha \beta^{-1} p^{-1/2}) (1 - \beta \gamma p^{-1}) (1 - \beta \gamma^{-1} p^{-1}).
\end{align*}
This completes the proof.
\end{proof}

Now Lemma \ref{l:unique-hom-p-depletion}\eqref{unique-hom-p-depletion-split} follows from Lemmas \ref{l:unique-hom-unram-split} and \ref{l:unique-hom-p-depletion-split}.

\section{Complex conjugation}
\label{s:complex_conjugation}

In this section, we describe the behaviour of the toric period under the complex conjugation.
We write $F = \Q$ and denote by $c$ the nontrivial automorphism of $E$ over $F$.

\subsection{An explicit formula}

Let $\varphi \in S(B(\A) \times \A^\times) = S(\A_K^2 \times \A^\times)$ be the Schwartz function as in \S \ref{ss:adelic_period}.
Define a Schwartz function $\varphi' \in S(B(\A) \times \A^\times)$ by
\[
 \varphi' = \Omega_\psi(h_0) \bar{\varphi}^c,
\]
where $h_0 = [\bi,\bi] \in \bH(F)$ and $\bar{\varphi}^c$ is as in \S \ref{ss:weil-u}.
Note that $h_0^2 = 1$, $h_0^c = h_0$, and $\nu(h_0) = 1$.

\begin{lem}
\label{l:complex_conjugation}
We have
\[
 \overline{\cP^H(\theta_\varphi(f_0^\mu), \chi)} = \cP^H(\theta_{\varphi'}(f_0^\mu), \chi).
\]
\end{lem}

\begin{proof}
By the assumption on $f_0, \chi, \mu$, we have
\[
 \overline{f_0^\mu(g)} = f_0^\mu(\delta g^c \delta), \quad
 \overline{\chi(h)} = \chi(h_0 h h_0)
\]
for $g \in \bG(\A)$ and $h \in H(\A)$, where 
\[
 \delta = \left[ \mat{-1}{0}{0}{1}, 1 \right] \in \bG(F).
\]
Hence we have
\begin{align*}
 \overline{\cP^H(\theta_\varphi(f_0^\mu), \chi)}
 & = \int_{Z_{\tilde{H}}(\A) H(F) \backslash H(\A)} \int_{\bG(F) \backslash \bG(\A)} \overline{\Theta_\varphi(g,h) f_0^{\mu}(g) \chi(h)} \, dg \, dh \\
 & = \int_{Z_{\tilde{H}}(\A) H(F) \backslash H(\A)} \int_{\bG(F) \backslash \bG(\A)} \overline{\Theta_\varphi(g,h)} f_0^\mu(\delta g^c \delta) \chi(h_0 h h_0) \, dg \, dh \\
 & = \int_{Z_{\tilde{H}}(\A) H(F) \backslash H(\A)} \int_{\bG(F) \backslash \bG(\A)} \overline{\Theta_\varphi(\delta g^c \delta, h_0 h h_0)} f_0^\mu(g) \chi(h) \, dg \, dh.
\end{align*}
On the other hand, by Lemma \ref{l:weil-conjugation}, we have
\begin{align*}
 \overline{\Theta_\varphi(\delta g^c \delta,h_0 h h_0)}
 & = \sum_{x \in B} \sum_{t \in F^\times} \overline{\Omega_\psi(\delta g^c \delta, h_0 h h_0) \varphi(x^c,t)} \\
 & = \sum_{x \in B} \sum_{t \in F^\times} \Omega_\psi(g, h_0 h^c h_0) \bar{\varphi}^c(x,t) \\
 & = \Theta_{\varphi'}(g,h^c).
\end{align*}
Since $h^c = h$ for $h \in H(\A)$, the assertion follows.
\end{proof}

\begin{lem}
\label{l:change_of_schwartz}
We have
\[
 \theta_{\varphi'}(f_0^\mu)(h) = (-1)^n \mu(z) \cdot \theta_\varphi(f_0^\mu)(h h_1')
\]
for $h \in \bH(\A)$, where $h_1' \in \A_{K,f}^\times$ and $z \in \A_{E,f}^\times$ are given by
\[
 h'_{1,q} = 
 \begin{cases}
  (1,q) & \text{if $q \in \Sigma_s^K \cap \Sigma_r^E$;} \\
  (q,1) & \text{if $q \in \Sigma^+$;} \\
  \bj & \text{if $q \in \Sigma_r^K$;} \\
  1 & \text{otherwise,}
 \end{cases} \quad
 z_q = 
 \begin{cases}
  \bi & \text{if $q \in \Sigma_s^K \cap \Sigma_r^E$;} \\
  (q,1) & \text{if $q \in \Sigma^+$;} \\
  (1,q) & \text{if $q \in \Sigma_r^K$;} \\
  1 & \text{otherwise.}
 \end{cases}
\]
\end{lem}

\begin{proof}
By Lemmas \ref{l:local-theta-split-1}, \ref{l:local-theta-split-2}, \ref{l:local-theta-inert}, \ref{l:local-theta-ramified}, \ref{l:local-theta-dyadic}, \ref{l:local-theta-real} below, we have
\[
 \varphi' = (-1)^n \cdot \Omega_\psi(h_1) \varphi
\]
with $h_1 = [h_1',z] \in \bH(\A_f)$.
This yields the assertion.
\end{proof}

Now we state the main result in this section.

\begin{prop}
\label{p:complex_conjugation}
We have
\[
 \cP^H(\theta_\varphi(f_0^\mu), \chi) = (-1)^{n+1} \cdot C_{\chi,\mu} \cdot \overline{\cP^H(\theta_\varphi(f_0^\mu), \chi)},
\]
where 
\[
 C_{\chi,\mu} = 
 \prod_{q \in \Sigma_s^K \cap \Sigma_r^E} \chi_q(1,q)
 \cdot \prod_{q \in \Sigma^+} \chi_q(q,1) \mu_q(1,q)
 \cdot \prod_{q \in \Sigma^-} \xi_q(q)
 \cdot \prod_{q \in \Sigma_r^K} \mu_q(q,1).
\]
\end{prop}

\begin{proof}
By Lemmas \ref{l:complex_conjugation} and \ref{l:change_of_schwartz}, we have
\[
 \overline{\cP^H(\theta_\varphi(f_0^\mu), \chi)}
 = (-1)^n \mu(z) \chi(h_1')^{-1} \cdot \cP^H(\theta_\varphi(f_0^\mu), \chi).
\]
On the other hand, we have
\begin{align*}
 \prod_{q \in \Sigma_r^K} \chi_q(\bj) & = \prod_{v \ne \infty} \chi_v(\bj) = \chi_\infty(\bj)^{-1} = -1, \\
 \prod_{q \in \Sigma_r^E} \mu_q(\bi) & = \prod_{v \ne \infty} \mu_v(\bi) = \mu_\infty(\bi)^{-1} = 1. 
\end{align*}
Since $\Sigma_r^E = (\Sigma_s^K \cap \Sigma_r^E) \cup \Sigma^-$ and $\mu_q = \xi_q \circ \N_{E_q/\Q_q}$ for $q \in \Sigma^-$, we have
\[
 \prod_{q \in \Sigma_s^K \cap \Sigma_r^E} \mu_q(\bi) = \prod_{q \in \Sigma^-} \mu_q(\bi)^{-1} = \prod_{q \in \Sigma^-} \xi_q(q).
\]
This yields the assertion.
\end{proof}

To finish the proof of Lemma \ref{l:change_of_schwartz} (and hence Proposition \ref{p:complex_conjugation}), we need to compute $\varphi'$ explicitly.
We fix a place $v$ of $F$ and omit the subscript $v$ from the notation.

\subsection{The split case I}

We use the setting of \S \ref{ss:choices-split-1}.
In particular, $\varphi$ is the characteristic function of $\fo_B \times \fo_F^\times$, where $\fo_B$ is a maximal order in $B$ given by $\fo_B = \M_2(\fo_F)$ under the identification $B = \M_2(F)$ as in \S \ref{ss:choices-split-1}.

\begin{lem}
\label{l:local-theta-split-1}
\begin{enumerate}
\item 
\label{local-theta-split-1-i}
If $E$ is split or inert, then we have 
\[
 \Omega_\psi(h_0) \bar{\varphi}^c = \varphi.
\]
\item 
\label{local-theta-split-1-ii}
If $E$ is ramified, then we have
\[
 \Omega_\psi(h_0) \bar{\varphi}^c = \Omega_\psi(h_1) \varphi,
\]
where $h_1 = [h_1', \bi] \in \bH$ with $h_1' = (1,\varpi_F) \in K^\times$ under the identification $K = F \oplus F$ as in \S \ref{ss:choices-split-1}.
\end{enumerate}
\end{lem}

\begin{proof}
Since
\[
 \mat{x_1}{x_2}{x_3}{x_4}^c = \mat{x_1}{-x_2}{-x_3}{x_4}, 
\]
we have $\bar{\varphi}^c = \varphi$.
Also, we have 
\[
 \bi = \mat{0}{1}{u}{0}.
\]
If $E$ is split or inert, then we have $\bi \in \fo_B^\times$.
Hence we have $h_0 \fo_B = \bi \fo_B \bi^{-1} = \fo_B$, so that $\Omega_\psi(h_0) \varphi = \varphi$.
This proves \eqref{local-theta-split-1-i}.
If $E$ is ramified, then we have 
\[
 h_1' = \mat{1}{0}{0}{\varpi_F} \in \bi \fo_B^\times
\]
and $\nu(h_1) \in \fo_F^\times$.
Hence we have $h_0 \fo_B = \bi \fo_B \bi^{-1} = h_1' \fo_B \bi^{-1} = h_1 \fo_B$, so that $\Omega_\psi(h_0) \varphi = \Omega_\psi(h_1) \varphi$.
This proves \eqref{local-theta-split-1-ii}.
\end{proof}

\subsection{The split case II}

We use the setting of \S \ref{ss:choices-split-2}.
In particular, $\varphi$ is the characteristic function of $\mathfrak{I}_B k_0 \times \fo_F^\times$, where $\mathfrak{I}_B$ is an Iwahori subalgebra of $B$ and $k_0$ is an element in $B^\times$ given by 
\[
 \mathfrak{I}_B = \mat{\fo_F}{\fo_F}{\fp_F}{\fo_F}, \quad
 k_0 = \mat{1}{\frac{1}{u'}}{1}{-\frac{1}{u'}}
\]
under the identification $B = \M_2(F)$ as in \S \ref{ss:choices-split-1}.

\begin{lem}
\label{l:local-theta-split-2}
We have
\[
 \Omega_\psi(h_0) \bar{\varphi}^c = \Omega_\psi(h_1) \varphi,
\]
where $h_1 = [h_1', z] \in \bH$ with $h_1' = (\varpi_F,1) \in K^\times$ and $z = (\varpi_F,1) \in E^\times$ under the identifications $K = F \oplus F$ and $E = F \oplus F$ as in \S \ref{ss:choices-split-1} and \S \ref{ss:choices-ramified}, respectively.
\end{lem}

\begin{proof}
Since
\[
 \mat{x_1}{x_2}{x_3}{x_4}^c = \mat{x_1}{-x_2}{-x_3}{x_4}, \quad 
 k_0^c = \mat{0}{1}{-1}{0} k_0, 
\]
$\bar{\varphi}^c$ is the characteristic function of $\tilde{\mathfrak{I}}_B k_0 \times \fo_F^\times$, where 
\[
 \tilde{\mathfrak{I}}_B = \mat{\fo_F}{\fo_F}{\fo_F}{\fp_F}.
\]
Also, we have
\[
 \bi = \mat{0}{1}{u}{0}, \quad
 k_0 \bi^{-1} k_0^{-1} = \mat{\frac{1}{u'}}{0}{0}{-\frac{1}{u'}}, \quad
 h_1' = \mat{\varpi_F}{0}{0}{1}, \quad
 k_0 z^{-1} k_0^{-1} = \mat{\varpi_F^{-1}}{0}{0}{1},
\]
and $\nu(h_1) = 1$.
Hence, putting 
\[
 \mathfrak{I}_B^- = \mat{\fo_F}{\fp_F}{\fo_F}{\fo_F}, 
\]
we have $h_0 \tilde{\mathfrak{I}}_B k_0 = \bi \tilde{\mathfrak{I}}_B k_0 \bi^{-1} = \mathfrak{I}_B^- k_0$ and $h_1 \mathfrak{I}_B k_0 = h_1' \mathfrak{I}_B k_0 z^{-1} = \mathfrak{I}_B^- k_0$, so that $\Omega_\psi(h_0) \bar{\varphi}^c = \Omega_\psi(h_1) \varphi$.
This completes the proof.
\end{proof}

\subsection{The inert case}

We use the setting of \S \ref{ss:choices-inert-1} or \S \ref{ss:choices-inert-2}.
In particular, $\varphi$ is the characteristic function of $\fo_B \times \fo_F^\times$, where $\fo_B$ is a maximal order in $B$ given by $\fo_B = \fo_K + \fo_K \bi$.

\begin{lem}
\label{l:local-theta-inert} 
We have
\[
 \Omega_\psi(h_0) \bar{\varphi}^c = \varphi.
\]
\end{lem}

\begin{proof}
Since $(x_1 + x_2 \bi)^c = x_1 - x_2 \bi$ for $x_1, x_2 \in K$, we have $\bar{\varphi}^c = \varphi$.
Also, we have $h_0 \fo_B = \bi \fo_B \bi^{-1} = \fo_B$, so that $\Omega_\psi(h_0) \varphi = \varphi$.
This completes the proof.
\end{proof}

\subsection{The ramified case}

We use the setting of \S \ref{ss:choices-ramified}.
In particular, $\varphi$ is the characteristic function of $\fo_B \times \fo_F^\times$, where $\fo_B$ is a maximal order in $B$ given by $\fo_B = \M_2(\fo_F)$ under the identification $B = \M_2(F)$ as in \S \ref{ss:choices-ramified}.

\begin{lem}
\label{l:local-theta-ramified}
We have
\[
 \Omega_\psi(h_0) \bar{\varphi}^c = \Omega_\psi(h_1) \varphi,
\]
where $h_1 = [\bj, z] \in \bH$ with $z = (1,\varpi_F) \in E^\times$ under the identification $E = F \oplus F$ as in \S \ref{ss:choices-ramified}.
\end{lem}

\begin{proof}
Since
\[
 \mat{x_1}{x_2}{x_3}{x_4}^c = \mat{x_4}{x_3 J^{-1}}{x_2 J}{x_1}, \quad
 z = \mat{1}{0}{0}{\varpi_F}, 
\]
$\bar{\varphi}^c$ is the characteristic function of $\fo_B' \times \fo_F^\times$, where $\fo_B' = z \fo_B z^{-1}$.
Also, we have
\[
 \bi = \mat{u'}{0}{0}{-u'} \in (\fo_B')^\times, \quad
 \bj = \mat{0}{1}{J}{0} \in z \fo_B^\times,
\]
and $\nu(h_1) \in \fo_F^\times$.
Hence we have $h_0 \fo_B' = \bi \fo_B' \bi^{-1} = \fo_B' = \bj \fo_B z^{-1} = h_1 \fo_B$, so that $\Omega_\psi(h_0) \bar{\varphi}^c = \Omega_\psi(h_1) \varphi$.
This completes the proof.
\end{proof}

\subsection{The dyadic case}

We use the setting of \S \ref{ss:choices-dyadic}.
In particular, $\varphi$ is the characteristic function of $\fo_B k_0 \times 2^{-1} \fo_F^\times$, where $\fo_B$ is a maximal order in $B$ and $k_0$ is an element in $B^\times$ given by
\[
 \fo_B = \M_2(\fo_F), \quad
 k_0 = \mat{1}{\frac{1}{u'}}{1}{-\frac{1}{u'}}
\]
under the identification $B = \M_2(F)$ as in \S \ref{ss:choices-split-1}.

\begin{lem}
\label{l:local-theta-dyadic}
We have
\[
 \Omega_\psi(h_0) \bar{\varphi}^c = \varphi.
\]
\end{lem}

\begin{proof}
Since
\[
 \mat{x_1}{x_2}{x_3}{x_4}^c = \mat{x_1}{-x_2}{-x_3}{x_4}, \quad
 k_0^c = \mat{0}{1}{-1}{0} k_0, 
\]
we have $\bar{\varphi}^c = \varphi$.
Also, we have 
\[
 \bi = \mat{0}{1}{u}{0} \in \fo_B^\times, \quad
 k_0 \bi^{-1} k_0^{-1} = \mat{\frac{1}{u'}}{0}{0}{-\frac{1}{u'}} \in \fo_B^\times.
\]
Hence we have $h_0 \fo_B k_0 = \bi \fo_B k_0 \bi^{-1} = \fo_B k_0$, so that $\Omega_\psi(h_0) \varphi = \varphi$.
This completes the proof.
\end{proof}

\subsection{The real case}

Suppose that $F=\R$, $E=\C$, and $B=\mathbb{H}$.
Let $\varphi = \ell \circ \Omega_\psi(g_0) \vec{\varphi}$ be the Schwartz function as in \S \ref{ss:adelic_period}.

\begin{lem}
\label{l:local-theta-real}
We have
\[
 \Omega_\psi(h_0) \bar{\varphi}^c = (-1)^n \cdot \varphi.
\]
\end{lem}

\begin{proof}
By Proposition \ref{p:phi-sharp} below, we have
\[
 \varphi(x,t) = (\sqrt{-1})^n \cdot t \phi(t) \cdot (\bar{x}_1^2 - u \bar{x}_2^2) \cdot P(t x_1 \bar{x}_1, t x_2 \bar{x}_2, \sqrt{-1} t x_1 \bar{x}_2, \sqrt{-1} t \bar{x}_1 x_2) \cdot e^{-2 \pi t (x_1 \bar{x}_1 - u x_2 \bar{x}_2)}
\]
for $x = x_1 + x_2 \bi$ with $x_1, x_2 \in K = \C$ and $t \in \R^\times$, where $\phi$ is the fixed $\R$-valued function as in \S \ref{ss:schwartz-def} and $P$ is some polynomial in four variables with coefficients in $\R$.
Hence, noting that $(h_0^{-1} x)^c = (\bar{x}_1 + \bar{x}_2 \bi)^c = \bar{x}_1 - \bar{x}_2 \bi$, the assertion follows.
\end{proof}

\section{Waldspurger's formula}
\label{s:waldspurger}

In this section, we prove an explicit version of Waldspurger's formula for toric periods \cite{wald85}.

\subsection{An explicit formula}
\label{ss:wald-explicit}

As in \S \ref{ss:toric_period}, we consider the toric period
\[
 \cP^H(\theta_\varphi(f_0^\mu), \chi) = \int_{\A^\times K^\times \backslash \A_K^\times} \theta_\varphi(f_0^\mu)(h) \chi(h) \, d\dot{h}, 
\]
where $d\dot{h}$ is the quotient measure induced by the Tamagawa measures on $\A_K^\times$ and $\A^\times$.
Note that $d \dot{h} = C_K^{-1} \cdot \prod_v d \dot{h}_v$ with $d \dot{h}_v = dh_v / d\nu_v$, where $dh_v$ is the standard measure on $K_v^\times$ as in \S \ref{ss:theta-Sp4} and $d\nu_v$ is the standard measure on $\Q_v^\times$ given as follows:
\begin{itemize}
\item if $v = \infty$, then $d\nu_v = d\nu_v^{\mathrm{Leb}} / |\nu_v|$, where $d\nu_v^{\mathrm{Leb}}$ is the Lebesgue measure on $\Q_v = \R$;
\item if $v \ne \infty$, then $d\nu_v$ is the Haar measure on $\Q_v^\times$ such that $\vol(\Z_v^\times) = 1$.
\end{itemize}

\begin{prop}
\label{p:wald-explicit}
We have
\[
 |\cP^H(\theta_\varphi(f_0^\mu), \chi)|^2 = \pi \cdot |D_K|^{1/2} C_K^{-2} \cdot \zeta(2) \cdot C_{\Sigma^+,\Sigma^-} \cdot \frac{L_{\fin}(\frac{1}{2}, \pi_K \times \chi)}{L_{\fin}(1, \pi, \Ad)} \cdot \langle{\theta_\varphi(f_0^\mu),\theta_\varphi(f_0^\mu)}\rangle,
\]
where
\[
 C_{\Sigma^+,\Sigma^-} = \prod_{q \in \Sigma^+} \frac{q+1}{q} \cdot \prod_{q \in \Sigma^-} \frac{q-1}{q}, 
\]
$L_{\fin}(s, \pi, \Ad)$ is the adjoint $L$-function of $\pi$, and $\langle \cdot, \cdot \rangle$ is the Petersson inner product with respect to the Tamagawa measure on $Z_\bH(\A) \backslash \bH(\A)$.
\end{prop}

\begin{proof}
Let $f^B$ be the restriction of $\theta_\varphi(f_0^\mu)$ to $B^\times(\A)$, where we regard $B^\times$ as a subgroup of $\bH$ via the map $h \mapsto [h,1]$.
Then we have
\begin{align*}
 \cP^H(\theta_\varphi(f_0^\mu), \chi) & = \int_{\A^\times K^\times \backslash \A_K^\times} f^B(h) \chi(h) \, d\dot{h}, \\
 \langle \theta_\varphi(f_0^\mu), \theta_\varphi(f_0^\mu) \rangle & = \int_{\A^\times B^\times(F) \backslash B^\times(\A)} |f^B(h)|^2 \, dh,
\end{align*}
where $dh$ is the Tamagawa measure on $\A^\times \backslash B^\times(\A)$.
We may assume that $f^B$ is nonzero.
Then by Proposition \ref{p:theta-u11-u2} and the choice of $\varphi$, $f^B = \otimes_v f^B_v$ is a decomposable vector in $\pi^B$ such that 
\begin{itemize}
\item 
$\pi^B_\infty(h) f^B_\infty = h \bar{h}^{-1} f^B_\infty$ for all $h \in K_\infty^\times = \C^\times$;
\item 
if $v \notin \Sigma \cup \{ \infty \}$, then $f_v^B$ is $\fo_{B_v}^\times$-fixed, where $\fo_{B_v}$ is the maximal order in $B_v$ as in \S \ref{ss:choices-split-1} or \S \ref{ss:choices-inert-1} or \S \ref{ss:choices-ramified} or \S \ref{ss:choices-dyadic};
\item 
if $v \in \Sigma^+$, then $f^B_v$ is $\mathfrak{I}_{B_v}^\times$-fixed, where $\mathfrak{I}_{B_v}$ is the Iwahori subalgebra of $B_v$ as in \S \ref{ss:choices-split-2};
\item 
if $v \in \Sigma^-$, then $f^B_v$ is $\fo_{B_v}^\times$-fixed, where $\fo_{B_v}$ is the unique maximal order in $B_v$.
\end{itemize}
Note that $f_B$ is uniquely determined by these conditions up to scalars.
Moreover, by \cite[Proposition 7]{wald85}, we have
\begin{align*}
 \frac{|\cP^H(\theta_\varphi(f_0^\mu), \chi)|^2}{\langle \theta_\varphi(f_0^\mu), \theta_\varphi(f_0^\mu) \rangle} & = 2^{-1} C_K^{-1} \cdot \frac{\zeta(2) L_{\fin}(\frac{1}{2}, \pi_K \times \chi)}{L_{\fin}(1, \pi, \Ad) L_{\fin}(1, \xi_{K/\Q})} \\
 & \quad \times \cP_{\infty}(f^B_\infty, \chi_\infty) \cdot \prod_{v \ne \infty} \left( \frac{\zeta_v(2) L_v(\frac{1}{2}, \pi_K \times \chi)}{L_v(1, \pi, \Ad) L_v(1, \xi_{K/\Q})} \right)^{-1} \cP_v(f^B_v, \chi_v) 
\end{align*}
with 
\[
 \cP_v(f^B_v, \chi_v) = \int_{\Q_v^\times \backslash K_v^\times} \Psi_{f^B_v}(h_v) \chi_v(h_v) \, d\dot{h}_v.
\]
Here $\Psi_{f^B_v}$ is a matrix coefficient of $\pi^B_v$ given by 
\[
 \Psi_{f^B_v}(h_v) = \frac{\langle \pi^B_v(h_v) f^B_v, f^B_v \rangle_v}{\langle f^B_v, f^B_v \rangle_v},
\]
where $\langle \cdot, \cdot \rangle_v$ is a $B^\times_v$-invariant inner product on $\pi^B_v$.
Hence the proposition follows from Lemmas \ref{l:wald-ps}, \ref{l:wald-st}, \ref{l:wald-1dim}, \ref{l:wald-real} below, where we compute $\cP_v(f^B_v, \chi_v)$ explicitly.
\end{proof}

To finish the proof of Proposition \ref{p:wald-explicit}, we need to compute the integrals of the matrix coefficients explicitly.
We fix a place $v$ of $F$ and omit the subscript $v$ from the notation.
For a matrix coefficient $\Psi$ of $\pi^B$, we will compute the integral
\[
 \cP(\Psi, \chi) = \int_{F^\times \backslash K^\times} \Psi(h) \chi(h) \, d\dot{h},
\]
where $d\dot{h}$ is the quotient measure induced by the standard measures on $K^\times$ and $F^\times$.

\subsection{The principal series case}

We use the setting of \S \ref{ss:choices-split-1} or \S \ref{ss:choices-inert-1} or \S \ref{ss:choices-ramified} or \S \ref{ss:choices-dyadic}.
In particular, $B$ is split. 
We identify $B$ with $\M_2(F)$ as before.
Let $\pi = \Ind(\xi_1 \otimes \xi_2)$ be a principal series representation of $\GL_2(F)$, where $\xi_1$ and $\xi_2$ are unramified characters of $F^\times$, and $\chi$ an unramified character of $K^\times$ such that $\xi_1 \cdot \xi_2 \cdot \chi|_{F^\times} = 1$.
Let $\Psi$ be the $\GL_2(\fo_F)$-bi-invariant matrix coefficient of $\pi$ such that $\Psi(1) = 1$.

\begin{lem}
\label{l:wald-ps}
We have
\[
 \cP(\Psi, \chi) = \frac{\zeta(2) L(\frac{1}{2}, \pi_K \times \chi)}{L(1, \pi, \Ad) L(1, \xi_{K/F})}.
\]
\end{lem}

\begin{proof}
Recall Macdonald's formula
\[
 \Psi\! \mat{\varpi_F^i}{0}{0}{1} = 
 \frac{q^{-i/2}}{1+q^{-1}}
 \left( \alpha_1^i \frac{1 - \alpha_1^{-1} \alpha_2 q^{-1}}{1 - \alpha_1^{-1} \alpha_2} + \alpha_2^i \frac{1 - \alpha_1 \alpha_2^{-1} q^{-1}}{1 - \alpha_1 \alpha_2^{-1}} \right)
\]
for $i \ge 0$, where $\alpha_1 = \xi_1(\varpi_F)$ and $\alpha_2 = \xi_2(\varpi_F)$.

First assume that $K$ is split.
We identify $K$ with $F \oplus F$ as in \S \ref{ss:choices-split-1}.
Put $\beta = \chi(\varpi_F, 1)$, so that $\chi(1, \varpi_F) = \alpha_1^{-1} \alpha_2^{-1} \beta^{-1}$.
Then we have 
\begin{align*}
 \cP(\Psi,\chi) & = \sum_{i \in \Z} \Psi\! \mat{\varpi_F^i}{0}{0}{1} \chi(\varpi_F^i,1) \\ 
 & = \sum_{i = 0}^\infty \frac{q^{-i/2}}{1+q^{-1}}
 \left( \alpha_1^i \frac{1 - \alpha_1^{-1} \alpha_2 q^{-1}}{1 - \alpha_1^{-1} \alpha_2} + \alpha_2^i \frac{1 - \alpha_1 \alpha_2^{-1} q^{-1}}{1 - \alpha_1 \alpha_2^{-1}} \right) \beta^i \\
 & + \sum_{i = 1}^\infty \frac{q^{-i/2}}{1+q^{-1}}
 \left( \alpha_1^i \frac{1 - \alpha_1^{-1} \alpha_2 q^{-1}}{1 - \alpha_1^{-1} \alpha_2} + \alpha_2^i \frac{1 - \alpha_1 \alpha_2^{-1} q^{-1}}{1 - \alpha_1 \alpha_2^{-1}} \right) \alpha_1^{-i} \alpha_2^{-i} \beta^{-i} \\
 & = \frac{1}{(1+q^{-1})(\alpha_1 - \alpha_2)} 
 \bigg( \frac{\alpha_1 - \alpha_2 q^{-1}}{1 - \alpha_1 \beta q^{-1/2}}
 - \frac{\alpha_2 - \alpha_1 q^{-1}}{1 - \alpha_2 \beta q^{-1/2}} \\
 & \qquad + \frac{\alpha_2^{-1} \beta^{-1} q^{-1/2} (\alpha_1 - \alpha_2 q^{-1})}{1 - \alpha_2^{-1} \beta^{-1} q^{-1/2}} 
 - \frac{\alpha_1^{-1} \beta^{-1} q^{-1/2} (\alpha_2 - \alpha_1 q^{-1})}{1 - \alpha_1^{-1} \beta^{-1} q^{-1/2}} 
 \bigg) \\
 & = \frac{(1-q^{-1}) (1-\alpha_1 \alpha_2^{-1} q^{-1}) (1-\alpha_1^{-1} \alpha_2 q^{-1})}{(1 + q^{-1}) (1 - \alpha_1 \beta q^{-1/2}) (1 - \alpha_2 \beta q^{-1/2}) (1 - \alpha_1^{-1} \beta^{-1} q^{-1/2}) (1 - \alpha_2^{-1} \beta^{-1} q^{-1/2})}.
\end{align*}
This yields the desired identity.

Next assume that $K$ is inert.
Then we have 
\[
 \cP(\Psi, \chi) = \Psi(1) = 1.
\]
This yields the desired identity.

Finally, assume that $K$ is ramified.
Then we may assume that 
\[
 \varpi_K = \mat{0}{1}{\varpi_F}{0}
\]
in $B$.
Put $\beta = \chi(\varpi_K)$, so that $\alpha_1 \alpha_2 \beta^2 = 1$.
Then we have
\begin{align*}
 \cP(\Psi,\chi) & = \Psi(1) + \Psi\! \mat{0}{1}{\varpi_F}{0} \chi(\varpi_K) \\ 
 & = 1 + \frac{q^{-1/2}}{1+q^{-1}}
 \left( \alpha_1 \frac{1 - \alpha_1^{-1} \alpha_2 q^{-1}}{1 - \alpha_1^{-1} \alpha_2} + \alpha_2 \frac{1 - \alpha_1 \alpha_2^{-1} q^{-1}}{1 - \alpha_1 \alpha_2^{-1}} \right) \beta \\
 & = \frac{(1 + \alpha_1 \beta q^{-1/2}) (1 + \alpha_2 \beta q^{-1/2})}{1+q^{-1}}.
\end{align*}
This yields the desired identity.
\end{proof}

\subsection{The twisted Steinberg case}

We use the setting of \S \ref{ss:choices-split-2}.
In particular, $B$ is split. 
We identify $B$ with $\M_2(F)$ as before.
Let $I$ be an Iwahori subgroup of $\GL_2(F)$ given by
\[
 I = \left\{ k \in \GL_2(\fo_F) \; \middle| \; k \equiv \mat{*}{*}{0}{*} \bmod \fp_F \right\}.
\]
Let $\pi = \St \otimes \xi$ be a twisted Steinberg representation of $\GL_2(F)$, where $\xi$ is an unramified character of $F^\times$, and $\chi$ an unramified character of $K^\times$ such that $\xi^2 \cdot \chi|_{F^\times} = 1$.
Let $\Psi$ be the $I$-bi-invariant matrix coefficient of $\pi$ such that $\Psi(1) = 1$.

\begin{lem}
\label{l:wald-st}
We have
\[
 \cP(\Psi, \chi) = \frac{q+1}{q} \cdot \frac{\zeta(2) L(\frac{1}{2}, \pi_K \times \chi)}{L(1, \pi, \Ad) L(1, \xi_{K/F})}.
\]
\end{lem}

\begin{proof}
We identify $K$ with $F \oplus F$ as in \S \ref{ss:choices-split-1}.
Then we have
\begin{align*}
 \cP(\Psi, \chi) = \sum_{i \in \Z} \Psi\! \mat{\varpi_F^i}{0}{0}{\varpi_F^{-i}} \chi(\varpi_F^i, \varpi_F^{-i}) + \sum_{i \in \Z} \Psi\! \mat{\varpi_F^{i+1}}{0}{0}{\varpi_F^{-i}} \chi(\varpi_F^{i+1}, \varpi_F^{-i}).
\end{align*}
As in \cite[\S 6.6.3]{periods1}, we have 
\[
 \Psi\! \mat{\varpi_F^i}{0}{0}{\varpi_F^{-i}} = q^{-2|i|}, \quad
 \Psi\! \mat{\varpi_F^{i+1}}{0}{0}{\varpi_F^{-i}} = \xi(\varpi_F) \cdot q^{-|2i+1|}.
\]
Hence, putting $\alpha = \xi(\varpi_F) \cdot \chi(\varpi_F, 1)$ (so that $\alpha^2 = \chi(\varpi_F, \varpi_F^{-1})$), we have
\begin{align*}
 \cP(\Psi, \chi) & = \sum_{i=0}^\infty q^{-2i} \alpha^{2i}
 + \sum_{i=1}^\infty q^{-2i} \alpha^{-2i}
 + \sum_{i=0}^\infty q^{-2i-1} \alpha^{2i+1}
 + \sum_{i=1}^\infty q^{-2i+1} \alpha^{-2i+1} \\
 & = \frac{1}{1 - \alpha^2 q^{-2}} + \frac{\alpha^{-2} q^{-2}}{1 - \alpha^{-2} q^{-2}} + \frac{\alpha q^{-1}}{1 - \alpha^2 q^{-2}} + \frac{\alpha^{-1} q^{-1}}{1 - \alpha^{-2} q^{-2}} \\
 & = \frac{1-q^{-2}}{(1 - \alpha q^{-1}) (1 - \alpha^{-1} q^{-1})}.
\end{align*}
This yields the desired identity.
\end{proof}

\subsection{The $1$-dimensional case}

We use the setting of \S \ref{ss:choices-inert-2}.
In particular, $B$ is ramified.
Let $\pi = \St \otimes \xi$ be a twisted Steinberg representation of $\GL_2(F)$, where $\xi$ is an unramified character of $F^\times$, and $\chi$ an unramified character of $K^\times$ such that $\xi^2 \cdot \chi|_{F^\times} = 1$.
Then $\pi^B$ is the $1$-dimensional representation $\xi \circ \nu$.
Let $\Psi$ be the matrix coefficient of $\pi^B$ such that $\Psi(1) = 1$.

\begin{lem}
\label{l:wald-1dim}
We have
\[
 \cP(\Psi, \chi) = \frac{q-1}{q} \cdot \frac{\zeta(2) L(\frac{1}{2}, \pi_K \times \chi)}{L(1, \pi, \Ad) L(1, \xi_{K/F})}.
\]
\end{lem}

\begin{proof}
Since $K$ is inert, we have 
\[
 \cP(\Psi, \chi) = \Psi(1) = 1.
\]
This yields the desired identity.
\end{proof}

\subsection{The real case}

Suppose that $F=\R$, $K=\C$, and $B=\mathbb{H}$.
Let $\pi$ be the discrete series representation of $\GL_2(\R)$ of weight $2n+4$ and $\chi$ a character of $\C^\times$ given by $\chi(z) = \bar{z}/z$.
Then $\pi^B$ is the $(2n+3)$-dimensional representation with trivial central character.
Let $\Psi$ be the matrix coefficient of $\pi^B$ such that 
\[
 \Psi(h_1 h h_2) = \chi(h_1 h_2)^{-1} \Psi(h)
\]
for all $h \in B^\times$ and $h_1, h_2 \in K^\times$, and such that $\Psi(1) = 1$.

\begin{lem}
\label{l:wald-real}
We have
\[
 \cP(\Psi, \chi) = 2 \pi.
\]
\end{lem}

\begin{proof}
We have
\[
 \cP(\Psi,\chi) = \vol(\C^\times/\R^\times, d\dot{h}) \cdot \Psi(1).
\]
If we write $h = r e^{\sqrt{-1} \theta}$ with $r \in \R^\times_+$ and $\theta \in \R / 2 \pi \Z$, then we have
\[
 \vol(\C^\times/\R^\times, d\dot{h}) = \frac{1}{[\R^\times : \R^\times_+]} \cdot \vol \left( \C^\times/\R^\times_+, \frac{2 r^{-1} \, dr \, d \theta}{r^{-1} \, dr} \right) = 2 \pi, 
\]
where $dr$ and $d\theta$ are the Lebesgue measures.
This completes the proof.
\end{proof}

\section{Rallis inner product formula}
\label{s:rallis}

In this section, we compute the Petersson inner product of the theta lift described in \S \ref{ss:theta-U11} explicitly.

\subsection{An abstract version}
\label{ss:rallis-abstract}

Let $F$ be a totally real number field and $B$ a totally definite quaternion algebra over $F$.
Let $E$ be a totally imaginary quadratic extension of $F$ which embeds into $B$.
Let $\bV = B$ be the $2$-dimensional hermitian $E$-space and $\bW = E^2$ the $2$-dimensional split skew-hermitian $E$-space.
Put
\begin{align*}
 \bG & = \GU(\bW) = (\GL_2(F) \times E^\times)/F^\times, &
 \bG_1 & = \U(\bW), \\
 \bH & = \GU(\bV) = (B^\times \times E^\times)/F^\times, &
 \bH_1 & = \U(\bV).
\end{align*}
Put
\[
 d(\nu) = \left[ \mat{1}{0}{0}{\nu}, 1 \right] \in \bG
\]
for $\nu \in F^\times$.

Recall the Weil representation $\omega_\psi$ (resp.~$\Omega_\psi$) of $\bG_1(\A) \times \bH_1(\A)$ (resp.~$\bG(\A) \times \bH(\A)$) on $\cS(B(\A))$ (resp.~$\cS(B(\A) \times \A^\times)$) as in \S \ref{ss:weil-u}.
We define a $\bG_1(\A)$-invariant inner product $\langle \cdot, \cdot \rangle$ on $\cS(B(\A))$ by 
\[
 \langle \varphi, \varphi' \rangle = \int_{B(\A)} \varphi(x) \overline{\varphi'(x)} \, dx,
\]
where $dx$ is the self-dual Haar measure on $B(\A)$ with respect to $\psi \circ \tr_{B/F}$.
Similarly, for each place $v$ of $F$, we define a $\bG_{1,v}$-invariant inner product $\langle \cdot, \cdot \rangle_v$ on $\cS(B_v)$, so that we have a decomposition $\langle \cdot, \cdot \rangle = \prod_v \langle \cdot, \cdot \rangle_v$.
For $h \in \bH(\A)$, we define a unitary operator $L(h)$ on $\cS(B(\A))$ by
\[
 L(h) \varphi(x) = |\nu(h)|^{-1} \cdot \varphi(h^{-1} x).
\]
Let $\pi$ be an irreducible unitary cuspidal automorphic representation of $\GL_2(\A)$ with central character $\xi_\pi$ and $\mu$ a unitary character of $\A_E^\times/E^\times$ such that $\xi_\pi \cdot \mu|_{\A^\times} = 1$.
We regard $\pi \boxtimes \mu$ as an automorphic representation of $\bG(\A)$.
Let $\langle \cdot, \cdot \rangle$ be the Petersson inner product on $\pi \boxtimes \mu$ given by
\[
 \langle f, f' \rangle = \int_{Z_\bG(\A) \bG(F) \backslash \bG(\A)} f(g) \overline{f'(g)} \, dg,
\]
where $dg$ is the Tamagawa measure on $Z_\bG(\A) \backslash \bG(\A)$.
For each place $v$ of $F$, we fix a $\bG_v$-invariant inner product $\langle \cdot, \cdot \rangle_v$ on $\pi_v \boxtimes \mu_v$.
We also fix decompositions of the Tamagawa measures $dg_1 = \prod_v dg_{1,v}$ and $d \nu = \prod_v d \nu_v$ on $\bG_1(\A)$ and $\A^\times$, respectively.

Let $\varphi = \otimes_v \varphi_v \in S(B(\A) \times \A^\times)$ be a decomposable Schwartz function and $f = \otimes_v f_v \in \pi \boxtimes \mu$ a nonzero decomposable vector.
Then we compute the Petersson inner product of the theta lift
\[
 \langle \theta_\varphi(f), \theta_\varphi(f) \rangle
 = \int_{Z_\bH(\A) \bH(F) \backslash \bH(\A)} |\theta_\varphi(f)(h)|^2 \, dh,
\]
where $dh$ is the Tamagawa measure on $Z_\bH(\A) \backslash \bH(\A)$.

\begin{prop}
\label{p:rallis-abstract}
We have
\[
 \frac{\langle \theta_\varphi(f), \theta_\varphi(f) \rangle}{\langle f, f \rangle} = \frac{1}{2} \cdot \frac{L^S(\frac{1}{2}, \pi_E \times \mu)}{L^S(1, \xi_{E/F}) \zeta^S(2)} \cdot \prod_{v \in S} Z_v(\varphi_v, f_v),
\]
where $S$ is a sufficiently large finite set of places of $F$, the superscript $S$ indicates that we take a product of local factors over all $v \notin S$, and $Z_v(\varphi_v, f_v)$ is a local zeta integral given by
\[
 Z_v(\varphi_v, f_v) = \int_{F_v^\times} \int_{F_v^\times} \int_{\bG_{1,v}} |\nu_v \nu_v'| \langle \omega_{\psi,v}(g_{1,v}) \varphi_{v, \nu_v}, \varphi_{v, \nu_v'} \rangle_v \Psi_{f_v}(d(\nu_v') g_{1,v} d(\nu_v)^{-1}) \, dg_{1,v} \, d \nu_v \, d \nu_v'
\]
with $\varphi_{v,\nu_v}(x_v) = \varphi_v(x_v,\nu_v)$ and 
\[
 \Psi_{f_v}(g_v) = \frac{\langle (\pi_v \boxtimes \mu_v)(g_v) f_v, f_v \rangle_v}{\langle f_v, f_v \rangle_v}.
\]
\end{prop}

\begin{proof}
Put $\varphi_\nu(x) = \varphi(x,\nu)$ and $f_\nu(g) = f(g d(\nu))$.
Recall that
\begin{align*}
 \theta_\varphi(f)(h) & = \int_{\bG(F) \backslash \bG(\A)} \Theta_\varphi(g,h) f(g) \, dg \\
 & = \int_{F^\times \backslash \A^\times} \int_{\bG_1(F) \backslash \bG_1(\A)}
 \sum_{x \in B} \sum_{t \in F^\times} \Omega_\psi(g_1 d(\nu), h) \varphi(x,t) f_\nu(g_1) \, dg_1 \, d \nu.
\end{align*}
Since
\begin{align*}
 \Omega_\psi(g_1 d(\nu), h) \varphi(x,t) 
 & = |\nu(h) \nu^{-1}| \omega_{t \psi}(g_1) L(h) \varphi_{\nu(h) \nu^{-1} t}(x) \\
 & = |\nu(h) \nu^{-1}| \omega_{\psi}(d(t)^{-1} g_1 d(t)) L(h) \varphi_{\nu(h) \nu^{-1} t}(x),
\end{align*}
we have
\begin{align*}
 \theta_\varphi(f)(h) & = \int_{F^\times \backslash \A^\times} \int_{\bG_1(F) \backslash \bG_1(\A)} \sum_{x \in B} \sum_{t \in F^\times} 
 |\nu(h) \nu^{-1}| \omega_{\psi}(d(t)^{-1} g_1 d(t)) L(h) \varphi_{\nu(h) \nu^{-1} t}(x) f_\nu(g_1) \, dg_1 \, d \nu \\
 & = \int_{F^\times \backslash \A^\times} \sum_{t \in F^\times} \int_{\bG_1(F) \backslash \bG_1(\A)} \sum_{x \in B}
 |\nu(h) \nu^{-1}| \omega_{\psi}(g_1) L(h) \varphi_{\nu(h) \nu^{-1} t}(x) f_{\nu t^{-1}}(g_1) \, dg_1 \, d \nu \\
 & = \int_{\A^\times} \int_{\bG_1(F) \backslash \bG_1(\A)} \sum_{x \in B}
 |\nu(h) \nu^{-1}| \omega_{\psi}(g_1) L(h) \varphi_{\nu(h) \nu^{-1}}(x) f_{\nu}(g_1) \, dg_1 \, d \nu \\
 & = \int_{\A^\times} |\nu| \int_{\bG_1(F) \backslash \bG_1(\A)} \sum_{x \in B}
 \omega_{\psi}(g_1) L(h) \varphi_{\nu}(x) f_{\nu(h) \nu^{-1}}(g_1) \, dg_1 \, d \nu.
\end{align*}
Hence, putting $\cH = Z_\bH(\A) \bH_1(\A) \bH(F) \backslash \bH(\A)$ and
\[
 I(h, \nu, \nu') = \int_{\bH_1(F) \backslash \bH_1(\A)} \phi(h_1 h,\nu) \overline{\phi(h_1 h,\nu')} \, dh_1
\]
with 
\[
 \phi(h,\nu) = \int_{\bG_1(F) \backslash \bG_1(\A)} \sum_{x \in B} \omega_{\psi}(g_1) L(h) \varphi_{\nu}(x) f_{\nu(h) \nu^{-1}}(g_1) \, dg_1,
\]
we have
\begin{align*}
 \langle \theta_\varphi(f), \theta_\varphi(f) \rangle
 & = \int_{\cH} \int_{\bH_1(F) \backslash \bH_1(\A)} |\theta_\varphi(f)(h_1 h)|^2 \, dh_1 \, d\dot{h} \\
 & = \int_{\A^\times} \int_{\A^\times} |\nu \nu'| \int_{\cH} I(h, \nu, \nu') \, d\dot{h} \, d \nu \, d \nu'.
\end{align*}
Here $dh_1$ is the Tamagawa measure on $\bH_1(\A)$ and $d\dot{h}$ is the Haar measure on $\cH$ such that 
\[
 \vol(\cH) = \frac{\vol(Z_\bH(\A) \bH(F) \backslash \bH(\A))}{\vol(\bH_1(F) \backslash \bH_1(\A))} = \frac{2}{2} = 1.
\]
By the Siegel-Weil formula \cite[Theorem 4.2]{ichino}, we have
\[
 I(h, \nu, \nu') = \frac{1}{2} \cdot J(0,h,\nu,\nu')
\]
with
\[
 J(s,h,\nu,\nu') = \int_{\bG_1(F) \backslash \bG_1(\A)} \int_{\bG_1(F) \backslash \bG_1(\A)} E(\iota(g_1,g_2), s, \cF_{h,\nu,\nu'}) f_{\nu(h) \nu^{-1}}(g_1) \overline{f_{\nu(h) \nu'^{-1}}(g_2)} \, d g_1 \, d g_2,
\]
where the notation is as follows.

Let $\bW^\square = \bW \oplus \bW$ be the $4$-dimensional $E$-vector space equipped with the skew-hermitian form
\[
 ((w_1,w_2), (w_1',w_2')) = (w_1,w_1') - (w_2,w_2')
\]
for $w_1,w_2,w_1',w_2' \in \bW$.
Then we have a natural embedding
\[
 \iota : \{ (g_1,g_2) \in \bG \times \bG \mid \nu(g_1) = \nu(g_2) \} \hookrightarrow \GU(\bW^\square).
\]
Let $\bW^\square = \bW^\bigtriangledown \oplus \bW^\triangle$ be a complete polarization given by 
\[
 \bW^\bigtriangledown = \{ (w, -w) \mid w \in \bW \}, \quad
 \bW^\triangle = \{ (w, w) \mid w \in \bW \}.
\]
Let $\bP$ be the Siegel parabolic subgroup of $\U(\bW^\square)$ stabilizing $\bW^\triangle$.
For $s \in \C$, we write 
\[
 \cI(s) = \Ind^{\U(\bW^\square)(\A)}_{\bP(\A)} (|{\det}_{\bW^\bigtriangledown} |^s)
\]
for the degenerate principal series representation of $\U(\bW^\square)(\A)$.
For a holomorphic section $\cF$ of $\cI(s)$, we define an Eisenstein series $E(s,\cF)$ by (the meromorphic continuation of)
\[
 E(g,s,\cF) = \sum_{\gamma \in \bP(F) \backslash \U(\bW^\square)(F)} \cF(\gamma g, s).
\]
Note that $E(s,\cF)$ is holomorphic at $s=0$.
For $\varphi^\square \in \cS((\bV \otimes \bW^\bigtriangledown)(\A))$, we may define a standard section $\cF_{\varphi^\square}$ of $\cI(s)$ (i.e.~a holomorphic section of $\cI(s)$ whose restriction to the standard maximal compact subgroup of $\U(\bW^\square)(\A)$ does not depend on $s$) such that
\[
 \cF_{\varphi^\square}(g,0) = \omega_{\psi}^\square(g) \varphi^\square(0),
\]
where $\omega_\psi^\square$ is the Weil representation of $\U(\bW^\square)(\A)$.
As in \cite[p.~182]{li92}, we define a partial Fourier transform
\[
 \delta : \cS(B(\A)) \mathbin{\hat{\otimes}} \cS(B(\A)) \rightarrow \cS((\bV \otimes \bW^\bigtriangledown)(\A))
\]
such that $\delta(\varphi_1 \otimes \overline{\varphi_2})(0) = \langle \varphi_1, \varphi_2 \rangle$ for $\varphi_1, \varphi_2 \in \cS(B(\A))$.
Note that
\begin{align*}
 \delta \circ \omega_\psi(g_1) \otimes \overline{\omega_\psi(g_2)} & = \omega_\psi^\square(\iota(g_1,g_2)) \circ \delta, \\
 \delta \circ L(h) \otimes \overline{L(h)} & = L^\square(h) \circ \delta
\end{align*}
for $g_1, g_2 \in \bG_1(\A)$ and $h \in \bH(\A)$, where $L^\square(h)$ is the unitary operator on $\cS((\bV \otimes \bW^\bigtriangledown)(\A))$ given by
\[
 L^\square(h) \varphi^\square(x) = \varphi^\square(h^{-1} x d^\square(\nu))
\]
with $d^\square(\nu) = \iota(d(\nu), d(\nu))$.
If $\varphi^\square = \delta(L(h) \varphi_\nu \otimes \overline{L(h) \varphi_{\nu'}})$, then we write $\cF_{h,\nu,\nu'} = \otimes_v \cF_{h_v,\nu_v,\nu'_v}$ for the corresponding decomposable standard section of $\cI(s)$.
Since the function $\nu \mapsto L(h) \varphi_\nu$ is compactly supported, so is the function $(\nu,\nu') \mapsto \cF_{h,\nu,\nu'}$.

Now we define another holomorphic (but not necessarily standard) section $\tilde{\cF}_{h,\nu,\nu'}$ of $\cI(s)$ by
\[
 \tilde{\cF}_{h,\nu,\nu'}(g,s) = \cF_{1,\nu,\nu'}(d^\square(\nu(h))^{-1} g d^\square(\nu(h)), s).
\]
Since 
\[
 L^\square(h)^{-1} \omega_\psi^\square(g) L^\square(h) = \omega_\psi(d^\square(\nu(h))^{-1} g d^\square(\nu(h))),
\]
we have
\begin{align*}
 \tilde{\cF}_{h,\nu,\nu'}(g,0) & = \omega_\psi(d^\square(\nu(h))^{-1} g d^\square(\nu(h))) \delta(\varphi_\nu \otimes \overline{\varphi_{\nu'}})(0) \\
 & = L^\square(h)^{-1} \omega_\psi^\square(g) L^\square(h) \delta(\varphi_\nu \otimes \overline{\varphi_{\nu'}})(0) \\
 & = \omega_\psi^\square(g) L^\square(h) \delta(\varphi_\nu \otimes \overline{\varphi_{\nu'}})(0) \\
 & = \omega_\psi^\square(g) \delta(L(h) \varphi_\nu \otimes \overline{L(h) \varphi_{\nu'}})(0) \\
 & = \cF_{h,\nu,\nu'}(g,0).
\end{align*}
Hence we have $E(0, \tilde{\cF}_{h,\nu,\nu'}) = E(0, \cF_{h,\nu,\nu'})$, so that
\[
 I(h, \nu, \nu') = \frac{1}{2} \cdot \tilde{J}(0,h,\nu,\nu')
\]
with 
\[
 \tilde{J}(s,h,\nu,\nu') = \int_{\bG_1(F) \backslash \bG_1(\A)} \int_{\bG_1(F) \backslash \bG_1(\A)} E(\iota(g_1,g_2), s, \tilde{\cF}_{h,\nu,\nu'}) f_{\nu(h) \nu^{-1}}(g_1) \overline{f_{\nu(h) \nu'^{-1}}(g_2)} \, d g_1 \, d g_2.
\]
By the doubling method of Piatetski-Shapiro and Rallis \cite{psr}, we can unfold the Eisenstein series to get 
\begin{align*}
 \tilde{J}(s,h,\nu,\nu') & = \int_{\bG_1(\A)} \tilde{\cF}_{h,\nu,\nu'}(\iota(g_1,1),s) \int_{\bG_1(F) \backslash \bG_1(\A)} f_{\nu(h) \nu^{-1}}(g_2 g_1) \overline{f_{\nu(h) \nu'^{-1}}(g_2)} \, dg_2 \, d g_1 \\
 & = \int_{\bG_1(\A)} \cF_{1,\nu,\nu'}(\iota(d(\nu(h))^{-1} g_1 d(\nu(h)),1),s) \int_{\bG_1(F) \backslash \bG_1(\A)} f_{\nu(h) \nu^{-1}}(g_2 g_1) \overline{f_{\nu(h) \nu'^{-1}}(g_2)} \, dg_2 \, d g_1 \\
 & = \int_{\bG_1(\A)} \cF_{1,\nu,\nu'}(\iota(g_1,1),s) \int_{\bG_1(F) \backslash \bG_1(\A)} f_{\nu^{-1}}(g_2 d(\nu(h)) g_1) \overline{f_{\nu'^{-1}}(g_2 d(\nu(h)))} \, dg_2 \, d g_1 
\end{align*}
for $\Re(s) \gg 0$.
On the other hand, since $\A^\times = F^\times \cdot \{ \nu(h) \mid h \in \bH(\A) \}$, the map $h \mapsto d(\nu(h))$ induces an isomorphism $\cH \simeq Z_{\bG}(\A) \bG_1(\A) \bG(F) \backslash \bG(\A)$.
Hence we have
\begin{align*}
 \int_{\cH} \tilde{J}(s,h,\nu,\nu') \, d \dot{h} 
 & = \int_{\bG_1(\A)} \cF_{1,\nu,\nu'}(\iota(g_1,1),s) \int_{Z_{\bG}(\A) \bG(F) \backslash \bG(\A)} f_{\nu^{-1}}(g g_1) \overline{f_{\nu'^{-1}}(g)} \, dg \, d g_1 \\
 & = \langle f, f \rangle \cdot \prod_v J_v(s,\nu_v, \nu'_v),
\end{align*}
where 
\[
 J_v(s,\nu_v, \nu'_v) = \int_{\bG_{1,v}} \cF_{1,\nu_v,\nu_v'}(\iota(g_{1,v},1),s) \Psi_{f_v}(d(\nu'_v) g_{1,v} d(\nu_v)^{-1}) \, dg_{1,v}.
\]
Note that the integral $J_v(s,\nu_v, \nu'_v)$ is absolutely convergent for $\Re(s) \ge 0$.
(When $\pi_v$ is tempered, this follows from \cite[Lemma 7.2]{yamana}. The general case can be proved similarly.)
Moreover, for almost all $v$, we have $J_v(s,\nu_v, \nu'_v) = 0$ unless $\nu_v, \nu_v' \in \fo_{F_v}^\times$, in which case
\[
 J_v(s,\nu_v, \nu'_v) = J_v(s,1,1) = \frac{L_v(s+\frac{1}{2}, \pi_E \times \mu)}{L_v(s+1, \xi_{E/F}) \zeta_v(s+2)}
\]
by \cite[\S 7]{lapid-rallis}.
Hence, noting that the function $(\nu_v, \nu_v') \mapsto J_v(s,\nu_v, \nu'_v)$ is compactly supported, we have
\begin{align*}
 & \int_{\A^\times} \int_{\A^\times} |\nu \nu'| \int_{\cH} \tilde{J}(s,h,\nu,\nu') \, d\dot{h} \, d\nu \, d\nu' \\ 
 & = \langle f, f \rangle \cdot \frac{L^S(s+\frac{1}{2}, \pi_E \times \mu)}{L^S(s+1, \xi_{E/F}) \zeta^S(s+2)} \cdot \prod_{v \in S} \int_{F_v^\times} \int_{F_v^\times} |\nu_v \nu_v'| J_v(s,\nu_v, \nu'_v) \, d\nu_v \, d\nu_v'.
\end{align*}
Finally, since 
\[
 \cF_{1,\nu_v,\nu_v'}(\iota(g_{1,v},1),0) = \langle \omega_{\psi, v}(g_{1,v}) \varphi_{v, \nu_v}, \varphi_{v, \nu_v'} \rangle_v, 
\]
we have
\[
 \int_{F_v^\times} \int_{F_v^\times} |\nu_v \nu_v'| J_v(0,\nu_v, \nu'_v) \, d\nu_v \, d\nu_v' = Z_v(\varphi_v, f_v).
\]
This completes the proof.
\end{proof}

\subsection{An explicit version}
\label{s:rallis-explicit}

From now on, we assume that $F=\Q$ and $\psi$ is the standard additive character of $\A/\Q$, i.e.~$\psi_\infty(x) = e^{2 \pi \sqrt{-1} x}$.
We also use the settings of \S \ref{ss:choices} and \S \ref{ss:choices-2}.
In particular, we have $\varepsilon_v(\pi_E \times \mu) = \varepsilon_v(B)$ for all $v$.
Let $\varphi = \otimes_v \varphi_v \in S(B(\A) \times \A^\times) = S(\A_K^2 \times \A^\times)$ be the Schwartz function as in \S \ref{ss:adelic_period}.
Let $f_0^\mu = \otimes_v f_v \in \pi \boxtimes \mu$ be the automorphic form on $\bG(\A)$ as in \S \ref{ss:adelic_period}, so that 
\begin{itemize}
\item $(\pi_\infty \boxtimes \mu_\infty)([k,z]) f_\infty = e^{-\sqrt{-1}(2n+4) \theta} f_\infty$ for all $k = \smat{\cos \theta}{\sin \theta}{-\sin \theta}{\cos \theta} \in \SO(2)$ and $z \in E_\infty^\times$;
\item if $v \notin \Sigma \cup \{ \infty \}$, then $f_v$ is $\bK_v$-fixed, where
\[
 \bK_v = \{ [k,z] \mid k \in \GL_2(\Z_v), \; z \in \fo_{E_v}^\times \};
\]
\item if $v \in \Sigma$, then $f_v$ is $\mathbf{I}_v$-fixed, where
\[
 \mathbf{I}_v = \left\{ [k,z] \in \bK_v \; \middle| \; k \equiv \mat{*}{*}{0}{*} \bmod v \Z_v \right\}.
\]
\end{itemize}
We write the Tamagawa measure on $\A^\times$ as $d\nu = \prod_v d\nu_v$, where $d\nu_v$ is the standard measure on $\Q_v^\times$ as in \S \ref{ss:wald-explicit}.
We also define the standard measure $dg_{1,v}$ on $\bG_{1,v}$ as follows:
\begin{itemize} 
\item if $v = \infty$, then
\[
 dg_{1,v} = \frac{dx \, dy}{y^2} \, d\kappa, \quad
 g_{1,v} = \left[ \mat{1}{x}{0}{1} \mat{y^{1/2}}{0}{0}{y^{-1/2}}, 1 \right] \cdot \kappa
\]
for $x \in \R$, $y \in \R_+^\times$, and $\kappa \in \bK_{1,v}$ with
\[
 \bK_{1,v} = \{ [k,z] \mid k \in \SO(2), \; z \in \C^1 \},
\]
where $dx, dy$ are the Lebesgue measures and $d\kappa$ is the Haar measure on $\bK_{1,v}$ such that $\vol(\bK_{1,v}) = 1$;
\item if $v \ne \infty$, then $dg_{1,v}$ is the Haar measure on $\bG_{1,v}$ such that $\vol(\bK_{1,v}) = 1$, where $\bK_{1,v} = \bK_v \cap \bG_{1,v}$.
\end{itemize}
Then the Tamagawa measure $dg_1$ on $\bG_1(\A)$ is given by the following.

\begin{lem}
\label{l:tamagawa_bG1}
We have
\[
 dg_1 = 2 \pi^2 \cdot C_E^{-1} \cdot \zeta(2)^{-1} \cdot \prod_v dg_{1,v},
\]
where $C_E = |D_E|^{1/2} \cdot L_{\fin}(1,\xi_{E/\Q})$.
\end{lem}

\begin{proof}
We regard $\bG_1$ as a subgroup of $\GL_2(E)$ via the map $[g,z] \mapsto z^{-1} g$.
Let $\mathfrak{G}_1$ be the Lie algebra of $\bG_1$, so that 
\[
 \mathfrak{G}_1 = \left\{ \mat{a}{b}{c}{-\bar{a}} \; \middle| \; a \in E, \; b, c \in \Q \right\}.
\]
Let $\mathfrak{L}$ be the lattice in $\mathfrak{G}_1$ spanned by 
\[
 \mat{1}{0}{0}{-1}, \quad
 \mat{0}{1}{0}{0}, \quad
 \mat{0}{0}{1}{0}, \quad
 \mat{a_0}{0}{0}{-\bar{a}_0}
\]
with $a_0 = (1+\bi)/2$ (so that $\fo_E = \Z + \Z a_0$).
For each place $v$ of $\Q$, let $dg_{1,v}^{\Tam}$ be the Haar measure on $\bG_{1,v}$ induced by the lattice $\wedge^4 \mathfrak{L}$ in $\wedge^4 \mathfrak{G}_1$ and the additive character $\psi_v$ of $\Q_v$.
Then we have
\[
 dg_1 = L(1,\xi_{E/\Q})^{-1} \cdot \prod_v L_v(1,\xi_{E/\Q}) \, dg_{1,v}^{\Tam}.
\]
On the other hand, $dg_{1,\infty}$ is the Haar measure on $\bG_{1,\infty}$ induced by the lattice $\wedge^4 \tilde{\mathfrak{L}}$ in $\wedge^4 \mathfrak{G}_{1,\infty}$ and the additive character $\psi_\infty$ of $\R$, where $\tilde{\mathfrak{L}}$ is the lattice in $\mathfrak{G}_{1,\infty}$ spanned by
\[
 \mat{0}{1}{0}{0}, \quad 
 \frac{1}{2} \cdot \mat{1}{0}{0}{-1}, \quad
 2 \pi \cdot \mat{0}{1}{-1}{0}, \quad
 \pi \cdot \mat{\sqrt{-1}}{0}{0}{\sqrt{-1}}.
\]
From this, we deduce that
\[
 dg_{1,\infty}^{\Tam} = 2 \pi^2 |u|^{-1/2} \cdot dg_{1,\infty}.
\]
Also, if $v \ne \infty$, then we have
\[
 dg_{1,v}^{\Tam} = \vol(\bK_{1,v}, dg_{1,v}^{\Tam}) \cdot dg_{1,v}.
\]
Since
\[
 \mathfrak{L} \otimes_\Z \Z_v = \left\{ \mat{a}{b}{c}{-\bar{a}} \; \middle| \; a \in \fo_{E_v}, \; b, c \in \Z_v \right\},
\]
we have
\[
 \vol(\bK_{1,v}, dg_{1,v}^{\Tam}) = L_v(1,\xi_{E/\Q})^{-1} \cdot \zeta_v(2)^{-1}.
\]
Hence we have
\[
 dg_1 = 2 \pi^2 |u|^{-1/2} \cdot L_{\fin}(1, \xi_{E/\Q})^{-1} \cdot \zeta(2)^{-1} \cdot \prod_v dg_{1,v}.
\]
This proves the lemma, noting that $u = D_E$.
\end{proof}

Now we state the main result in this section.

\begin{prop}
\label{p:rallis-explicit}
We have
\[
 \langle \theta_\varphi(f_0^\mu), \theta_\varphi(f_0^\mu) \rangle = 2^{2n-4} \pi^{-1} \cdot |D_E|^{3/2} C_E^{-2} \cdot \zeta(2)^{-2} \cdot C_{n,\Sigma^+,\Sigma^-} \cdot L_{\fin}(\tfrac{1}{2}, \pi_E \times \mu) \cdot \langle f_0^\mu, f_0^\mu \rangle,
\]
where
\[
 C_{n,\Sigma^+,\Sigma^-} = \frac{n!(n+1)!^2 (n+2)!}{(2n+3)!} \cdot \prod_{q \in \Sigma^+} \frac{1}{(q+1)^2} \cdot \prod_{q \in \Sigma^-} \frac{2}{q^2-1}.
\]
\end{prop}

\begin{proof}
By Proposition \ref{p:rallis-abstract} and Lemma \ref{l:tamagawa_bG1}, we have
\begin{align*}
 \frac{\langle \theta_\varphi(f_0^\mu), \theta_\varphi(f_0^\mu) \rangle}{\langle f_0^\mu, f_0^\mu \rangle} & = \pi^2 \cdot |D_E|^{1/2} C_E^{-2} \cdot \zeta(2)^{-2} \cdot L_\fin(\tfrac{1}{2}, \pi_E \times \mu) \\
 & \quad \times Z_\infty(\varphi_\infty, f_\infty) \cdot \prod_{v \ne \infty} \left( \frac{L_v(\frac{1}{2}, \pi_E \times \mu)}{L_v(1, \xi_{E/\Q}) \zeta_v(2)} \right)^{-1} Z_v(\varphi_v, f_v).
\end{align*}
Hence the proposition follows from Lemmas \ref{l:rallis-ps}, \ref{l:rallis-dyadic}, \ref{l:rallis-st}, \ref{l:rallis-1dim} and Proposition \ref{p:rallis-real} below, where we compute $Z_v(\varphi_v, f_v)$ explicitly.
\end{proof}

\section{Local zeta integrals}
\label{s:local_zeta}

To finish the proof of Proposition \ref{p:rallis-explicit}, we need to compute the local zeta integrals explicitly.
We fix a place $v$ of $F$ and omit the subscript $v$ from the notation.
When $F$ is non-archimedean, we denote by $q$ the cardinality of the residue field of $F$ and by $\zeta(s) = (1-q^{-s})^{-1}$ the local zeta function of $F$.
For a Schwartz function $\varphi$ on $B \times F^\times$ and a matrix coefficient $\Psi$ of $\pi \boxtimes \mu$, we will compute the local zeta integral
\[
 Z(\varphi, \Psi) = \int_{F^\times} \int_{F^\times} \int_{\bG_1} |\nu \nu'| \langle \omega_{\psi}(g_1) \varphi_\nu, \varphi_{\nu'} \rangle \Psi(d(\nu') g_1 d(\nu)^{-1}) \, dg_1 \, d \nu \, d \nu',
\]
where $\varphi_\nu(x) = \varphi(x,\nu)$.

\subsection{The principal series case}
\label{ss:rallis-ps}

We use the setting of \S \ref{ss:choices-split-1} or \S \ref{ss:choices-inert-1} or \S \ref{ss:choices-ramified} or \S \ref{ss:choices-dyadic}.
In particular, $B$ is split and $\varphi$ is the characteristic function of $\fo_B \times \fo_F^\times$ (resp.~$\fo_B k_0 \times 2^{-1} \fo_F^\times$) if $2 \in \fo_F^\times$ (resp.~$2 \in \fp_F$), where $\fo_B$ is the maximal order in $B$ and $k_0$ is the element in $B^\times$ as before.
Note that $\vol(\fo_B) = 1$, where we take the Haar measure on $B$ as in \S \ref{ss:rallis-abstract}.
Let $\bK$ be a maximal compact subgroup of $\bG$ given by 
\[
 \bK = \{ [k,z] \mid k \in \GL_2(\fo_F), \; z \in \fo_E^\times \}
\]
and put $\bK_1 = \bK \cap \bG_1$.
We take the standard Haar measures $dg_1$ and $d\nu$ on $\bG_1$ and $F^\times$, respectively, so that $\vol(\bK_1) = \vol(\fo_F^\times) = 1$.
Let $\pi = \Ind(\xi_1 \otimes \xi_2)$ be a principal series representation of $\GL_2(F)$, where $\xi_1$ and $\xi_2$ are unramified characters of $F^\times$, and $\mu$ an unramified character of $E^\times$ such that $\xi_1 \cdot \xi_2 \cdot \mu|_{F^\times} = 1$.
Let $\Psi$ be the $\bK$-bi-invariant matrix coefficient of $\pi \boxtimes \mu$ such that $\Psi(1) = 1$.

\begin{lem}
\label{l:rallis-ps}
If $2 \in \fo_F^\times$, then we have
\[
 Z(\varphi, \Psi) = \frac{L(\frac{1}{2}, \pi_E \times \mu)}{L(1, \xi_{E/F}) \zeta(2)}.
\]
\end{lem}

\begin{proof}
Define a $\bK_1$-bi-invariant function $\Phi$ on $\bG_1$ by $\Phi(g_1) = \langle \omega_{\psi}(g_1) \varphi_1, \varphi_1 \rangle$.
Then we have
\[
 Z(\varphi, \Psi) = \int_{\bG_1} \Phi(g_1) \Psi(g_1) \, dg_1.
\]
We also recall Macdonald's formula
\[
 \Psi \left( \left[ \mat{\varpi_F^i}{0}{0}{1}, 1 \right] \right) = 
 \frac{q^{-i/2}}{1+q^{-1}}
 \left( \alpha_1^i \frac{1 - \alpha_1^{-1} \alpha_2 q^{-1}}{1 - \alpha_1^{-1} \alpha_2} + \alpha_2^i \frac{1 - \alpha_1 \alpha_2^{-1} q^{-1}}{1 - \alpha_1 \alpha_2^{-1}} \right)
\]
for $i \ge 0$, where $\alpha_1 = \xi_1(\varpi_F)$ and $\alpha_2 = \xi_2(\varpi_F)$.

First assume that $E$ is split.
Fix $u' \in \fo_F^\times$ such that $(u')^2 = u$.
We identify $E$ with $F \oplus F$ via the map
\[
 a + b \bi \mapsto (a+bu', a-bu').
\]
Put $\beta_1 = \mu(\varpi_F, 1)$ and $\beta_2 = \mu(1, \varpi_F)$, so that $\alpha_1 \alpha_2 \beta_1 \beta_2 = 1$.
Put 
\[
 g_{i,j} = \left[ \mat{\varpi_F^{i+j}}{0}{0}{1}, (\varpi_F^i, \varpi_F^j) \right] \in \bG_1.
\]
Then we can take $\{ g_{i,j} \mid i, j \in \Z, \; i+j \ge 0 \}$ as a set of representatives for $\bK_1 \backslash \bG_1 / \bK_1$, so that 
\[
 Z(\varphi,\Psi) = \sum_{k=0}^\infty Z_k
\]
with 
\[
 Z_k = \sum_{\substack{i,j \in \Z \\ i+j=k}} \Phi(g_{i,j}) \Psi(g_{i,j}) \vol(\bK_1 g_{i,j} \bK_1).
\]
By a direct computation, we have
\[
 \Phi(g_{i,j}) = q^{-i-j} \times
 \begin{cases}
  1 & \text{if $i, j \ge 0$;} \\
  q^{2i} & \text{if $i < 0$ and $j \ge 0$;} \\
  q^{2j} & \text{if $i \ge 0$ and $j < 0$,}
 \end{cases}
\]
noting that $\fo_B = \fo_E + \fo_E \bj$.
We also have
\[
 \vol(\bK_1 g_{i,j} \bK_1) =
\begin{cases}
 1 & \text{if $i+j=0$;} \\
 q^{i+j} (1+q^{-1}) & \text{if $i+j \ge 1$.}
\end{cases}
\]
Hence we have
\begin{align*}
 Z_0 & = 1 + \sum_{i=1}^\infty q^{-2i} \beta_1^i \beta_2^{-i} + \sum_{i=1}^\infty q^{-2i} \beta_1^{-i} \beta_2^i \\
 & = 1 + \frac{\beta_1 \beta_2^{-1} q^{-2}}{1 - \beta_1 \beta_2^{-1} q^{-2}} + \frac{\beta_1^{-1} \beta_2 q^{-2}}{1 - \beta_1^{-1} \beta_2 q^{-2}} \\
 & = \frac{1-q^{-4}}{(1 - \beta_1 \beta_2^{-1} q^{-2})(1 - \beta_1^{-1} \beta_2 q^{-2})}
\end{align*}
and 
\begin{align*}
 Z_k & = q^{-k/2} \left( \alpha_1^k \frac{1 - \alpha_1^{-1} \alpha_2 q^{-1}}{1 - \alpha_1^{-1} \alpha_2} + \alpha_2^k \frac{1 - \alpha_1 \alpha_2^{-1} q^{-1}}{1 - \alpha_1 \alpha_2^{-1}} \right) \\
 & \quad \times \left( \sum_{i=0}^k \beta_1^{k-i} \beta_2^i + \sum_{i=1}^\infty q^{-2i} \beta_1^{k+i} \beta_2^{-i} + \sum_{i=1}^\infty q^{-2i} \beta_1^{-i} \beta_2^{k+i} \right) \\
 & = q^{-k/2} \frac{\alpha_1^k (\alpha_1 - \alpha_2 q^{-1}) - \alpha_2^k (\alpha_2 - \alpha_1 q^{-1})}{\alpha_1 - \alpha_2} \\
 & \quad \times \left( \frac{\beta_1^{k+1} - \beta_2^{k+1}}{\beta_1 - \beta_2} + \frac{\beta_1^{k+1} \beta_2^{-1} q^{-2}}{1 - \beta_1 \beta_2^{-1} q^{-2}} + \frac{\beta_1^{-1} \beta_2^{k+1} q^{-2}}{1 - \beta_1^{-1} \beta_2 q^{-2}} \right) \\
 & = q^{-k/2} \frac{\alpha_1^k (\alpha_1 - \alpha_2 q^{-1}) - \alpha_2^k (\alpha_2 - \alpha_1 q^{-1})}{\alpha_1 - \alpha_2} \\
 & \quad \times (1-q^{-2}) \frac{\beta_1^k (\beta_1 - \beta_2 q^{-2}) - \beta_2^k (\beta_2 - \beta_1 q^{-2})}{(\beta_1 - \beta_2) (1 - \beta_1 \beta_2^{-1} q^{-2}) (1 - \beta_1^{-1} \beta_2 q^{-2})}
\end{align*}
for $k \ge 1$.
Thus we obtain 
\begin{align*}
 Z(\varphi,\Psi) & = \frac{1-q^{-4}}{(1 - \beta_1 \beta_2^{-1} q^{-2})(1 - \beta_1^{-1} \beta_2 q^{-2})} \\
 & + \frac{1-q^{-2}}{(\alpha_1 - \alpha_2) (\beta_1 - \beta_2) (1 - \beta_1 \beta_2^{-1} q^{-2}) (1 - \beta_1^{-1} \beta_2 q^{-2})} \\
 & \quad \times \Bigg( \frac{\alpha_1 \beta_1 q^{-1/2} (\alpha_1 - \alpha_2 q^{-1}) (\beta_1 - \beta_2 q^{-2})}{1 - \alpha_1 \beta_1 q^{-1/2}} - \frac{\alpha_1 \beta_2 q^{-1/2} (\alpha_1 - \alpha_2 q^{-1}) (\beta_2 - \beta_1 q^{-2})}{1 - \alpha_1 \beta_2 q^{-1/2}} \\
 & \qquad - \frac{\alpha_2 \beta_1 q^{-1/2} (\alpha_2 - \alpha_1 q^{-1}) (\beta_1 - \beta_2 q^{-2})}{1 - \alpha_2 \beta_1 q^{-1/2}} + \frac{\alpha_2 \beta_2 q^{-1/2} (\alpha_2 - \alpha_1 q^{-1}) (\beta_2 - \beta_1 q^{-2})}{1 - \alpha_2 \beta_2 q^{-1/2}} \Bigg) \\
 & = \frac{(1 - q^{-1})(1 - q^{-2})}{(1 - \alpha_1 \beta_1 q^{-1/2})(1 - \alpha_1 \beta_2 q^{-1/2})(1 - \alpha_2 \beta_1 q^{-1/2})(1 - \alpha_2 \beta_2 q^{-1/2})}.
\end{align*}
This yields the desired identity.

Next assume that $E$ is inert.
Put $\beta = \mu(\varpi_F)$, so that $\alpha_1 \alpha_2 \beta = 1$.
Put
\[
 g_i = \left[ \mat{\varpi_F^{2i}}{0}{0}{1}, \varpi_F^i \right] \in \bG_1.
\]
Then we can take $\{ g_i \mid i \ge 0 \}$ as a set of representatives for $\bK_1 \backslash \bG_1 / \bK_1$, so that
\[
 Z(\varphi,\Psi) = \sum_{i=0}^\infty \Phi(g_i) \Psi(g_i) \vol(\bK_1 g_i \bK_1).
\]
By a direct computation, we have
\[
 \Phi(g_i) = q^{-2i}
\]
for $i \ge 0$, noting that $\fo_B = \fo_E + \fo_E \bj$.
We also have
\[
 \vol(\bK_1 g_i \bK_1) =
\begin{cases}
 1 & \text{if $i=0$;} \\
 q^{2i} (1+q^{-1}) & \text{if $i \ge 1$.}
\end{cases}
\]
Hence we have
\begin{align*}
 Z(\varphi,\Psi) & = 1 + \sum_{i=1}^\infty q^{-i}
 \left( \alpha_1^{2i} \frac{1 - \alpha_1^{-1} \alpha_2 q^{-1}}{1 - \alpha_1^{-1} \alpha_2} + \alpha_2^{2i} \frac{1 - \alpha_1 \alpha_2^{-1} q^{-1}}{1 - \alpha_1 \alpha_2^{-1}} \right) \beta^i \\
 & = 1 + \frac{\alpha_1^2 \beta q^{-1} (\alpha_1 - \alpha_2 q^{-1})}{(1 - \alpha_1^2 \beta q^{-1}) (\alpha_1 - \alpha_2)} + \frac{\alpha_2^2 \beta q^{-1} (\alpha_2 - \alpha_1 q^{-1})}{(1 - \alpha_2^2 \beta q^{-1}) (\alpha_2 - \alpha_1)} \\
 & = \frac{(1 + q^{-1}) (1 - q^{-2})}{(1 - \alpha_1^2 \beta q^{-1}) (1 - \alpha_2^2 \beta q^{-1})}.
\end{align*}
This yields the desired identity.

Finally, assume that $E$ is ramified.
Then we may assume that $\N_{E/F}(\varpi_E) = \varpi_F$.
Put $\beta = \mu(\varpi_E)$, so that $\alpha_1 \alpha_2 \beta^2 = 1$.
Put
\[
 g_i = \left[ \mat{\varpi_F^i}{0}{0}{1}, \varpi_E^i \right] \in \bG_1.
\]
Then we can take $\{ g_i \mid i \ge 0 \}$ as a set of representatives for $\bK_1 \backslash \bG_1 / \bK_1$, so that
\[
 Z(\varphi,\Psi) = \sum_{i=0}^\infty \Phi(g_i) \Psi(g_i) \vol(\bK_1 g_i \bK_1).
\]
By a direct computation, we have
\[
 \Phi(g_i) = q^{-i}
\]
for $i \ge 0$, noting that $\fo_B \supset \fo_E$.
We also have
\[
 \vol(\bK_1 g_i \bK_1) =
\begin{cases}
 1 & \text{if $i=0$;} \\
 q^i (1+q^{-1}) & \text{if $i \ge 1$.}
\end{cases}
\]
Hence we have
\begin{align*}
 Z(\varphi,\Psi) & = 1 + \sum_{i=1}^\infty q^{-i/2}
 \left( \alpha_1^i \frac{1 - \alpha_1^{-1} \alpha_2 q^{-1}}{1 - \alpha_1^{-1} \alpha_2} + \alpha_2^i \frac{1 - \alpha_1 \alpha_2^{-1} q^{-1}}{1 - \alpha_1 \alpha_2^{-1}} \right) \beta^i \\
 & = 1 + \frac{\alpha_1 \beta q^{-1/2} (\alpha_1 - \alpha_2 q^{-1})}{(1 - \alpha_1 \beta q^{-1/2}) (\alpha_1 - \alpha_2)} + \frac{\alpha_2 \beta q^{-1/2} (\alpha_2 - \alpha_1 q^{-1})}{(1 - \alpha_2 \beta q^{-1/2}) (\alpha_2 - \alpha_1)} \\
 & = \frac{1 - q^{-2}}{(1 - \alpha_1 \beta q^{-1/2})(1 - \alpha_2 \beta q^{-1/2})}.
\end{align*}
This yields the desired identity.
\end{proof}

\begin{lem}
\label{l:rallis-dyadic}
If $2 \in \fp_F$, then we have
\[
 Z(\varphi, \Psi) = \frac{L(\frac{1}{2}, \pi_E \times \mu)}{L(1, \xi_{E/F}) \zeta(2)}.
\]
\end{lem}

\begin{proof}
Define a $\bK_1$-bi-invariant function $\Phi$ on $\bG_1$ by $\Phi(g_1) = \langle \omega_{2^{-1} \psi}(g_1) \varphi_{2^{-1}}, \varphi_{2^{-1}} \rangle$.
Then we have
\begin{align*}
 Z(\varphi, \Psi) & = |2|^{-2} \int_{\bG_1} \langle \omega_{\psi}(g_1) \varphi_{2^{-1}}, \varphi_{2^{-1}} \rangle \Psi(d(2)^{-1} g_1 d(2)) \, dg_1 \\ 
 & = |2|^{-2} \int_{\bG_1} \Phi(g_1) \Psi(g_1) \, dg_1.
\end{align*}
By a direct computation, we have
\[
 \Phi(g_{i,j}) = |2|^2 q^{-i-j} \times 
 \begin{cases}
  1 & \text{if $i, j \ge 0$;} \\
  q^{2i} & \text{if $i < 0$ and $j \ge 0$;} \\
  q^{2j} & \text{if $i \ge 0$ and $j < 0$,}
 \end{cases}
\]
where $g_{i,j} \in \bG_1$ is as in the proof of Lemma \ref{l:rallis-ps}.
Hence the rest of the proof is the same as that of Lemma \ref{l:rallis-ps}.
\end{proof}

\subsection{The twisted Steinberg case I}
\label{ss:rallis-st1}

We use the setting of \S \ref{ss:choices-split-2}.
In particular, $B$ is split and $\varphi$ is the characteristic function of $\mathfrak{I}_B k_0 \times \fo_F^\times$, where $\mathfrak{I}_B$ is the Iwahori subalgebra of $B$ and $k_0$ is the element in $B^\times$ as before.
Note that $\vol(\mathfrak{I}_B) = q^{-1}$, where we take the Haar measure on $B$ as in \S \ref{ss:rallis-abstract}.
Let $\bI$ be an Iwahori subgroup of $\bG$ given by
\[
 \bI = \{ [k,z] \mid k \in I, \; z \in \fo_E^\times \}, 
\]
where
\[
 I = \left\{ k \in \GL_2(\fo_F) \; \middle| \; k \equiv \mat{*}{*}{0}{*} \bmod \fp_F \right\}, 
\]
and put $\bI_1 = \bI \cap \bG_1$.
We take the standard Haar measures $dg_1$ and $d\nu$ on $\bG_1$ and $F^\times$, respectively, so that $\vol(\bI_1) = (q+1)^{-1}$ and $\vol(\fo_F^\times) = 1$.
Let $\pi = \St \otimes \xi$ be a twisted Steinberg representation of $\GL_2(F)$, where $\xi$ is an unramified character of $F^\times$, and $\mu$ an unramified character of $E^\times$ such that $\xi^2 \cdot \mu|_{F^\times} = 1$.
Let $\Psi$ be the $\bI$-bi-invariant matrix coefficient of $\pi \boxtimes \mu$ such that $\Psi(1) = 1$.

\begin{lem}
\label{l:rallis-st}
We have
\[
 Z(\varphi,\Psi) = \frac{1}{(q+1)^2} \cdot \frac{L(\tfrac{1}{2}, \pi_E \times \mu)}{L(1, \xi_{E/F}) \zeta(2)}.
\]
\end{lem}

\begin{proof}
We identify $B$ with $\M_2(F)$ as in \S \ref{ss:choices-split-1}.
Since $E$ is split, we may identify $E$ with $F \oplus F$ as in the proof of Lemma \ref{l:rallis-ps}.
Under this identification, we have
\[
 \mathfrak{I}_B = \left\{ x \in \M_2(\fo_F) \; \middle| \; x \equiv \mat{*}{*}{0}{*} \bmod \fp_F \right\} 
\]
and 
\[
 k_0 x k_0^{-1} = \mat{x_1}{0}{0}{x_2}
\]
for $x = (x_1, x_2)  \in E$.
Define an $\bI_1$-bi-invariant function $\Phi$ on $\bG_1$ by $\Phi(g_1) = \langle \omega_{\psi}(g_1) \varphi_1, \varphi_1 \rangle$.
Then we have
\[
 Z(\varphi, \Psi) = \int_{\bG_1} \Phi(g_1) \Psi(g_1) \, dg_1.
\]
Put
\[
 g_{i,j,k,l} = \left[ \mat{\varpi_F^{i+k}}{0}{0}{\varpi_F^{-i}} w^l, (\varpi_F^{j+k}, \varpi_F^{-j}) \right] \in \bG_1,
\]
where 
\[
 w = \mat{0}{-1}{1}{0}.
\]
Then we can take $\{ g_{i,j,k,l} \mid i, j \in \Z, \; 0 \le k, l \le 1 \}$ as a set of representatives for $\bI_1 \backslash \bG_1 / \bI_1$, so that
\[
 Z(\varphi,\Psi) = \sum_{i=-\infty}^\infty \sum_{j=-\infty}^\infty \sum_{k=0}^1 \sum_{l=0}^1 \Phi(g_{i,j,k,l}) \Psi(g_{i,j,k,l}) \vol(\bI_1 g_{i,j,k,l} \bI_1).
\]
By a direct computation, we have
\begin{align*}
 \Phi(g_{i,j,0,0}) & = q^{-2i} \times
 \begin{cases}
  q^{-1} & \text{if $i \ge j \ge -i$;} \\
  q^{2i-2j-1} & \text{if $j > i \ge -j$;} \\
  q^{2i+2j-1} & \text{if $-j > i \ge j$;} \\
  q^{4i-1} & \text{if $-i > j > i$,}
 \end{cases} \\
 \Phi(g_{i,j,1,0}) & = q^{-2i-1} \times
 \begin{cases}
  q^{-1} & \text{if $i \ge j \ge -i-1$;} \\
  q^{2i-2j-1} & \text{if $j > i \ge -j-1$;} \\
  q^{2i+2j+1} & \text{if $-j-1 > i \ge j$;} \\
  q^{4i+1} & \text{if $-i-1 > j > i$,}
 \end{cases} \\
 \Phi(g_{i,j,0,1}) & = q^{-2i-1} \times
 \begin{cases}
  q^{-1} & \text{if $i \ge j \ge -i$;} \\
  q^{2i-2j} & \text{if $j > i \ge -j$;} \\
  q^{2i+2j} & \text{if $-j > i \ge j$;} \\
  q^{4i+1} & \text{if $-i > j > i$,}
 \end{cases} \\
 \Phi(g_{i,j,1,1}) & = q^{-2i-2} \times
 \begin{cases}
  q^{-1} & \text{if $i \ge j \ge -i-1$;} \\
  q^{2i-2j} & \text{if $j > i \ge -j-1$;} \\
  q^{2i+2j+2} & \text{if $-j-1 > i \ge j$;} \\
  q^{4i+3} & \text{if $-i-1 > j > i$,}
 \end{cases}
\end{align*}
noting that $\omega_\psi(w) \varphi_1$ is $q^{-1}$ times the characteristic function of
\[
 \left\{ \mat{x_1}{x_2}{x_3}{x_4} k_0 \; \middle| \; x_1, x_3, x_4 \in \fo_F, \; x_2 \in \fp_F^{-1} \right\}.
\]
As in \cite[\S 6.6.3]{periods1}, we also have
\[
 \Psi(g_{i,j,k,l}) \vol(\bI_1 g_{i,j,k,l} \bI_1) = \frac{1}{q+1} \times
 \begin{cases}
  \alpha^{2j} & \text{if $k=0$, $l=0$;} \\
  -\alpha^{2j} & \text{if $k=0$, $l=1$;} \\
  \alpha^{2j+1} & \text{if $k=1$, $l=0$;} \\
  - \alpha^{2j+1} & \text{if $k=1$, $l=1$,}
 \end{cases}
\]
where $\alpha = \xi(\varpi_F) \cdot \mu(\varpi_F, 1)$ (so that $\alpha^2 = \mu(\varpi_F, \varpi_F^{-1})$).
Hence we have 
\[
 Z(\varphi, \Psi) = \frac{1}{q+1} \times (Z_{0,0} - Z_{0,1} + Z_{1,0} - Z_{1,1})
\]
with 
\begin{align*}
 Z_{0,0} & = \sum_{i \ge j \ge -i} q^{-2i-1} \alpha^{2j} + \sum_{j > i \ge -j} q^{-2j-1} \alpha^{2j} + \sum_{j > i \ge -j} q^{-2j-1} \alpha^{-2j} + \sum_{i > j > -i} q^{-2i-1} \alpha^{2j}, \\
 Z_{0,1} & = \sum_{i \ge j \ge -i} q^{-2i-2} \alpha^{2j} + \sum_{j > i \ge -j} q^{-2j-1} \alpha^{2j} + \sum_{j > i \ge -j} q^{-2j-1} \alpha^{-2j} + \sum_{i > j > -i} q^{-2i} \alpha^{2j}, \\
 Z_{1,0} & = \sum_{i \ge j \ge -i-1} q^{-2i-2} \alpha^{2j+1} + \sum_{j > i \ge -j-1} q^{-2j-2} \alpha^{2j+1} + \sum_{j-1 > i \ge -j} q^{-2j} \alpha^{-2j+1} + \sum_{i-1 > j > -i} q^{-2i} \alpha^{2j+1}, \\
 Z_{1,1} & = \sum_{i \ge j \ge -i-1} q^{-2i-3} \alpha^{2j+1} + \sum_{j > i \ge -j-1} q^{-2j-2} \alpha^{2j+1} + \sum_{j-1 > i \ge -j} q^{-2j} \alpha^{-2j+1} + \sum_{i-1 > j > -i} q^{-2i+1} \alpha^{2j+1}.
\end{align*}
Thus we obtain
\begin{align*}
 Z_{0,0} - Z_{0,1} & = (1-q^{-1}) \sum_{i \ge j \ge -i} q^{-2i-1} \alpha^{2j} - (1-q^{-1}) \sum_{i > j > -i} q^{-2i} \alpha^{2j} \\
 & = (1-q^{-1}) \left( \sum_{i=0}^\infty q^{-2i-1} \frac{\alpha^{2i+1} - \alpha^{-2i-1}}{\alpha - \alpha^{-1}} - \sum_{i=1}^\infty q^{-2i} \frac{\alpha^{2i-1} - \alpha^{-2i+1}}{\alpha - \alpha^{-1}} \right) \\
 & = \frac{1-q^{-1}}{\alpha - \alpha^{-1}} \left( \frac{\alpha q^{-1}}{1 - \alpha^2 q^{-2}} - \frac{\alpha^{-1} q^{-1}}{1 - \alpha^{-2} q^{-2}} - \frac{\alpha q^{-2}}{1 - \alpha^2 q^{-2}} + \frac{\alpha^{-1} q^{-2}}{1 - \alpha^{-2} q^{-2}} \right) \\
 & = \frac{q^{-1}(1-q^{-1})^2(1+q^{-2})}{(1 - \alpha^2 q^{-2}) (1 - \alpha^{-2} q^{-2})}, \\
 Z_{1,0} - Z_{1,1} & = (1-q^{-1}) \sum_{i \ge j \ge -i-1} q^{-2i-2} \alpha^{2j+1} - (1-q^{-1}) \sum_{i-1 > j > -i} q^{-2i+1} \alpha^{2j+1} \\
 & = (1-q^{-1}) \left( \sum_{i=0}^\infty q^{-2i-2} \frac{\alpha^{2i+2} - \alpha^{-2i-2}}{\alpha - \alpha^{-1}} - \sum_{i=2}^\infty q^{-2i+1} \frac{\alpha^{2i-2} - \alpha^{-2i+2}}{\alpha - \alpha^{-1}} \right) \\
 & = \frac{1-q^{-1}}{\alpha - \alpha^{-1}} \left( \frac{\alpha^2 q^{-2}}{1 - \alpha^2 q^{-2}} - \frac{\alpha^{-2} q^{-2}}{1 - \alpha^{-2} q^{-2}} - \frac{\alpha^2 q^{-3}}{1 - \alpha^2 q^{-2}} + \frac{\alpha^{-2} q^{-3}}{1 - \alpha^{-2} q^{-2}} \right) \\
 & = \frac{(\alpha + \alpha^{-1})q^{-2}(1-q^{-1})^2}{(1 - \alpha^2 q^{-2}) (1 - \alpha^{-2} q^{-2})},
\end{align*}
so that
\[
 Z(\varphi, \Psi) = \frac{q^{-2} (1-q^{-1})^2}{(1+q^{-1}) (1 - \alpha q^{-1}) (1 - \alpha^{-1} q^{-1})}.
\]
This yields the desired identity.
\end{proof}

\subsection{The twisted Steinberg case II}
\label{ss:rallis-st2}

We use the setting of \S \ref{ss:choices-inert-2}.
In particular, $B$ is ramified and $\varphi$ is the characteristic function of $\fo_B \times \fo_F^\times$, where $\fo_B$ is the maximal order in $B$.
Note that $\vol(\fo_B) = q^{-1}$, where we take the Haar measure on $B$ as in \S \ref{ss:rallis-abstract}.
Let $\bI$ be an Iwahori subgroup of $\bG$ given by 
\[
 \bI = \{ [k,z] \mid k \in I, \; z \in \fo_E^\times \},
\]
where $I$ is as in \S \ref{ss:rallis-st1}, and put $\bI_1 = \bI \cap \bG_1$.
We take the standard Haar measures $dg_1$ and $d\nu$ on $\bG_1$ and $F^\times$, respectively, so that $\vol(\bI_1) = (q+1)^{-1}$ and $\vol(\fo_F^\times) = 1$.
Let $\pi = \St \otimes \xi$ be a twisted Steinberg representation of $\GL_2(F)$, where $\xi$ is an unramified character of $F^\times$, and $\mu$ an unramified character of $E^\times$ such that $\xi^2 \cdot \mu|_{F^\times} = 1$.
Let $\Psi$ be the $\bI$-bi-invariant matrix coefficient of $\pi \boxtimes \mu$ such that $\Psi(1) = 1$.

\begin{lem}
\label{l:rallis-1dim}
We have $Z(\varphi, \Psi) = 0$ if $\varepsilon(\pi_E \times \mu) = 1$ and 
\[
 Z(\varphi, \Psi) = \frac{2}{q^2-1} \cdot \frac{L(\tfrac{1}{2}, \pi_E \times \mu)}{L(1, \xi_{E/F}) \zeta(2)}
\]
if $\varepsilon(\pi_E \times \mu) = -1$.
\end{lem}

\begin{proof}
Since $E$ is ramified, we may assume that $\N_{E/F}(\varpi_E) = \varpi_F$.
Define an $\bI_1$-bi-invariant function $\Phi$ on $\bG_1$ by $\Phi(g_1) = \langle \omega_{\psi}(g_1) \varphi_1, \varphi_1 \rangle$.
Then we have
\[
 Z(\varphi, \Psi) = \int_{\bG_1} \Phi(g_1) \Psi(g_1) \, dg_1.
\]
Put
\[
 g_{i,k,l} = \left[ \mat{\varpi_F^{i+k}}{0}{0}{\varpi_F^{-i}} w^l, \varpi_E^k \right] \in \bG_1,
\]
where 
\[
 w = \mat{0}{-1}{1}{0}.
\]
Then we can take $\{ g_{i,k,l} \mid i \in \Z, \; 0 \le k, l \le 1 \}$ as a set of representatives for $\bI_1 \backslash \bG_1 / \bI_1$, so that
\[
 Z(\varphi,\Psi) = \sum_{i = - \infty}^\infty \sum_{k=0}^1 \sum_{l=0}^1 \Phi(g_{i,k,l}) \Psi(g_{i,k,l}) \vol(\bI_1 g_{i,k,l} \bI_1).
\]
By a direct computation, we have
\begin{align*}
 \Phi(g_{i,k,0}) & = q^{-2i-k} \times 
  \begin{cases}
  q^{-1} & \text{if $2i + k \ge 0$;} \\
  q^{4i+2k-1} & \text{if $2i + k < 0$,}
 \end{cases} \\
 \Phi(g_{i,k,1}) & = - q^{-2i-k-1} \times 
  \begin{cases}
  q^{-1} & \text{if $2i + k \ge -1$;} \\
  q^{4i+2k+1} & \text{if $2i + k < -1$,}
 \end{cases} \\
\end{align*}
noting that $\fo_B = \fo_E + \fo_E \bj$ and $\omega_\psi(w) \varphi_1 = -q^{-1} \cdot \mathbb{I}_{\varpi_E^{-1} \fo_B}$.
As in \cite[\S 6.6.3]{periods1}, we also have
\[
 \Phi(g_{i,k,l}) \vol(\bI_1 g_{i,k,l} \bI_1) = \frac{1}{q+1} \times
 \begin{cases}
  1 & \text{if $k=0$, $l=0$;} \\
  -1 & \text{if $k=0$, $l=1$;} \\
  - \varepsilon & \text{if $k=1$, $l=0$;} \\
  \varepsilon & \text{if $k=1$, $l=1$,}
 \end{cases}
\]
where $\varepsilon = \varepsilon(\pi_E \times \mu) = - \xi(\varpi_F) \cdot \mu(\varpi_E)$.
Hence we have 
\[
 Z(\varphi, \Psi) = \frac{1}{q+1} \times (Z_{0,0} + Z_{0,1} - \varepsilon Z_{1,0} - \varepsilon Z_{1,1})
\]
with 
\begin{align*}
 Z_{0,0} & = \sum_{i=0}^\infty q^{-2i-1} + \sum_{i=1}^\infty q^{-2i-1}, \\ 
 Z_{0,1} & = \sum_{i=0}^\infty q^{-2i-2} + \sum_{i=1}^\infty q^{-2i}, \\
 Z_{1,0} & = \sum_{i=0}^\infty q^{-2i-2} + \sum_{i=1}^\infty q^{-2i} = Z_{0,1}, \\
 Z_{1,1} & = \sum_{i=-1}^\infty q^{-2i-3} + \sum_{i=2}^\infty q^{-2i+1} = Z_{0,0}.
\end{align*}
Thus we obtain $Z(\varphi,\Psi) = 0$ if $\varepsilon = 1$ and 
\begin{align*}
 Z(\varphi,\Psi) & = \frac{2}{q+1} \left( \frac{q^{-1}}{1-q^{-2}} + \frac{q^{-3}}{1-q^{-2}} + \frac{q^{-2}}{1-q^{-2}} + \frac{q^{-2}}{1-q^{-2}} \right) \\
 & = \frac{2 q^{-2}}{1-q^{-1}}
\end{align*}
if $\varepsilon = -1$.
This yields the desired identity.
\end{proof}

\subsection{The real case}

Suppose that $F=\R$, $E=\C$, and $B=\mathbb{H}$.
We take the Schwartz function $\varphi$ as in \S \ref{ss:adelic_period}.
Namely, $\varphi$ is given by 
\[
 \varphi_\nu = \nu \phi(\nu) \cdot \omega_{\nu \psi}(g_0) \varphi^\sharp_{n,\nu}
\]
with 
\[
 \varphi_{n,\nu}^\sharp = \varphi_{n,\nu} + \varphi_{n,\nu}''
\]
for $\nu \in \R^\times$, where $\varphi_{n,\nu}, \varphi_{n,\nu}'' \in \cS(\C^2)$ are as in \S \ref{ss:schwartz-derivatives} (with $t=\nu$), $\phi \in C_c^\infty(\R^\times)$ is as in \S \ref{ss:schwartz-def}, and $g_0 \in \Sp_4(\R)$ is as in \S \ref{ss:periods-adelic}.
Note that $\supp(\phi) \subset \R_+^\times$, so that $\varphi_\nu = 0$ if $\nu<0$.
Take the $\Sp_4(\R)$-invariant inner product $\langle \cdot, \cdot \rangle$ on $\cS(\C^2)$ with respect to the Haar measure on $B$ as in \S \ref{ss:rallis-abstract}, which is $4|u|$ times the Lebesgue measure on $\C^2$.
Define an embedding $\iota_0 : \bG \hookrightarrow \GSp_4(\R)$ by 
\[
 \iota_0(g) = g_0^{-1} \iota(g) g_0,
\]
where $\iota$ is as in \S \ref{ss:periods-adelic}.
Then we have 
\[
 \iota_0 \left( \left[ \mat{a}{b}{c}{d}, r e^{\sqrt{-1} \theta} \right] \right) = 
 \begin{pmatrix}
  a & 0 & b & 0 \\
  0 & a & 0 & b \\
  c & 0 & d & 0 \\
  0 & c & 0 & d
 \end{pmatrix} \cdot
 \frac{1}{r}
 \begin{pmatrix}
  \cos \theta & - \sin \theta & 0 & 0 \\
  \sin \theta & \cos \theta & 0 & 0 \\
  0 & 0 & \cos \theta & - \sin \theta \\
  0 & 0 & \sin \theta & \cos \theta
 \end{pmatrix}.
\]
We take the standard Haar measures $dg_1$ and $d\nu$ on $\bG_1$ and $\R^\times$, respectively.
Let $\pi$ be the discrete series representation of $\GL_2(\R)$ of weight $2n+4$ and $\mu$ the trivial character of $\C^\times$.
Let $\Psi$ be the matrix coefficient of $\pi \boxtimes \mu$ such that 
\[
 \Psi([k_1 g k_2, z]) = e^{-\sqrt{-1}(2n+4) (\theta_1 + \theta_2)} \Psi([g,1])
\]
for all $g \in \GL_2(\R)$, $k_i = \smat{\cos \theta_i}{\sin \theta_i}{-\sin \theta_i}{\cos \theta_i} \in \SO(2)$, and $z \in \C^\times$, and such that $\Psi(1) = 1$.

\begin{lem}
\label{l:rallis-real-bG1}
We have
\[
 Z(\varphi, \Psi) = \frac{1}{4 \pi^2} \int_{\bG_1} \langle \omega_{\psi}(\iota_0(g_1)) \varphi^\sharp_{n,1}, \varphi^\sharp_{n,1} \rangle \Psi(g_1) \, dg_1.
\]
\end{lem}

\begin{proof}
By Lemma \ref{l:weil-seesaw}, we have
\[
 Z(\varphi, \Psi) = \int_{\R_+^\times} \int_{\R_+^\times} \int_{\bG_1} \nu \nu' \langle \omega_{\psi}(\iota(g_1)) \varphi_\nu, \varphi_{\nu'} \rangle \Psi(d(\nu') g_1 d(\nu)^{-1}) \, dg_1 \, d \nu \, d \nu',
\]
noting that $\varphi_\nu = 0$ if $\nu<0$.
Put
\begin{align*}
 \mathbf{t}(\nu) & = \left[ \mat{\nu^{1/2}}{0}{0}{\nu^{-1/2}}, 1 \right] \in \bG_1, \\
 t(\nu) & = \iota(\mathbf{t}(\nu)) = \diag(\nu^{1/2}, \nu^{1/2}, \nu^{-1/2}, \nu^{-1/2}) \in \Sp_4(\R)
\end{align*}
for $\nu \in \R_+^\times$.
Since 
\[
 \Psi(d(\nu') g_1 d(\nu)^{-1}) = \Psi(\mathbf{t}(\nu')^{-1} g_1 \mathbf{t}(\nu)),
\]
we have
\[
 Z(\varphi, \Psi) = \int_{\R_+^\times} \int_{\R_+^\times} \int_{\bG_1} \nu \nu' \langle \omega_{\psi}(t(\nu') \iota(g_1) t(\nu)^{-1}) \varphi_\nu, \varphi_{\nu'} \rangle \Psi(g_1) \, dg_1 \, d \nu \, d \nu'.
\]
Moreover, since 
\[
 \varphi_\nu = \nu \phi(\nu) \cdot \omega_\psi(t(\nu) g_0 t(\nu)^{-1}) \varphi^\sharp_{n,\nu}, 
\]
we have
\begin{align*}
 Z(\varphi, \Psi) & = \int_{\R_+^\times} \int_{\R_+^\times} \int_{\bG_1} \nu^2 \nu'^2 \phi(\nu) \phi(\nu') \langle \omega_{\psi}(\iota(g_1) g_0 t(\nu)^{-1}) \varphi^\sharp_{n,\nu}, \omega_\psi(g_0 t(\nu')^{-1}) \varphi^\sharp_{n,\nu'} \rangle \Psi(g_1) \, dg_1 \, d \nu \, d \nu' \\
 & = \int_{\R_+^\times} \int_{\R_+^\times} \int_{\bG_1} \nu^2 \nu'^2 \phi(\nu) \phi(\nu') \langle \omega_{\psi}(\iota_0(g_1) t(\nu)^{-1}) \varphi^\sharp_{n,\nu}, \omega_\psi(t(\nu')^{-1}) \varphi^\sharp_{n,\nu'} \rangle \Psi(g_1) \, dg_1 \, d \nu \, d \nu'.
\end{align*}
Finally, since 
\[
 \omega_\psi(t(\nu)^{-1}) \varphi^\sharp_{n,\nu} = \nu^{-2} \varphi^\sharp_{n,1},
\]
we have
\[
 Z(\varphi, \Psi) = \int_{\R_+^\times} \phi(\nu) \, d \nu \cdot \int_{\R_+^\times} \phi(\nu') \, d \nu' \cdot \int_{\bG_1} \langle \omega_{\psi}(\iota_0(g_1)) \varphi^\sharp_{n,1}, \varphi^\sharp_{n,1} \rangle \Psi(g_1) \, dg_1.
\]
This yields the assertion.
\end{proof}

Put $G' = \SL_2(\R)$ and define the standard measure $dg'$ on $G'$ by 
\[
 dg' = \frac{dx \, dy}{y^2} \, dk, \quad
 g' = \mat{1}{x}{0}{1} \mat{y^{1/2}}{0}{0}{y^{-1/2}} k
\]
for $x \in \R$, $y \in \R_+^\times$, and $k \in \SO(2)$, where $dx, dy$ are the Lebesgue measures and $dk$ is the Haar measure on $\SO(2)$ such that $\vol(\SO(2)) = 1$.
We regard $G'$ as a subgoup of $\bG_1$ via the map $g' \mapsto [g',1]$.

\begin{lem}
\label{l:rallis-real-SL2}
We have
\[
 Z(\varphi, \Psi) = \frac{1}{4 \pi^2} \int_{G'} \langle \omega_{\psi}(\iota_0(g')) \varphi^\sharp_{n,1}, \varphi^\sharp_{n,1} \rangle \Psi(g') \, dg'.
\]
\end{lem}

\begin{proof}
Recall that $\Psi$ is $Z_{\bG_1}$-invariant, where $Z_{\bG_1} = \{ [1, z] \mid z \in \C^1 \}$ is the center of $\bG_1$.
By Lemma \ref{l:S(C^2)_n}\eqref{S(C^2)_n-2}, the function $g_1 \mapsto \langle \omega_{\psi}(\iota_0(g_1)) \varphi^\sharp_{n,1}, \varphi^\sharp_{n,1} \rangle$ is also $Z_{\bG_1}$-invariant.
Hence the assertion follows from Lemma \ref{l:rallis-real-bG1}, notingh that $\bG_1 = Z_{\bG_1} \cdot G'$.
\end{proof} 

Now we fix $t \in \R_+^\times$ and compute $\varphi_n^\sharp = \varphi_{n,t}^\sharp$ explicitly.
We use the notation of \S \ref{s:schwartz_forms}, so that
\[
 \varphi_n^\sharp = \varphi_n + \varphi_n''.
\]
Recall the Gaussian $\varphi^0 \in \cS(\C^2)$ given by
\[
 \varphi^0(x) = e^{- 2 \pi t (x_1 \bar{x}_1 + x_2 \bar{x}_2)}
\]
for $x = (x_1,x_2) \in \C^2$.
Write 
\[
 \varphi_n = P_n \cdot \varphi^0, \quad
 \varphi_n' = P_n' \cdot \varphi^0, \quad
 \varphi_n'' = P_n'' \cdot \varphi^0, 
\]
where $P_n, P_n', P_n''$ are polynomials in $x_1, \bar{x}_1, x_2, \bar{x}_2$.
Put $\kappa = - 2 \pi t$ and $\kappa_0 = 2 \sqrt{\pi t}$, so that $\kappa_0^2 = -2 \kappa$.
Define differential operators $\delta_1, \delta'_1, \delta_2, \delta'_2$ by
\[
 \delta_i = \sqrt{-1} \kappa_0 x_i + \frac{1}{\sqrt{-1} \kappa_0} \frac{\partial}{\partial \bar{x}_i}, \quad
 \delta'_i = \sqrt{-1} \kappa_0 \bar{x}_i + \frac{1}{\sqrt{-1} \kappa_0} \frac{\partial}{\partial x_i}.
\]
Note that $\C[\delta_1, \delta'_1, \delta_2, \delta'_2]$ is a commutative subalgebra of $\C[x_1, \bar{x}_1, x_2, \bar{x}_2, \frac{\partial}{\partial x_1}, \frac{\partial}{\partial \bar{x}_1}, \frac{\partial}{\partial x_2}, \frac{\partial}{\partial \bar{x}_2}]$.

\begin{lem}
\label{l:diff-diag}
We have
\[
 P_n = \delta_1'^2 (\delta'_1 \delta_2 - \delta_1 \delta'_2)^n P^0, \quad
 P_n' = \delta_1' \delta_2' (\delta'_1 \delta_2 - \delta_1 \delta'_2)^n P^0, \quad
 P_n'' = \delta_2'^2 (\delta'_1 \delta_2 - \delta_1 \delta'_2)^n P^0,
\]
where $P^0 = \frac{1}{2 \kappa}$.
\end{lem}

\begin{proof}
Define differential operators $\nabla_1, \nabla_2, \nabla_3$ by 
\begin{align*}
 \nabla_1 & = 1 + 2 \kappa x_1 \bar{x}_1 + x_1 \frac{\partial}{\partial x_1} + \bar{x}_1 \frac{\partial}{\partial \bar{x}_1} + \frac{1}{2\kappa} \frac{\partial^2}{\partial x_1 \partial \bar{x}_1}, \\
 \nabla_2 & = 2 \kappa x_1 \bar{x}_2 + 2 \kappa \bar{x}_1 x_2 
 + x_1 \frac{\partial}{\partial x_2} + \bar{x}_1 \frac{\partial}{\partial \bar{x}_2} + x_2 \frac{\partial}{\partial x_1} + \bar{x}_2 \frac{\partial}{\partial \bar{x}_1} 
 + \frac{1}{2\kappa} \frac{\partial^2}{\partial x_1 \partial \bar{x}_2} + \frac{1}{2\kappa} \frac{\partial}{\partial \bar{x}_1 \partial x_2}, \\
 \nabla_3 & = 1 + 2 \kappa x_2 \bar{x}_2 + x_2 \frac{\partial}{\partial x_2} + \bar{x}_2 \frac{\partial}{\partial \bar{x}_2} + \frac{1}{2\kappa} \frac{\partial^2}{\partial x_2 \partial \bar{x}_2}.
\end{align*}
Then by a direct computation, we have
\[
 \omega(\mathtt{X}_{2,0}) \varphi = \nabla_1 P \cdot \varphi^0, \quad
 \omega(\mathtt{X}_{1,1}) \varphi = \nabla_2 P \cdot \varphi^0, \quad
 \omega(\mathtt{X}_{0,2}) \varphi = \nabla_3 P \cdot \varphi^0 
\]
for $\varphi = P \cdot \varphi^0$ with a polynomial $P$ in $x_1, \bar{x}_1, x_2, \bar{x}_2$.
This implies that 
\[
 \begin{pmatrix}
  P_n \\
  P'_n \\
  P''_n 
 \end{pmatrix}
 = A^n
 \begin{pmatrix}
  P_0 \\
  P'_0 \\
  P''_0 
 \end{pmatrix},
\]
where 
\[
 A = 
 \begin{pmatrix}
  \nabla_2 & -2 \nabla_1 & 0 \\
  \nabla_3 & 0 &  -\nabla_1 \\
  0 & 2 \nabla_3 & - \nabla_2
 \end{pmatrix}.
\]
On the other hand, noting that 
\[
 \frac{\partial}{\partial x_i} x_j - x_j \frac{\partial}{\partial x_i} = 
 \begin{cases}
  1 & \text{if $i = j$;} \\
  0 & \text{if $i \ne j$,}
 \end{cases}
\]
we have 
\[
 \nabla_1 = \delta_1 \delta'_1, \quad
 \nabla_2 = \delta_1 \delta'_2 + \delta'_1 \delta_2, \quad
 \nabla_3 = \delta_2 \delta'_2.
\]
Hence, putting 
\[
 Q = 
 \begin{pmatrix}
  \delta_1^2 & 2 \nabla_1 & \delta_1'^2 \\
  \delta_1 \delta_2 & \nabla_2 & \delta'_1 \delta'_2 \\
  \delta_2^2 & 2 \nabla_3 & \delta_2'^2
 \end{pmatrix},
\]
we have
\[
 AQ = Q
 \begin{pmatrix}
  \delta_1 \delta'_2 - \delta'_1 \delta_2 & 0 & 0 \\
  0 & 0 & 0 \\
  0 & 0 & \delta'_1 \delta_2 - \delta_1 \delta'_2
 \end{pmatrix}.
\]
Since
\[
 \begin{pmatrix}
  P_0 \\
  P'_0 \\
  P''_0 
 \end{pmatrix}
 = 
 \begin{pmatrix}
  \bar{x}_1^2 \\
  \bar{x}_1 \bar{x}_2 \\
  \bar{x}_2^2
 \end{pmatrix}
 = Q
 \begin{pmatrix}
  0 \\
  0 \\
  P^0 
 \end{pmatrix}, 
\]
we have
\[
 \begin{pmatrix}
  P_n \\
  P'_n \\
  P''_n 
 \end{pmatrix}
 = A^n Q
 \begin{pmatrix}
  0 \\
  0 \\
  P^0 
 \end{pmatrix} 
 = Q
 \begin{pmatrix}
  (\delta_1 \delta'_2 - \delta'_1 \delta_2)^n & 0 & 0 \\
  0 & 0 & 0 \\
  0 & 0 & (\delta'_1 \delta_2 - \delta_1 \delta'_2)^n
 \end{pmatrix}
 \begin{pmatrix}
  0 \\
  0 \\
  P^0 
 \end{pmatrix}.
\]
This completes the proof.
\end{proof}

Now we make a change of variables
\[
 y_1 = \frac{1}{\sqrt{2}} (x_1 + \sqrt{-1} x_2), \quad 
 y_2 = \frac{1}{\sqrt{2}} (x_1 - \sqrt{-1} x_2).
\]
Note that 
\[
 \varphi^0(x_1, x_2) = e^{- 2 \pi t (y_1 \bar{y}_1 + y_2 \bar{y}_2)}. 
\]

\begin{lem}
\label{l:diff-y}
We have
\[
 \delta'_1 \delta_2 - \delta_1 \delta'_2 = - \sqrt{-1} (D_1 - D_2), \quad
 \delta_1'^2 + \delta_2'^2 = 2 D_1' D_2',
\]
where 
\[
 D_i = 1 + 2 \kappa y_i \bar{y}_i + y_i \frac{\partial}{\partial y_i} + \bar{y}_i \frac{\partial}{\partial \bar{y}_i} + \frac{1}{2 \kappa} \frac{\partial^2}{\partial y_i \partial \bar{y}_i}, \quad
 D_i' = \sqrt{-1} \kappa_0 \bar{y}_i + \frac{1}{\sqrt{-1} \kappa_0} \frac{\partial}{\partial y_i}.
\]
\end{lem}

\begin{proof}
The lemma follows from a direct computation, noting that 
\begin{align*}
 x_1 & = \frac{1}{\sqrt{2}} (y_1 + y_2), &
 x_2 & = - \frac{\sqrt{-1}}{\sqrt{2}} (y_1 - y_2), \\
 \bar{x}_1 & = \frac{1}{\sqrt{2}} (\bar{y}_1 + \bar{y}_2), & 
 \bar{x}_2 & = \frac{\sqrt{-1}}{\sqrt{2}} (\bar{y}_1 - \bar{y}_2), \\
 \frac{\partial}{\partial x_1} & = \frac{1}{\sqrt{2}} \left( \frac{\partial}{\partial y_1} + \frac{\partial}{\partial y_2} \right), &
 \frac{\partial}{\partial x_2} & = \frac{\sqrt{-1}}{\sqrt{2}} \left( \frac{\partial}{\partial y_1} - \frac{\partial}{\partial y_2} \right), \\
 \frac{\partial}{\partial \bar{x}_1} & = \frac{1}{\sqrt{2}} \left( \frac{\partial}{\partial \bar{y}_1} + \frac{\partial}{\partial \bar{y}_2} \right), & 
 \frac{\partial}{\partial \bar{x}_2} & = - \frac{\sqrt{-1}}{\sqrt{2}} \left( \frac{\partial}{\partial \bar{y}_1} - \frac{\partial}{\partial \bar{y}_2} \right).
\end{align*}
\end{proof}

Using this lemma, we will separate variables.
Thus we consider polynomials in $z, \bar{z}$, where $z$ is a complex variable, and differential operators $D,D'$ given by 
\[
 D = 1 + 2 \kappa z \bar{z} + z \frac{\partial}{\partial z} + \bar{z} \frac{\partial}{\partial \bar{z}} + \frac{1}{2 \kappa} \frac{\partial^2}{\partial z \partial \bar{z}}, \quad
 D' = \sqrt{-1} \kappa_0 \bar{z} + \frac{1}{\sqrt{-1} \kappa_0} \frac{\partial}{\partial z}.
\]
We also use polar coordinates
\[
 z = r e^{\sqrt{-1} \theta}.
\]

\begin{lem}
\label{l:diff-polar}
We have
\[
 D = 1 + 2 \kappa r^2 + r \frac{\partial}{\partial r} + \frac{1}{8 \kappa} \left( \frac{\partial^2}{\partial r^2} + \frac{1}{r} \frac{\partial}{\partial r} + \frac{1}{r^2} \frac{\partial^2}{\partial \theta^2}\right).
\]
In particular, if $f$ is a polynomial of the form $f(z, \bar{z}) = g(r^2)$, where $g$ is a polynomial in one variable, then we have
\[
 Df(z,\bar{z}) = (1 + 2 \kappa r^2) g(r^2) + \left( 2r^2 + \frac{1}{2 \kappa} \right) g'(r^2) + \frac{r^2}{2 \kappa} g''(r^2).
\]
\end{lem}

\begin{proof}
The lemma follows from a direct computation, noting that 
\begin{align*}
 \frac{\partial}{\partial z} & = \frac{1}{2} (\cos \theta - \sqrt{-1} \sin \theta) \frac{\partial}{\partial r} - \frac{1}{2r} (\sin \theta + \sqrt{-1} \cos \theta) \frac{\partial}{\partial \theta}, \\
 \frac{\partial}{\partial \bar{z}} & = \frac{1}{2} (\cos \theta + \sqrt{-1} \sin \theta) \frac{\partial}{\partial r} - \frac{1}{2r} (\sin \theta - \sqrt{-1} \cos \theta) \frac{\partial}{\partial \theta}, \\
 \frac{\partial^2}{\partial z \partial \bar{z}} & = \frac{1}{4} \left( \frac{\partial^2}{\partial r^2} + \frac{1}{r} \frac{\partial}{\partial r} + \frac{1}{r^2} \frac{\partial^2}{\partial \theta^2}\right).
\end{align*}
\end{proof}

Let $L_n^{(\alpha)}$ be the generalized Laguerre polynomial given by 
\[
 L_n^{(\alpha)}(x) = \frac{x^{-\alpha} e^x}{n!} \frac{d^n}{dx^n} (x^{n+\alpha} e^{-x}).
\]
Write $L_n = L_n^{(0)}$.

\begin{lem}
\label{l:diff-laguerre}
Let $f^0 = 1$.
\begin{enumerate}
\item
\label{diff-f0-1}
We have
\[
 D^n f^0(z,\bar{z}) = n! L_n(\kappa_0^2 z \bar{z}).
\]
\item 
\label{diff-f0-2}
We have
\[
 D' D^n f^0(z,\bar{z}) = \sqrt{-1} n! \kappa_0 \bar{z} L_n^{(1)}(\kappa_0^2 z \bar{z}).
\]
\end{enumerate}
\end{lem}

\begin{proof}
Define a family of polynomials $\{ g_n \}_{n \ge 0}$ inductively by $g_0 = 1$ and 
\[
 g_{n+1}(x) = (1 - \kappa_0^2 x) g_n(x) + \left( 2x - \frac{1}{\kappa_0^2} \right) g_n'(x) - \frac{x}{\kappa_0^2} g_n''(x).
\]
To prove \eqref{diff-f0-1}, it suffices to show that
\[
 g_n(x) = n! L_n(\kappa_0^2 x)
\]
by Lemma \ref{l:diff-polar}.
We proceed by induction on $n$.
For $n=0$, this is obvious.
For general $n$, recall that 
\[
 x L_n''(x) + (1-x) L_n'(x) + nL_n(x) = 0
\]
and
\[
 x L_n'(x) = (n+1) L_{n+1}(x) - (n+1-x) L_n(x).
\]
From this and the induction hypothesis, we deduce that 
\begin{align*}
 g_{n+1}(x) & = n! ((1 - \kappa_0^2 x) L_n(\kappa_0^2 x) + (2 \kappa_0^2 x - 1) L_n'(\kappa_0^2 x) - \kappa_0^2 x L_n''(\kappa_0^2 x)) \\
 & = n! ((1 - \kappa_0^2 x) L_n(\kappa_0^2 x) + (2 \kappa_0^2 x - 1) L_n'(\kappa_0^2 x) + n L_n(\kappa_0^2 x) + (1 -\kappa_0^2 x) L_n'(\kappa_0^2 x)) \\
 & = n! ((n + 1 - \kappa_0^2 x) L_n(\kappa_0^2 x) + \kappa_0^2 x L_n'(\kappa_0^2 x)) \\
 & = (n+1)! L_{n+1}(\kappa_0^2 x).
\end{align*}
This completes the proof of \eqref{diff-f0-1}.

For \eqref{diff-f0-2}, recall that 
\[
 L_n^{(1)}(x) = - L_{n+1}'(x) = L_n(x) - L_n'(x).
\]
From this and \eqref{diff-f0-1}, we deduce that 
\begin{align*}
 D' D^n f^0(z,\bar{z}) & = n! (\sqrt{-1} \kappa_0 \bar{z} L_n(\kappa_0^2 z \bar{z}) - \sqrt{-1} \kappa_0 \bar{z} L_n'(\kappa_0^2 z \bar{z})) \\
 & = \sqrt{-1} n! \kappa_0 \bar{z} L_n^{(1)}(\kappa_0^2 z \bar{z}).
\end{align*}
This completes the proof of \eqref{diff-f0-2}.
\end{proof}

For $n \ge 0$, we define a polynomial $f_n$ by
\[
 f_n(z, \bar{z}) = \bar{z} L_n^{(1)}(\kappa_0^2 z \bar{z}).
\]

\begin{prop}
\label{p:phi-sharp}
We have $\varphi_n^\sharp = P_n^\sharp \cdot \varphi^0$, where
\[
 P_n^\sharp(y_1, \bar{y}_1, y_2, \bar{y}_2) = 2 (-\sqrt{-1})^n n! \sum_{i=0}^n (-1)^i f_{n-i}(y_1, \bar{y}_1) f_i(y_2, \bar{y}_2).
\]
\end{prop}

\begin{proof}
By Lemmas \ref{l:diff-diag} and \ref{l:diff-y}, we have
\[
 P_n^\sharp = (\delta_1'^2 + \delta_2'^2) (\delta_1' \delta_2 - \delta_1 \delta_2')^n P^0 = 2 (- \sqrt{-1})^n D_1' D_2' (D_1 - D_2)^n P^0.
\]
Hence, by Lemma \ref{l:diff-laguerre}, we have
\begin{align*}
 P_n^\sharp(y_1, \bar{y}_1, y_2, \bar{y}_2) & = \frac{(-\sqrt{-1})^n}{\kappa} \sum_{i=0}^n \binom{n}{i} (-1)^i D_1' D_1^{n-i} f^0(y_1,\bar{y}_1) D_2' D_2^i f^0(y_2,\bar{y}_2) \\
 & = \frac{(-\sqrt{-1})^n}{\kappa} \sum_{i=0}^n \binom{n}{i} (-1)^{i+1} (n-i)! i! \kappa_0^2 \bar{y}_1 \bar{y}_2 L_{n-i}^{(1)}(\kappa_0^2 y_1 \bar{y}_1) L_i^{(1)}(\kappa_0^2 y_2 \bar{y}_2) \\
 & = 2 (-\sqrt{-1})^n n! \bar{y}_1 \bar{y}_2 \sum_{i=0}^n (-1)^i L_{n-i}^{(1)}(\kappa_0^2 y_1 \bar{y}_1) L_i^{(1)}(\kappa_0^2 y_2 \bar{y}_2).
\end{align*}
This completes the proof.
\end{proof}

To compute $Z(\varphi,\Psi)$, we consider a certain projection of $\varphi_n^\sharp$.
For this, we need to introduce the Weil representation of $G' = \SL_2(\R)$.
Let $K' = \SO(2)$ be the standard maximal compact subgroup of $G'$.
Let $\fg'$ be the complexified Lie algebra of $G'$.
Take the following basis of $\fg'$:
\[
 \mathtt{H} = \mat{0}{-\sqrt{-1}}{\sqrt{-1}}{0}, \quad
 \mathtt{R} = \frac{1}{2} \cdot \mat{1}{\sqrt{-1}}{\sqrt{-1}}{-1}, \quad
 \mathtt{L} = \frac{1}{2} \cdot \mat{1}{-\sqrt{-1}}{-\sqrt{-1}}{-1}.
\]
Let $\omega$ be the Weil representation of $G'$ on $\cS(\C)$ relative to $t \psi$.
As in \S \ref{ss:schwartz-setup}, the associated action of $\fg'$ is given by
\[
 \omega \mat{1}{0}{0}{-1} = 1 + z \frac{\partial}{\partial z} + \bar{z} \frac{\partial}{\partial \bar{z}}, \quad
 \omega \mat{0}{1}{0}{0} = 2 \pi t \sqrt{-1} z \bar{z}, \quad
 \omega \mat{0}{0}{1}{0} = \frac{\sqrt{-1}}{2 \pi t} \frac{\partial^2}{\partial z \partial \bar{z}}.
\]
Take the $G'$-invariant inner product $\langle \cdot, \cdot \rangle'$ on $\cS(\C)$ with respect to twice the Lebesgue measure on $\C$.
Let $\phi^0 \in \cS(\C)$ be the Gaussian given by
\[
 \phi^0(z) = e^{- 2 \pi t z \bar{z}}.
\]
For $n \ge 0$, we define a Schwartz function $\phi_n \in \cS(\C)$ by 
\[
 \phi_n = f_n \cdot \phi^0.
\]

\begin{lem}
\label{l:phi-1var}
\begin{enumerate}
\item We have
\label{phi-1var-1}
\[
 \omega(\mathtt{H}) \phi_n = (2n+2) \phi_n, \quad
 \omega(\mathtt{R}) \phi_n = (n+1) \phi_{n+1}, \quad
 \omega(\mathtt{L}) \phi_n = - (n+1) \phi_{n-1},
\]
where we interpret $\phi_{-1} = 0$.
\item 
\label{phi-1var-2}
We have
\[
 \langle \phi_n, \phi_m \rangle' =
 \begin{cases}
  2 \pi \kappa_0^{-4} (n+1) & \text{if $n=m$;} \\
  0 & \text{if $n \ne m$.}
 \end{cases}
\]
\end{enumerate}
\end{lem}

\begin{proof}
We have 
\begin{align*}
 \omega(\mathtt{H}) \phi & = \left( 1 + z \frac{\partial}{\partial z} + \bar{z} \frac{\partial}{\partial \bar{z}} + \frac{1}{\kappa} \frac{\partial^2}{\partial z \partial \bar{z}} \right) f \cdot \phi^0, \\
 \omega(\mathtt{R}) \phi & = \left( 1 + 2 \kappa z \bar{z} + z \frac{\partial}{\partial z} + \bar{z} \frac{\partial}{\partial \bar{z}} + \frac{1}{2 \kappa} \frac{\partial^2}{\partial z \partial \bar{z}} \right) f \cdot \phi^0, \\
 \omega(\mathtt{L}) \phi & = - \frac{1}{2 \kappa} \frac{\partial^2}{\partial z \partial \bar{z}} f \cdot \phi^0
\end{align*}
for $\phi = f \cdot \phi^0$ with a polynomial $f$ in $z, \bar{z}$.
Recall also that
\[
 x {L_n^{(1)}}''(x) + (2-x) {L_n^{(1)}}'(x) + nL_n(x) = 0
\]
and 
\begin{align*}
 x {L_n^{(1)}}'(x) 
 & = (n+1) L_{n+1}^{(1)}(x) - (n+2-x) L_n^{(1)}(x) \\
 & = n L_n^{(1)}(x) - (n+1) L_{n-1}^{(1)}(x).
\end{align*}
From this, we can deduce \eqref{phi-1var-1}.

We have
\begin{align*}
 \langle \phi_n, \phi_m \rangle' & = \int_{\C} z \bar{z} L_n^{(1)}(\kappa_0^2 z \bar{z}) L_m^{(1)}(\kappa_0^2 z \bar{z}) e^{-4 \pi t z \bar{z}} \, dz \\
 & = \int_0^\infty \int_0^{2 \pi} r^2 L_n^{(1)}(\kappa_0^2 r^2) L_m^{(1)}(\kappa_0^2 r^2) e^{- \kappa_0^2 r^2} \cdot 2r \, dr \, d \theta \\
 & = 2 \pi \kappa_0^{-4} \int_0^\infty r L_n^{(1)}(r) L_m^{(1)}(r) e^{-r} \, dr.
\end{align*}
Hence \eqref{phi-1var-2} follows from the orthogonality of the generalized Laguerre polynomials.
\end{proof}

For any integer $k \ge 2$, we denote by $\HDS_k$ the underlying Harish-Chandra module of the holomorphic discrete series representation of $G'$ of weight $k$.
By Lemma \ref{l:phi-1var}\eqref{phi-1var-1}, we may realize $\HDS_2$ on the space $\cS(\C)^\natural$ spanned by $\{ \phi_i \mid i \ge 0 \}$.

Now we consider the identification
\[
 \cS(\C^2) = \cS(\C) \mathbin{\hat{\otimes}} \cS(\C)
\]
as a representation of $G' \times G'$, where we regard $G' \times G'$ as a subgroup of $\Sp_4(\R)$ via the embedding
\[
 \left( \mat{a_1}{b_1}{c_1}{d_1}, \mat{a_2}{b_2}{c_2}{d_2} \right) \mapsto 
 \begin{pmatrix}
  a_1 & 0 & b_1 & 0 \\
  0 & a_2 & 0 & b_2 \\
  c_1 & 0 & d_1 & 0 \\
  0 & c_2 & 0 & d_2
 \end{pmatrix}.
\]
Let $\langle \cdot, \cdot \rangle'$ be an $\Sp_4(\R)$-invariant inner product on $\cS(\C^2)$ given by
\[
 \langle \phi \otimes \psi, \phi' \otimes \psi' \rangle' = \langle \phi, \phi' \rangle' \cdot \langle \psi, \psi' \rangle'
\]
for $\phi, \psi, \phi', \psi' \in \cS(\C)$, so that $\langle \cdot, \cdot \rangle = |u| \langle \cdot, \cdot \rangle'$.
Define an isometry $T$ of $\cS(\C^2)$ by
\[
 T \varphi(x_1, x_2) = \varphi\left( (x_1, x_2) \mat{\frac{1}{\sqrt{2}}}{- \frac{\sqrt{-1}}{\sqrt{2}}}{\frac{1}{\sqrt{2}}}{\frac{\sqrt{-1}}{\sqrt{2}}} \right).
\]
Note that $T$ is equivariant under the action of the diagonal subgroup $G'$ of $G' \times G'$.
Then Proposition \ref{p:phi-sharp} says that
\begin{equation}
\label{eq:Tphi-sharp}
 T \varphi_n^\sharp = 2 (-\sqrt{-1})^n n! \sum_{i=0}^n (-1)^i \phi_{n-i} \otimes \phi_i. 
\end{equation}
Put 
\[
 \cS(\C^2)^\natural = \cS(\C)^\natural \otimes \cS(\C)^\natural, 
\]
which is spanned by $\{ \phi_i \otimes \phi_j \mid i, j \ge 0 \}$.
As a $(\fg',K')$-module, $\cS(\C^2)^\natural$ is isomorphic to 
\begin{equation}
\label{eq:repka}
 \HDS_2 \otimes \HDS_2 = \bigoplus_{n=0}^\infty \HDS_{2n+4} 
\end{equation}
by \cite[Theorem 7.1]{repka}.
Let $\cS(\C^2)^\natural_n$ be the subspace of $\cS(\C^2)^\natural$ corresponding to $\HDS_{2n+4}$.
Let $\pr_n : \cS(\C^2)^\natural \rightarrow \cS(\C^2)^\natural_n$ be the associated orthogonal projection.
Define a Schwartz function $\varphi_n^\natural \in \cS(\C^2)^\natural$ by 
\[
 \varphi_n^\natural = \sum_{i=0}^n (-1)^i \binom{n+2}{i+1} \phi_{n-i} \otimes \phi_i.
\]

\begin{lem}
\label{l:schwartz-lowest-wt}
We have
\[
 \omega(\Delta \mathtt{H}) \varphi_n^{\natural} = (2n+4) \varphi_n^{\natural}, \quad 
 \omega(\Delta \mathtt{L}) \varphi_n^{\natural} = 0, 
\]
where $\Delta : \fg' \rightarrow \fg' \oplus \fg'$ is the diagonal embedding.
In particular, $\varphi_n^\natural$ is a lowest weight vector in $\cS(\C^2)^\natural_n$.
\end{lem}

\begin{proof}
By Lemma \ref{l:phi-1var}\eqref{phi-1var-1}, we have 
\begin{align*}
 \omega(\Delta \mathtt{H}) \varphi_n^\natural
 & = \sum_{i=0}^n (-1)^i \binom{n+2}{i+1} (2n-2i+2) \phi_{n-i} \otimes \phi_i \\
 & + \sum_{i=0}^n (-1)^i \binom{n+2}{i+1} (2i+2) \phi_{n-i} \otimes \phi_i \\
 & = (2n+4) \varphi_n^\natural, \\
 \omega(\Delta \mathtt{L}) \varphi_n^\natural
 & = \sum_{i=0}^{n-1} (-1)^{i+1} \binom{n+2}{i+1} (n-i+1) \phi_{n-i-1} \otimes \phi_i \\
 & + \sum_{i=1}^n (-1)^{i+1} \binom{n+2}{i+1} (i+1) \phi_{n-i} \otimes \phi_{i-1} \\
 & = \sum_{i=1}^n (-1)^i \left[ \binom{n+2}{i} (n-i+2) - \binom{n+2}{i+1} (i+1) \right] \phi_{n-i} \otimes \phi_{i-1} \\
 & = 0.
\end{align*}
This completes the proof.
\end{proof}

\begin{lem}
\label{l:pr(Tvarphi^sharp)}
\begin{enumerate}
\item
\label{pr(Tvarphi^sharp)-1}
We have
\[
 \pr_n(T \varphi_n^\sharp) = 2^{n+1} (-\sqrt{-1})^n \cdot 
 \frac{n! (n+1)!^2}{(2n+2)!} \cdot \varphi_n^\natural.
\]
\item
\label{pr(Tvarphi^sharp)-2}
We have
\[
 \langle \pr_n(T \varphi_n^\sharp), \pr_n(T \varphi_n^\sharp) \rangle' = 2^{2n-4} \pi^{-2} t^{-4} \cdot \frac{n! (n+1)!^2 (n+2)!}{(2n+2)!}.
\]
\end{enumerate}
\end{lem}

\begin{proof}
By \eqref{eq:Tphi-sharp} and Lemma \ref{l:phi-1var}\eqref{phi-1var-1}, we have
\[
 \omega(\Delta \mathtt{H}) T \varphi_n^{\sharp} = (2n+4) T \varphi_n^{\sharp}, 
\]
so that
\[
 T \varphi_n^{\sharp} \in \bigoplus_{m=0}^n \cS(\C^2)^\natural_m.
\]
Hence, by Lemma \ref{l:schwartz-lowest-wt}, we have 
\[
 \pr_n(T \varphi_n^{\sharp}) = C \varphi_n^\natural
\]
for some $C \in \C$.
On the other hand, by Lemma \ref{l:phi-1var}\eqref{phi-1var-2}, we have
\begin{align*}
 \langle T \varphi_n^\sharp, \varphi_n^\natural \rangle'
 & = 2 (-\sqrt{-1})^n n! \sum_{i=0}^n \sum_{j=0}^n (-1)^{i+j} \binom{n+2}{j+1}
 \langle \phi_{n-i}, \phi_{n-j} \rangle' \langle \phi_i, \phi_j \rangle' \\
 & = 2 (-\sqrt{-1})^n n! \sum_{i=0}^n \binom{n+2}{i+1} \cdot 4 \pi^2 \kappa_0^{-8} (n-i+1) (i+1) \\
 & = 2 (-\sqrt{-1})^n n! \cdot 4 \pi^2 \kappa_0^{-8} (n+2)(n+1)
 \sum_{i=0}^n \binom{n}{i} \\
 & = 2 (-\sqrt{-1})^n n! \cdot 4 \pi^2 \kappa_0^{-8} (n+2)(n+1) \cdot 2^n
\end{align*}
and
\begin{align*}
 \langle \varphi_n^\natural, \varphi_n^\natural \rangle'
 & = \sum_{i=0}^n \sum_{j=0}^n (-1)^{i+j} \binom{n+2}{i+1} \binom{n+2}{j+1}
 \langle \phi_{n-i}, \phi_{n-j} \rangle' \langle \phi_i, \phi_j \rangle' \\
 & = \sum_{i=0}^n \binom{n+2}{i+1}^2 \cdot 4 \pi^2 \kappa_0^{-8} (n-i+1) (i+1) \\
 & = 4 \pi^2 \kappa_0^{-8} (n+2)(n+1) \sum_{i=0}^n \binom{n+2}{i+1} \binom{n}{i} \\
 & = 4 \pi^2 \kappa_0^{-8} (n+2)(n+1) \cdot \binom{2n+2}{n+1}.
\end{align*}
Here the last equality follows by comparing the coefficients of $X^{n+1} Y^{n+1}$ on both sides of 
\[
 (X+Y)^{n+2}(X+Y)^n = (X+Y)^{2n+2}.
\]
Hence we have
\[
 C = \frac{\langle \pr_n(T \varphi_n^\sharp), \varphi_n^\natural \rangle'}{\langle \varphi_n^\natural, \varphi_n^\natural \rangle'} = \frac{\langle T \varphi_n^\sharp, \varphi_n^\natural \rangle'}{\langle \varphi_n^\natural, \varphi_n^\natural \rangle'} = 2^{n+1} (-\sqrt{-1})^n n! \cdot \binom{2n+2}{n+1}^{-1}
\]
and
\[
 \langle \pr_n(T \varphi_n^\sharp), \pr_n(T \varphi_n^\sharp) \rangle' = |C|^2 \cdot \langle \varphi_n^\natural, \varphi_n^\natural \rangle' = 2^{2n+4} \pi^2 \kappa_0^{-8} n! (n+2)! \cdot \binom{2n+2}{n+1}^{-1}.
\]
This completes the proof.
\end{proof}

Finally, we obtain the following.

\begin{prop}
\label{p:rallis-real}
We have
\[
 Z(\varphi,\Psi) = 2^{2n-4} \pi^{-3} |u| \cdot \frac{n! (n+1)!^2 (n+2)!}{(2n+3)!}.
\]
\end{prop}

\begin{proof}
Take $t=1$.
Then by Lemma \ref{l:rallis-real-SL2} and \eqref{eq:repka}, we have
\begin{align*}
 4 \pi^2 \cdot Z(\varphi, \Psi) & = \int_{G'} \langle \omega_{\psi}(\iota_0(g')) \varphi^\sharp_n, \varphi^\sharp_n \rangle \Psi(g') \, dg' \\
 & = \int_{G'} \langle T \omega_{\psi}(\iota_0(g')) \varphi^\sharp_n, T \varphi^\sharp_n \rangle \Psi(g') \, dg' \\
 & = \int_{G'} \langle \omega_{\psi}(\iota_0(g')) T \varphi^\sharp_n, T \varphi^\sharp_n \rangle \Psi(g') \, dg' \\
 & = \int_{G'} \langle \omega_{\psi}(\iota_0(g')) \pr_n(T \varphi^\sharp_n), \pr_n(T \varphi^\sharp_n) \rangle \Psi(g') \, dg'.
\end{align*}
Moreover, by Lemmas \ref{l:schwartz-lowest-wt} and \ref{l:pr(Tvarphi^sharp)}\eqref{pr(Tvarphi^sharp)-1}, we have
\[
 \langle \omega_{\psi}(\iota_0(g')) \pr_n(T \varphi^\sharp_n), \pr_n(T \varphi^\sharp_n) \rangle = \langle \pr_n(T \varphi^\sharp_n), \pr_n(T \varphi^\sharp_n) \rangle \cdot \overline{\Psi(g')}
\]
for $g' \in G'$, so that
\[
 4 \pi^2 \cdot Z(\varphi, \Psi) = \langle \pr_n(T \varphi^\sharp_n), \pr_n(T \varphi^\sharp_n) \rangle \cdot \int_{G'} |\Psi(g')|^2 \, dg'.
\]
Hence, noting that
\begin{align*}
 \int_{G'} |\Psi(g')|^2 \, dg'
 & = 4 \pi \int_0^\infty \left| \Psi\! \mat{e^x}{0}{0}{e^{-x}} \right|^2 \sinh(2x) \, dx \\
 & = 4 \pi \int_0^\infty \cosh(x)^{-4n-8} \sinh(2x) \, dx \\
 & = \frac{4\pi}{2n+3}
\end{align*}
as in \cite[\S 12]{ichino-ikeda}, the assertion follows from Lemma \ref{l:pr(Tvarphi^sharp)}\eqref{pr(Tvarphi^sharp)-2}.
\end{proof}

\section{Proof of Theorems \ref{t:period-L-value-1} and \ref{t:period-L-value-2}}
\label{s:proof}

By Lemmas \ref{l:adelic_period} and \ref{l:kill-p-depletion}, we have
\[
 P(\delta^n f^\flat, \tilde{\bff}_\alpha, \mu')
 = \varepsilon_{\fin}(\pi) \cdot C_p \cdot  P(\delta^n f, \bff, \mu').
\]
This reduces the proof of Theorem \ref{t:period-L-value-2} to Theorem \ref{t:period-L-value-1}.

It remains to prove Theorem \ref{t:period-L-value-1}.
We first recall the formula for $\langle f_0^\mu, f_0^\mu \rangle$.

\begin{prop}
\label{p:petersson-norm-f}
We have
\[
 \langle f_0^\mu, f_0^\mu \rangle = 2^{-4n-7} \pi^{-2n-4} \cdot \zeta(2)^{-1} \cdot C_{n,\Sigma} \cdot L_{\fin}(1, \pi, \Ad), 
\]
where
\[
 C_{n,\Sigma} = (2n+3)! \cdot \prod_{q \in \Sigma} \frac{q}{q+1}.
\]
\end{prop}

\begin{proof}
We have
\[
 \langle f_0^\mu, f_0^\mu \rangle = \int_{\A^\times \GL_2(F) \backslash \GL_2(\A)} |f_0(g)|^2 \, dg,
\]
where $f_0$ is the automorphic form on $\GL_2(\A)$ as in \S \ref{ss:adelic_period} and $dg$ is the Tamagawa measure on $\A^\times \backslash \GL_2(\A)$.
This inner product (but with respect to the standard measure, which is $(2\pi)^{-1} \zeta(2) \, dg$ by \cite[Lemma 6.1.1]{periods1}) is computed in \cite[Proposition 6.3.1]{periods1}.
This completes the proof.
\end{proof}

By Lemmas \ref{l:adelic_period} and \ref{l:seesaw-explicit}, we have
\begin{align*}
 P(\delta^n f, \bff, \mu') & = (-1)^n 2^{-2n-2} \pi^{-n} \cdot |D_E|^{-(n+1)/2} \cdot \cP^{\bG}(\theta_{\varphi}(\chi)^\std, f_0^{\mu'}) \\
 & = (-1)^n 2^{-2n-2} \pi^{-n} \cdot |D_E|^{-(n+1)/2} \cdot C_K \cdot \cP^H(\theta_\varphi(f_0^\mu), \chi).
\end{align*}
Hence, by Propositions \ref{p:complex_conjugation}, \ref{p:wald-explicit}, \ref{p:rallis-explicit}, \ref{p:petersson-norm-f}, we have
\begin{align*}
 & P(\delta^n f, \bff, \mu')^2 \\
 & = 2^{-4n-4} \pi^{-2n} \cdot |D_E|^{-n-1} \cdot C_K^2 \cdot \cP^H(\theta_\varphi(f_0^\mu), \chi)^2 \\
 & = (-1)^{n+1} 2^{-4n-4} \pi^{-2n} \cdot |D_E|^{-n-1} \cdot C_K^2 \cdot C_{\chi,\mu} \cdot |\cP^H(\theta_\varphi(f_0^\mu), \chi)|^2 \\
 & = (-1)^{n+1} 2^{-4n-4} \pi^{-2n+1} \cdot |D_K|^{1/2} \cdot |D_E|^{-n-1} \cdot \zeta(2) \\
 & \quad \times C_{\chi,\mu} \cdot C_{\Sigma^+,\Sigma^-} \cdot \frac{L_{\fin}(\frac{1}{2}, \pi_K \times \chi)}{L_{\fin}(1, \pi, \Ad)} \cdot \langle{\theta_\varphi(f_0^\mu),\theta_\varphi(f_0^\mu)}\rangle \\
 & = (-1)^{n+1} 2^{-2n-8} \pi^{-2n} \cdot |D_K|^{1/2} \cdot |D_E|^{-n-1/2} \rho_E^{-2} \cdot \zeta(2)^{-1} \\
 & \quad \times C_{\chi,\mu} \cdot C_{\Sigma^+,\Sigma^-} \cdot C_{n,\Sigma^+,\Sigma^-} \cdot \frac{L_{\fin}(\frac{1}{2}, \pi_K \times \chi)}{L_{\fin}(1, \pi, \Ad)} \cdot L_{\fin}(\tfrac{1}{2}, \pi_E \times \mu) \cdot \langle f_0^\mu, f_0^\mu \rangle \\
 & = (-1)^{n+1} 2^{-6n-15} \pi^{-4n-4} \cdot |D_K|^{1/2} \cdot |D_E|^{-n-1/2} \rho_E^{-2} \cdot \zeta(2)^{-2} \\
 & \quad \times C_{\chi,\mu} \cdot C_{\Sigma^+,\Sigma^-} \cdot C_{n,\Sigma^+,\Sigma^-} \cdot C_{n,\Sigma} \cdot L_{\fin}(\tfrac{1}{2}, \pi_K \times \chi) \cdot L_{\fin}(\tfrac{1}{2}, \pi_E \times \mu).
\end{align*}
This yields the desired identity and completes the proof, noting that
\[
 L_\infty(\tfrac{1}{2}, \pi_K \times \chi) \cdot L_\infty(\tfrac{1}{2}, \pi_E \times \mu) = 2^{-4n-4} \pi^{-4n-8} \cdot n! (n+1)!^2 (n+2)!.
\]

\end{document}